\documentclass[english, hidelinks, 11pt]{article}
\usepackage{geometry}
\usepackage{setspace}
\usepackage{bm}
\usepackage{graphicx}
\graphicspath{ {.} }
\usepackage{hyperref}
\usepackage{xr}

\usepackage{amssymb, amsmath, amsthm, amsfonts, mathtools, bbm}
\usepackage{float}
\usepackage{color, xcolor}
\usepackage[dvipsnames]{xcolor}
\usepackage{url}
\usepackage{enumitem}
\usepackage{csquotes}
\usepackage{multirow}
\usepackage{siunitx}
\usepackage{placeins}
\usepackage{authblk}
\usepackage{appendix}
\usepackage{lscape}
\usepackage{booktabs}
\usepackage{stackrel}
\usepackage{xurl}

\allowdisplaybreaks[1]

\usepackage{tikz}
\usetikzlibrary{arrows.meta, decorations.pathreplacing}

\usepackage{algorithm}
\usepackage[noend]{algpseudocode}

\usepackage{xr-hyper}
\usepackage{hyperref}
\hypersetup{
    colorlinks=true,
    linkcolor=blue,
    citecolor=blue,
    filecolor=magenta,
    urlcolor=blue
}

\usepackage[round]{natbib}

\newtheorem{theorem}{Theorem}
\newtheorem{lemma}{Lemma}
\newtheorem{proposition}{Proposition}

\newtheorem{definition}{Definition}
\newtheorem{remark}{Remark}
\newtheorem{assumption}{Assumption}

\newtheorem{example}{Example}
\newtheorem{condition}{Condition}

\numberwithin{equation}{section}

\usepackage{soul}       
\usepackage{ulem}   
\DeclareMathOperator*{\argmin}{arg\,min}
\DeclareMathOperator*{\argmax}{arg\,max}

\DeclarePairedDelimiter{\ceil}{\lceil}{\rceil}
\DeclarePairedDelimiter{\floor}{\lfloor}{\rfloor}

\newtheorem{assumptionalt}{Assumption}[assumption]
\newenvironment{assumptionp}[1]{
  \renewcommand\theassumptionalt{#1}
  \assumptionalt
}{\endassumptionalt}

\newcommand{\bnorm}[1]{\bigl\lVert #1 \bigr\rVert}

\def\pmk{p_{\text{-}k}}

\def\rmk{r_{\text{-}k}}
\def\wmk{w_{\text{-}k}}
\def\Lambdamk{{\Lambda}_{\text{-}k}}
\def\Lambdamk{\Lambda_{\text{-}k}}
\def\wtLambdamk{\wt{\Lambda}_{\text{-}k}}
\def\whLambdamk{\wh{\Lambda}_{\text{-}k}}
\def\Aemk{A_{e,\text{-}k}}
\newcommand{\Ajmk}[1]{A_{#1,\text{-}k}}
\def\mat{\textnormal{\textsc{mat}}}
\def\tr{\textnormal{\text{tr}}}
\def\wt#1{\widetilde{#1}}
\def\wh#1{\widehat{#1}}

\def\C{\mathbb{C}} 
\def\E{\mathbb{E}} 
\def\L{\mathbb{L}} 
\def\M{\mathbb{M}} 
\def\N{\mathbb{N}} 
\def\P{\mathbb{P}}
\def\R{\mathbb{R}}

\def\Z{\mathbb{Z}} 

\def\cB{\mathcal{B}} 
\def\cC{\mathcal{C}} 
\def\cD{\mathcal{D}} 
\def\cE{\mathcal{E}} 
\def\cF{\mathcal{F}}
\def\cG{\mathcal{G}} 
\def\cH{{\mathcal{H}}} 
\def\cI{\mathcal{I}} 
\def\cII{\mathcal{II}} 
\def\cIII{\mathcal{III}} 
\def\cIV{\mathcal{IV}} 
\def\cJ{{\mathcal{J}}} 
\def\cK{{\mathcal{K}}} 

\def\cM{\mathcal{M}} 
\def\cN{\mathcal{N}} 
\def\cO{\mathcal{O}}

\def\cQ{\mathcal{Q}} 
\def\cS{\mathcal{S}}
\def\cT{\mathcal{T}} 
\def\cU{\mathcal{U}} 
 
\def\cV{\mathcal{V}} 
\def\cX{\mathcal{X}} 
\def\cY{\mathcal{Y}}

\def\bi{\boldsymbol{i}}

\def\bu{\mathbf{u}}

\def\bg{\mathbf{g}}

 \def\bxi{\boldsymbol{\xi}}

\def\thh{\widehat{\theta}} 
 
\def\thj{\theta_j}

\renewcommand{\l}{\left}
\renewcommand{\r}{\right}

\def\wh{\widehat}
\def\wt{\widetilde}

\def\d{\mathrm{d}}

\def\trans{\intercal}
 
\def\F{{F}}

\def\cov{\mathbb{C}\mathrm{ov}} 
\def\var{\mathbb{V}\mathrm{ar}}

\def\vec{\mathrm{vec}} 
\def\Vech{\mathrm{vech}} 
\def\col{\mathrm{col}}

\def\rank{\mathrm{rank}}

\def\diag{\mathrm{diag}}

\def\lc{\ell^{\circ}} 
\def\al{a_{\ell}} 
\def\bl{b_{\ell}} 
\def\sll{s_{\ell}} 
\def\el{e_{\ell}} 
\def\alc{a_{\ell^{\circ}}} 
\def\blc{b_{\ell^{\circ}}}
\def\delc{\Delta^{\circ}} 
\def\gc{\mathbf{g}^{\circ}}

\allowdisplaybreaks

\usepackage[makeroom]{cancel}

\begin{document}

\begin{titlepage}

\title{Detection and Mode-Identification of Multiple Change Points in Tensor Factor Models}

\author[1]{Yuqi Zhang\thanks{Equal contribution. Email: \url{yuqi.zhang@bristol.ac.uk}. Supported by Engineering and Physical Sciences Research Council for Doctoral Training in Computational Statistics and Data Science (Compass, EP/S023569/1).}}
\author[1]{Zetai Cen\thanks{Equal contribution; corresponding author. Email: \url{zetai.cen@bristol.ac.uk}. Supported by Engineering and Physical Sciences Research Council (EP/Z531327/1).}}
\author[1]{Haeran Cho\thanks{Email: \url{haeran.cho@bristol.ac.uk}. Supported by Engineering and Physical Sciences Research Council (EP/Z531327/1).}}

\affil[1]{School of Mathematics, University of Bristol}

\date{}

\maketitle

\begin{abstract}
We study the problem of modeling high-dimensional tensor-valued time series under a Tucker decomposition-based factor model while permitting the presence of multiple structural change points. 
Such data call for a method that (i)~handles high dimensionality and complex dependence while achieving consistency in multiple change points and further, (ii)~addresses a unique challenge posed by the multi-dimensional array setting, of identifying the modes associated with each change while the shifts in different modes interact with one another.
To this end, we propose a change point detection procedure utilizing the low-rank structure of the data for statistical and computational efficiency. 
In addition, we formalize the \textit{mode-identifiability} of each change and propose an algorithm for detecting the modes at which the data undergo a mode-identifiable shift.
We establish the consistency of both change point detection and mode-identification methods under a weak moment condition, and demonstrate their good performance on simulated datasets where, in particular, it is shown that the mode-identification step can improve the post-segmentation estimation of the mode-wise loading space. We also analyze the datasets on New York City taxi usage and Fama--French portfolio returns using the proposed suite of methods.
\end{abstract}

\bigskip


\end{titlepage}

\setcounter{page}{2}

\maketitle

\enlargethispage{\baselineskip}

\tableofcontents

\clearpage

\section{Introduction}
\label{sec: introduction}

High-dimensional time series are increasingly observed in the form of tensors, where each observation is a multi-dimensional array organized along several modes. Examples include dynamic transport and trade networks across product groups, multi-dimensional finance panels indexed by firms, characteristics and time, spatiotemporal environmental fields, and biomedical imaging data, to name a few \citep{li2020tensor, chen2022factor, sedighin2024tensor, babii2025tensor}, where tensor formatting not only preserves the innate interpretation of the data, but also can underline inherent structure therein.

Tensor factor models provide a natural framework for analyzing such data by imposing a low-rank representation.
Specifically, each tensor-valued observation is assumed to be driven by a handful of latent factors, as characterized by some mode-specific loading matrices, thereby achieving dimension reduction while preserving the role of each mode \citep{wang2019factor,chen2022factor,han2024cp}. 
Based on the two popular tensor decompositions \citep{kolda2009tensor}, Tucker and CanDecomp/PARAFAC (CP) factor models have frequently been studied, where the latter is a special case of the former, and the recent literature has seen a wide range of developments from determining the ranks \citep{Hanetal2022}, to estimating the mode-wise loading spaces and latent factors \citep{ChenLam2024, han2024tensor, Barigozzietal2025_robust, barigozzi2026statistical} and handling missing observations by imputation \citep{CenLam2025}. 

Most of these contributions are made supposing that the factor structure is time-invariant, departure from such an assumption is often observed in many applications.
On the other hand, macroeconomic and financial dependence patterns may shift after major policy changes, crises or market re-organizations, leading to structural changes in how certain factors drive the observed time series \citep{stock2009forecasting, breitung2011testing, chen2014detecting}. In these scenarios, ignoring such changes could undermine parameter estimation and hinder the interpretation of the evolution in the latent structures. 

Change point detection has studied extensively in vector time series factor modeling.
Existing work includes methods that test for structural instability in factor loadings \citep{chen2014detecting,yamamoto2015testing,baltagi2017identification} and estimate the break dates \citep{baltagi2017identification,Duanetal2023,Duanetal2025, Baietal2024}, as well as multiple change point detection algorithms \citep{chan2014group, barigozzi2018simultaneous, Baltagietal2021, li2023detection, Baietal2024, barigozzi2025moving}.
By contrast, the problem of detecting change points in matrix- or, more generally, tensor-valued time series remains largely unexplored. For matrix factor models with potential changes, \citet{Heetal2024AOS} 
proposed an online monitoring scheme, while \citet{peng2025detection, peng2025estimation} developed offline procedures. 
Beyond the matrix settings, \citet{huang2022multiple} and \citet{pevsta2025tensor} study the problem of detecting changes in the mean of a tensor-variate sequence, \citet{zhang2024change} and \citet{wang2025multilayer} investigate the problem in dynamic networks, 
and \citet{anastasiou2025tensor} propose an approach to detect changes in cryptocurrency networks by leveraging CP decomposition.
We also mention \citet{zhang2026spectral} where a low-rank decomposition of an order-three tensor formed by cumulative sums over time and frequency, is utilized for high-dimensional time series segmentation. 

Despite the growing interests in tensor data with change points, to the best of our knowledge, a formal treatment towards offline detection in tensor factor models is still absent.
Moreover, due to the multi-dimensional array nature of the data, changes in a tensor factor model may affect one or several modes of the tensor data and manifest differently across the modes, and these changes in different modes may interact with one another.
Vectorizing the data would annihilate this property, thus obscuring the relations between any detected change and the tensor modes, which calls for a method that preserves and exploits the tensor structure.

Recognizing these new challenges, we develop a suite of methods for modeling tensor time series under a factor model with multiple changes.
Firstly, we propose an algorithm for detecting multiple change points in the tensor factor model.
It extends the idea which has frequently been utilized in the vector factor modeling literature, that different types of changes in the factor structure are observationally equivalent to the changes in the covariance of so-called pseudo-factor of a fixed dimension \citep{han2015tests, Baltagietal2021, Duanetal2023}, to the tensor setting.
This is particularly beneficial for the tensor-valued data that quickly becomes high-dimensional: For a series of order-$K$ tensors $\cX_t \in \R^{p_1 \times \dots \times p_K}$, our proposed method performs multiple change point detection on the objects of dimension $\sum_{k = 1}^K r_k(r_k + 1)/2$, where $r_k$'s denote the dimensions of the order-$K$ pseudo-factor tensor which are typically much smaller than $p_k$. 

\begin{figure}[h!t!b!]
\centering
\includegraphics[width = \linewidth]{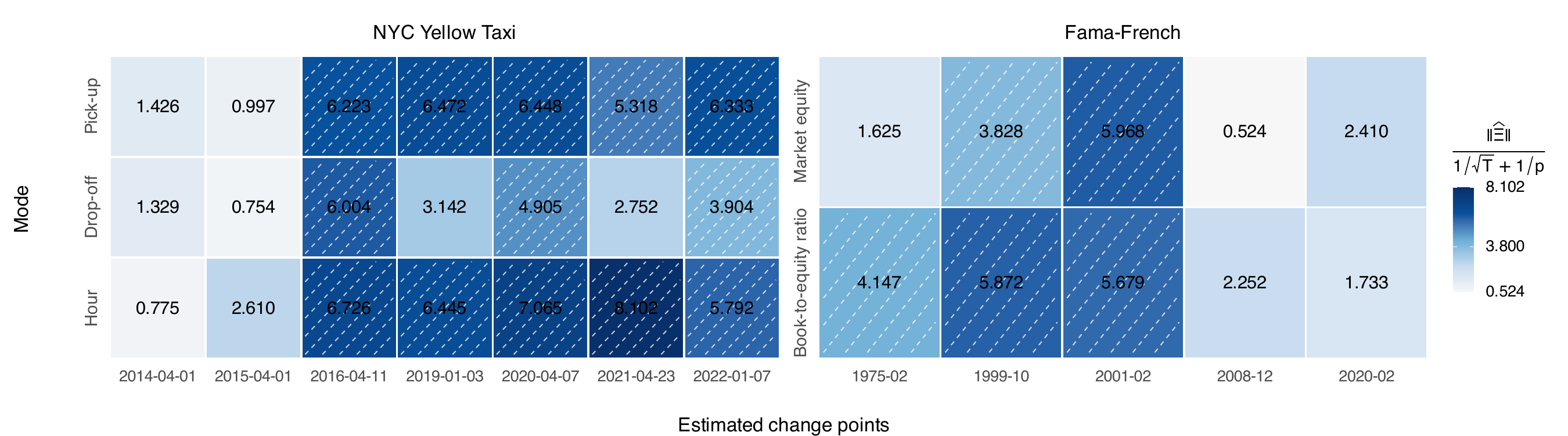}
\caption{Heatmaps of the scaled mode-identification statistics 
for New York City Yellow Taxi data (left, with factor numbers $(r_1, r_2, r_3) = (4, 2, 2)$) and equal-weighted Fama--French portfolio returns (right, $(r_1, r_2) = (2, 2)$).
The $x$-axis marks the dates of the detected change points, and the $y$-axis marks the modes. 
Tiles with darker shades indicate stronger evidence for the association between the corresponding modes and the changes, and those with light diagonal hatching have the mode-identification statistic exceed the threshold.
}
\label{fig: heat_combined}
\end{figure}

We also formulate and study the problem of mode-identification for tensor factor models with multiple change points, a first contribution of its kind.
To this end, we formally characterize the \textit{mode-identifiability} of a change in Definition~\ref{def: classification_change} and Theorem~\ref{thm: identification_change}, and propose a post-detection mode-identification procedure that outputs the tensor modes undergoing identifiable shifts at each detected change point once the change points are detected.
One direct downstream benefit of mode-identification is that it enables \textit{mode-informed} loading space estimation which brings demonstrable numerical benefits over estimating the loading spaces segment by segment. 
Besides, it offers further insights that aid the interpretation of the detected change points: As an illustration, we analyze the New York City Yellow Taxi trip records where on each day, we observe an order-3 tensor with pick-up and drop-off locations and the hour of the day as the three modes. 
The left panel of Figure~\ref{fig: heat_combined} shows that earlier changes are deemed mode-unidentifiable, which may be due to the overall decline in the volume of trips taken via traditional service with the emergence of app-based services; on the other hand, Covid-19 lockdown (late March/April 2020) and the subsequent easing off of the restrictions (April 2021) are associated with changes to the overall taxi usage patterns along the different modes.
In the right panel, among the change points detected from the matrix time series formed by Fama--French portfolio returns, the last two associated with the financial crisis (December 2008) and the pandemic (February 2020), are characterized as the shifts in the overall variability rather than the changes along specific modes. See Section~\ref{sec: real} for the complete descriptions. 

Theoretically, we establish the consistency of the combined methodology, both in estimating the total number and locations of the change points and in selecting the subset of modes undergoing identifiable changes at each change point, under a weak moment condition and permitting temporal and cross-sectional dependence. 

The rest of this paper is organized as follows. Section~\ref{sec:model} introduces the tensor factor model with multiple change points and defines the notion of mode-identifiability of a change. 
In Section~\ref{sec: method}, we describe the methods for multiple change point detection, post-detection mode-identification, 
and a corresponding, mode-informed method for loading estimation.
Section~\ref{sec: theorem} presents the assumptions and establishes the asymptotic consistency of the proposed methods. 
Section~\ref{sec: simulation} 
provides a detailed discussion on the selection of tuning parameters, and
reports the simulation results.
In Section~\ref{sec: real}, we analyze the datasets on New York City taxi usage and Fama–French portfolio returns.
Section~\ref{sec: conclusion} concludes the paper, and all the proofs, additional discussions and simulation results are deferred to the supplement.

\paragraph{Notations.} 
Given a positive integer $m$, we write $[m]=\{1,2,\dots,m\}$. 
For sequences $\{a_n\}$ and $\{b_n\}$, We use $a_n \lesssim b_n$ to denote $a_n =\cO(b_n)$, $a_n\gtrsim b_n$ to denote $b_n=\cO(a_n)$, and $a_n \asymp b_n$ to denote $a_n = \cO(b_n)$ and $b_n = \cO(a_n)$; the stochastic boundness is denoted by $\cO_P(\cdot)$.
For a vector $u = (u_1, \ldots, u_p)^\trans \in \R^p$, let $|u|_2 = (\sum_{i=1}^p u_i^2)^{1/2}$ denote the Euclidean norm. 
For a matrix $A = [a_{ij}]\in\R^{m\times n}$, let $\|A\|$ denote the operator norm, $\|A\|_\F = (\sum_{i=1}^m\sum_{j=1}^n a_{ij}^2)^{1/2}$ the Frobenius norm, $\Vert A \Vert_1 = \max_{1\le j\le n}\sum_{i=1}^m|a_{ij}|$ and $\Vert A \Vert_\infty = \max_{1\le i\le m}\sum_{j=1}^n|a_{ij}|$ denote the induced $\ell_1$- and $\ell_{\infty}$-norm, respectively. For a matrix $A$, let $\col(A)$ denote its column space.
Also, $\vec(A)$ denotes the vector in $\R^{mn}$ obtained by stacking the columns of $A$ in column-major order and for a symmetric matrix $A$, let $\Vech(A)$ denote the vector obtained by stacking the lower triangular entries of $A$ (including the diagonal) in column-major order.
For an order-$K$ tensor $\cX = (X_{i_1,\ldots,i_K}) \in \mathbb{R}^{p_1\times \cdots\times p_K}$, with $p = \prod_{k=1}^K p_k$ and $\pmk =p/p_k$, 
we denote by $\mat_k(\cX) \in \R^{p_k\times \pmk}$ its \textit{mode-$k$ unfolding/matricization}, and by $\cX \times_k A$ the \textit{mode-$k$ product} of a tensor $\cX$ with a matrix $A$, defined by $\mat_k(\cX \times_k A) = A\, \mat_k(\cX)$.
Finally, we use $\otimes$ to represent the Kronecker product, and $\circ$ the Hadamard product. 

\section{Model}
\label{sec:model}

\subsection{Tensor factor model with multiple change points}\label{subsec: model}

Consider a mean-zero time series of order-$K$ tensors $\cX_t\in\R^{p_1\times \dots \times p_K}, \, t\in[T]$, with $K \ge 1$, which follows a Tucker tensor factor model with $q$ change points such that
\begin{equation}
\label{eqn: tfm_change}
\cX_t = \cC_t + \cE_t, \text{ \ where \ } \cC_t = \left\{
    \begin{array}{ll}
	\cF_t \times_{k=1}^K \Lambda_{1,k} & \text{for $\theta_0 + 1 = 1 \le t\le \theta_1$,} \\
    \cF_t \times_{k=1}^K \Lambda_{2,k} & \text{for $\theta_1 + 1 \le t\le \theta_2$,} \\
    \vdots \\
    \cF_t \times_{k=1}^K \Lambda_{q + 1,k} & \text{for $\theta_q + 1 \le t\le \theta_{q+1} = T$,}
    \end{array}
    \right.
\end{equation}
with 
$\Theta=\{\theta_1, \ldots, \theta_{q}\} \subseteq [T]$ denoting the set of change point locations.
Here, $\cF_t\in \R^{r_1\times \dots\times r_K}$ denotes the core factor tensor with its dimensions $r_k$'s fixed regardless of growing $p_k$, $\cC_t$ denotes the common component driven by the core factor, and $\cE_t$ is the idiosyncratic component unaccounted by $\cF_t$.
We denote by $\Lambda_{j,k}\in \R^{p_k\times r_k}$ the mode-$k$ factor loading matrix in the $j$-th segment such that at each change point $\theta_j$, at least one of the loading matrices undergoes a shift.
We permit the numbers of factors to change so that the column rank of each $\Lambda_{j, k}$, which effectively represents the mode-$k$ factor number in the $j$-th segment, may be strictly smaller than $r_k$; later we introduce an alternative representation of the model~\eqref{eqn: tfm_change} with which the meaning of $r_k$ is made clearer. 
Throughout, we make the following condition on $\{\cF_t\}_{t = 1}^T$:
\begin{condition}
\label{cond:factor}
At any $t\in[T]$, the elements of the core factor $\cF_t$ are uncorrelated with each other and have zero mean and unit variance, i.e.\ $\cov(\vec(\cF_t)) = I_r$. 
\end{condition}
For each given $j$, while the space spanned by the columns of each $\Lambda_{j, k}$, i.e.\ $\col(\Lambda_{j, k})$, is uniquely defined, the loading matrices and the core factor tensor are not identifiable regardless of $K$. 
From this perspective, Condition~\ref{cond:factor} is akin to the assumption commonly found in the literature on vector factor modeling \citep{bai2003inferential}, if $\cF_t$ admits a separable covariance structure \citep{hoff2011separable}, i.e.\ $\cov(\vec(\cF_t)) = \otimes_{k = K}^1 \Gamma^{(k)}_F$ for some positive definite $\Gamma^{(k)}_F \in \mathbb{R}^{r_k \times r_k}$, whereby the rotation and normalization therein can be absorbed into the factor loadings.
Later we make Assumption~\ref{assum: core_factor} on the core factor time series which, while fulfilling Condition~\ref{cond:factor}, further characterizes the degree of serial dependence.
Even under Condition~\ref{cond:factor}, we have $\Lambda_{j, k}$ identifiable up to an orthogonal transform only, see Lemma~\ref{lemma: identifiable} in the supplement.
Moreover, changes in the loading matrices may equivalently be represented as the changes in the covariance of $\cF_t$ and vice versa:
\begin{example}
\label{ex:equiv}
Suppose that we observe $\cX_t = \cG_t \times_{k = 1}^K \Lambda_k + \cE_t, \, t \in [T]$, where $\{\cG_t\}_{t = 1}^T$ undergoes a single change point in its covariance, as
\begin{align*}
\cG_t = \left\{\begin{array}{l}
\cF_t \\ 
\cF_t \times_{k = 1}^K A_k \\ 
\end{array} \right.
\text{ \ whereby \ }
\cov(\vec(\cG_t)) = \left\{\begin{array}{ll}
I_r = I_{r_K} \otimes \cdots \otimes I_{r_1} 
& \text{for \ } 1 \le t \le \theta_1, \\
\Gamma_G^{(K)} \otimes \cdots \otimes \Gamma_G^{(1)} & \text{for \ } \theta_1 +  1 \le t \le T,
\end{array} \right.
\end{align*}
with $\cF_t$ satisfying Condition~\ref{cond:factor}, and $\Gamma_G^{(k)} = A_k A_k^\trans \ne I_{r_k}$ for at least one $k \in [K]$.
Then, we may re-write this time series as in~\eqref{eqn: tfm_change}, with $\cF_t$ as the core factor and segment-wise mode-$k$ loading matrices $\Lambda_{1, k} = \Lambda_k$ and $\Lambda_{2, k} = \Lambda_k A_k$. 
\end{example}
Example~\ref{ex:equiv} demonstrates that model~\eqref{eqn: tfm_change} accommodates different types of changes arising from the Tucker tensor factor model; 
see also Section~\ref{subsec: change_scenario} for further example scenarios.

In fact, there exists a global mode-$k$ loading matrix $\Lambda_k \in\R^{p_k\times r_k}$ 
such that for each $\Lambda_{j,k}$, there is a transformation matrix $A_{j,k} \in\R^{r_k\times r_k}$, not necessarily of full rank, with which we can write $\Lambda_{j,k} = \Lambda_k A_{j,k}$; 
we set $r_k$ to be the smallest integer so that such~$\Lambda_k$ is of full column rank. 
Then, model~\eqref{eqn: tfm_change} is equivalently re-written as
\begin{equation}
\label{eqn: tfm_change_rewrite}
\begin{split}
    \cX_t 
    &=
    \l(\sum_{j=1}^{q+1} \cF_t \times_{k=1}^K A_{j,k} \cdot \mathbb{I}_{\{\theta_{j-1} <t\le \theta_j\}} \r) \times_{k=1}^K \Lambda_k + \cE_t
    =: \cG_t \times_{k=1}^K \Lambda_k + \cE_t,
\end{split}
\end{equation}
where $\cG_t$ is the pseudo-factor tensor of the same dimensions as $\cF_t$, and is piecewise stationary with the change points in its covariance at $\theta_j, \, j \in [q]$.
To see this, at any $\theta_j$, there exists some $k\in[K]$ whereby the corresponding mode-$k$ covariance of $\cG_t$ undergoes a change, provided that
\begin{align}
& 
\E\Big\{ \mat_k(\cG_{\theta_j}) \mat_k(\cG_{\theta_j})^\trans \Big\} =
A_{j,k} \E\Big\{ \mat_k(\cF_{\theta_j}) A_{j, \text{-}k}^\trans A_{j, \text{-}k} \mat_k(\cF_{\theta_j})^\trans \Big\} A_{j,k}^\trans \notag \\
=& \,
\tr\Big( A_{j, \text{-}k}^\trans A_{j, \text{-}k} \Big) \cdot A_{j,k} A_{j,k}^\trans \neq
\tr\Big( A_{j + 1, \text{-}k}^\trans A_{j + 1, \text{-}k} \Big) \cdot A_{j + 1,k} A_{j + 1,k}^\trans \notag \\
=& \,
\E\Big\{ \mat_k(\cG_{\theta_j+1}) \mat_k(\cG_{\theta_j+1})^\trans \Big\},
\label{eqn: cov_G_signal}
\end{align}
with $A_{j, \text{-}k} := \otimes_{i = K, i \ne k}^1 A_{j, i}$, where the second and last equalities hold under Condition~\ref{cond:factor} (see Lemma~\ref{lemma: FMF_expectation} in the supplement). 
The re-formulation in~\eqref{eqn: tfm_change_rewrite} extends the idea popularly exploited for change point detection under vector factor models 
to higher-order tensor time series, that structural changes in the factor loading can equivalently be represented as changes in the (mode-wise) covariance matrix of the pseudo-factor.
Thus motivated, we define the size of change at~$\theta_j$.
\begin{definition}[Size of change]\label{def: jump}
Let
$
\Gamma_{G,(j)}^{(k)} = \E\{ \mat_k(\cG_{\theta_j}) \mat_k(\cG_{\theta_j})^\trans \}
$
for each $j\in[q+1]$ and $k\in[K]$. 
Then the size of change in mode-$k$ at $\theta_j$ is defined as $\omega_j^{(k)} =\|\Omega_j^{(k)}\|_F$ with $\Omega_j^{(k)}=\Gamma_{G,(j+1)}^{(k)}-\Gamma_{G,(j)}^{(k)}$.
With these, we define the size of change at $\theta_j$ as $\omega_j = \big\{\sum_{k=1}^K \big(\omega_j^{(k)}\big)^2\big\}^{1/2}$.
\end{definition}

\subsection{Mode-identifiability of each change point}
\label{sec:mode}


Inspecting~\eqref{eqn: cov_G_signal} reveals that the size of change $\omega^{(k)}_j$ in Definition~\ref{def: jump}, depends not only on the change in the mode-$k$ transformation matrix, but also those corresponding to other modes.
Accordingly, the tensor setting with $K > 1$ brings complications previously unobserved in the vector factor modeling literature, and naturally raises the following questions.
\begin{enumerate}[label = (Q\arabic*)]
    \item \label{q:one} Can we identify all the modes that undergo a shift at the $j$-th change point, i.e.\ all $k \in [K]$ such that $A_{j,k} \neq A_{j + 1,k}$?
    \item \label{q:two} If multiple modes have $A_{j, k} \neq A_{j + 1, k}$ at the $j$-th change point, can these changes cancel each other out, leading to $\omega_j = 0$? 
\end{enumerate}
Although the problem of change point detection under matrix factor models has previously been studied, these questions remain unanswered:
\cite{Heetal2024AOS} focus on the situation where a change occurs in a single mode only, in developing and analyzing an online change test; \cite{peng2025detection, peng2025estimation} study the detection of the changes in the covariance of $\vec(\cG_t)$ which ignores the cross-mode interactions. 
We make a first contribution in addressing the inherently unique challenges brought by the tensor setting.
To this end, we first define 
the notion of mode-identifiability of a change.
\begin{definition}[Mode-identification of changes]\label{def: classification_change}
Consider the model in~\eqref{eqn: tfm_change_rewrite} and suppose that Condition~\ref{cond:factor} holds.
Then for some $j \in [q]$, the mode-$k$ change from $A_{j,k}$ to $A_{j + 1,k}$ is defined to be mode-unidentifiable if and only if there exists a set of matrices $\{B_\ell\}_{\ell\in[K]\setminus\{k\}}$ such that
\[
\cF_t \times_{\ell\in[K]\setminus\{k\}} A_{j + 1,\ell} \times_k A_{j + 1,k} = \cF_t \times_{\ell\in[K]\setminus\{k\}} B_\ell \times_k A_{j,k}.
\]
Otherwise the change point $\theta_j$ is mode-$k$ identifiable. A change point is referred to as mode-unidentifiable when it is mode-$k$ unidentifiable for all $k \in [K]$; otherwise it is mode-identifiable.
\end{definition}

With Definition~\ref{def: classification_change}, we can associate each change point $\theta_j$ with the subset of the tensor modes $[K]$ for which it is mode-identifiable,
in the sense that characterizing the mode-$k$ change only depends upon the relation between $A_{j,k}$ and $A_{j + 1,k}$ regardless of the transformation matrices at modes $\ell \ne k$. 
We formally justify this notion of mode-(un)identifiability with the following theorem.
\begin{theorem}[The only type of mode-unidentifiable change]\label{thm: identification_change}
Consider the model in~\eqref{eqn: tfm_change_rewrite} and suppose that Condition~\ref{cond:factor} holds.
Then for any $j \in [q]$ and $k \in [K]$, the mode-$k$ change is mode-unidentifiable for $\theta_j$ if and only if $A_{j + 1, k} = c A_{j, k} Q$ for some constant $c$ and orthogonal matrix~$Q$.
\end{theorem}

Notice that under Condition~\ref{cond:factor}, the transformation matrices $A_j$ are identifiable up to orthogonal transforms (cf.\ Lemma~\ref{lemma: identifiable}).
Therefore, Theorem~\ref{thm: identification_change} establishes that the only type of mode-unidentifiable change is in the form of a scalar multiplication.
Jointly, Definition~\ref{def: classification_change} and Theorem~\ref{thm: identification_change} allow us to regard the changes across different modes individually in higher-order tensor factor models.
Back to the earlier questions, Theorem~\ref{thm: identification_change} directly addresses~\ref{q:one}.
As for~\ref{q:two}, observing~\eqref{eqn: cov_G_signal} in combination with Lemma~\ref{lemma: AA=cBB} in the supplement, the theorem shows that any mode-identifiable change is detectable (in the sense that $\omega_j > 0$) and, in fact, the cross-mode cancellation of the change occurs in the trivial scenario only: When $\theta_j$ is mode-unidentifiable with the scalar multiples on all modes canceling one another. 
For completeness, we provide a categorization of the types of mode-identifiable changes in Section~\ref{sec: anatomy_mode_id}.

In summary, model~\eqref{eqn: tfm_change} poses an interesting challenge of mode-identification which, following the detection of multiple change points, can enrich their interpretation and reveal cross-mode interactions. 
We propose a mode-identification algorithm in Section~\ref{sec: mode_identify} that accounts for the stochasticity arising from the change point detection step.
While the primary motivation for mode-identification arises when $K > 1$, we close this section with an example demonstrating that the notion of mode-identifiability is of use even for the vector factor model:
\begin{example} 
\label{ex:vector_mode_id}
When $K = 1$, the tensor factor model~\eqref{eqn: tfm_change} reduces to a vector factor model with $\cX_t \in \R^{p_1}$.
In this setting, as shown in Theorem~\ref{thm: identification_change}, the change point $\theta_j$ is mode-unidentifiable if and only if $A_{j + 1,1} = c A_{j,1} Q$ for some scalar $c$ and orthogonal matrix $Q$. 
This corresponds to the situation where the column space of the loading matrix undergoes a change characterized by scalar multiplication, which can be absorbed into the factor without altering the column space. 
Thus, a mode-unidentifiable change does not reflect a genuine shift in the loading space 
and consequently, we can benefit from pooling the data from pre- and post-change segments for its estimation even when the change is detected;
see Section~\ref{subsec: mode_informed} for further discussion.
\end{example}

\section{Methodology}
\label{sec: method}


\subsection{Change point detection}
\label{sec: detection}

Recall that under the model~\eqref{eqn: tfm_change_rewrite}, we are able to detect the change points by scanning for the changes in the covariance of the global pseudo-factor tensor $\cG_t$.
As a first step for estimating the unobservable $\cG_t$, we estimate the mode-$k$ loading matrix $\Lambda_k$ from the (scaled) mode-$k$ sample covariance matrix of $\cX_t$ for each $k\in[K]$, i.e.\ $\wh{\Gamma}_X^{(k)} = (Tp)^{-1} \sum_{t=1}^T \mat_k(\cX_t) \mat_k(\cX_t)^\trans$.
In the time series tensor factor modeling literature \citep{yu2022projected, zhang2025tucker, barigozzi2026statistical}, a popularly adopted estimation approach takes two steps: 
We first perform eigendecomposition on $\wh{\Gamma}_X^{(k)}$, and retain the $r_k$ leading eigenvectors multiplied by $\sqrt{p_k}$, as the preliminary estimator $\wt{\Lambda}_k$; this is justifiable under the orthogonality imposed on $\Lambda_k$, see Assumption~\ref{assum: loadings}.
Then, writing
\begin{equation}
\label{eqn: est_lam}
\wtLambdamk = \otimes_{j\in[K]\setminus\{k\}}\wt\Lambda_k, \quad
Y_{k,t}= \frac{1}{\pmk} \mat_{k}(\cX_t) \wtLambdamk, \text{\ and \ } 
\wh\Gamma_Y^{(k)} = \frac{1}{Tp_k} \sum_{t=1}^T Y_{k,t} Y_{k,t}^\trans,
\end{equation}
we obtain a refined estimator $\wh\Lambda_k$ as $\sqrt{p_k}$ times the $r_k$ leading eigenvectors of $\wh\Gamma_Y^{(k)}$.
The additional step of projecting $\mat_{k}(\cX_t)$ onto the column space of $\wtLambdamk$, enhances the signal-to-noise ratio and thus improves the estimation performance.
Finally, the pseudo-factor tensor is estimated as $\wh\cG_t = p^{-1} \cX_t \times_{k=1}^K \wh\Lambda_k^\trans, \, t\in[T]$.

For effectively scanning for the changes in the mode-wise covariance of $\wh{\cG}_t$, we construct a cumulative sum (CUSUM) statistic. 
For its definition and motivation, let us write
\begin{equation}\label{eqn: gamma_G}
\wh{\Gamma}_{G,a,b}^{(k)} = \frac{1}{b-a} \sum_{t=a+1}^b \mat_k(\wh\cG_t) \mat_k(\wh\cG_t)^\trans
\text{ \ and \ }
{\Gamma}_{G,a,b}^{(k)} = \frac{1}{b-a} \sum_{t=a+1}^b \E\l\{ \mat_k(\cG_t) \mat_k(\cG_t)^\trans \r\},
\end{equation}
for any $k \in [K]$ and $0 \le a < b \le T$.
We have, for each $j \in [q]$ and any integers $a, b$ such that $\theta_{j-1} \leq a <\theta_j < b \leq \theta_{j+1}$,
\begin{align*}
{\Gamma}_{G,\tau,b}^{(k)} - {\Gamma}_{G,a,\tau}^{(k)}
&=
\left\{
\begin{array}{ll}
\frac{b - \theta_j}{b - \tau} \Big( {\Gamma}_{G,(j + 1)}^{(k)} - {\Gamma}_{G,(j)}^{(k)} \Big)  & \text{for $a < \tau \le \theta_j$,}
\\
\frac{\theta_j-a}{\tau-a}
\Big( {\Gamma}_{G,(j + 1)}^{(k)} - {\Gamma}_{G,(j)}^{(k)} \Big)  & \text{for $\theta_j + 1 \le \tau \le b$.} 
\end{array}
\right.
\end{align*}
Then, provided that $\omega_j^{(k)} \ne 0$ (see Definition~\ref{def: jump}), any matrix norm of ${\Gamma}_{G,\tau,b}^{(k)} - {\Gamma}_{G,a,\tau}^{(k)}$ is maximized at $\tau = \theta_j$, which justifies the scanning of its estimated counterpart for change point detection.
Instead of detecting the change points separately for each mode, we advocate an approach that searches for the changes in all $K$ modes simultaneously. 
The latter enhances the detection power by pooling the information across multiple modes, while avoiding any necessity for post-processing as a change in mode $k$, noting that any $\theta_j$ may manifest itself with $\omega^{(k)}_j \ne 0$ for multiple modes $k \in [K]$, see~\eqref{eqn: cov_G_signal} and the discussions in Section~\ref{sec:mode}. 
To this end, we consider the following CUSUM statistics over appropriately selected intervals $(a, b], \, 0 \le a < b \le T$,
\begin{equation} 
\begin{split}
\cM_{a, \tau, b} &= \sqrt{\frac{(\tau - a) (b - \tau)}{b - a}} \begin{bmatrix}
\Vech\left(\wh{\Gamma}_{G,\tau,b}^{(1)} - \wh{\Gamma}_{G,a,\tau}^{(1)}\right) \\
\vdots \\
\Vech\left(\wh{\Gamma}_{G,\tau,b}^{(K)} - \wh{\Gamma}_{G,a,\tau}^{(K)} \right)
\end{bmatrix} \in \R^d, \, \text{ \ for \ } a < \tau < b,
\end{split} 
\nonumber
\end{equation}
with $d=\sum_{k=1}^K r_k(r_k+1)/2$. Then finally, with a (possibly) location-dependent, positive definite weight matrix $W_{a, \tau, b} \in \R^{d \times d}$, we define the detector statistic
\begin{equation}
\cT_{a, \tau, b}=\big\vert \cM_{a, \tau, b}^\trans W_{a, \tau, b}^{-1} \cM_{a, \tau, b}\big\vert^{1/2}. \label{eqn: test_stat_agg_std}
\end{equation}
We propose to scan the data for multiple changes using $\cT_{a, \tau, b}$, looking for local maximizers of the detector statistic over strategically selected intervals $(a, b]$. 
Specifically, we embed this statistic in a multiscale search based on seeded intervals \citep{kovacs2023seeded} and the narrowest-over-threshold principle \citep{baranowski2019narrowest}, the combination of which has successfully been applied to a wide range of change point detection problems. 

Specifically, we evaluate $\cT_{a, \tau, b}$ over a collection of the seeded intervals generated as
\begin{align}
\label{def: seeded_interval}
\M = \bigcup_{h = 1}^{\lfloor \mu_T\rfloor} \left\{ \Big( \floor[\big]{(i-1) m_{h}}, \ceil[\big]{(i+1) m_{h}} \Big] : \; i\in\big[ \ceil{T/m_{h}}-1 \big], \; m_{h}=T /2^{h} \right\},
\end{align}
with $\mu_T \asymp \log_2\log(T)$.
By construction, the number of intervals in $\M$ is $\cO(\log(T))$, and the lengths of the intervals are at least of order $T/\log(T)$.
In the literature on change point detection under factor models, it is often assumed that the change points are linearly spaced (see Assumption~\ref{assum: trans_mat_alt}), to ensure that pseudo-factors induced by the changes are well-captured by the global estimation procedure leading to $\wh{\cG}_t, \, t \in [T]$, which is compatible with this choice of $\M$.
Scanning the data with $\cT_{a, \tau, b}$ for $(a, b] \in \M$, enables us to systematically zoom in the neighborhoods of individual change points and isolate each in a sufficiently large interval for its detection and estimation.

Next, we iteratively select the shortest interval, say $(a, b] \in \M$, over which the maximal local covariance difference aggregated over the $K$ modes, as measured by $\max_{a + \varpi_T < \tau < b - \varpi_T} \cT_{a, \tau, b}$, exceeds a given threshold $\pi_{T, p}$, with some trimming parameter $\varpi_T \ge 0$; if there are ties, we select the interval with the largest detector statistic. 
This step returns a seeded interval containing exactly one change point with high probability, effectively breaking down the multiple change point problem to that of detecting each change point individually. 
Then we identify a change point estimator as $\wh\theta = \mathop{\arg\max}_{a + \varpi_T < \tau < b - \varpi_T} \cT_{a, \tau, b}$, remove any intervals in $\M$ containing $\wh\theta$ from future consideration, and repeatedly perform the above steps until no seeded interval is left to be considered.
Referred to as \underline{t}ensor \underline{f}actor \underline{m}odel \underline{seg}mentation algorithm (TFMseg), its full description is given in Algorithm~\ref{alg: mad}.
We detail the theoretical requirements on all the tuning parameters, namely $W_{a, \tau, b}$, $\mu_T$, $\varpi_T$ and $\pi_{T, p}$, in Section~\ref{sec: theorem} (see Theorem~\ref{thm: asymp_consistency_detection}), and give guidance on their choice and the estimator of $r_k, \, k \in [K]$, in Section~\ref{sec: tuning} for good finite sample performance.

\begin{algorithm}[h!t!]
\caption{TFMseg: Multiple change point detection in tensor factor models}
\label{alg: mad}
\begin{algorithmic}[1]
\setstretch{1.25}
\State \textbf{Input:} Trimming parameter $\varpi_T \ge 0$, threshold $\pi_{T, p} \ge 0$, seeded interval collection $\M \gets \{(a_\ell, b_\ell], \, \ell \in [M] \}$ from~\eqref{def: seeded_interval}
\State \textbf{Initialize:} Set $\wh\Theta \gets \emptyset$ and $\L\gets [M]$
\For{$\ell \in \L$}
\If{$b_\ell -a_\ell > 2\varpi_T$}
\State Set $\tau_\ell \gets \argmax_{a_\ell + \varpi_T <\tau < b_\ell -\varpi_T} \cT_{a_\ell, \tau, b_\ell}$
and $\cT_\ell \gets \cT_{a_\ell, \tau_\ell, b_\ell}$
\Else \State Set $\tau_\ell \gets b_\ell$ and $\cT_\ell \gets 0$
\EndIf
\If{$\cT_\ell \leq \pi_{T, p}$}
\State Set $\L \gets \L\setminus\{\ell\}$
\EndIf
\EndFor
\While{$\L\neq \emptyset$}
\State Set $\ell^\circ \gets \argmin_{\ell\in \L} (b_\ell -a_\ell)$; if there are ties, select the index with the largest $\cT_\ell$
\State Set $\wh\theta \gets \tau_{\ell^\circ}$, $\wh\Theta \gets \wh\Theta \cup \{\wh\theta\}$ and $\L \gets \L\setminus \{\ell\in\L : \wh\theta \in (a_\ell, b_\ell]\}$
\EndWhile
\State \textbf{Output:} Set of change point estimators $\wh\Theta$
\end{algorithmic}
\end{algorithm}

\begin{remark}[Utilizing mode-wise separability for change point detection.]
\label{rem:tensor}
In the context of matrix factor time series segmentation, \citet{peng2025detection,peng2025estimation} propose to search for change points in the $(r_1r_2) \times (r_1r_2)$ covariance of $\Vec(\wh{\cG}_t)$.
While not reported here, in preliminary numerical experiments, we considered its tensor-analogue by modifying the CUSUM statistic $\cM_{a, \tau, b}$ accordingly while keeping the rest of the algorithm identical, and found that the resultant procedure suffered from the relatively large dimensionality.
For example, for the order-$3$ tensor factor time series in Section~\ref{sec: general_dgp} with $(r_1, r_2, r_3) = (3, 3, 3)$, this approach handles a CUSUM process of dimension $27 \times 28 / 2= 378$, compared to $3 \times 3 \times 4 / 2 = 18$ in the case of TFMseg.
It illustrates the advantage of the proposed detector statistic which utilizes the separability of the covariance of $\cG_t$ implied by Condition~\ref{cond:factor}, in further reducing the dimensionality of the structure of interest from $r (r + 1) / 2$ to $\sum_{k = 1}^K r_k (r_k + 1)/2$, the gap between which grows with the tensor order~$K$. 
\end{remark}

\begin{remark}[Change point detection with missing data]\label{rem: miss}
Often the data contain missingness particularly in high-dimensional time series, 
and the problem of handling missingness in the change point detection framework has received considerable interest \citep{Xieetal2012, follain2022high, LiuSafikhani2025}. 
Analyzing a fully-observed subset of the data only, may significantly reduce the realized sample size and potentially compromise the interpretation of the outcome, all of which may be intensified for tensor data. 
We may leverage an imputation procedure proposed by \cite{CenLam2025} for tensor time series admitting a representation in \eqref{eqn: tfm_change} (without change points). 
Specifically, in constructing the detector statistic in~\eqref{eqn: test_stat_agg_std}, we may employ the estimator of the core factor time series given in Equation~(3.7) of \cite{CenLam2025} to plug in \eqref{eqn: gamma_G}, with the rest of the steps of TFMseg remaining the same.
We empirically validate this modification of TFMseg in the analysis of Fama--French portfolio returns in Section~\ref{sec:fama}.
\end{remark}

\subsection{Post-detection mode-identification}\label{sec: mode_identify}

In this section, we propose a method for identifying the subset of modes for which each detected change point is identifiable according to Definition~\ref{def: classification_change}. 
While we present and analyze the method in combination with the set of change point estimators $\wh{\Theta} = \{\wh\theta_j, \, j \in [\wh q]: \, \wh\theta_1 < \ldots < \wh\theta_{\wh q} \}$ returned by TFMseg (Algorithm~\ref{alg: mad}), it is applicable with any other change point detection procedure or even with pre-specified change points.

To this end, we measure the degree of the mode-$k$ identifiability at change point $\theta_j$ with $\|\Xi_j^{(k)}\|$ where, with $\Gamma_{G,(j+1)}^{(k)}$ given in Definition~\ref{def: jump},
\begin{equation}
\label{eqn: scaled_mode_change}
\Xi_j^{(k)}:= \tr\Big( \Gamma_{G,(j+1)}^{(k)} \Big)^{-1} \Gamma_{G,(j+1)}^{(k)} - \tr\Big( \Gamma_{G,(j)}^{(k)} \Big)^{-1} \Gamma_{G,(j)}^{(k)} .
\end{equation}
By Theorem~\ref{thm: identification_change} and from that $\Gamma^{(k)}_{G, (j)} = \tr( A_{j, \text{-}k}^\trans A_{j, \text{-}k} ) \cdot A_{j,k} A_{j,k}^\trans$ (see~\eqref{eqn: cov_G_signal}), we have $\|\Xi_j^{(k)}\| = 0$ if the change at $\theta_j$ is mode-$k$ unidentifiable. Furthermore, by Lemma~\ref{lemma: AA=cBB} in the supplement, $\|\Xi_j^{(k)}\| = 0$ only if the change is mode-unidentifiable.
In the case of $c = \prod_{k=1}^K \|A_{j,k}\|_F =\prod_{k=1}^K \|A_{j + 1,k}\|_F$, we have $\omega_j^{(k)}$ directly related to $\|\Xi_j^{(k)}\|$ as $\omega_j^{(k)} = c \cdot \|\Xi_j^{(k)}\|$.
That is, if the overall magnitude of the transformation matrices remains invariant, the difficulty in detecting the change using its mode information is proportional to the difficulty in mode-identification.

Then, our aim is to identify $\cK_j := \{k \in [K]: \, \Vert \Xi^{(k)}_j \Vert  > 0 \}$ for all $j \in [q]$, with $\cK_j = \emptyset$ if and only if the change at $\theta_j$ is mode-unidentifiable.
For the estimation of $\Xi_j^{(k)}$, we propose to adopt the set of \textit{finer} seeded intervals, 
\begin{align}
\M^* = \left\{ \Big( \floor[\big]{(i-1) m^*}, \ceil[\big]{(i+1) m^*} \Big] : \; i\in\big[ \ceil{T/m^*}-1 \big], \; m^* = T /2^{\mu_T + 1} \right\},
\nonumber
\end{align}
with $\mu_T$ used in the construction of $\M$ in~\eqref{def: seeded_interval}.
Compared to $\M$, the set $\M^*$ contains single-scale seeded intervals at the finer scale $\mu_T + 1$.
Setting $\wh\theta_0 = \wh\theta_0^+ = 0$ and $\wh\theta_{\wh{q}+1} = \wh\theta_{\wh{q}+1}^{-} = T$, we define 
\begin{equation}
\label{eqn: seeded_adj_cp}
\wh\theta_j^{-} := \max\left\{ a: (a,b] \in \M^*, b\leq \wh\theta_j \right\} \text{ \ and \ } 
\wh\theta_j^{+} := \min\left\{ b: (a,b] \in \M^*, a> \wh\theta_j \right\},
\end{equation}
for each $j\in [\wh q]$.
With these, $\Gamma^{(k)}_{G, (j)}, \, k \in [K]$, are estimated as $\wh{\Gamma}_{G,(j)}^{(k)} = \wh{\Gamma}_{G, \wh\theta_{j-1}^{+}, \wh\theta_j^{-}}^{(k)}$, see~\eqref{eqn: gamma_G}.
In doing so, upon the consistent estimation of $\Theta$ by $\wh{\Theta}$,
we select $(\wh\theta_{j-1}^{+}, \wh\theta_j^{-}]$ using the end points of the intervals $\M^*$ to be the maximal interval lying between $\theta_{j - 1}$ and $\theta_j$ with probability tending to one. 
At the same time, selecting $\wh\theta_j^{\pm}$ from the deterministic grid allows us to control for multiplicity arising from the stochasticity of $\wh\Theta$ effectively.
Next, we estimate $\Xi_j^{(k)}$ by $\wh\Xi_j^{(k)}$ obtained according to~\eqref{eqn: scaled_mode_change} by plugging in $\wh{\Gamma}_{G,(j)}^{(k)}$ in place of ${\Gamma}_{G,(j)}^{(k)}$.
Finally, for each $j\in [\wh{q}]$, we determine the change detected by $\wh\theta_j$ to be mode-$k$ identifiable if $\|\wh\Xi_j^{(k)}\| > \zeta_{T, p}$ for some threshold $\zeta_{T, p}$, i.e.\ $\wh{\cK}_j = \{k \in [K]: \, \|\wh\Xi_j^{(k)}\| > \zeta_{T, p}\}$.
We conclude that change detected by $\wh\theta_j$ as mode-unidentifiable if $\|\wh\Xi_j^{(k)}\| \leq \zeta_{T, p}$ for all $k\in[K]$ or, equivalently, if $\wh{\cK}_j = \emptyset$. 
Theorem~\ref{thm: mode_identify} outlines the rate requirement on $\zeta_{T, p}$, and Section~\ref{sec: tuning} gives practical guidance on its choice.

\subsection{Mode-informed estimation of loading spaces}
\label{subsec: mode_informed}

\begin{figure}[t!]
\centering
\includegraphics[width=0.8\linewidth]{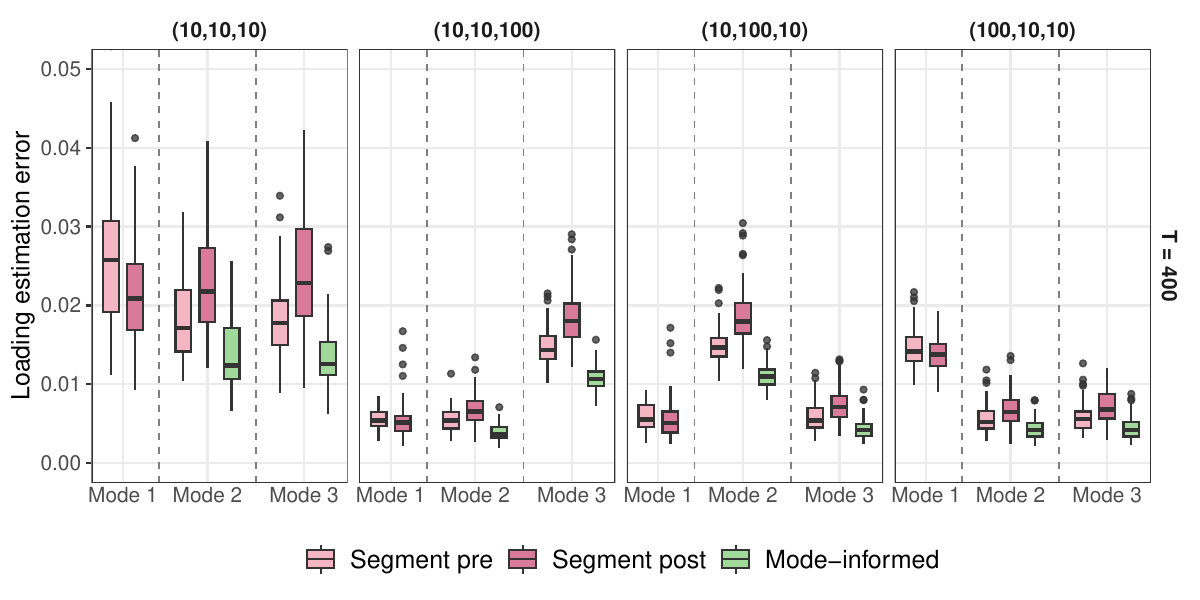}
\caption{Boxplots of mode-wise loading estimation errors with $T = 400$ and varying $(p_1, p_2, p_3)$ (left to right). Segment~pre (resp.\ Segment~post) refers to segment-wise estimation errors using the data before (resp.\ after) the change point estimator $\wh\theta_1$. See Section~\ref{sec:num:loading} for the complete description.}
\label{fig: plot_reest_rnull_T400}
\end{figure}

Once the change points are detected, one approach to estimating the mode-$k$ loading space for the $j$-th segment is to apply the projection estimator of Section~\ref{sec: detection} to the data over the interval $(\wh\theta_{j - 1}, \wh\theta_j]$.  
This segment-wise estimator is consistent for the column space of $\Lambda_{j, k}$ as long as the change point estimators are sufficiently accurate, but it ignores the possibility that some loading spaces may remain unchanged before and after certain change points.
Specifically, in view of Theorem~\ref{thm: identification_change}, if a change at $\theta_j$ is mode-$k$ unidentifiable, then the mode-$k$ loading space $\col(\Lambda_{j, k})$ is identical to $\col(\Lambda_{j + 1, k})$ up to an orthogonal rotation and a scalar multiplication, and thus may be regarded as unchanged from the perspective of the column space.  
Similarly, if the change is mode-$k$ identifiable, the column space genuinely shifts.  
The mode-identification procedure of Section~\ref{sec: mode_identify} provides, for each estimated change point $\wh\theta_j$, an estimated set $\wh{\cK}_j \subseteq [K]$ of modes that are deemed identifiable.
Such information naturally leads to a more efficient estimation strategy: For each mode $k \in [K]$, 
we define the set of boundary points that is formed by the mode-$k$ identifiable changes, as
\[
\cB^{(k)} = \{0, T\} \cup \{\wh\theta_j : j \in [\wh{q}], \ k \in \wh{\cK}_j\}
=:
\l\{ \eta_\ell^{(k)}, \, \ell \in [m_k] \cup \{0\}: \, \eta_0^{(k)} < \eta_1^{(k)} < \ldots < \eta_{m_k}^{(k)} \r\}.
\]
Then for each $\ell \in [m_k]$, we estimate the common mode-$k$ loading space by applying the projection estimator of Section~\ref{sec: detection} to the sub-sample $\{\cX_t : t \in (\eta_{\ell - 1}^{(k)}, \eta_{\ell}^{(k)}]\}$. 
This \textit{mode-informed} approach pools information across the segments that are not separated by a mode-$k$ identifiable change, and thus is expected to increase the effective sample size, see Figure~\ref{fig: plot_reest_rnull_T400} demonstrating the numerical benefit of the mode-informed approach over segment-wise estimation.
We emphasize that such arguments carry over to the vector factor model setting with $K = 1$ (see Example~\ref{ex:vector_mode_id}), and the mode-identification procedure in Section~\ref{sec: mode_identify} can be deployed to determine whether pooling the adjacent segments is advantageous or not. 


\section{Theory}\label{sec: theorem}

\subsection{Assumptions}\label{subsec: assumption}

We introduce the assumptions on the degree of serial and cross-sectional dependence in $\{ \cX_t \}_{t = 1}^T$, and the strength of the factors under model~\eqref{eqn: tfm_change_rewrite}.
We first define a class of general linear tensor time series.
\begin{definition}[]\label{def: general_linear_tensor}
A tensor time series $\{\cY_t\}_{t \in \Z}$ with $\cY_t \in \R^{p_1 \times \ldots \times p_K}$, is a $(\nu,\beta)$-general linear tensor time series for some $\nu \geq 4$ and $\beta\geq 0$, if and only if it holds that $\cY_t = \sum_{h\geq 0} a_{y,h} \cX_{y,t-h}$ where (i)~the innovation process $\{\cX_{y,t}\}$ has independent elements with zero mean, unit variance, and uniformly bounded $\nu$-th order moments, i.e.\ denoting the elements of $\cX_{y, t}$ by $\cX_{y, t, \bi}, \, \bi = (i_1, \ldots, i_K)$, there exists some constant $C_\nu > 0$ depending on $\nu$, such that ${\sup_{t\in\mathbb Z}}\max_{\bi \in \prod_{k = 1}^K [p_k]} \E(\vert \cX_{y, t, \bi} \vert^\nu) \le C_\nu$; (ii)~the scalar coefficients $a_{y,h}$ satisfy $\sum_{h\geq 0} a_{y,h}^2=1$ and $\sum_{h\geq 0} |a_{y,h}| \leq c$ for some constant $c > 0$; and (iii)~there exists some constant $c_\beta > 0$ depending on $\beta$, such that $\sup_{m \ge 0} (m + 1)^\beta \sum_{s = m}^\infty \vert a_{y, s} \vert \le c_{\beta}$.
\end{definition}

Definition~\ref{def: general_linear_tensor} formalizes the data generating process first used in \cite{ChenLam2024}.
Essentially, for each $\bi \in \prod_{k = 1}^K [p_k]$, the time series $\{\cY_{t, \bi}\}_{t \in \Z}$ follows a general linear process with weak serial correlation, bounded $\nu$-th moment, and $\beta$ characterizing the tail summability. 
By construction, the linearity of $\cY_t$ in $\{\cX_{y, t - h}\}_{h \ge 0}$, is preserved under various operations on $\cY_t$ such as vectorization, unfolding and tensor reshaping, which makes it amenable to modeling tensor time series with temporal dependence structure. 
In the definition, (iii) is implied by (ii), e.g.\ when $\beta = 0$. 
With Definition~\ref{def: general_linear_tensor}, and $\Delta_j = \min(\theta_j -\theta_{j-1}, \theta_{j+1} -\theta_j)$ denoting the spacing between $\theta_j$ and its neighboring change points for each $j \in [q]$, we make the following assumptions.

\begin{assumption}[Core factor]\label{assum: core_factor}
The series $\{\cF_{t}\}_{t \in \Z}$ is a $(\nu,\beta)$-general linear tensor time series.
\end{assumption}

\begin{assumption}[Global factor loadings]\label{assum: loadings}
For each $k\in[K]$, there exists some constant $c_{\lambda}$ such that (i)~the largest elements of $\Lambda_k$ in modulus is bounded by $c_{\lambda}$ and (ii)~$p_k^{-1} \Lambda_k^\trans \Lambda_k = I_{r_k}$ for all  sufficiently large $p_k$.
\end{assumption}

\begin{assumption}[Transformation matrices]\label{assum: trans_mat_alt}
The transformation matrices satisfy the following.
\begin{enumerate}[itemsep=0pt, label = (\roman*), left = 0pt]
    \item For each $j\in[q + 1]$ and $k\in[K]$, we have $\|A_{j,k}\|_F$ bounded
    away from zero and infinity.
    \item For each $k\in[K]$, there is $\Sigma_{A,k}$ such that
    $T^{-1} \sum_{j=1}^{q+1} \big\|\otimes_{\ell = K, \, \ell \ne k}^1 A_{j,\ell} \big\|_F^2 \cdot (\theta_j - \theta_{j-1}) \cdot A_{j,k} A_{j,k}^\trans \to \Sigma_{A,k}$ as $\min(T,p_1,\dots, p_K) \to\infty$. 
    \item For each $k\in[K]$, $[A_{1,k}, A_{2,k}, \dots, A_{q + 1,k}]$ has full row rank. 
    
    \item There exists some constant $c_\Delta \in (0, 1/2]$ such that $\min_{1 \le j \le q} \Delta_j \ge c_\Delta T. $
\end{enumerate}
\end{assumption}

\begin{assumption}[Idiosyncratic noise]\label{assum: noise}
There are $(\nu,\beta)$-general linear tensor time series $\{\cF_{e,t}\}_{t \in \Z}$ and $\{\epsilon_t\}_{t \in \Z}$ independent of each other, such that the noise $\cE_t$, $t\in[T]$, admits the representation
$\cE_t = \cF_{e,t}\times_{k=1}^K A_{e,k} + \Sigma_{\epsilon} \circ \epsilon_t$,
where each $A_{e,k} \in \R^{p_k\times r_{e,k}}$ satisfies $\|A_{e,k}\|_1 = \cO(1)$ and $\|A_{e,k}\|_\infty = \cO(1)$, and the order-$K$ tensor $\Sigma_{\epsilon}$ has uniformly bounded elements. 
Finally, $\{\cF_{e,t}\}_{t \in \Z}$ and $\{\epsilon_t\}_{t \in \Z}$ are independent of $\{\cF_t\}_{t \in \Z}$.
\end{assumption}

Under Assumption~\ref{assum: core_factor}, we have $\cF_t$ satisfy Condition~\ref{cond:factor}.
Assumption~\ref{assum: loadings} is a tensor analogue of the condition found in \cite{BaiNg2013}, and similar conditions are made in the literature on factor modeling for vector- \citep{stock2002forecasting}, matrix- \citep{yu2022projected}, and tensor-valued \citep{barigozzi2026statistical} time series.
Assumption~\ref{assum: trans_mat_alt} regularizes the transformation matrices so that, together with Assumption~\ref{assum: loadings}, the global loading matrices $\Lambda_k, \, k \in [K]$, can be estimated consistently. Part~(i) prevents the factor structure from vanishing along some modes in any segment; parts~(ii) and (iii), in combination with Assumption~\ref{assum: loadings}, generalize the standard full column rank condition imposed on the loading matrices in factor models, to the change point setting.
When $K = 1$, our part~(ii) is akin to Assumption~5~(iii) in \cite{barigozzi2025moving} which constrains the covariance of the pseudo-factor rather than the transformation matrices; part~(iii) boils down to Assumption~11~(i) in \cite{Duanetal2023}, see also their Remark~2 for more discussion.
The linear spacing condition made in~(iv) is invariably found in the related literature; see, for example, \cite{barigozzi2018simultaneous, barigozzi2025moving, Baltagietal2021, Duanetal2025}. 
We may relax this which, however, calls for the expansion of the interval collection $\M$ with $\mu_T \asymp \log_2(T)$ (see~\eqref{def: seeded_interval}), and has further implication on the detection and estimation performance as detailed in Appendix~\ref{sec:assum:lin:space}.

Under Assumption~\ref{assum: noise}, the idiosyncratic component is decomposed into a factor-driven component with approximately sparse loadings, and a linear tensor time series,
and thus contains both cross-sectional and serial correlation.
At the same time, the cross-sectional correlations are sufficiently weak so that with Assumptions~\ref{assum: loadings} and~\ref{assum: trans_mat_alt}, there exists a large gap between the $r_k$-th and the $(r_k + 1)$-th largest eigenvalues of the mode-$k$ covariance of $\mat_k(\cX_t)$ for all $k \in [K]$, which ensures the asymptotic identifiability of the (unobserved) $\cC_t$ and $\cE_t$ across all the modes. 
We specify our requirements on $(\nu, \beta)$ in Assumptions~\ref{assum: core_factor} and~\ref{assum: noise} later; it suffices to assume that $\nu \ge 4$ and $\beta \ge 0$ for consistent estimation of $\Lambda_k$ (Proposition~\ref{prop: consistency_proj}), whereas we require that $\nu > 8$ and $\beta > 1$ for the consistency in change point detection and mode-identification (Theorems~\ref{thm: asymp_consistency_detection} and~\ref{thm: mode_identify}).

\begin{assumption}[Signal-to-noise ratio]\label{assum: signal-to-noise}
The jump magnitudes $\omega_j, \, j \in [q]$, defined in Definition~\ref{def: jump} satisfy
$\min_{j \in [q]} \, \sqrt{\Delta_j}\,\omega_j / \sqrt{\log(T)} \to \infty$ and $p \min_{j \in [q]} \, \omega_j \to \infty$ as $\min(T, p) \to \infty$.
\end{assumption}

\begin{assumption}[Dimensionality]
\label{assum: dim}
$p^{-1}\sqrt{T}  = \cO( \sqrt{\log(T)})$ as $\min(T,p_1,\ldots,p_K)\to \infty$. 
\end{assumption}
Assumption~\ref{assum: signal-to-noise} permits $\omega_j \to 0$ as $\min(T, p_1, \ldots, p_K) \to \infty$, at a sufficiently slow rate and accordingly,  the rate of estimation for each $\theta_j$ explicitly depends on $\omega_j$ (Theorem~\ref{thm: asymp_consistency_detection}).
Assumption~\ref{assum: dim} is weaker than the conditions on the relative divergence of $p$ and $T$ found in \citet{Baltagietal2021, Duanetal2023,Baietal2024,barigozzi2025moving} for change point detection and testing, that $p^{-1} \sqrt{T} = o(1)$. 
Under Assumptions~\ref{assum: trans_mat_alt}~(iv) and~\ref{assum: dim}, the second requirement in Assumption~\ref{assum: signal-to-noise} implies the first. 
Finally, we impose the following condition on $\Xi_j^{(k)}$ defined in~\eqref{eqn: scaled_mode_change} for the change at $\theta_j$ to be identifiable with the modes $k \in \cK_j$ (where $\Vert \Xi_j^{(k)} \Vert \ne 0$), which still allows for $\min_{k \in \cK_j} \!\! \Vert \Xi_j^{(k)} \Vert \to 0$.
\begin{assumption}[The degree of mode-identifiability]
\label{assum: mode-id}
$\max(T^{-1/2}, p^{-1}) = o\big( \min_{j\in[q], \, k \in \cK_j} \| \Xi_j^{(k)} \| \big)$.
\end{assumption}

\subsection{Asymptotic results}\label{subsec: theorem}

We first establish that the column space of each loading matrix $\Lambda_k$ in~\eqref{eqn: tfm_change_rewrite} is consistently estimated by that of $\wh{\Lambda}_k$ given in Section~\ref{sec: detection}.
This plays a key role in the subsequent analysis, and is of independent interest as a first-of-its-kind result derived in the presence of multiple change points under a weak moment condition.

\begin{proposition}[Loading estimation]
\label{prop: consistency_proj}
Let Assumptions~\ref{assum: core_factor}, \ref{assum: loadings}, \ref{assum: trans_mat_alt} and~\ref{assum: noise} hold, with $\nu\geq 4$ and $\beta\geq 0$ for Assumptions~\ref{assum: core_factor} and~\ref{assum: noise}. 
Then there exists a matrix $\wh{H}_k \in \R^{r_k \times r_k}$ which satisfies $\wh{H}_k^\trans \wh{H}_k = I_{r_k} + o_P(1)$ as $\min\{T,p_1,\dots, p_K\} \to \infty$, such that
\[
\frac{1}{p_k}\Big\| \wh\Lambda_k - \Lambda_k \wh{H}_k \Big\|_F^2 = \cO_P\l\{ \frac{1}{T\pmk} + \frac{1}{p^2} + \Big( \sum_{\ell\in [K]\setminus\{k\}} \frac{1}{Tp_{\text{-}\ell}} \Big) \Big( \frac{1}{T} + \frac{1}{p_k} \Big) \r\}.
\]
\end{proposition}


\begin{theorem}[Consistency of change point detection]\label{thm: asymp_consistency_detection}
Suppose that Assumptions~\ref{assum: core_factor}--\ref{assum: dim} hold, with $\nu > 8$ and $\beta > 1$ for Assumptions~\ref{assum: core_factor} and~\ref{assum: noise}. Suppose that the threshold $\pi_{T, p}$ is chosen to satisfy $\sqrt{\log(T)}/\pi_{T, p} = o(1)$ and  $\pi_{T, p} = o(\min_{j \in [q]} \sqrt{\Delta_j} \omega_j)$, 
$\M$ defined in~\eqref{def: seeded_interval} is constructed with $\mu_T \asymp \log_2\log(T)$, the trimming parameter satisfies $\varpi_T = c_\varpi T/2^{\mu_T}$ for some constant $c_\varpi \in (0, 1/2)$, and $W_{a, \tau, b} = W$ is used in~\eqref{eqn: test_stat_agg_std} for all $(a, b] \in \M$ and $\tau \in (a, b)$, where $W$ is a positive definite matrix with $\max(\Vert W \Vert, \Vert W^{-1} \Vert) \le C_W$ for some constant $C_W > 0$. Then TFMseg (Algorithm~\ref{alg: mad}) returns a set of change point estimators $\wh\Theta = \{\wh{\theta}_j, \, j \in [\wh{q}]: \, \wh{\theta}_1 < \ldots < \wh{\theta}_{\wh{q}} \}$ which satisfies, for any sequence $\upsilon_T \to \infty$ arbitrarily slowly,
\begin{align*}
\P\l\{ \wh{q} = q, \, \max_{j\in[q]}\, \omega_j^{2} \vert \thh_j-\thj\vert \le \upsilon_T \r\} \to 1 \text{ \ as \ } \min(T,p_1,\ldots,p_K)\to \infty.
\end{align*}
\end{theorem}

Theorem~\ref{thm: asymp_consistency_detection} establishes that TFMseg consistently estimates both the total number and locations of the change points, which holds for any positive definite matrix $W$ with bounded spectra in place of $W_{a, \tau, b}$ in~\eqref{eqn: test_stat_agg_std}, which facilitates the investigation into the rate of estimation.
The specific form of $W$ may only affect the finite-sample power. This viewpoint is in line with the approach taken by \cite{kirch2015detection} for multivariate change point detection where the weight matrix in their CUSUM statistic is treated as a tuning parameter, and structured choices are discussed as a way of balancing between numerical stability and finite sample detection power.
The trimming parameter $\varpi_T$, as well as the finest level $\mu_T$ in constructing $\M$, are chosen to uniformly control the errors arising from scanning the data multiple times, particularly under the weaker moment condition imposed on the data (Assumptions~\ref{assum: core_factor} and~\ref{assum: noise}). 
The choice of $\pi_{T, p}$ is compatible with Assumption~\ref{assum: signal-to-noise}. 
The rate of estimation for each $\theta_j$ is characterized by the change size $\omega_j$, in the sense that $\vert \wh\theta_j - \theta_j \vert = \cO_P(\omega_j^{-2})$, i.e.\ larger changes are estimated with better accuracy. 
Under Assumption~\ref{assum: signal-to-noise}, we have $\vert \wh\theta_j - \theta_j \vert/\Delta_j = o_P(1)$, i.e.\ the change point estimators are consistent in the re-scaled time. 

We remark that Theorem~\ref{thm: asymp_consistency_detection} is established under a condition on the tail behavior of $\{\cF_t\}_{t \in \Z}$ and $\{\cE_t\}_{t \in \Z}$ characterized through $\nu > 8$, which is comparable to those found in vector \citep{Duanetal2023, Baietal2024} and matrix \citep{peng2025detection, peng2025estimation} settings. 
We may relax this and require $\nu > 4$, which calls for e.g.\ a stronger condition on the signal-to-noise, namely $T^{-\nu/8 + 1} = o(\min_{j \in [q]} \omega_j \sqrt{\Delta_j})$, than that found in Assumption~\ref{assum: signal-to-noise}.

Finally, we establish the consistency in mode-identification at each change point, 
which builds upon the consistency in change point detection in Theorem~\ref{thm: asymp_consistency_detection}.
\begin{theorem}[Consistency of mode-identification]\label{thm: mode_identify}
Let all the assumptions in Theorem~\ref{thm: asymp_consistency_detection} hold, 
and additionally impose Assumption~\ref{assum: mode-id}. 
Setting the mode-identification threshold $\zeta_{T, p}$ to fulfill $T^{-1/2} + p^{-1} = o(\zeta_{T, p})$ and $\zeta_{T, p} / \min_{j \in [q], \, k \in \cK_j} \Vert \Xi_j^{(k)} \Vert = o(1)$ as $\min(T, p_1, \ldots, p_K) \to \infty$, 
we have $\P\big\{ \wh{\cK}_j = \cK_j \text{ for all } j \in [q] \big\} \to 1$.
\end{theorem}


\section{Numerical experiments}
\label{sec: simulation}

\subsection{Tuning parameter selection}
\label{sec: tuning}

We provide guidance on selecting parameters used in our method.

\paragraph{Core factor dimensions.}
Following \cite{barigozzi2026statistical}, we estimate $r_k, \, k \in [K]$, by inspecting the eigenvalues of $\wh{\Gamma}^{(k)}_Y$ defined in~\eqref{eqn: est_lam}, say $\lambda_\ell(\wh\Gamma^{(k)}_Y), \, \ell \ge 1$, in the decreasing order. 
Specifically, setting the maximum mode-$k$ factor number at $\bar{r}_k = \lceil p_k/3 \rceil$, we first obtain $\wt r_k =\argmax_{1 \le \ell \le \bar r_k}{\lambda_\ell(\wt{\Gamma}^{(k)}_X)}/{\lambda_{\ell+1}(\wt{\Gamma}^{(k)}_X)}$, with which we form $\wh{\Gamma}^{(k)}_Y$ as in~\eqref{eqn: est_lam} and set $\wh r_k = \argmax_{1 \le \ell \le \bar{r}_k} \lambda_\ell(\wh{\Gamma}^{(k)}_Y)/\lambda_{\ell+1}(\wh{\Gamma}^{(k)}_Y)$. 
\vspace{-10pt}

\paragraph{Weight matrix $W_{a, \tau, b}$.} 
The asymptotic consistency in change point detection is established for any positive definite matrix $W$ with bounded spectrum in place of $W_{a, \tau, b}$ (see Theorem~\ref{thm: asymp_consistency_detection}). 
Unlike in the mean change detection, the variability of the CUSUM process itself also undergoes a shift in the current setting, which favors a localized approach to the choice of the weight matrix where it is chosen for each seeded interval, i.e.\ $W_{a, \tau, b} = W_{a, b}$ for all $(a, b] \in \M$ and $\tau \in (a, b)$.
In the context of detecting mean shifts in multivariate time series, \cite{cho2026multivariate} discuss the computational, numerical and statistical trade-off between the use of a weight matrix estimating the full (long-run) covariance of the noise, and that only contains its diagonal elements. 
They show that neither of the two choices uniformly outperforms the other in terms of the detection power, and the latter choice is to be preferred for its numerical stability for moderately large dimensions, particularly as the inverse of the matrix is used in the detector statistic.
From these considerations, for each seeded interval $(a, b]$, we set $W_{a, b}$ to be diagonal populated with the diagonal entries of the following local Bartlett kernel long-run covariance matrix estimator $\wh{W}_{a, b}$: With $\wh{g}_{k, t}^{(a, b)} = \Vech(\mat_k(\wh\cG_t) \mat_k(\wh\cG_t)^\trans -\wh{\Gamma}_{G,a, b}^{(k)})$ and bandwidth $w= \min\{\lfloor T^{0.15}\rfloor, (b - a)/2\}$,
\begin{align*}
\wh{W}_{a, b} &= \wh{\Sigma}_{a, b}(0) + \sum_{\ell = 1}^w \left( 1-\frac{\ell}{w+1} \right) \left(\wh{\Sigma}_{a, b}(\ell)+ (\wh{\Sigma}_{a, b}(\ell))^\trans \right), \\
\text{where \ } 
\wh{\Sigma}_{a, b}(\ell) &= \frac{1}{b - a} \sum_{t = a + \ell + 1}^b \begin{bmatrix}
\l(\wh{g}_{1, t}^{(a, b)} \r)^\trans, \ldots, \l(\wh{g}_{K, t}^{(a, b)} \r)^\trans
\end{bmatrix}^\trans 
\begin{bmatrix}
\l(\wh{g}_{1, t-\ell}^{(a, b)} \r)^\trans, \ldots, \l(\wh{g}_{K, t-\ell}^{(a, b)} \r)^\trans
\end{bmatrix}.
\end{align*}
\vspace{-10pt}

\paragraph{Seeded intervals and trimming parameter.}
We set $\mu_T$ in~\eqref{def: seeded_interval} so that the minimum interval length is $0.5T/\log(T)$, i.e.\ $\mu_T = \log_2(4\log(T))$. Accordingly, we set $\varpi_T = 0.25T/\log(T)$.
\vspace{-10pt}

\paragraph{Threshold $\pi_{T, p}$ for change point detection.}
Theorem~\ref{thm: asymp_consistency_detection} below gives a rate range for $\pi_{T, p}$ which, however, involves quantities that are unknown in practice. 
This is a common issue in change point detection where a practical solution is through large-scale simulations, an approach which we also adopt here. 
Specifically, we simulate tensor time series under the null model ($q = 0$, see Appendix~\ref{app: num_setting}). 
For each realization, we record $\max_{(s, e] \in \M} \max_{\tau \in (s + \varpi_T, e - \varpi_T]} \cT_{s, \tau, e}$, 
take the empirical $0.9$-quantiles of these maxima over $100$ realizations for each combination of dimensions and other data generating parameters, 
and regress them on the functions of $T$ and $(r_1, r_2, r_3)$.
We use the fitted model ($R^2_{\text{adj}}$ is roughly 0.93) with $T$ and $\wh r_k$ as regressors, to determine $\pi_{T, p}$; see Appendix~\ref{app: num_det_oracle} for further description of the experiment. 
\vspace{-10pt}

\paragraph{Threshold $\zeta_{T, p}$ for mode-identification.}
Based on Theorem~\ref{thm: mode_identify} below, we examine the rescaled statistic $\Vert \wh{\Xi}_j^{(k)} \Vert/(T^{-1/2} + p^{-1})$, computed on simulated datasets (with $q = 3$, see Appendix~\ref{app: num_setting}) while treating $\theta_j$ as known, over varying configurations.
The empirical $0.99$-quantiles of the statistics for the modes $k \in [K] \setminus \cK_j$, lie around $3.8$, which suggests the threshold $\zeta_{T,p} = 3.8(T^{-1/2} + p^{-1})$; see Appendix~\ref{app: num_idt_oracle} for details of the experiment.

\subsection{Simulation settings}
\label{sec: general_dgp}

We generate $\{\cX_t\}_{t=1}^T$ from an order-$3$ tensor factor model under~\eqref{eqn: tfm_change}
with $q = 3$ and $(r_1,r_2,r_3)=(3,3,3)$.
The factor process is generated as
$\vec(\cF_t) =  \rho_f\,\vec(\cF_{t-1}) + \varepsilon_t$ with $\varepsilon_t \stackrel{\mathrm{i.i.d.}}{\sim}
\cN_r({0},\, (1- \rho_f^{2})I_r)$ and $\rho_f \in\{0,0.7\}$, 
and the vectorized idiosyncratic component follows $\vec(\cE_t) =  \rho_e \, \vec(\cE_{t-1}) + \eta_t$ with $\eta_t \stackrel{\mathrm{i.i.d.}}{\sim}
\cN_p({0},\, (1- \rho_e^{2})\Sigma_e)$, where $\rho_e \in\{0,0.3\}$ and $(\Sigma_e)_{ij} = 0.3^{\vert i-j\vert}$.
The loading matrix $\Lambda_{1, k}$ has its entries generated independently from the uniform distribution $\cU(-1, 1)$.
Then at each $j \in [3]$, we set $\Lambda_{j + 1, k} = \Lambda_{j, k} A_{j, k}$ for $k = j$ while $\Lambda_{j + 1, k} = \Lambda_{j, k}, \, k \ne j$, where
\begin{equation*}
A_{1,1} =
\begin{bmatrix}
0.5&0&0\\
a_{1,21}&1&0\\
a_{1,31}&a_{1,32}&1.5
\end{bmatrix},
\quad
A_{2,2}=
\begin{bmatrix}
1&0&0\\
0&1&0\\
0&0&0
\end{bmatrix},
\text{ \ and \ }
A_{3,3}= \begin{bmatrix} a_{3, 11} & a_{3, 12} & a_{3, 13} \\
a_{3, 21} & a_{3, 22} & a_{3, 23} \\
a_{3, 31} & a_{3, 32} & a_{3, 33} 
\end{bmatrix},
\end{equation*}
with $a_{1,\ell h}\stackrel{\mathrm{i.i.d.}}{\sim}\cN(0,1)$ for $(\ell,h)\in\{(2,1),(3,1),(3,2)\}$, and $a_{3,\ell h}\stackrel{\mathrm{i.i.d.}}{\sim}\cN(0,1/3)$ for $\ell,h\in [3]$; this way, we have the dimensions of the pseudo-factor $\cG_t$ at $(r_1, r_2, r_3) = (3, 3, 3)$ while locally, the mode-$2$ factor number from the third segment onward is two.
In the main text, we report the results on change point detection and mode-identification under the above model with varying $T \in \{400, 800, 1600, 3200\}$ and $(p_1,p_2,p_3) \in \{(10,10,10),(10,10,100),(10,20,40),(20,20,20)\}$, when the change points are \enquote{unbalanced}, i.e.\ $\Theta = \bigl\{\lfloor 0.25T \rfloor,\lfloor 0.5T \rfloor,\lfloor 0.625T \rfloor\bigr\}$, and the data are serially correlated with $(\rho_f, \rho_e) = (0.7, 0.3)$, which we consider as the more challenging scenario. For each setting, $N = 100$ realizations are generated.

Additionally, we consider further change point scenarios including the null scenarios (i.e.\ $\Theta = \emptyset$), and explore the performance of the proposed methods for change point detection and mode-identification in Appendices~\ref{app: num_det} and~\ref{app: num_idt}, respectively; we further verify the discussion in Section~\ref{subsec: mode_informed} on improving the mode-wise loading space estimation via mode-identification, see Appendix~\ref{app: num_reest}; finally, we report in Appendix~\ref{app: num_miss} the performance TFMseg in the presence of missingness with the modification as discussed in Remark~\ref{rem: miss}.
Throughout, we apply TFMseg and the mode-identification procedure with the tuning parameters generated as in Section~\ref{sec: tuning}. 

\subsection{Results for change point detection}\label{sec: num_detect}

For each $n \in [N]$, denote by $\wh{q}^{(n)}$ and $\wh{\theta}_j^{(n)}$ the number of change point estimators and their locations for the $n$-th realization. 
Figure~\ref{fig: subplot_unbal_rnull} displays the distributions of the change point estimators $\wh{\theta}_j^{(n)}, \, j \in [\wh{q}^{(n)}]$, returned by TFMseg over $N = 100$ realizations, across varying dimensions and sample size.
We additionally report the accuracy of the change point estimators conditional on the successful detection of all three change points, by considering the realizations where $\wh q^{(n)} = 3$ and  $\wh\theta^{(n)}_j \in ((\theta_{j-1} + \theta_j)/2, (\theta_j + \theta_{j + 1})/2]$ for all $j \in [3]$; the index set of those realizations is denoted by $\wh{\cD} \subset [N]$.

For comparison, we consider the binary segmentation procedure of \citet{Baietal2024} which, referred to as \enquote{LR}, recursively performs likelihood-ratio tests. As the method is proposed for vector time series, we first vectorize the tensor time series prior to its application.
In its implementation, we set the minimum segment length to be $0.5T/\log(T)$ while all other tuning parameters are set as recommended by the authors. 
\citet{peng2025detection, peng2025estimation} consider the problem of matrix time series segmentation under a Tucker factor model, by looking for changes in the covariance of the vectorized pseudo-factor matrix time series, see also Remark~\ref{rem:tensor}. We consider its \enquote{oracle} version, referred to as \enquote{oraclePZW}: It proceeds as in Algorithm~\ref{alg: mad} with the modified detector statistic formed with the vectorized $\wh{\cG}_t$ until exactly $q = 3$ change points are detected. This allows us to compare the localization accuracy of their approach against ours, numerically complementing the discussion in Remark~\ref{rem:tensor}. 
While other methods exist that perform change point analysis of tensor-valued time series \citep{anastasiou2025tensor}, their focus lies in detecting changes in the stochastic structures not necessarily well-captured by the Tucker factor model in~\eqref{eqn: tfm_change}, and thus we opt to include them in the real data applications reported in Section~\ref{sec: real}.

Overall, the performance of all three methods improves as $T$ and $(p_1, p_2, p_3)$ grow, which is in line with Assumption~\ref{assum: signal-to-noise}.
TFMseg tends to perform better than oraclePZW and LR, particularly when $T$ and/or $(p_1, p_2, p_3)$ are smaller, and the difference in performance is more striking when focusing on how frequently all three change points are detected by each method.
This is perhaps not surprising as (i)~vectorization of $\cX_t$ taken before applying LR, or (ii)~vectorization of $\wh{\cG}_t$ involved in oraclePZW, do not fully utilize the tensor structure of the data. 
It also proves to bring computational benefit to retain the tensor structure as TFMseg is considerably faster than LR as the dimensions increase (see Figure~\ref{fig:runtime}).
Additionally, the results in Appendix~\ref{app: num_det} confirm that its calibration suggested in Section~\ref{sec: tuning} works well so that TFMseg returns few spurious estimator in the absence of any change point, and its performance improves in the absence of serial correlations (with $(\rho_f, \rho_e) = (0, 0)$).

\begin{figure}[h!t!b!p!]
\centering
\includegraphics[width = .9\linewidth]{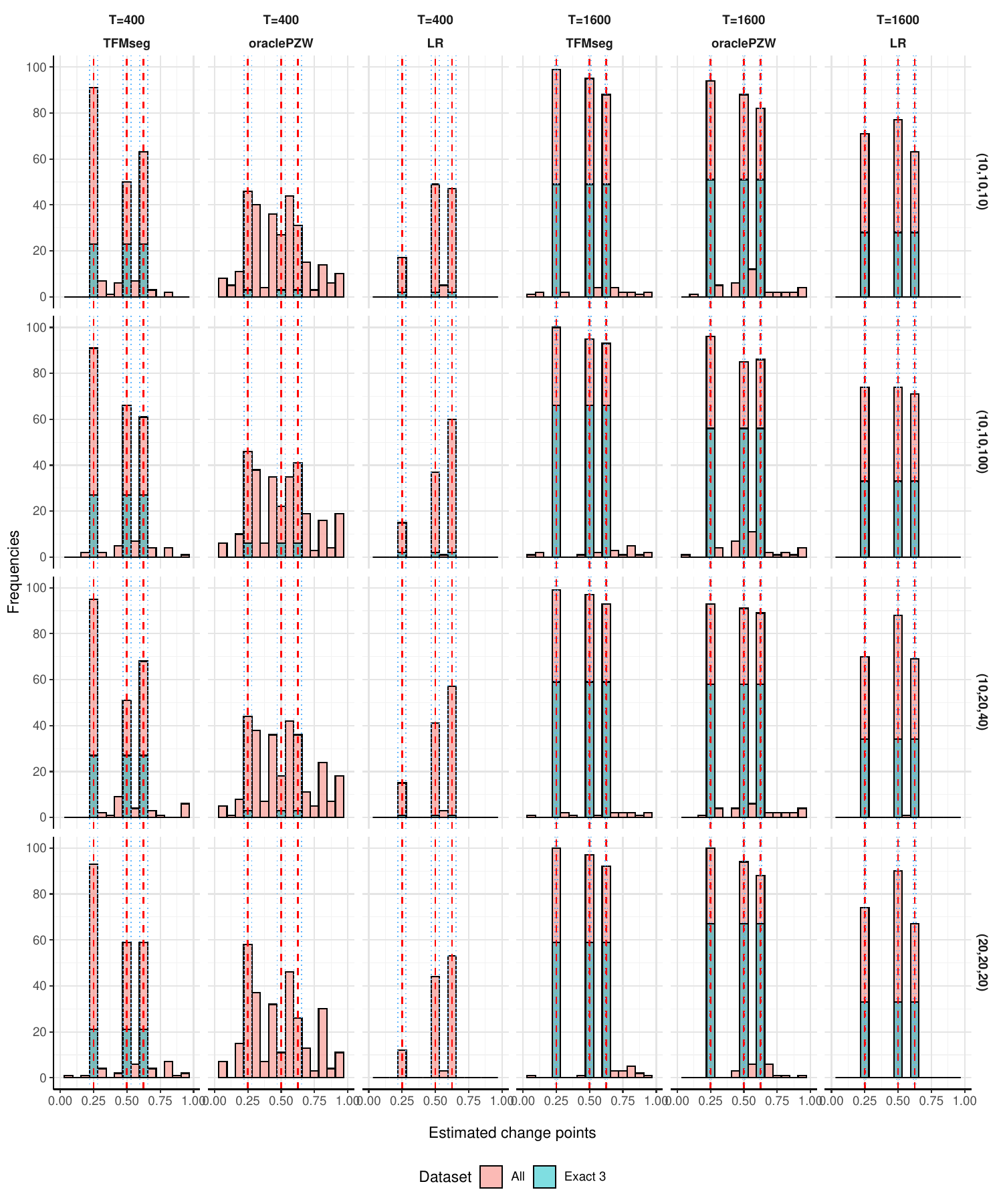}
\caption{Barplots of the scaled change point estimators $\{ \wh\theta^{(n)}_j/T, \, j \in [\wh q^{(n)}]\}$, returned by TFMseg, oraclePZW and LR with varying $T \in \{400, 1600\}$ and $(p_1,p_2,p_3)$ (top to bottom) over $N = 100$ realizations. The red bars give the total frequency of estimated change points for all $n \in [N]$, while the blue bars give the frequency from the subset of realizations $n \in \wh{\cD}$ where all three change points are detected for each method.}
\label{fig: subplot_unbal_rnull}
\end{figure}

\subsection{Results for mode-identification}
\label{sec:num:mode}

While performing the mode-identification with $\wh\theta^\pm_j$ defined as in~\eqref{eqn: seeded_adj_cp} facilitates the theoretical analysis, for practical implementation, we propose to simply set $\wh\theta^\pm_j = \wh\theta_j$ as it makes little difference numerically. 
To separately treat the problem of change point detection and mode-identification, we investigate the performance of the mode-identification procedure on the subset of realizations $\wh{\cD}$ defined in Section~\ref{sec: num_detect} for each setting, where all three change points are detected by TFMseg. 
Recalling that $\cK_j \subset [K]$ denotes the subset of modes for which the change at $\theta_j$ is mode-identifiable according to Definition~\ref{def: classification_change}, we report the following performance metrics:
With $\wh\cK_j^{(n)} = \{ k \in [K]: \, \Vert \wh\Xi_j^{(k), (n)} \Vert > \zeta_{T,p} \}$,
\[
\text{TPR}_j = \frac{1}{\vert \wh{\cD} \vert}\sum_{n \in \wh{\cD}} \frac{\vert \wh{\cK}_j^{(n)} \cap \cK_j \vert}{ \max(\vert \cK_j \vert, \, 1) } \text{ \ and \ } \text{FPR}_j = \frac{1}{\vert \wh{\cD} \vert}\sum_{n \in \wh{\cD}} \frac{\vert \wh\cK_j^{(n)} \setminus \cK_j \vert}{\max(K - \vert \cK_j \vert, \, 1)}.
\]
For the data generating process described in Section~\ref{sec: general_dgp}, we have $\cK_j=\{j\}$ for all $j \in [3]$. 

\begin{table}[h!t!b!p!]
\centering
\caption{Summary of the mode-identification results for varying $T$ and $(p_1,p_2,p_3)$.}
\label{tab: Mode_unbal_rnull_est}
\centering
\resizebox{\ifdim\width>\linewidth\linewidth\else\width\fi}{!}{
\small
\begin{tabular}[t]{ccccccccc}
\toprule
\multicolumn{2}{c}{ } & \multicolumn{2}{c}{$j=1$} & \multicolumn{2}{c}{$j=2$} & \multicolumn{2}{c}{$j=3$} & \multicolumn{1}{c}{  } \\
\cmidrule(l{3pt}r{3pt}){3-4} \cmidrule(l{3pt}r{3pt}){5-6} \cmidrule(l{3pt}r{3pt}){7-8}
$T$ & $(p_1,p_2,p_3)$ & TPR & FPR & TPR & FPR & TPR & FPR & $\vert \wh{\cD} \vert$\\
\cmidrule(lr){1-2} 
\cmidrule(lr){3-4} 
\cmidrule(lr){5-6} 
\cmidrule(lr){7-8}
\cmidrule(lr){9-9}
 & (10,10,10) & 0.74 & 0 & 1 & 0.04 & 0.87 & 0 & 23\\
 & (10,10,100) & 0.67 & 0 & 0.96 & 0.04 & 0.96 & 0.04 & 27\\
 & (10,20,40) & 0.74 & 0 & 1 & 0.02 & 0.96 & 0.02 & 27\\
\multirow{-4}{*}{\centering\arraybackslash 400} & (20,20,20) & 0.90 & 0 & 1 & 0 & 0.90 & 0.02 & 21\\
\cmidrule(lr){1-2} \cmidrule(lr){3-4} \cmidrule(lr){5-6} \cmidrule(lr){7-8} \cmidrule(lr){9-9}
 & (10,10,10) & 0.79 & 0 & 1 & 0.05 & 0.95 & 0.03 & 38\\
 & (10,10,100) & 0.80 & 0 & 1 & 0.07 & 0.96 & 0.01 & 46\\
 & (10,20,40) & 0.85 & 0 & 1 & 0.07 & 0.98 & 0.02 & 46\\
\multirow{-4}{*}{\centering\arraybackslash 800} & (20,20,20) & 0.88 & 0.01 & 1 & 0.03 & 1 & 0.03 & 48\\
\cmidrule(lr){1-2} \cmidrule(lr){3-4} \cmidrule(lr){5-6} \cmidrule(lr){7-8} \cmidrule(lr){9-9}
 & (10,10,10) & 0.80 & 0 & 1 & 0.04 & 1 & 0 & 49\\
 & (10,10,100) & 0.80 & 0 & 0.98 & 0.07 & 1 & 0.01 & 66\\
 & (10,20,40) & 0.86 & 0 & 1 & 0.03 & 0.98 & 0.01 & 59\\
\multirow{-4}{*}{\centering\arraybackslash 1600} & (20,20,20) & 0.93 & 0 & 1 & 0.04 & 0.98 & 0.02 & 59\\
\cmidrule(lr){1-2} \cmidrule(lr){3-4} \cmidrule(lr){5-6} \cmidrule(lr){7-8} \cmidrule(lr){9-9}
 & (10,10,10) & 0.71 & 0 & 1 & 0.05 & 1 & 0.01 & 75\\
 & (10,10,100) & 0.76 & 0 & 1 & 0.09 & 1 & 0 & 76\\
 & (10,20,40) & 0.75 & 0 & 1 & 0.07 & 1 & 0.01 & 76\\
\multirow{-4}{*}{\centering\arraybackslash 3200} & (20,20,20) & 0.92 & 0.01 & 1 & 0.03 & 0.99 & 0.02 & 76\\
\bottomrule
\end{tabular}}
\end{table}

Table~\ref{tab: Mode_unbal_rnull_est} shows that the proposed mode-identification method performs well, with TPR and FPR improving as the sample size and the dimensions increase. 
Additional numerical experiments not reported here, indicate that the relatively lower TPR when $T = 400$, is attributed to that $\wh{\cD}$ is relatively small, and that the factor number estimator (see Section~\ref{sec: tuning}) occasionally under-estimates $r_k$ which leads to loss of signal---on such occasions, TPR improves when $\wh r_k$ is replaced with $r_k$.
Appendix~\ref{app: num_idt} reports the results from additional scenarios with more complex change structures including some changes which are not unidentifiable for some modes.

\subsection{Results for mode-wise loading space estimation}
\label{sec:num:loading}

We examine the impact of mode-identification for the downstream task of estimating the loading spaces.
To this end, we compare the mode-informed approach discussed in Section~\ref{subsec: mode_informed}, against the naive approach that estimates the column space $\col(\Lambda_{j, k})$ over each estimated segment for each $k \in [K]$, focusing on the single change point scenario ($q = 1$) where $K = 3$, $(r_1, r_2, r_3) = (3, 3, 3)$, $\theta_1 = \lfloor T/2 \rfloor$ and $\cK_1 = \{1\}$; in other words, 
$\col(\Lambda_{1, 1}) \ne \col(\Lambda_{2, 1})$ while
$\col(\Lambda_{1, k}) = \col(\Lambda_{2, k}) = \col(\Lambda_k)$ for $k \in \{2, 3\}$.
We refer to Appendix~\ref{app: num_reest} for the complete description of the experiments.
After applying TFMseg and the mode-identification procedure to $N = 100$ realizations per setting, we denote by $\wt{\cD} \subset \wh{\cD} \subset [N]$, the index set of realizations where the single change is detected and $\wh{\cK}_1 = \cK_1$ is returned.
On those realizations, the mode-informed approach estimates $\col(\Lambda_k), \, k \in \{2, 3\}$, using the whole sample, while $\col(\Lambda_{j, 1})$ is estimated separately on $(\wh\theta_{j - 1}, \wh\theta_j]$ for each $j \in [2]$.
We report the loading space estimation error as measured by~\eqref{eq:fle_dist} in Appendix~\ref{app: num_reest}, see Figure~\ref{fig: plot_reest_rnull_T400} in Section~\ref{subsec: mode_informed} for a representative example.
Besides that the mode-$k$ loading estimation error decreases with $\pmk$, which is in line with Proposition~\ref{prop: consistency_proj}, it is apparent that the mode-informed approach clearly benefits from the increased sample size for the modes $k \notin \cK_1$. 


\section{Real data applications}\label{sec: real}

\subsection{New York City Yellow Taxi data}\label{sec: real_nyc}

We analyze the New York City Yellow Taxi trip records released by the city's Taxi and Limousine Commission (TLC). \footnote{Retrieved from \url{https://www.nyc.gov/site/tlc/about/tlc-trip-record-data.page}.} 
We focus on trips within Manhattan, which account for the majority of the observations. The pick-up and drop-off locations are aggregated into $p_1 = p_2 = 69$ pre-defined taxi zones, and each day is partitioned into $p_3 = 24$ hourly intervals.
Then on each business day $t$, we observe an order-$3$ tensor $\cX^\circ_t \in \R^{69 \times 69 \times 24}$, where the entry $X^\circ_{i_1,i_2,i_3,t}$ records the number of trips from zone $i_1$ to zone $i_2$ during the $i_3$-th hourly interval, over the period from 1 January 2013 to 31 December 2022 ($T=2519$).
As the data exhibit potential seasonal effects and overall trends, we de-mean each $\{X^\circ_{i_1,i_2,i_3,t}\}_{t = 1}^T$ using a rolling window of one month radius, to form the time series of the residuals $\{\cX_t\}_{t = 1}^T$.

We analyze the data for possible structural changes using TFMseg with all the tuning parameters set as in Section~\ref{sec: tuning}, unless it is mentioned specifically. 
For comparison, we consider the LR procedure of \cite{Baietal2024} applied to the vectorized data, see also Section~\ref{sec: num_detect} where the method was applied to the simulated datasets.
Additionally, we include \enquote{TenSeg} of \citet{anastasiou2025tensor} which employs the CP decomposition of the data to extract the low-rank structure of the data on which change point detection is performed. 

When applied to $\{\cX_t\}_{t = 1}^T$, the factor number estimator given in Section~\ref{sec: tuning} returns $(\wh{r}_1, \wh{r}_2, \wh{r}_3) = (1,1,1)$; \cite{chen2022factor} argue that a dominant rank-one component may be masking other factors, and suggest to use $(r_1,r_2,r_3)=(4,2,2)$.
Given this background, we propose to apply the proposed method with varying specifications of the factor numbers $(r_1,r_2,r_3)$ lying between and adjacent to $(1,1,1)$ and $(4,2,2)$, and examine its sensitivity, see~\eqref{eq:real:rank:taxi} in Appendix~\ref{app: real_nyc} for detail. 
We apply the vectorization-based approach of LR with $r = r_1r_2r_3$ varying accordingly while for TenSeg, we rely on the default choice of the CP rank available in its implementation.
From Figure~\ref{fig:taxi_rank_sensitivity}, we observe that the sets of change point estimators returned by TFMseg tend to exhibit a nested structure, in that increasing the factor numbers generally allows for the detection of more change points which is as expected. 
We also see that some change points around 2014-04, 2015-04, 2019-01, 2020-04 and 2021-04, are detected for a wide factor number range, and the mode-identification results for such changes remain largely insensitive to the choice of $(r_1, r_2, r_3)$.
The change point detected around 2016-04, to a lesser degree, also satisfies this description and the mode-identification result therein remains stable provided that $r_1 \ge 3$.
While the performance of LR remains stable for the most of $r$ values in consideration, it detects only the two change points near 2020-04 and 2021-04.

\begin{figure}[h!t!]
\centering
\includegraphics[width=\linewidth]{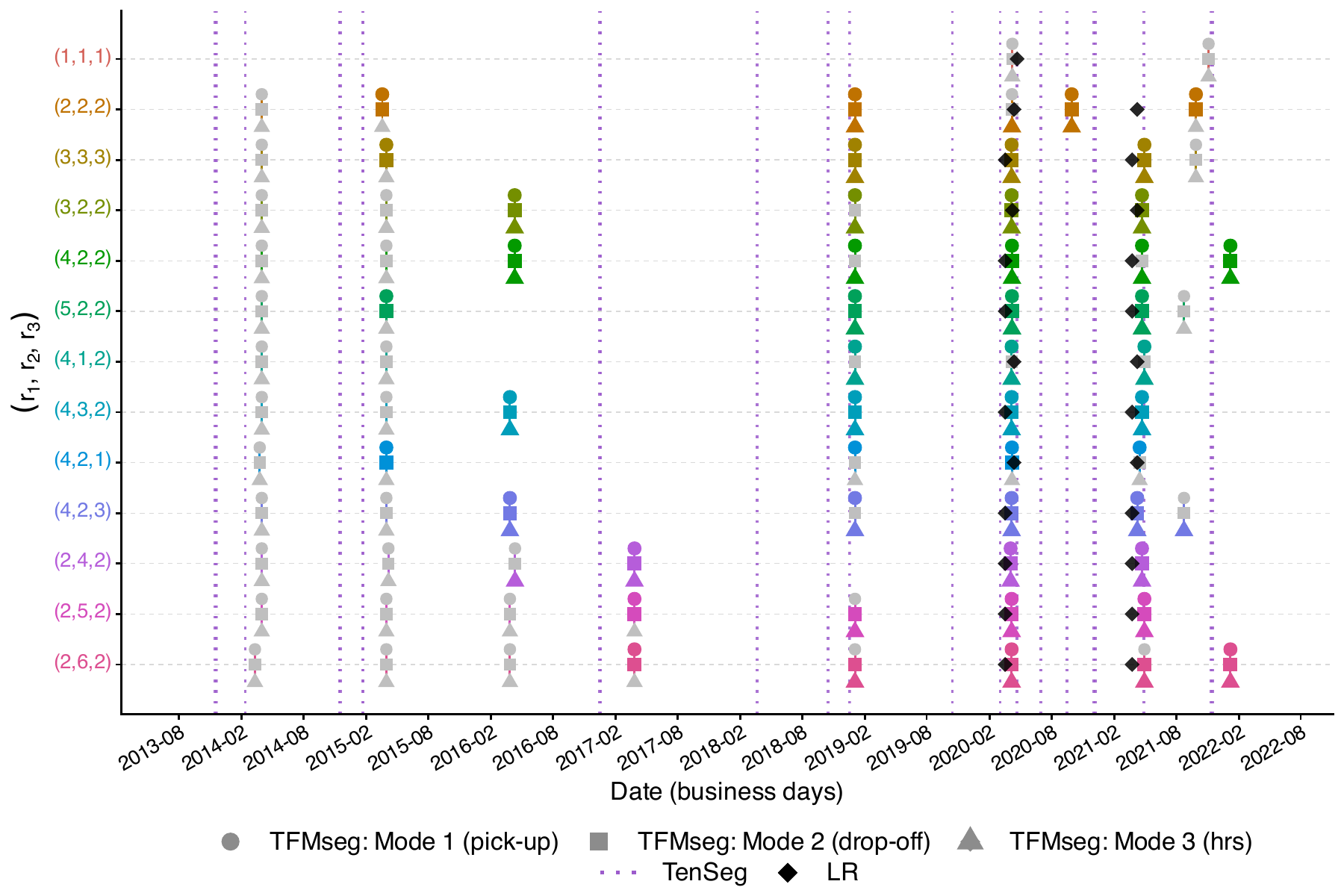}
\caption{New York City Yellow Taxi data: Change points estimated with varying factor number specifications. 
Each triplet of circle, square and triangle marks the locations of the change points ($x$-axis) detected with given $(r_1, r_2, r_3)$ ($y$-axis) and, when the plotting character is highlighted, it indicates that the change is mode-identifiable for the corresponding mode while grey plotting characters indicate mode-unidentifiability for the particular modes.
We also mark the change point estimators returned by LR \citep{Baietal2024} with the varying factor number $r = r_1r_2r_3$ by black diamonds, and those returned by TenSeg \citep{anastasiou2025tensor} by the vertical dotted lines.}
\label{fig:taxi_rank_sensitivity}
\end{figure}


Based on this sensitivity analysis, we focus on interpreting the results returns with $(r_1,r_2,r_3)=(4,2,2)$.
TFMseg detects change points at the following dates: 2014-04-01, 2015-04-01, 2016-04-11, 2019-01-03, 2020-04-07, 2021-04-23 and 2022-01-07; LR detects two (2020-03-11, 2021-03-18), and TenSeg detects 22 change points in total. 
The three methods largely agree on the large shifts around March--April 2020 and March--April 2021.
The changes at 2020-04, 2021-04 and 2022-01 have clear interpretations in association with the commencement of the COVID-19 lockdown in New York on March 22, 2020, the easing off of the regulations with the large-scale vaccine rollout in April 2021, and the re-opening of restaurants and large venues in February 2022, which in turn are related to abrupt shifts in taxi usage patterns. 
The left panel of Figure~\ref{fig: heat_combined} further strengthens this interpretation as these changes are deemed to have brought changes which are mode-identifiable in terms of pick-up locations, hours and, in the case of the 2020-04 and 2022-01 changes, drop-off locations.

Outside this period, some prominent changes may be accounted for by a documented weakening of the traditional Yellow Taxi market in 2015, see the report by \cite{nyctlc2015apr23}.
Those changes pre-dating the report are largely considered mode-unidentifiable (Figure~\ref{fig: heat_combined}), which is consistent with the view that they are attributed to the overall decline of the trip volume and its variability, rather than a shift in the taxi usage pattern. 

\subsection{Fama--French portfolio returns with missingness}
\label{sec:fama}

We study the Fama--French portfolio returns data,\footnote{Retrieved from \url{https://mba.tuck.dartmouth.edu/pages/faculty/ken.french/Data_Library/det_100_port_sz.html}.}
where the stocks are categorized into ten groups by market equity (ME) and ten groups by book-to-equity ratio (BE) to form $10 \times 10$ portfolios at each time stamp. We consider the time series of order-2 tensors $\cX_t \in \R^{10 \times 10}$ on value- and equal-weighted returns, respectively, where the former weights stocks by market capitalization and the latter assigns equal weights within each portfolio.
We use the monthly data from January 1974 to June 2021 ($T = 570$).
The data are not fully observed with 0.28\% missingness, and complete observations are available from July 2009 only.
We leverage the approach described in Remark~\ref{rem: miss} and estimate the pseudo-factors directly from the incomplete data.

The factor number estimator described in Section~\ref{sec: tuning} returns $(\wh r_1,\wh r_2)=(1,1)$, while \citet{CenLam2025} recommend using $(r_1, r_2) = (2,2)$ for the same data.
Similarly as in Section~\ref{sec: real_nyc}, we perform a sensitivity analysis by varying $r_1, r_2 \in [3]$. 
We set the rest of the tuning parameters according to Section~\ref{sec: tuning}, apart from the trimming parameter $\varpi_T$ (see Algorithm~\ref{alg: mad}), which is set to be $\varpi_T = 6$ (half-year) to account for the relatively small sample size and the monthly frequency.
Figure~\ref{fig:eq_fama_rank_sens} demonstrates that, on equal-weighted return data, TFMseg is relatively insensitive to the choice of $(r_1, r_2)$.
Most commonly detected changes (1975-02, 1999-10, 2001-02, 2008-12, 2020-02) overlap with the estimators returned by TenSeg, which also detects additional ones including those lying close to other detected changes; on this dataset, LR does not detect any change point.
See Appendix~\ref{app: real_fama} for the results obtained from the value-weighted return data. 

\begin{figure}[h!t!b!]
\centering
\includegraphics[width=.75\linewidth]{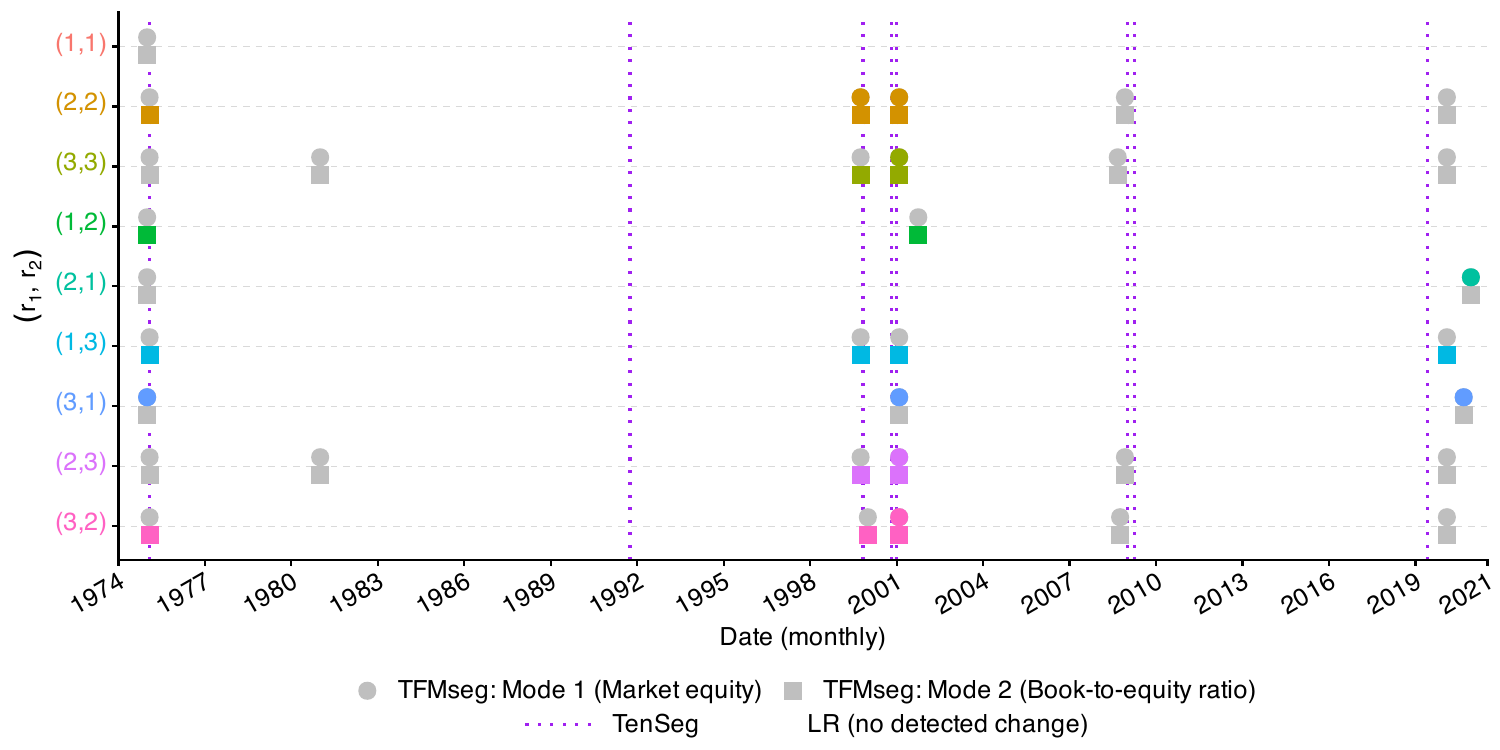}
\caption{Fama--French portfolio returns (equal-weighted):
Change points estimated with varying factor number specifications. 
Each pair of circle and square marks the locations of the change points ($x$-axis) detected with given $(r_1, r_2)$ ($y$-axis) and, when the plotting character is highlighted, it indicates that the change is mode-identifiable for the corresponding mode while grey plotting characters indicate mode-unidentifiability for the particular modes.
We also mark the change point estimators returned by TenSeg \citep{anastasiou2025tensor} with the vertical dotted lines; LR does \citep{Baietal2024} not detect any change for any $r = r_1r_2$.}
\label{fig:eq_fama_rank_sens}
\end{figure}

The figure also reveals that earlier three changes are identifiable with one or both of the two modes, relatively consistently across the factor number specifications, see also the right panel of Figure~\ref{fig: heat_combined} obtained with $(r_1, r_2) = (2, 2)$.
The change point estimator at 2008-12 is associated with the Global Financial Crisis, while the one at 2020-02 is relatable to Covid-19 pandemic, and both changes are deemed mode-unidentifiable as such events are likely to have brought shifts in the overall variability, rather than interactions across the cross-sections along any particular mode.

\section{Conclusion}\label{sec: conclusion}

This paper studies the modeling of high-dimensional tensor time series under a factor model with multiple change points. 
Beyond the detection of change points, by formalizing the concept of mode-identifiable changes, we develop a statistic that consistently assigns the detected change points to their corresponding modes. 
This provides an interpretation of structural changes in tensor factor models that is unavailable from detection alone, 
and underpins the necessity of preserving the multi-dimensional array structure of the data. 
Moreover, loading matrices for the modes where, either no change occurs or the change is not mode-identifiable, can be estimated by pooling the information across the segments, which leads to a mode-informed estimation strategy with evidently improved estimation accuracy.
We construct rigorous theoretical guarantees for our proposed methods, which is further backed up by extensive numerical results. 
We remark that estimating the number of pseudo-factors, $\{r_k\}_{k=1}^K$, might suffer from potentially weak factors or short intervals in practice. To the best of our knowledge, a consistent estimator of factor numbers that can adapt non-linearly spaced changes remains absent in the literature, which certainly merits future investigation.


\clearpage
\bibliographystyle{apalike}
\bibliography{reference.bib}	


\clearpage
\appendix
\renewcommand{\thesection}{\Alph{section}}
\renewcommand{\theHsection}{\Alph{section}} 
\renewcommand{\thesubsection}{\Alph{section}.\arabic{subsection}}
\renewcommand{\theHsubsection}{\Alph{section}.\arabic{subsection}}

\numberwithin{equation}{section}
\numberwithin{figure}{section}
\numberwithin{table}{section}
\numberwithin{theorem}{section}
\numberwithin{lemma}{section}
\numberwithin{proposition}{section}
\numberwithin{definition}{section}

\section{Additional discussions on the model and assumption}

\subsection{Change point scenarios and their representations}\label{subsec: change_scenario}

To demonstrate the generality of our model formulation in~\eqref{eqn: tfm_change} and~\eqref{eqn: tfm_change_rewrite}, we showcase some examples of different types of changes, complementing Example~\ref{ex:equiv}. For this purpose, it suffices to consider a vector factor model with a single change point, as the conclusions are easily generalized to tensor factor models with multiple change points in different combinations of the modes. Below, all $\cX_t$ and $\cE_t$ are in $\R^{p_1}$, where the use of calligraphic letters are simply for consistency with the notations in~\eqref{eqn: tfm_change} and~\eqref{eqn: tfm_change_rewrite}.

First, rotational changes (i.e., a scale or space change according to Proposition~\ref{prop: identification_change_type} in Appendix~\ref{sec: anatomy_mode_id})
are straightforwardly encoded in our formulation. That is,
\begin{align*}
\cX_t = \left\{
    \begin{array}{ll}
	\cF_t \times_1 \Lambda_{1,1} + \cE_t  & \text{for $1 \le t\le \theta_1$,} \\
    \cF_t \times_1 (\Lambda_{1,1} A_{2,1}) + \cE_t  & \text{for $\theta_1 + 1 \le t\le T$,} 
    \end{array}
    \right.
\end{align*}
where $\Lambda_{1,1} \in\R^{p_1\times r_1}$ and $A_{2,1} \in\R^{r_1\times r_1}$ are some matrices with rank $r_1$. In this case, the global loading matrix is $\Lambda_1 = \Lambda_{1,1}$, with the pre- and post-change transformation matrices $A_{1,1} =I_{r_1}$ and $A_{2,1}$, respectively.

All other types of changes are characterized with a rank-deficient transformation matrix either before or after the change point. These include, for example, an increase in the number of factors such that
\begin{align*}
\cX_t = \left\{
    \begin{array}{ll}
	\cF_{t,1} \times_1 Z_1 + \cE_t  & \text{for $1 \le t\le \theta_1$,} \\
    \cF_{t,1} \times_1 Z_1 + \cF_{t,2} \times_1 Z_2 + \cE_t  & \text{for $\theta_1 + 1 \le t\le T$,} 
    \end{array}
    \right.
\end{align*}
where $Z_1\in\R^{p_1\times r_1}$ and $Z_2\in\R^{p_1\times (r_2 -r_1)}$, which we re-formulate as
\begin{align*}
\cX_t = \left\{
    \begin{array}{ll}
	(\cF_{t,1}^\trans, \cF_{t,2}^\trans)^\trans \times_1 (Z_1, 0_{p_1\times (r_2 -r_1)}) + \cE_t  & \text{for $1 \le t \le \theta_1$,} \\
    (\cF_{t,1}^\trans, \cF_{t,2}^\trans)^\trans \times_1 (Z_1, Z_2) + \cE_t  & \text{for $\theta_1 + 1 \le t \le T$,} 
    \end{array}
    \right.
\end{align*}
i.e.\ $\Lambda_1 = (Z_1, Z_2) \in\R^{p_1\times r_2}$ with transformation matrix $A_{1,1}$ being a diagonal matrix with 1 on the first $r_1$ elements and zeros otherwise, and $A_{2,1}=I_{r_2}$. A decrease in the number of factors can be formulated similarly.

Next, we consider a change in the column space of the loading matrix:
\begin{align*}
\cX_t = \left\{
    \begin{array}{ll}
	\cF_t \times_1 Z_1 + \cE_t  & \text{for $1 \le t \le \theta_1$,} \\
    \cF_t \times_1 Z_2 + \cE_t  & \text{for $\theta_1 + 1 \le t \le T$,} 
    \end{array}
    \right.
\end{align*}
where $Z_1, Z_2\in\R^{p_1\times u}$ such that $\text{col}(Z_1) \cap \text{col}(Z_2) =\emptyset$. This leads to an interesting re-formulation as
\begin{align*}
\cX_t = \left\{
    \begin{array}{ll}
	(\cF_{t}^\trans, \cG_{t}^\trans)^\trans \times_1 (Z_1, 0_{p_1\times u}) + \cE_t  & \text{for $1 \le t \le \theta_1$,} \\
    (\cF_{t}^\trans, \cG_{t}^\trans)^\trans \times_1 (Z_2, 0_{p_1\times u})  + \cE_t  & \text{for $\theta_1 + 1 \le t \le T$,} 
    \end{array}
    \right.
\end{align*}
for some hypothetical $\cG_t$ as an independent copy of $\cF_t$ so that the core factor requirement Assumption~\ref{assum: core_factor} can be satisfied. The core factor rank is therefore $r_1=2u$ according to our formulation in~\eqref{eqn: tfm_change}. This introduction of $\cG_t$ is needed to have the global loading $\Lambda_1=(Z_1,Z_2)$, which it is not possible for $r_1<2u$ due to the different column spaces $Z_1$ and $Z_2$ span. The transformation matrices are then read as
\[
A_{1,1} = \begin{bmatrix}
    I_{u} & 0_{u\times u} \\
    0_{u\times u} & 0_{u\times u}
\end{bmatrix} ,\quad
A_{2,1} = \begin{bmatrix}
    0_{u\times u} & 0_{u\times u} \\
    I_{u} & 0_{u\times u}
\end{bmatrix} .
\]
Note that $A_{1,1}$ and $A_{2,1}$ easily satisfy Assumption~\ref{assum: trans_mat_alt}~(iii). 

To conclude all the previous discussion, the formulation~\eqref{eqn: tfm_change} or, equivalently, that in~\eqref{eqn: tfm_change_rewrite} is general enough to cover most, if not all (depending on the scope of interests), types of changes in the factor structure. One might argue that changes in the serial correlation in the core factors cannot be formulated, but the serial correlation is not manifested in the static factor representation anyway. This implies the merit of investigating structural changes in dynamic factor models (cf.\ \citeauthor{cho2024high}, \citeyear{cho2024high}, for the vector time series case), which we leave for the future study.

\subsection{An anatomy of mode-identifiable changes}\label{sec: anatomy_mode_id}

We complement Theorem~\ref{thm: identification_change} by categorizing mode-identifiable changes into three types. 

\begin{proposition}[Three types of mode-identifiable change]\label{prop: identification_change_type}
Consider any change from $A_{j,k}$ to $A_{j+1,k}$ for $j\in[q]$ and $k\in[K]$, and denote by $W_{j, k} \in \R^{r_k \times r_k}$ the diagonal matrix of singular values of $A_{j,k}$ in descending order.
Let us define the types of changes as follows.
\begin{enumerate}[itemsep=0pt, label = (\roman*), left = 0pt]
    \item Rank change: $A_{j,k}$ and $A_{j+1,k}$ have different numbers of non-zero singular values.
    \item Scale change: The change is not a rank change, there does not exist a constant $c$ such that $W_{j + 1,k} =c W_{j,k}$,
    and the space spanned by the non-zero singular values of $A_{j,k}$ is the same as that of $A_{j+1,k}$.
    \item Space change: The change is not a rank change nor a scale change, and for any $i \in [r_k]$, the left singular space of $A_{j,k}$ corresponding to $(W_{j,k})_{ii}$ is different from that of $A_{j + 1,k}$ corresponding to $(W_{j + 1,k})_{ii}$.
\end{enumerate}
Then the change is mode-identifiable if and only if it falls into one of the three types above.
\end{proposition}

Proposition~\ref{prop: identification_change_type} provides a taxonomy of structural changes in tensor factor models, which is comparable with
that given by \cite{Duanetal2025} in the context of vector time series factor modeling. 
We remark that the distinctions between scale and space changes are ad-hoc, and they only serve to provide more insight on the possible changes in a tensor factor model.
It is technically difficult 
to identify a scale change or a space change, due to the possibility of repeated singular values in the transformation matrix, which cannot be ruled out.
On the other hand, estimating the factor numbers \citep[e.g.][]{Lam2021, Hanetal2022, barigozzi2026statistical} 
can lead to the identification of rank changes.

\begin{proof}[Proof of Proposition~\ref{prop: identification_change_type}]
It is straightforward that when the change falls into any of the three types described in Proposition~\ref{prop: identification_change_type}, it is mode-identifiable according to Theorem~\ref{thm: identification_change}. It suffices to show if a change is mode-identifiable, then it must fall into one type (note that the types are mutually exclusive by how they are defined), or equivalently, if a change does not fall into any type, then it is mode-unidentifiable.

For any $j=1,\dots,q + 1$, we denote the singular value decomposition as
\begin{equation}
\label{eqn: re-write_A_jk}
A_{j,k} = Q_{j,k} W_{j,k} R_{j,k}^\trans ,
\end{equation}
for some orthogonal matrices $Q_{j,k}, R_{j,k} \in\R^{r_k\times r_k}$ and diagonal matrix $W_{j,k}$ of the singular values of $A_{j,k}$ in descending order, such that $A_{j,k} A_{j,k}^\trans = Q_{j,k} W_{j,k}^2 Q_{j,k}^\trans$. Then for a change from $A_{j,k}$ to $A_{j+1,k}$ not falling into any type, we may write
\[
A_{j,k} = (Q_{j,k}^\dagger W_{j,k}^\dagger, 0_{r_k\times w}) R_{j,k}^\trans, \quad
A_{j+1,k} = (Q_{j+1,k}^\dagger W_{j+1,k}^\dagger, 0_{r_k\times w}) R_{j+1,k}^\trans ,
\]
where $0_{r_k\times w}$ is a zero matrix of dimension $r_k\times w$ with $0 \le w < r_k$, $W_{j,k}^\dagger$ represents the diagonal sub-matrix consisting of only non-zero entries in $W_{j,k}$, $Q_{j,k}^\dagger$ represents the sub-matrix of $Q_{j, k}$ with the columns corresponding the non-zero diagonal entries of $W_{j, k}$, and analogously for $W_{j+1,k}^\dagger$ and $Q_{j+1,k}^\dagger$. Since the change is not of any type, we have
\[
\col(Q_{j,k,i}^\dagger) = \col(Q_{j+1,k,i}^\dagger) \text{ \ for all \ } 1 \le i \le r_k - w, \text{ \ and \ } W_{j+1,k}^\dagger = c W_{j,k}^\dagger ,
\]
where $Q_{j,k,i}^\dagger$ denotes the column(s) of $Q_{j,k}^\dagger$ corresponding to the $i$-th largest element in $W_{j,k}^\dagger$. 
This implies there exists some block diagonal matrix $Q$, where each block is an orthogonal matrix corresponding to a repeated singular value, or $\pm 1$ for a unique singular value, such that $Q_{j+1,k}^\dagger = Q_{j,k}^\dagger Q$. Thus,
\begin{align*}
    &\;\quad
    A_{j+1,k} A_{j+1,k}^\trans = (Q_{j+1,k}^\dagger W_{j+1,k}^\dagger, 0_{r_k\times w})
    \begin{bmatrix}
        (W_{j+1,k}^\dagger)^\trans (Q_{j+1,k}^\dagger)^\trans \\ 0_{r_k\times w}
    \end{bmatrix}
    = Q_{j+1,k}^\dagger (W_{j+1,k}^\dagger)^2 (Q_{j+1,k}^\dagger)^\trans \\
    &=
    c^2 \cdot Q_{j,k}^\dagger Q (W_{j,k}^\dagger)^2 Q^\trans (Q_{j,k}^\dagger)^\trans 
    = c^2 \cdot Q_{j,k}^\dagger (W_{j,k}^\dagger)^2 (Q_{j,k}^\dagger)^\trans
    = c^2 \cdot A_{j,k} A_{j,k}^\trans ,
\end{align*}
where the second last equality used the fact that either the orthogonal blocks in $Q$ correspond to a scale multiple of identity matrices in the diagonal matrix $(W_{j,k}^\dagger)^2$, or the blocks in $Q$ are diagonal with $\pm 1$ corresponding to unique diagonal entries in $(W_{j,k}^\dagger)^2$.

Applying Lemma~\ref{lemma: AA=cBB} to the above, and with Theorem~\ref{thm: identification_change}, the change is mode-unidentifiable, as desired. This completes the proof of the theorem.
\end{proof}

\subsection{Relaxation of Assumption~\ref{assum: trans_mat_alt}~(iv)}
\label{sec:assum:lin:space} 

Assumption~\ref{assum: trans_mat_alt}~(iv) requires the spacing of the change points to be linear in $T$.
We discuss how this assumption may be relaxed, by presenting an alternative assumption in Appendix~\ref{sec:trans_mat}, and discussing an alteration to TFMseg and its impact on the theoretical results in Appendix~\ref{sec:mod:tfmseg}.

\subsubsection{Alternative of Assumption~\ref{assum: trans_mat_alt}}
\label{sec:trans_mat}

In place of Assumption~\ref{assum: trans_mat_alt}, let us consider the following assumption.

\begin{assumptionp}{\ref*{assum: trans_mat_alt}'}[Alternative of Assumption~\ref{assum: trans_mat_alt}]\label{assum: trans_mat}
The transformation matrices satisfy:
\begin{enumerate}[itemsep=0pt, label = (\roman*), left = 0pt]
    \item For each $j\in[q + 1]$ and $k\in[K]$, we have $\|A_{j,k}\|_F$ bounded away from zero and infinity.
    \item For each $k\in[K]$, there is $\Sigma_{A,k}$ such that
    $T^{-1} \sum_{j=1}^{q+1} \big\|\otimes_{\ell = K, \, \ell \ne k}^1 A_{j,\ell} \big\|_F^2 \cdot (\theta_j - \theta_{j-1}) \cdot A_{j,k} A_{j,k}^\trans \to \Sigma_{A,k}$ as $\min(T,p_1,\dots, p_K) \to\infty$.
    \item The matrix $\Sigma_{A,k}$ is positive definite.
\end{enumerate}
\end{assumptionp}

In fact, Proposition~\ref{prop: consistency_proj} holds under this condition replacing Assumption~\ref{assum: trans_mat_alt} (see Lemma~\ref{lemma: consistency_proj}).
Lemma~\ref{lemma: eigenvalue_PCA_consistency}~(i) in Appendix~\ref{sec:pf:loading} shows that parts~(iii) in both Assumptions~\ref{assum: trans_mat_alt} and~\ref{assum: trans_mat} are in fact equivalent if all change points are linearly spaced, \text{viz.} Assumption~\ref{assum: trans_mat_alt}~(iv). 
This indicates that Assumption~\ref{assum: trans_mat_alt} is mildly stronger than Assumption~\ref{assum: trans_mat}. 
Below we further discuss the scenarios where Assumption~\ref{assum: trans_mat}~(iii) holds even though Assumption~\ref{assum: trans_mat_alt}~(iv) does not. 

In what follows, we write 
\[
\bar{\Sigma}_{A, k} = \frac{1}{T} \sum_{j=1}^{q+1} \big\|A_{j, \text{-}k} \big\|_F^2 \cdot (\theta_j - \theta_{j-1}) \cdot A_{j,k} A_{j,k}^\trans,
\]
where $A_{j, \text{-}k} = \otimes_{\ell = K, \, \ell \ne k}^1 A_{j,\ell}$.
Focusing on the mode $1$, we consider the case with $q = 1$ and $r_1 = 2$. 
\begin{enumerate}[label = (\alph*)]
\item Setting $A_{1,1}=I_2$ and $A_{2,1}=2 I_2$, we have
\[
\bar{\Sigma}_{A, k} = \frac{1}{T} \left\{
    \begin{bmatrix}
        \theta_1 & 0 \\ 0 & \theta_1
    \end{bmatrix} + \begin{bmatrix}
        4(T-\theta_1) & 0 \\ 0 &  4(T-\theta_1)
    \end{bmatrix}
    \right\} =
    \l(4 - \frac{3\theta_1}{T}\r) \cdot I_2
\]
which always converges to some positive definite matrix, regardless of the order of $\theta_1$ relative to $T$.

\item Consider the following with $a \neq 0$ (otherwise there is no change):
\begin{align*}
    A_{1,1} &= \begin{bmatrix}
        1 & 0 \\ 0 & 1
    \end{bmatrix} \text{ \ and \ }
    A_{2,1} = \begin{bmatrix}
        1 & 0 \\ a & 1
    \end{bmatrix}.
\end{align*}
Then,
\begin{align*}
    \bar{\Sigma}_{A, k} &= \frac{1}{T}\left\{
    \begin{bmatrix}
        \theta_1 & 0 \\ 0 & \theta_1
    \end{bmatrix} + \begin{bmatrix}
        T-\theta_1 & a(T-\theta_1) \\ a(T-\theta_1) & (a^2 + 1)(T- \theta_1)
    \end{bmatrix}
    \right\}
    = \begin{bmatrix}
        1 & \frac{a(T-\theta_1)}{T} \\ \frac{a(T-\theta_1)}{T} & 1 + \frac{a^2(T - \theta_1)}{T}
    \end{bmatrix} ,
\end{align*}
which also always converges to some positive definite matrix. Indeed, it suffices to ensure the determinant is nonzero, and the determinant is $1+a^2 z(1-z)$ with $z=(T-\theta_1)/T \in [0,1]$.

\item We now demonstrate some scenarios where the linear spacing condition is necessary. Consider
\begin{align*}
    A_{1,1} &= \begin{bmatrix}
        1 & 0 \\ 0 & 0
    \end{bmatrix} \text{ \ and \ }
    A_{2,1} = \begin{bmatrix}
        0 & 0 \\ 0 & 1
    \end{bmatrix}.
\end{align*}
Then, for
\begin{align*}
    \bar{\Sigma}_{A, k} &= \frac{1}{T}\left\{
    \begin{bmatrix}
        \theta_1 & 0 \\ 0 & 0
    \end{bmatrix} + \begin{bmatrix}
        0 & 0 \\ 0 & T-\theta_1
    \end{bmatrix}
    \right\}
    = \begin{bmatrix}
        \theta_1/T & 0 \\ 0 &  1-\theta_1/T
    \end{bmatrix} ,
\end{align*}
to be positive definite, it requires $\min(\theta_1, T - \theta_1)$ to be of order $T$. 
\end{enumerate}
We now consider the case with $q = 2$:
\begin{enumerate}[label = (\alph*)]
\setcounter{enumi}{3}
\item Let
\begin{align*}
    A_{1,1} &= \begin{bmatrix}
        1 & 0 \\ 0 & 0
    \end{bmatrix} , \,
    A_{2,1} = \begin{bmatrix}
        0 & 0 \\ 0 & 1
    \end{bmatrix} \text{ \ and \ }
    A_{3,1} = \begin{bmatrix}
        1 & 0 \\ 0 & 1
    \end{bmatrix}.
\end{align*}
Then, we have
\begin{align*}
    \bar{\Sigma}_{A, k} &= \frac{1}{T}\left\{
    \begin{bmatrix}
        \theta_1 & 0 \\ 0 & 0
    \end{bmatrix} + \begin{bmatrix}
        0 & 0 \\ 0 & \theta_2-\theta_1
    \end{bmatrix} + \begin{bmatrix}
        T-\theta_2 & 0 \\ 0 & T-\theta_2
    \end{bmatrix}
    \right\}
    = \begin{bmatrix}
        1-\frac{(\theta_2-\theta_1)}{T} & 0 \\ 0 &  1-\frac{\theta_1}{T}
    \end{bmatrix},
\end{align*}
which can be positive definite even when $\theta_1/T=o(1)$ as long as $T- \theta_2 \asymp T$. 
\end{enumerate}

We conclude this subsection by presenting the following proposition which encompasses all the examples discussed previously.

\begin{proposition}\label{prop: nonlinear_spacing}
Let Assumptions~\ref{assum: trans_mat}~(i)--(ii) hold. Then given any index set of segments $\cQ \subset [q+1]$ such that $\sum_{j\in \cQ} (\theta_j - \theta_{j-1})/T =o(1)$, Assumption~\ref{assum: trans_mat}~(iii) is equivalent to
\[
\frac{1}{T} \sum_{j\in [q+1]\setminus\cQ} \big\| \otimes_{\ell = K, \, \ell \ne k}^1 A_{j,\ell} \big\|_F^2 \cdot (\theta_j - \theta_{j-1}) \cdot A_{j,k} A_{j,k}^\trans \to \Sigma_{A,k}^\ast ,
\]
for some positive definite matrix $\Sigma_{A,k}^\ast$.
\end{proposition}

\begin{proof}[Proof of Proposition~\ref{prop: nonlinear_spacing}]
Note that $\cQ$ can be empty, and the equivalence is trivial. Suppose that $\cQ$ is non-empty. Noting that $\cQ \ne [q+1]$ (otherwise it contradicts that $\sum_{j\in \cQ} (\theta_j - \theta_{j-1})/T =o(1)$), both directions are shown as follows.
\begin{itemize}[leftmargin=35pt]
    \item [``$\Rightarrow$'':] We show the implication using (in fact, the unique choice) $\Sigma_{A,k}^\ast =\Sigma_{A,k}$ from Assumption~\ref{assum: trans_mat}~(ii), which is direct by
    \begin{align*}
        \Sigma_{A,k} &= \lim_{T,p_1,\dots,p_K\to\infty} \frac{1}{T} \sum_{j\in [q+1]} \big\|\otimes_{\ell = K, \, \ell \ne k}^1 A_{j,\ell} \big\|_F^2 \cdot (\theta_j - \theta_{j-1}) \cdot A_{j,k} A_{j,k}^\trans \\
        &=
        \lim_{T,p_1,\dots,p_K\to\infty} \frac{1}{T} \sum_{j\in [q+1]\setminus \cQ} \big\|\otimes_{\ell = K, \, \ell \ne k}^1 A_{j,\ell} \big\|_F^2 \cdot (\theta_j - \theta_{j-1}) \cdot A_{j,k} A_{j,k}^\trans \\
        &\quad
        + \lim_{T,p_1,\dots,p_K\to\infty} \frac{1}{T} \sum_{j\in\cQ} \big\|\otimes_{\ell = K, \, \ell \ne k}^1 A_{j,\ell} \big\|_F^2 \cdot (\theta_j - \theta_{j-1}) \cdot A_{j,k} A_{j,k}^\trans \\
        &= \lim_{T,p_1,\dots,p_K\to\infty} \frac{1}{T} \sum_{j\in [q+1]\setminus \cQ} \big\|\otimes_{\ell = K, \, \ell \ne k}^1 A_{j,\ell} \big\|_F^2 \cdot (\theta_j - \theta_{j-1}) \cdot A_{j,k} A_{j,k}^\trans ,
    \end{align*}
    where the last equality used $\sum_{j\in \cQ} (\theta_j - \theta_{j-1})/T =o(1)$ and Assumption~\ref{assum: trans_mat}~(i).
    \item [``$\Leftarrow$'':] Similar to the above, we have
    \begin{align*}
        &\quad \frac{1}{T} \sum_{j\in [q+1]} \big\|\otimes_{\ell = K, \, \ell \ne k}^1 A_{j,\ell} \big\|_F^2 \cdot (\theta_j - \theta_{j-1}) \cdot A_{j,k} A_{j,k}^\trans \\
        &\to
        \Sigma_{A,k}^\ast + \frac{1}{T} \sum_{j\in\cQ} \big\|\otimes_{\ell = K, \, \ell \ne k}^1 A_{j,\ell} \big\|_F^2 \cdot (\theta_j - \theta_{j-1}) \cdot A_{j,k} A_{j,k}^\trans 
        =: \Sigma_{A,k}^\ast + \Sigma_{A,k}^\cQ.
    \end{align*}
    Combining $\sum_{j\in \cQ} (\theta_j - \theta_{j-1})/T =o(1)$ and Assumption~\ref{assum: trans_mat}~(i), we have $\|\Sigma_{A,k}^\cQ\|=o(1)$ and hence $\Sigma_{A,k} = \Sigma_{A,k}^\ast$ by Weyl's inequality and the definition of $\Sigma_{A,k}$ from Assumption~\ref{assum: trans_mat}~(ii). Thus Assumption~\ref{assum: trans_mat}~(iii) is true.
\end{itemize}
The equivalence is then concluded and so is the proof of this proposition.
\end{proof}

\subsubsection{Modification of TFMseg}
\label{sec:mod:tfmseg}

In the absence of Assumption~\ref{assum: trans_mat_alt}~(iv), it is required to expand the search space for the change points, to ensure that for each $\theta_j, \, j \in [q]$, there exists $(a_\ell, b_\ell] \in \M$ containing $\theta_j$ (and no other change point) whose length is sufficiently large for detecting the change point therein.
It is achieved by setting $\mu_T \asymp \log_2(T)$ and 
$\varpi_T \asymp \psi_T^2$, where $\psi_T = \max\{T^{2/(\nu-2)}, \sqrt{\log(T)}\}$
for $\nu > 6$, see Assumptions~\ref{assum: core_factor} and~\ref{assum: noise}.
Then, upon carefully inspecting the proof of Theorem~\ref{thm: asymp_consistency_detection}, provided that $\min_{j \in [q]} \omega_j^2 \Delta_j/\psi_T^2 \to \infty$, we have detection consistency ($\P\{ \wh q = q \} \to 1$) as well as
\[
\max_{j\in[q]} \omega_j^2 \vert \hat\theta_j-\theta_j \vert = \cO_P\l(\psi_T^2\r).
\]
This shows that Assumption~\ref{assum: trans_mat_alt}~(iv) enables considerably improving the theoretical performance of TFMseg, with Assumption~\ref{assum: signal-to-noise} that only requires $\min_{j \in [q]} \omega_j^2 \Delta_j/\log(T) \to \infty$, and the sharper rate of estimation.

\clearpage 

\section{Proofs}

We define some notations used throughout the proof in addition to those given in Introduction.

For a random variable $X$ and $\nu> 0$, write $\|X\|_\nu = [\E(|X|^\nu)]^{1/\nu}$.
Unless specified otherwise, we denote vectors, matrices and tensors by lower-case letters, capital letters, and calligraphic letters, i.e.\ $x$, $X$, and $\cX$, respectively. We also use $x_i, X_{ij}, X_{i\cdot}, X_{\cdot i}$ to denote, respectively, the $i$-th element of a vector $x$, the $(i,j)$-th element of $X$, the $i$-th row vector (as a column vector) of $X$, and the $i$-th column vector of $X$. 
We use $a\lesssim b$ to denote $a=\cO(b)$, $a\gtrsim b$ to denote $b=\cO(a)$, and $a\asymp b$ to denote $a=\cO(b)$ and $b=\cO(a)$; their versions of stochastic order, i.e.\ with $\cO(\cdot)$ replaced by $\cO_P(\cdot)$, are respectively denoted by $a\lesssim_P b$, $b\gtrsim_P a$, and $a \asymp_P b$.
The $i$-th largest eigenvalue of a matrix $A$ is denoted by $\lambda_i(A)$.
We write $\rank(A)$ for the rank of a matrix $A$. For a square matrix $A$, we write $\tr(A)$ for its trace.
Also, for a matrix $A \in \R^{m \times m}$, $\vec(A)$ denotes the vector in $\R^{mn}$ obtained by stacking the columns of $A$ in column-major order, with $|A|_a = \|\vec(A)\|_a$ for $a\in\{1,2,\infty\}$. 
We use $\otimes$ to represent the Kronecker product. By convention, the total Kronecker product for an index set is computed in descending index.
We write $r= \prod_{k=1}^K r_k$ and $\rmk = r/r_k$.

\subsection{Proof of Theorem~\ref{thm: identification_change}}

\begin{lemma}\label{lemma: AA=cBB}
Given any matrices $A,B\in \R^{r\times r}$, not necessarily of full rank, such that $AA^\trans = c BB^\trans$ for some nonzero scalar $c$, we can express $B=c^{-1/2} AQ$ for some orthogonal matrix $Q\in \R^{r\times r}$.
\end{lemma}

\begin{proof}[Proof of Lemma~\ref{lemma: AA=cBB}]
Both $AA^\trans$ and $BB^\trans$ are positive semi-definite, and the fact that $AA^\trans = c BB^\trans$ implies that they have the same eigenvalues (up to scaling by $c$) and eigenspaces. Let $A=U\Sigma V^\trans$ be the SVD of $A$, where $U$ and $V$ are orthogonal matrices and $\Sigma$ is diagonal. Then $AA^\trans = U\Sigma^2 U^\trans$, and hence $BB^\trans = c^{-1} U\Sigma^2 U^\trans$. The SVD of $B$ is then $B=c^{-1/2}U\Sigma W$ for some orthogonal matrix $W$, and with $AV=U\Sigma$, we conclude that $B=c^{-1/2} A(VW)$ where $VW$ is orthogonal, as required.
\end{proof}

\begin{lemma}\label{lemma: identifiable}
Under the model in~\eqref{eqn: tfm_change_rewrite} and Condition~\ref{cond:factor}, $A_{j,k}$ can only be identified up to an orthogonal transformation for all $j \in [q + 1]$ and $k \in [K]$.
\end{lemma}

\begin{proof}
Suppose there are two sets of parameters $\{\cF_t, A_{j,1}, \dots, A_{j,K}\}$ and $\{\wt\cF_t, \wt{A}_{j,1}, \dots, \wt{A}_{j,K}\}$ such that $\cF_t \times_{k=1}^K A_{j,k} =\wt\cF_t \times_{k=1}^K \wt{A}_{j,k}$. For identification, without loss of generality, fix $A_{j,k}A_{j,k}^\trans = \wt{A}_{j,k} \wt{A}_{j,k}^\trans = \Sigma_{j,k}$ for some positive semi-definite matrix $\Sigma_{j,k}$. Taking any mode-$k$ unfolding on both sides, we have
\[
A_{j,k} \mat_k(\cF_t) \Big( \otimes_{i = K, i \ne k}^1 A_{j,i} \Big)^\trans = \wt{A}_{j,k} \mat_k(\wt\cF_t) \Big( \otimes_{i = K, i \ne k}^1 \wt{A}_{j,i} \Big)^\trans.
\]
Thus, using Lemma~\ref{lemma: FMF_expectation},
\begin{align*}
    &\quad
    A_{j,k} \E\Big\{ \mat_k(\cF_t) \Big( \otimes_{i = K, i \ne k}^1 A_{j,i} \Big)^\trans \Big( \otimes_{i = K, i \ne k}^1 A_{j,i} \Big) \mat_k(\cF_t)^\trans \Big\} A_{j,k}^\trans \\
    &=
    \tr\Big\{ \Big( \otimes_{i = K, i \ne k}^1 A_{j,i} \Big)^\trans \Big( \otimes_{i = K, i \ne k}^1 A_{j,i} \Big) \Big\} \cdot A_{j,k} A_{j,k}^\trans \\
    &=
    \tr\Big\{ \Big( \otimes_{i = K, i \ne k}^1 \wt{A}_{j,i} \Big)^\trans \Big( \otimes_{i = K, i \ne k}^1 \wt{A}_{j,i} \Big) \Big\} \cdot \wt{A}_{j,k} \wt{A}_{j,k}^\trans .
\end{align*}
Under the trivial condition that any $A_{j,i}$ for $j \in [q + 1]$ and $i\in[K]$, is not a zero matrix, the above shows that we can write $\wt{A}_{j,k} \wt{A}_{j,k}^\trans = c A_{j,k} A_{j,k}^\trans$ for some nonzero constant $c$. Using Lemma~\ref{lemma: AA=cBB}, we have $\wt{A}_{j,k} = c^{1/2} A_{j,k} Q$ for some orthogonal matrix $Q$. Then we conclude that $c = 1$ since
\[
\Sigma_{j,k} = \wt{A}_{j,k} \wt{A}_{j,k}^\trans = c A_{j,k} A_{j,k}^\trans = c \Sigma_{j,k}.
\]
This shows that $A_{j,k}$ can only be identified up to an orthogonal transformation.
\end{proof}

\begin{proof}[Proof of Theorem~\ref{thm: identification_change}]
First, by inspecting the proof of Lemma~\ref{lemma: FMF_expectation}, we note that statement~\eqref{eqn: lemma_FMF_expectation_1} remains to hold without Assumption~\ref{assum: core_factor}---all we need is Condition~\ref{cond:factor}.
Also by Lemma~\ref{lemma: identifiable}, $A_{j,k}$ can only be identified up to an orthogonal transformation.

Next, for any $j\in[q]$, we show that a change of the form $A_{j,\ell} = c A_{j + 1,\ell}$, for some positive constant~$c$ (without loss of generality, due to the identification), is mode-unidentifiable. This is trivial by noting that we can write
\[
\cF_t \times_{k\in[K]\setminus\{\ell\}} A_{j,k} \times_\ell (c A_{j,\ell}) = \cF_t \times_{k\in[K]\setminus\{\ell\}} (c^{1/(K-1)} A_{j + 1,k}) \times_\ell A_{j,\ell}.
\]

Finally, let a change from $A_{j,k}$ to $A_{j+1,k}$ be an arbitrary mode-unidentifiable change. It remains to show the change is in the form of a scalar multiple. By Definition~\ref{def: classification_change}, there exists a set of parameters $\{B_\ell\}_{\ell \in [K] \setminus \{k\}}$ 
such that
\begin{align*}
    &\quad
    A_{j + 1,k} \E\Big\{ \mat_k(\cF_t) \Big( \otimes_{\ell \in[K] \setminus \{k\}} A_{j + 1, \ell} \Big)^\trans \Big( \otimes_{\ell \in[K] \setminus \{k\}} A_{j + 1, \ell} \Big) \mat_k(\cF_t)^\trans \Big\} A_{j + 1, k}^\trans \\
    &=
    \tr\Big\{ \otimes_{\ell \in[K] \setminus \{k\}} A_{j + 1, \ell} \Big)^\trans \Big( \otimes_{\ell \in[K] \setminus \{k\}} A_{j + 1, \ell} \Big) \Big\} \cdot A_{j + 1, k} A_{j + 1, k}^\trans \\
    &=
    \tr\Big\{ \Big( \otimes_{\ell \in [K] \setminus \{k\}} B_\ell \Big)^\trans \Big( \otimes_{\ell \in [K] \setminus \{k\}} B_\ell \Big) \Big\} \cdot A_{j, k} A_{j, k}^\trans ,
\end{align*}
which again follows from Lemma~\ref{lemma: FMF_expectation}. Hence by Lemma~\ref{lemma: AA=cBB}, we conclude that $A_{j + 1, k} = c A_{j, k} Q$ for some nonzero constant $c$ and orthogonal matrix $Q$, and this completes the proof of the theorem by noting that any $A_{j, k}$ is identified up to an arbitrary rotation.
\end{proof}


\subsection{Consistency in loading estimation}
\label{sec:pf:loading}

Throughout, we denote the global mode-$k$ covariance matrix for the pseudo-factor $\cG_t$ and its sample version, and the (scaled) covariance matrices for $\cX_t$ as follows.
\begin{align*}
&{\Gamma}_G^{(k)} = \frac{1}{T} \sum_{t=1}^T \E\Big\{ \mat_k(\cG_t) \mat_k(\cG_t)^\trans \Big\} ,  \quad    
\wh{\Gamma}_G^{(k)} = \frac{1}{T} \sum_{t=1}^T \mat_k(\cG_t) \mat_k(\cG_t)^\trans
,
\\
&{\Gamma}_X^{(k)} = \frac{1}{Tp} \sum_{t=1}^T \E\Big\{ \mat_k(\cX_t) \mat_k(\cX_t)^\trans \Big\} 
.
\end{align*}

\subsubsection{Supporting lemmas}

Throughout this subsection, anytime we refer to Assumptions~\ref{assum: core_factor} and~\ref{assum: noise}, we assume $(\nu,\beta)=(4,0)$.

\begin{lemma}\label{lemma: FMF_expectation}
Let Assumption~\ref{assum: core_factor} hold. Then for any $k\in[K]$, $t\in[T]$, and constant matrix $M\in\R^{\rmk\times \rmk}$, we have
\begin{equation}
\label{eqn: lemma_FMF_expectation_1}
\E\Big\{ \mat_k(\cF_t) \, M \, \mat_k(\cF_t)^\trans \Big\} = \tr(M) \, I_{r_k} .
\end{equation}
For any interval of timestamps $\cS$, if it further holds that $\|M\|_F^2= o(|\cS|)$, then
\[
\frac{1}{|\cS|} \sum_{t\in\cS} \mat_k(\cF_t) \, M \, \mat_k(\cF_t)^\trans \xrightarrow{P} \tr(M) \, I_{r_k} .
\]
\end{lemma}

\begin{proof}[Proof of Lemma~\ref{lemma: FMF_expectation}]
Consider the $(i,j)$-th entry of $\mat_k(\cF_t) \, M \, \mat_k(\cF_t)^\trans$. If $i\neq j$, then
\begin{align*}
    \E\Big[ \big\{ \mat_k(\cF_t) \, M \, \mat_k(\cF_t)^\trans \big\}_{i,j} \Big] = \E\Big[ \mat_k(\cF_t)_{i\cdot}^\trans \, M \, \mat_k(\cF_t)_{j\cdot} \Big] = 0,
\end{align*}
where the last equality used Assumption~\ref{assum: core_factor} that entries of $\cF_t$ are cross-sectionally uncorrelated. For the diagonal entries, i.e.\ $i=j$, we have
\begin{align*}
    &\quad
    \E\Big[ \big\{ \mat_k(\cF_t) \, M \, \mat_k(\cF_t)^\trans \big\}_{i,i} \Big] = \E\Big[ \mat_k(\cF_t)_{i\cdot}^\trans \, M \, \mat_k(\cF_t)_{i\cdot} \Big] \\
    &=
    \var(\cF_t)_{i,1} M_{1,1} +\dots+ \var(\cF_t)_{i,r_k} M_{r_k,r_k}
    = \tr(M),
\end{align*}
where the second line also used Assumption~\ref{assum: core_factor}. Hence~\eqref{eqn: lemma_FMF_expectation_1} is true.

To show the second result, it suffices to show the result for $\cS=[T]$ and from its proof the similar argument follows for any interval $\cS$. To this end, note that by Assumption~\ref{assum: core_factor}, there exist innovation $\cX_{f,t}$ and coefficients $a_{f,w}$ as in Definition~\ref{def: general_linear_tensor} such that $\cF_t = \sum_{w\geq 0} a_{f,w} \cX_{f,t-w}$. Then consider first that for any $t$,
\begin{equation}
\label{eqn: Ft_coef_bound}
\begin{split}
    &\quad
    \sum_{s=1}^T \Big(\sum_{w\geq 0} a_{f,w} a_{f,w-(t-s)} \Big) \Big(\sum_{z\geq 0} a_{f,z} a_{f,z-(t-s)} \Big) \\
    &\leq
    \Big(\sum_{s=1}^T \sum_{w\geq 0} |a_{f,w}| |a_{f,w-(t-s)}| \Big) \Big(\sum_{z\geq 0} |a_{f,z}| \Big) \cdot \max_z |a_{f,z}| =
    \cO(1) \cdot \Big( \sum_{w\geq 0} |a_{f,w}| \Big)^2 =\cO(1),
\end{split}
\end{equation}
where the last two equalities used the decaying coefficients in Definition~\ref{def: general_linear_tensor}. Similarly, it holds that
\begin{equation}
\label{eqn: Ft_coef_bound2}
\sum_{s=1}^T \sum_{w\geq 0} a_{f,w}^2 a_{f,w-(t-s)}^2 \leq \Big(\sum_{w\geq 0} a_{f,w}^2 \Big)^2 =\cO(1).
\end{equation}

Finally, with the notation $X_{f,t,(k)} =\mat_k(\cX_{f,t})$, for any $i,j\in [r_k]$,
\begin{align*}
    &\quad
    \var\Big( \frac{1}{T} \sum_{t=1}^T \mat_k(\cF_t)_{i\cdot}^\trans \, M \, \mat_k(\cF_t)_{j\cdot} \Big) \\
    &=
    \var\Big( \frac{1}{T} \sum_{t=1}^T \sum_{l=1}^{\rmk} \sum_{v=1}^{\rmk} \mat_k(\cF_t)_{il} M_{lv} \mat_k(\cF_t)_{jv} \Big) \\
    &=
    \frac{1}{T^2} \cov\Big( \sum_{t=1}^T \sum_{l=1}^{\rmk} \sum_{v=1}^{\rmk} \sum_{w\geq 0} \sum_{z\geq 0} a_{f,w} a_{f,z} X_{f,t-w,(k),il} M_{lv} X_{f,t-z,(k),jv}, \\
    &\quad \quad
    \sum_{s=1}^T \sum_{g=1}^{\rmk} \sum_{u=1}^{\rmk} \sum_{h\geq 0} \sum_{m\geq 0} a_{f,h} a_{f,m} X_{f,s-h,(k),ig} M_{gu} X_{f,s-m,(k),ju} \Big) \\
    &=
    \frac{1}{T^2} \sum_{t=1}^T \sum_{s=1}^T \Big(\sum_{w\geq 0} a_{f,w} a_{f,w-(t-s)} \Big) \Big(\sum_{z\geq 0} a_{f,z} a_{f,z-(t-s)} \Big) \Big(\sum_{l=1}^{\rmk} \sum_{v=1}^{\rmk} M_{lv}^2\Big) \\
    &\quad
    + \frac{1}{T^2} \sum_{t=1}^T \sum_{w\geq 0} \sum_{z\geq 0} \sum_{s=1}^T \sum_{h\geq 0} \sum_{m\geq 0} a_{f,w} a_{f,z} a_{f,h} a_{f,m} \Big(\sum_{l=1}^{\rmk} M_{ll}^2\Big) \\
    &\quad
    \cdot \cov\Big( X_{f,t-w,(k),il} X_{f,t-z,(k),il}, X_{f,s-h,(k),il} X_{f,s-m,(k),il} \Big) \\
    &=
    \cO\Big( \frac{1}{T}\Big) \cdot \Big(\sum_{l=1}^{\rmk} \sum_{v=1}^{\rmk} M_{lv}^2\Big) + \cO\Big( \frac{1}{T}\Big) \cdot \Big(\sum_{l=1}^{\rmk} M_{ll}^2\Big) = \cO\Big( \frac{1}{T}\Big) \cdot \|M \|_F^2 =o(1),
\end{align*}
where the third equality considered $i=j$ and $i\neq j$ separately, and the fourth used~\eqref{eqn: Ft_coef_bound} and~\eqref{eqn: Ft_coef_bound2}. Combining the above with~\eqref{eqn: lemma_FMF_expectation_1}, the proof of the lemma is completed.
\end{proof}

\begin{lemma}
\label{lemma: inequality_sandwich}
\begin{enumerate}[itemsep=0pt, label = (\roman*), left = 0pt]
    \item For any square matrix $M$ and sequence of matrices $A_t, \, t \in [T]$, of compatible dimensions, we have
    \[
    \left\| \sum_{t=1}^T A_t M A_t^\trans \right\|_F^2 \leq \|M\|^2 \cdot \left\| \sum_{t=1}^T A_t A_t^\trans \right\|_F^2 .
    \]
    Note that this inequality is tight, which can be seen by taking $M$ as the identity matrix.
    \item (Lemma~B.16 in \cite{Barigozzietal2025_robust}) For any sequence of matrices $A_t=[A_{t,ij}]$, $B_t$ for $t\in[T]$, and any matrix $M$ of compatible dimensions, we have
    \[
    \left\| \sum_{t=1}^T A_t M B_t^\trans \right\|_F^2 \leq \|M\|_F^2 \cdot \sum_{i,j} \left\| \sum_{t=1}^T A_{t,ij} B_t \right\|_F^2 .
    \]
\end{enumerate}
\end{lemma}

\begin{proof}[Proof of Lemma~\ref{lemma: inequality_sandwich}]
Part~(ii) is a direct quotation of Lemma~B.16 in \cite{Barigozzietal2025_robust}, so it is enough to show part~(i) below. Expanding the squared Frobenius norm, we have
\begin{equation}
\label{eqn: sum_AMA_decomp}
\begin{split}
    \left\| \sum_{t=1}^T A_t M A_t^\trans \right\|_F^2 &= \tr\left\{ \left( \sum_{t=1}^T A_t M A_t^\trans \right) \left( \sum_{s=1}^T A_s M^\trans A_s^\trans \right) \right\} = \sum_{t=1}^T \sum_{s=1}^T \tr(A_t M A_t^\trans A_s M^\trans A_s^\trans ) \\
    &=
    \sum_{t=1}^T \sum_{s=1}^T \tr(A_s^\trans A_t M A_t^\trans A_s M^\trans )
    =: \sum_{t=1}^T \sum_{s=1}^T \tr(A_{st} M A_{st}^\trans M^\trans ) ,
\end{split}
\end{equation}
where the second line used the cyclic property of trace, and the definition $A_{st}= A_s^\trans A_t$. For any $t,s\in[T]$, we may read $\tr(A_{st} M A_{st}^\trans M^\trans)$ as a Frobenius inner product and apply the Cauchy--Schwarz inequality such that
\begin{equation}
\label{eqn: trace_AMAM}
\tr(A_{st} M A_{st}^\trans M^\trans) = \langle A_{st} M, M A_{st} \rangle_F \leq \|A_{st} M\|_F \cdot \|M A_{st}\|_F \leq \|M\|^2 \|A_{st}\|_F^2 ,
\end{equation}
where the last inequality used the fact that for any matrix products $AB$, we have $\|AB\|_F=\sqrt{\sum_{i=1}^n \|A_{i\cdot}^\trans B\|^2} \leq \sqrt{\sum_{i=1}^n \|A_{i\cdot}^\trans\|^2 \|B\|^2} = \|A\|_F \cdot \|B\|$.

Combining~\eqref{eqn: sum_AMA_decomp} and~\eqref{eqn: trace_AMAM}, we have
\begin{align*}
    &\quad
    \left\| \sum_{t=1}^T A_t M A_t^\trans \right\|_F^2 = \sum_{t=1}^T \sum_{s=1}^T \tr(A_{st} M A_{st}^\trans M^\trans)
    \leq \|M\|^2 \cdot \sum_{t=1}^T \sum_{s=1}^T \|A_{st}\|_F^2 \\
     &=
     \|M\|^2 \sum_{t=1}^T \sum_{s=1}^T \tr(A_s^\trans A_t A_t^\trans A_s)
     = \|M\|^2 \cdot \tr\left\{ \left( \sum_{t=1}^T A_t A_t^\trans \right) \left( \sum_{s=1}^T A_s A_s^\trans \right) \right\}
     = \|M\|^2 \left\| \sum_{t=1}^T A_t A_t^\trans \right\|_F^2 ,
\end{align*}
as desired. This completes the proof of this lemma.
\end{proof}

For ease of notation, for any $t\in[T]$, $k\in[K]$, $j\in[q+1]$, denote
\begin{equation}
\label{eqn: set_of_notation}
\begin{split}
    & X_{k,t} = \mat_k(\cX_t), \quad
    G_{k,t} = \mat_k(\cG_t), \quad
    E_{k,t} = \mat_k(\cE_t), \quad
    F_{k,t} = \mat_k(\cF_t), \\
    & \Lambdamk = \otimes_{\ell\in[K]\setminus\{k\}} \Lambda_\ell , \quad
    \Ajmk{j} = \otimes_{\ell\in[K]\setminus\{k\}} A_{j,\ell} .
\end{split}
\end{equation}

\begin{lemma} 
\label{lemma: prelim_rate}
Under Assumptions~\ref{assum: core_factor}, \ref{assum: loadings}, \ref{assum: trans_mat}, and~\ref{assum: noise},
we have the following for any $k\in[K]$ and $j\in[q+1]$.
\begin{enumerate}[itemsep=0pt, label = (\roman*), left = 0pt]
    \item $\Big\| \sum_{t = a + 1}^b \Lambda_k G_{k,t} (\otimes_{\ell\in[K]\setminus\{k\}} \Lambda_\ell)^\trans E_{k,t}^\trans \Big\|_F^2 = 
    \cO_P\Big\{ (b-a) p_k^2 \pmk \Big\}$ for any integers $a\leq b$; in particular, $\Big\| \sum_{t = \theta_{j-1} + 1}^{ \theta_j} (\Lambda_k A_{j,k}) F_{k,t} (\otimes_{\ell\in[K]\setminus\{k\}} \Lambda_\ell A_{j,\ell})^\trans E_{k,t}^\trans \Big\|_F^2 = 
    \cO_P\Big\{ (\theta_j -\theta_{j-1}) p_k^2 \pmk \Big\}$;
    \item $\Big\| \sum_{t=1}^T E_{k,t} E_{k,t}^\trans \Big\|_F^2 = \cO_P\Big(Tp_k^2 \pmk + T^2 p_k \pmk^2 \Big)$, where in particular the first term is attributed to the bound on $\E\Big\{\big\|\big( \sum_{t=1}^T \big\{E_{k,t} E_{k,t}^\trans - \E(E_{k,t} E_{k,t}^\trans)\big\} \big)^2 \big\|_F\Big\}$, and the second term on $\Big\|\E\Big( \sum_{t=1}^T E_{k,t} E_{k,t}^\trans \Big)\Big\|_F^2$; moreover, $\Big\|\E\Big( \sum_{t=1}^T E_{k,t} E_{k,t}^\trans \Big)\Big\|^2 \leq c T^2 \pmk^2$ for some constant $c >0$;
    \item $\Big\| \sum_{i=1}^{p_k} \Lambda_{k,i\cdot} \sum_{t=1}^T E_{k,t,i\cdot}^\trans E_{k,t,z\cdot} \Big\|^2 = \cO_P\Big( Tp + T^2 \pmk^2 \Big)$ for any $z\in[p_k]$;
    \item $\Big\| \sum_{t=1}^T E_{k,t} VV^\trans E_{k,t}^\trans \Lambda_k \Big\|_F^2 = \cO_P\Big\{ \E(\|V\|_F^4) \Big( T p_k^2 + T^2 p_k \Big) \Big\}$ for any $V$ independent of $\cE_t$;
    \item $\Big\| \sum_{t=1}^T \Lambda_{k,i\cdot}^\trans G_{k,t} \Lambda_{\text{-}k,z\cdot} E_{k,t} \Big\|_F^2 =\cO_P(Tp)$ for any $i\in[p_k]$, $z\in[\pmk]$.
\end{enumerate}
\end{lemma}

\begin{proof}[Proof of Lemma~\ref{lemma: prelim_rate}]
For part~(i), we only show the special case
\[
\Bigg\| \sum_{t = \theta_{j-1} + 1}^{\theta_j} (\Lambda_k A_{j,k}) F_{k,t} (\otimes_{\ell\in[K]\setminus\{k\}} \Lambda_\ell A_{j,\ell})^\trans E_{k,t}^\trans \Bigg\|_F^2 = \cO_P\Big\{ (\theta_j -\theta_{j-1}) p_k^2 \pmk \Big\} ,
\]
while the general statement follows in exactly the same way. To do this, we apply the first result of Lemma~3 in \cite{CenLam2025_KronProd} with all factors being pervasive therein, observing that: (1) their results continue to hold even if the loading matrices have deficient column rank ($\Lambda_\ell A_{j,\ell}$ in our case); (2) the factor model is considered segment by segment here and hence the total length is replaced by each interval length $\theta_j -\theta_{j-1}$; and (3) all other related assumptions are satisfied by our Assumptions~\ref{assum: core_factor}, \ref{assum: loadings}, \ref{assum: noise} and \ref{assum: trans_mat}, except that we only need independent innovation process in our Definition~\ref{def: general_linear_tensor} but their results remain to hold from the proof. Then result is then direct. Similarly, the first claim in part~(ii) is by applying the second result of Lemma~3 in \cite{CenLam2025_KronProd}, together with its proof. For the second claim in part~(ii), note that
\[
\Big\|\E\Big( \sum_{t=1}^T E_{k,t} E_{k,t}^\trans \Big)\Big\|_1^2 = \Big\|\E\Big( \sum_{t=1}^T E_{k,t} E_{k,t}^\trans \Big)\Big\|_\infty^2 = \cO\Big(T^2 \pmk^2 \Big) ,
\]
where the second equality is direct from Proposition~1.1 in the supplement of \cite{CenLam2025}.
Then the claim holds by Hölder's inequality.

Consider part~(iii), and for simplicity, let $\cIV_{\Lambda,j} = \Big\| \sum_{i=1}^{p_k} \Lambda_{k,i\cdot} \sum_{t=1}^T E_{k,t,i\cdot}^\trans E_{k,t,j\cdot} \Big\|^2$. We will leverage the proof of Lemma~5 in \cite{CenLam2025_KronProd}, where their term $\cI_{3,j}$ is exactly our $\cIV_{\Lambda,j}$ scaled by $1/p_k$ due to the different normalization of the loading matrix. To verify the related conditions are fulfilled, note that our set of Assumption~\ref{assum: loadings} and Assumption~\ref{assum: noise} are the same as their Assumptions~(L1), (L2), (E1), and (E2), except that we only require independent innovation process in our Definition~\ref{def: general_linear_tensor} while they require \text{i.i.d.} innovation in their parallel conditions. However, from investigating their proof of Lemma~5, the results remain to hold if the innovations are not of the same distribution. Together with the fact that all our factors are strong and hence their Assumption~(R1) trivially holds, we may directly apply their result on $\cI_{3,j}$, so that as desired:
\[
\cIV_{\Lambda,j} = p_k \cdot \cO_P\Big( T\pmk + T^2 p_k^{-1} \pmk^2 \Big) = \cO_P\Big( Tp + T^2 \pmk^2 \Big) .
\]

Next, consider part~(iv). Note that by Assumption~\ref{assum: loadings}~(i), we have
\begin{equation}
\label{eqn: ELambdaLambdaE_Lambda}
\begin{split}
    \Big\| \sum_{t=1}^T E_{k,t} VV^\trans E_{k,t}^\trans \Lambda_k \Big\|_F^2 &= \sum_{i=1}^{p_k} \Big\| \sum_{j=1}^{p_k} \sum_{t=1}^T E_{k,t,i\cdot}^\trans VV^\trans E_{k,t,j\cdot} \Lambda_{k,j\cdot}^\trans \Big\|^2 \\
    &\lesssim
    \sum_{j=1}^{p_k} \sum_{i=1}^{p_k} \Big( \sum_{t=1}^T E_{k,t,i\cdot}^\trans VV^\trans E_{k,t,j\cdot} \Big)^2 .
\end{split}
\end{equation}
For any $i,j\in[p_k]$, consider first $\E\big( \sum_{t=1}^T E_{k,t,i\cdot}^\trans VV^\trans E_{k,t,j\cdot} \big)$. 
Under Assumption~\ref{assum: noise}, there exist innovation processes $\{\cX_{e,t}\}$, $\{\cX_{\epsilon,t}\}$ and coefficients $\{a_{e,w}\}$, $\{a_{\epsilon,h}\}$ as in Definition~\ref{def: general_linear_tensor} such that
\begin{equation}
\cE_t = \Big( \sum_{w\geq 0} a_{e,w} \cX_{e,t-w} \Big) \times_{k=1}^K A_{e,k} + \Sigma_{\epsilon} \circ \Big( \sum_{h\geq 0} a_{\epsilon,h} \cX_{\epsilon,t-h} \Big). \nonumber
\end{equation}
Then it holds immediately for any $k\in[K]$,
\begin{align*}
    & E_{k,t} = \sum_{w\geq 0} a_{e,w} A_{e,k} X_{e,t-w,k} \Aemk^\trans + \sum_{h\geq 0} a_{\epsilon,h} \Big( \Sigma_{\epsilon,k} \circ X_{\epsilon,t-h,k} \Big) , 
    \quad \text{where} \\
    X_{e,t-w,k} &= \mat_k(\cX_{e,t-w}), \;
    X_{\epsilon,t-h,k} = \mat_k(\cX_{\epsilon,t-h}), \;
    \Aemk = \otimes_{\ell\in [K]\setminus \{k\}} A_{e,\ell}, \;
    \Sigma_{\epsilon,k} = \mat_k(\Sigma_{\epsilon}) .
\end{align*}
Then we have
\begin{align*}
    &\quad
    \E\Big( \sum_{t=1}^T E_{k,t,i\cdot}^\trans VV^\trans E_{k,t,j\cdot} \Big) \\
    &=
    \E\Bigg\{ \sum_{t=1}^T \Big( \sum_{w\geq 0} a_{e,w} A_{e,k,i\cdot}^\trans X_{e,t-w,k} \Aemk^\trans + \sum_{h\geq 0} a_{\epsilon,h} \Sigma_{\epsilon,k,i\cdot}^\trans \circ X_{\epsilon,t-h,k,i\cdot}^\trans \Big) VV^\trans \\
    &\quad
    \cdot \Big( \sum_{w\geq 0} a_{e,w} \Aemk X_{e,t-w,k}^\trans A_{e,k,j\cdot} + \sum_{h\geq 0} a_{\epsilon,h} \Sigma_{\epsilon,k,j\cdot} \circ X_{\epsilon,t-h,k,j\cdot} \Big) \Bigg\} \\
    &=
    \E\Bigg( \sum_{t=1}^T \sum_{w\geq 0} \sum_{m\geq 0} a_{e,w} a_{e,m} A_{e,k,i\cdot}^\trans X_{e,t-w,k} \Aemk^\trans VV^\trans \Aemk X_{e,t-m,k}^\trans A_{e,k,j\cdot} \Bigg) \\
    &\quad
    + \E\Bigg( \sum_{t=1}^T \sum_{h\geq 0} \sum_{\ell\geq 0} a_{\epsilon,h} a_{\epsilon,\ell} (\Sigma_{\epsilon,k,i\cdot}^\trans \circ X_{\epsilon,t-h,k,i\cdot}^\trans) VV^\trans (\Sigma_{\epsilon,k,j\cdot} \circ X_{\epsilon,t-\ell,k,j\cdot}) \Bigg) \\
    &\lesssim
    \E[\tr(\Aemk^\trans VV^\trans \Aemk)] \cdot T \cdot A_{e,k,i\cdot}^\trans A_{e,k,j\cdot}
    + \E[\tr(VV^\trans)] \cdot T \cdot \mathbbm{I}_{\{i=j\}} ,
\end{align*}
so that together with the sparsity of $A_{e,\ell}$, $\ell\in[K]$, from Assumption~\ref{assum: noise},
\begin{equation}
\label{eqn: ELambdaLambdaE_mean}
\sum_{j=1}^{p_k} \sum_{i=1}^{p_k} \Big\{ \E\Big( \sum_{t=1}^T E_{k,t,i\cdot}^\trans VV^\trans E_{k,t,j\cdot} \Big) \Big\}^2 
\lesssim T^2 \cdot [\E(\|V\|_F^2)]^2 \cdot p_k .
\end{equation}
Next, for any $i,j\in[p_k]$, consider $\var\big( \sum_{t=1}^T E_{k,t,i\cdot}^\trans VV^\trans E_{k,t,j\cdot} \big)$ and we may simplify
\begin{align*}
    &\quad
    \var\Big( \sum_{t=1}^T E_{k,t,i\cdot}^\trans VV^\trans E_{k,t,j\cdot} \Big)
    = \var\Big( \sum_{t=1}^T \sum_{\ell=1}^{\pmk} \sum_{h=1}^{\pmk} E_{k,t,i\ell} (VV^\trans)_{\ell h} E_{k,t,jh} \Big) \\
    &=
    \var\Big\{ \sum_{t=1}^T \sum_{\ell=1}^{\pmk} \sum_{h=1}^{\pmk} (VV^\trans)_{\ell h} \Big( \sum_{w\geq 0} a_{e,w} A_{e,k,i\cdot}^\trans X_{e,t-w,k} A_{e,\text{-}k,\ell\cdot} + \sum_{m\geq 0} a_{\epsilon,m} \Sigma_{\epsilon,k,i\ell} X_{\epsilon,t-m,k,i\ell} \Big) \\
    &\quad
    \cdot \Big( \sum_{c\geq 0} a_{e,c} A_{e,k,j\cdot}^\trans X_{e,t-c,k} A_{e,\text{-}k,h\cdot} + \sum_{z\geq 0} a_{\epsilon,z} \Sigma_{\epsilon,k,jh} X_{\epsilon,t-z,k,jh} \Big) \Big\} \\
    &\lesssim
    T \cdot \sum_{\ell=1}^{\pmk} \sum_{h=1}^{\pmk} \E[(VV^\trans)_{\ell h}^2] \left( \sum_{w\geq 0} a_{e,w}^2 \left\Vert A_{e,k,i\cdot}\right\Vert^{2} \left\Vert A_{e,\text{-}k,\ell\cdot}\right\Vert^{2} + \sum_{m\geq 0} a_{\epsilon,m}^2 \Sigma_{\epsilon,k,i\ell}^{2} \right) \\
    &\quad
    \cdot \left( \sum_{c\geq 0} a_{e,c}^2 \left\Vert A_{e,k,j\cdot}\right\Vert^{2} \left\Vert A_{e,\text{-}k,h\cdot}\right\Vert^{2} + \sum_{z\geq 0} a_{\epsilon,z}^2 \Sigma_{\epsilon,k,jh}^{2} \right)
    \lesssim T \cdot \E(\|VV^\trans\|_F^2)
    \leq T \cdot \E(\|V\|_F^4) .
\end{align*}
Hence, we have
\begin{equation}
\label{eqn: ELambdaLambdaE_var}
\begin{split}
    &\quad
    \sum_{j=1}^{p_k} \sum_{i=1}^{p_k} \var\Big( \sum_{t=1}^T E_{k,t,i\cdot}^\trans VV^\trans E_{k,t,j\cdot} \Big)
    \lesssim T p_k^2 \cdot \E(\|V\|_F^4) .
\end{split}
\end{equation}

Combining~\eqref{eqn: ELambdaLambdaE_Lambda},~\eqref{eqn: ELambdaLambdaE_mean} and~\eqref{eqn: ELambdaLambdaE_var}, we have the desired:
\[
\E\Big(\Big\| \sum_{t=1}^T E_{k,t} VV^\trans E_{k,t}^\trans \Lambda_k \Big\|_F^2 \Big)
\lesssim
T^2 p_k \cdot [\E(\|V\|_F^2)]^2 + T p_k^2 \cdot \E(\|V\|_F^4) 
\leq
(T^2 p_k + T p_k^2) \cdot \E(\|V\|_F^4) .
\]
This completes the proof of claim (iv). For part~(v), it is similar to part~(i) and direct from the proof of Lemma~3 in \cite{CenLam2025_KronProd}.
\end{proof}

\begin{lemma}\label{lemma: eigenvalue_PCA_consistency}
We have the following. 
\begin{enumerate}[itemsep=0pt, label = (\roman*), left = 0pt]
\item Let Assumption~\ref{assum: trans_mat_alt}~(i)--(ii) hold. Then Assumption~\ref{assum: trans_mat_alt}~(iii)--(iv) implies Assumption~\ref{assum: trans_mat}~(iii). In fact, conditional on Assumption~\ref{assum: trans_mat_alt}~(iv), its part~(iii) is a necessary and sufficient condition for $\Sigma_{A,k}$ to be positive definite.
    \item Let Assumptions~\ref{assum: core_factor}, \ref{assum: loadings} and~\ref{assum: noise} hold. Further let Assumption~\ref{assum: trans_mat} 
    hold. For any $k\in[K]$, let $\wt{D}_k = \wt\Lambda_k^\trans \wh{\Gamma}_X^{(k)} \wt\Lambda_k /p_k$ be the diagonal matrix consisting of the leading $r_k$ eigenvalues of $\wh{\Gamma}_X^{(k)}$. Then as $\min\{T,p_1,\dots, p_K\} \to \infty$, the eigenvalues of $\wt{D}_k$ satisfy $\lambda_i(\wt{D}_k) = \lambda_i(\Sigma_{A,k}) + o_P(1)$ for any $i\leq r_k$.
\end{enumerate}
\end{lemma}

\begin{proof}[Proof of Lemma~\ref{lemma: eigenvalue_PCA_consistency}]
We first show result~(i), for which it is sufficient to show that given Assumption~\ref{assum: trans_mat_alt}~(iv), its part~(iii) is a necessary and sufficient condition for $\Sigma_{A,k}$ to be positive definite. To see this, note
\begin{equation}
\label{eqn: linear_combination_AA}
\frac{1}{T} \sum_{j=1}^{q+1} \|\Ajmk{j}\|_F^2 \cdot (\theta_j - \theta_{j-1}) \cdot A_{j,k} A_{j,k}^\trans = \sum_{j=1}^{q+1} c_j A_{j,k} A_{j,k}^\trans
\end{equation}
is a linear combination of $A_{j,k} A_{j,k}^\trans$ over $j\in[q+1]$ where $c_j =\|\Ajmk{j}\|_F^2 \cdot (\theta_j - \theta_{j-1})/T$ which is a positive constant. Then conditional on part~(iv) in Assumption~\ref{assum: trans_mat_alt}, we have Assumption~\ref{assum: trans_mat_alt}~(iii) if and only if Assumption~\ref{assum: trans_mat}~(iii) holds, as shown below.
\begin{itemize}[leftmargin=35pt]
    \item [``$\Rightarrow$'':] Fix any $k\in[K]$, the quadratic form of~\eqref{eqn: linear_combination_AA} satisfies, for any nonzero $u\in \R^{r_k}$, 
    \begin{align*}
    &\quad u^\trans \Big( \sum_{j=1}^{q+1} c_j A_{j,k} A_{j,k}^\trans \Big) u \\
    &=
    u^\trans [A_{1,k}, A_{2,k}, \dots, A_{q + 1,k}] \, \diag(c_1 I_{r_k}, c_2 I_{r_k}, \dots, c_{q + 1} I_{r_k}) \, \big\{ u^\trans [A_{1,k}, A_{2,k}, \dots, A_{q + 1,k}] \big\}^\trans ,
    \end{align*}
    which is positive as $\min\{T,p_1,\dots,p_K\} \to\infty$, since $[A_{1,k}, \dots, A_{q + 1,k}] \in \R^{r_k \times (q+1)r_k}$ is asymptotically of full row rank from Assumption~\ref{assum: trans_mat_alt}~(iii), and $c_j$ is uniformly bounded away from zero over $j\in[q+1]$ from Assumption~\ref{assum: trans_mat}~(i) and Assumption~\ref{assum: trans_mat_alt}~(iv). Thus, $\Sigma_{A,k}$ is positive definite.
    \item [``$\Leftarrow$'':] As $\min\{T,p_1,\dots,p_K\} \to\infty$, if $[A_{1,k}, \dots, A_{q + 1,k}]$ does not have full row rank, then there must exists some nonzero $u\in \R^{r_k}$ making the quadratic form in the above ``$\Rightarrow$'' part zero, and hence $\Sigma_{A,k}$ cannot be positive definite.
\end{itemize}

Next, we show result~(ii) of the lemma. By~\eqref{eqn: tfm_change_rewrite} and the definition of $\wh{\Gamma}_X^{(k)}$, we have
\begin{equation}
\label{eqn: whGamma_X_decomp}
\begin{split}
    \wh{\Gamma}_X^{(k)} &= \frac{1}{Tp} \sum_{t=1}^T X_{k,t} X_{k,t}^\trans
    = \frac{1}{Tp} \sum_{t=1}^T (\Lambda_k G_{k,t} \Lambdamk^\trans + E_{k,t}) (\Lambda_k G_{k,t} \Lambdamk^\trans + E_{k,t})^\trans \\
    &=
    \frac{1}{Tp} \sum_{t=1}^T \Lambda_k G_{k,t} \Lambdamk^\trans \Lambdamk G_{k,t}^\trans \Lambda_k^\trans + \frac{1}{Tp} \sum_{t=1}^T \Lambda_k G_{k,t} \Lambdamk^\trans E_{k,t}^\trans \\
    &\quad
    + \frac{1}{Tp} \sum_{t=1}^T E_{k,t} \Lambdamk G_{k,t}^\trans \Lambda_k^\trans + \frac{1}{Tp} \sum_{t=1}^T E_{k,t} E_{k,t}^\trans
    =:
    \cI + \cII + \cIII + \cIV .
\end{split}
\end{equation}
We first show find the eigenvalues of $\cI$. Using~\eqref{eqn: tfm_change_rewrite}, any mode-$k$ unfolding of $\cG_t$ is read as
\begin{align*}
    G_{k,t} &= \sum_{j=1}^{q+1} A_{j,k} F_{k,t} \Ajmk{j}^\trans \cdot \mathbb{I}_{\{\theta_{j-1} <t\le \theta_j\}} ,
\end{align*}
so that immediately, as $\min\{T,p_1,\dots,p_K\} \to\infty$,
\begin{align}
    &\quad
    \frac{1}{T\pmk} \sum_{t=1}^T G_{k,t} \Lambdamk^\trans \Lambdamk G_{k,t}^\trans = \frac{1}{T\pmk} \sum_{t=1}^T \sum_{j=1}^{q+1} A_{j,k} F_{k,t} \Ajmk{j}^\trans \Lambdamk^\trans \Lambdamk \Ajmk{j} F_{k,t}^\trans A_{j,k}^\trans \cdot \mathbb{I}_{\{\theta_{j-1} <t\le \theta_j\}} \nonumber \\
    &=
    \frac{1}{T} \sum_{j=1}^{q+1} (\theta_j - \theta_{j-1}) A_{j,k} \Big\{ \frac{1}{\theta_j - \theta_{j-1}} \sum_{t = \theta_{j-1} + 1}^{\theta_j} F_{k,t} \Ajmk{j}^\trans \Big( \frac{\Lambdamk^\trans \Lambdamk}{\pmk} \Big) \Ajmk{j} F_{k,t}^\trans \Big\} A_{j,k}^\trans \nonumber \\
    &\xrightarrow{P}
    \frac{1}{T} \sum_{j=1}^{q+1} \|\Ajmk{j}\|_F^2 \cdot (\theta_j - \theta_{j-1}) \cdot A_{j,k} A_{j,k}^\trans
    \to \Sigma_{A,k},
\label{eqn: GLambdaLambdaG}
\end{align}
where the last line used Assumption~\ref{assum: loadings}, Assumption~\ref{assum: trans_mat}~(i)--(ii), and Lemma~\ref{lemma: FMF_expectation}. Together with Assumption~\ref{assum: trans_mat}~(iii) (or Assumption~\ref{assum: trans_mat_alt}~(iii)--(iv), according to part~(a) of this lemma), we have
\[
\cI \xrightarrow{P} \frac{\Lambda_k \Sigma_{A,k} \Lambda_k^\trans}{p_k} ,
\]
which has the leading $r_k$ eigenvalues asymptotically equal to those of $\Sigma_{A,k}$ by Assumption~\ref{assum: loadings} again. Hence $\lambda_i(\cI) = \lambda_i(\Sigma_{A,k}) + o_P(1)$ for all $i\leq r_k$, while $\lambda_i(\cI) =0$ for $i> r_k$ by $\rank(\cI)\leq r_k$.

Next, consider the squared Frobenius norm of $\cII$ in~\eqref{eqn: whGamma_X_decomp}. If we re-write
\begin{align*}
    \|\cII\|_F^2 &= \Big\| \frac{1}{Tp} \sum_{t=1}^T \Lambda_k G_{k,t} \Lambdamk^\trans E_{k,t}^\trans \Big\|_F^2 \\
    &=
    \Big\| \frac{1}{Tp} \sum_{t=1}^T \sum_{j=1}^{q+1} (\Lambda_k A_{j,k}) F_{k,t} (\otimes_{\ell\in[K]\setminus\{k\}} \Lambda_\ell A_{j,\ell})^\trans E_{k,t}^\trans \cdot \mathbb{I}_{\{\theta_{j-1} <t\le \theta_j\}} \Big\|_F^2 \\
    &=
    \frac{1}{(Tp)^2} \sum_{j=1}^{q+1} \Big\| \sum_{t = \theta_{j-1} + 1}^{\theta_j} (\Lambda_k A_{j,k}) F_{k,t} (\otimes_{\ell\in[K]\setminus\{k\}} \Lambda_\ell A_{j,\ell})^\trans E_{k,t}^\trans \Big\|_F^2 \\
    &=
    \frac{1}{(Tp)^2} \sum_{j=1}^{q+1} \cO_P\Big\{ (\theta_j -\theta_{j-1}) p_k^2 \pmk \Big\} = \cO_P\Big\{ \frac{Tp p_k}{(Tp)^2} \Big\} = \cO_P\Big( \frac{1}{T\pmk} \Big) ,
\end{align*}
where the last line used Lemma~\ref{lemma: prelim_rate}~(i). The above is also the same rate for $\|\cIII\|_F^2$ by exactly the same arguments. Similarly, we have
\begin{align*}
    \|\cIV\|_F^2 &= \Big\| \frac{1}{Tp} \sum_{t=1}^T E_{k,t} E_{k,t}^\trans \Big\|_F^2 = \frac{1}{(Tp)^2} \Big\| \sum_{t=1}^T E_{k,t} E_{k,t}^\trans \Big\|_F^2 \\
    &=
    \frac{1}{(Tp)^2} \cdot \cO_P\Big(Tp_k^2 \pmk + T^2 p_k \pmk^2 \Big)
    = \cO_P\Big( \frac{1}{T\pmk} + \frac{1}{p_k} \Big) ,
\end{align*}
where the last line used Lemma~\ref{lemma: prelim_rate}~(ii). Combining all the above and applying Weyl's inequality iteratively, the first $r_k$ eigenvalues of $\wh{\Gamma}_X^{(k)}$ are led by those of $\cI$, among the decomposition in~\eqref{eqn: whGamma_X_decomp}. This completes the proof of the lemma.
\end{proof}

\begin{lemma}\label{lemma: consistency_PCA}
Let all the assumptions in Lemma~\ref{lemma: eigenvalue_PCA_consistency}~(ii) hold. For any $k\in[K]$, with $\wt{D}_k$ from Lemma~\ref{lemma: eigenvalue_PCA_consistency}, define
\[
\wt{H}_k =\frac{1}{Tp} \Big(\sum_{t=1}^T \mat_k(\cG_t) \Lambdamk^\trans \Lambdamk \mat_k(\cG_t)^\trans \Big) \Lambda_k^\trans \wt\Lambda_k \wt{D}_k^{-1} .
\]
Then $\wt{H}_k$ is asymptotically invertible as $\min\{T,p_1,\dots, p_K\} \to \infty$ and, moreover, satisfies $\wt{H}_k^\trans \wt{H}_k = I_{r_k} + \cO_P( 1/\sqrt{T\pmk} + 1/p_k)$. Also, the initial estimator $\wt\Lambda_k$ satisfies that
\[
\frac{1}{p_k}\Big\| \wt\Lambda_k - \Lambda_k \wt{H}_k \Big\|_F^2 = \cO_P\Big( \frac{1}{T\pmk} + \frac{1}{p_k^2} \Big) .
\]
\end{lemma}

\begin{proof}[Proof of Lemma~\ref{lemma: consistency_PCA}]
By the definitions of eigenvalues and eigenvectors, we have $\wt\Lambda_k \wt{D}_k =\wh{\Gamma}_X^{(k)} \wt\Lambda_k$ and hence $\wt\Lambda_k =\wh{\Gamma}_X^{(k)} \wt\Lambda_k \wt{D}_k^{-1}$ by the fact that $\wt{D}_k$ is asymptotically invertible from Lemma~\ref{lemma: eigenvalue_PCA_consistency}, which also spells out $\|\wt{D}_k\|_F=\cO_P(1)$ and $\|\wt{D}_k^{-1}\|_F =\cO_P(1)$. Moreover, note that with $\Lambdamk^\trans \Lambdamk$ replaced by $\wtLambdamk^\trans \wtLambdamk$ in~\eqref{eqn: GLambdaLambdaG}, the arguments therein remain exactly the same
and hence for the matrix $\wt{H}_k$, we have $\|\wt{H}_k\|_F \asymp_P \|\Lambda_k^\trans \wt\Lambda_k\|_F /p_k = \cO_P(1)$ which used $\|\Lambda_k\|_F =\cO(p_k^{1/2})$ from Assumption~\ref{assum: loadings} and $\|\wt\Lambda_k\|_F =\cO_P(p_k^{1/2})$ from its definition. All the above auxiliary results are summarized below: 
\begin{equation}
\label{eqn: aux_rates}
\textcolor{black}{
\max_{1 \le k \le K} \l\{
    \|\wt{D}_k\|_F ,\; \|\wt{D}_k^{-1}\|_F ,\; \|\wt{H}_k\|_F \r\} = \cO_P(1)
    \text{ \ and \ }
    \max_{1 \le k \le K} \l\{ \|\Lambda_k\|_F ,\; \|\wt\Lambda_k\|_F \r\} = \cO_P(p_k^{1/2}).
}
\end{equation}
With~\eqref{eqn: whGamma_X_decomp}, we have $\wt\Lambda_k - \Lambda_k \wt{H}_k = (\cII + \cIII + \cIV) \wt\Lambda_k \wt{D}_k^{-1}$. Thus using~\eqref{eqn: aux_rates},
\begin{equation}
\label{eqn: PCA_decomp}
\begin{split}
    &\quad \frac{1}{p_k}\Big\| \wt\Lambda_k - \Lambda_k \wt{H}_k \Big\|_F^2 = \frac{1}{p_k} \Big\|(\cII + \cIII + \cIV) \wt\Lambda_k \wt{D}_k^{-1} \Big\|_F^2 \\
    &\lesssim_P 
    \frac{1}{p_k} \Big\|(\cII + \cIII + \cIV) \Lambda_k \wt{H}_k \Big\|_F^2 + \frac{1}{p_k} \Big\|(\cII + \cIII + \cIV) (\wt\Lambda_k - \Lambda_k \wt{H}_k) \Big\|_F^2 \\
    &=
    \cO_P\Big\{ \frac{1}{p_k} \Big\|(\cII + \cIII + \cIV) \Lambda_k \Big\|_F^2 \Big\} + o_P\Big( \frac{1}{p_k}\Big\| \wt\Lambda_k - \Lambda_k \wt{H}_k \Big\|_F^2 \Big) ,
\end{split}
\end{equation}
where the last equality used the rates for $\|\cII\|_F^2$, $\|\cIII\|_F^2$, and $\|\cIV\|_F^2$ in the proof of Lemma~\ref{lemma: eigenvalue_PCA_consistency}.

For the first term in the last line of~\eqref{eqn: PCA_decomp}, note immediately by the rates for $\|\cII\|_F^2$ and $\|\cIII\|_F^2$ in the proof of Lemma~\ref{lemma: eigenvalue_PCA_consistency} and~\eqref{eqn: aux_rates} that
\begin{equation}
\label{eqn: cII_cIII_Lambda}
\frac{1}{p_k} \Big\|(\cII + \cIII) \Lambda_k \Big\|_F^2 \lesssim \Big(\|\cII\|_F^2 + \|\cIII\|_F^2\Big) \cdot \frac{1}{p_k} \|\Lambda_k\|_F^2 = \cO_P\Big( \frac{1}{T\pmk} \Big) .
\end{equation}
Next, we study $\|\cIV \cdot \Lambda_k\|_F^2 /p_k$. To this end, we write
\begin{equation}
\label{eqn: cIV_Lambda}
\begin{split}
    \frac{1}{p_k} \|\cIV  \cdot \Lambda_k \|_F^2 &= \frac{1}{p_k} \Big\| \frac{1}{Tp} \sum_{t=1}^T E_{k,t} E_{k,t}^\trans \Lambda_k \Big\|_F^2
    = \frac{1}{p_k} \sum_{z=1}^{p_k} \Big\| \frac{1}{Tp} \sum_{t=1}^T (E_{k,t} E_{k,t}^\trans \Lambda_{k})_{z\cdot} \Big\|^2 \\
    &=
    \frac{1}{(Tp)^2 p_k} \sum_{z=1}^{p_k} \Big\| \sum_{i=1}^{p_k} \Lambda_{k,i\cdot} \sum_{t=1}^T E_{k,t,i\cdot}^\trans E_{k,t,z\cdot} \Big\|^2 \\
    &=
    \frac{1}{(Tp)^2 p_k} \sum_{z=1}^{p_k} \cO_P\Big( Tp + T^2 \pmk^2 \Big)
    = \cO_P\Big( \frac{1}{Tp} + \frac{1}{p_k^2} \Big).
\end{split}
\end{equation}
where the second last equality used Lemma~\ref{lemma: prelim_rate}~(iii).

Combining~\eqref{eqn: PCA_decomp},~\eqref{eqn: cII_cIII_Lambda} and~\eqref{eqn: cIV_Lambda}, we conclude $\| \wt\Lambda_k - \Lambda_k \wt{H}_k \|_F^2 /p_k =\cO_P(1/(T\pmk) + 1/p_k^2)$, as desired. To end the proof of the lemma, it remains to show $\wt{H}_k \wt{H}_k^\trans = I_{r_k} + \cO_P( 1/\sqrt{T\pmk} + 1/p_k)$ and hence $\wt{H}_k^\trans \wt{H}_k \xrightarrow{P} I_{r_k}$, which follows from that
\begin{align*}
    I_{r_k} &= \frac{1}{p_k} \wt\Lambda_k^\trans \wt\Lambda_k = \frac{1}{p_k} \wt\Lambda_k^\trans (\wt\Lambda_k - \Lambda_k \wt{H}_k) + \frac{1}{p_k} \wt\Lambda_k^\trans \Lambda_k \wt{H}_k \\
    &=
    \frac{1}{p_k} \wt\Lambda_k^\trans (\wt\Lambda_k - \Lambda_k \wt{H}_k) + \frac{1}{p_k} (\wt\Lambda_k - \Lambda_k \wt{H}_k)^\trans \Lambda_k \wt{H}_k + \frac{1}{p_k} \wt{H}_k^\trans \Lambda_k^\trans \Lambda_k \wt{H}_k \\
    &=
    \wt{H}_k^\trans \wt{H}_k + \cO_P\Big( \frac{1}{\sqrt{T\pmk}} + \frac{1}{p_k} \Big),
\end{align*}
where the last equality used Assumption~\ref{assum: loadings},~\eqref{eqn: aux_rates} and the rate of $\|\wt\Lambda_k - \Lambda_k \wt{H}_k\|_F$.
\end{proof}

\begin{lemma}\label{lemma: population_cov_PCA_consistency}
Let all the assumptions in Lemma~\ref{lemma: eigenvalue_PCA_consistency}~(ii) hold. For any $k\in[K]$, with $\cI$ from~\eqref{eqn: whGamma_X_decomp}, define
\begin{equation}
\label{eqn: DdotGamma_X_decomp}
\begin{split}
    \Ddot{\Gamma}_X^{(k)} &= \frac{1}{Tp} \sum_{t=1}^T \Lambda_k G_{k,t} \Lambdamk^\trans \Lambdamk G_{k,t}^\trans \Lambda_k^\trans + \frac{1}{Tp} \sum_{t=1}^T \E(E_{k,t} E_{k,t}^\trans)
    =: \cI + {\Gamma}_E^{(k)} .
\end{split}
\end{equation}
Let $\Ddot{\Lambda}_k$ be $\sqrt{p_k}$ times the $r_k$ leading eigenvectors of $\Ddot{\Gamma}_X^{(k)}$. Then as $\min\{T,p_1,\dots, p_K\} \to \infty$, there exists orthogonal matrix $J_k\in \R^{r_k\times r_k}$ such that
\[
\frac{1}{p_k} \Big\| \wt\Lambda_k - \Ddot{\Lambda}_k J_k \Big\|_F^2 = \cO_P\Big(\frac{1}{T\pmk} \Big).
\]
\end{lemma}

\begin{proof}[Proof of Lemma~\ref{lemma: population_cov_PCA_consistency}]
For our proof, we will adapt Lemma~B.4~(i) in \cite{Barigozzietal2025_robust}, by equating our $p_k \wh{\Gamma}_X^{(k)}$, $p_k \Ddot{\Gamma}_X^{(k)}$, and $p_k \cI$ as their $\wh{\mathbf{S}}$, $\wt{\mathbf{S}}$, and $\mathbf{S}$, respectively. Instead of verifying the exact condition (C1) in \cite{Barigozzietal2025_robust}, we note that the eigenvalues required therein need not be necessarily distinct, in which case the diagonal matrix $\wt{\mathbf{J}}$ in their Lemma~B.4~(i) is compromised to some orthogonal matrix, say matrix $J_k$. It suffices to show $p_k^{-1}$ times the first $r_k$ eigenvalues of $p_k \cI$, or equivalently, the first $r_k$ eigenvalues of $\cI$, are bounded away from zero and infinity, which is true from the proof of Lemma~\ref{lemma: eigenvalue_PCA_consistency}~(ii).

It remains to verify the conditions (C2) and (C3) in Lemma~B.4 in \cite{Barigozzietal2025_robust}. For the condition (C2), using the decomposition of $\wh{\Gamma}_X^{(k)}$ in~\eqref{eqn: whGamma_X_decomp}, we have
\begin{align*}
\begin{split}
    p_k^{-2} \Big\| p_k\wh{\Gamma}_X^{(k)} - p_k \Ddot{\Gamma}_X^{(k)} \Big\|_F^2 &= \Bigg\| \frac{1}{Tp} \sum_{t=1}^T \Big\{ E_{k,t} E_{k,t}^\trans - \E(E_{k,t} E_{k,t}^\trans) \Big\}
    + \cII + \cIII \Bigg\|_F^2 \\
    &=: \Big\| \Ddot\cIV + \cII + \cIII \Big\|_F^2.
\end{split}
\end{align*}
By the first claim in Lemma~\ref{lemma: prelim_rate}~(ii), we have $\|\Ddot\cIV\|_F^2 = \cO_P(1/T\pmk)$, the rate of which applies also to $\|\cII\|_F^2$ and $\|\cIII\|_F^2$ from the proof of Lemma~\ref{lemma: eigenvalue_PCA_consistency}. Thus, the term $\zeta_{n,p}$ in condition (C2) is $1/\sqrt{T\pmk}$ here. Lastly, condition (C3) is direct by noting that~\eqref{eqn: DdotGamma_X_decomp} gives
\[
\Big\| p_k \Ddot{\Gamma}_X^{(k)} -p_k \cI \Big\| = \frac{1}{T\pmk} \Big\| \sum_{t=1}^T \E(E_{k,t} E_{k,t}^\trans) \Big\| = \frac{1}{T\pmk} \cO\Big(T \pmk \Big) = \cO(1) ,
\]
where the second equality used the second claim in Lemma~\ref{lemma: prelim_rate}~(ii). With the relevant conditions in Lemma~B.4 of \cite{Barigozzietal2025_robust} verified, and with the fact that $\wt\Lambda_k$ and $\Ddot{\Lambda}_k$ are $\sqrt{p_k}$ times the first $r_k$ eigenvectors of $p_k \wh{\Gamma}_X^{(k)}$ and $p_k \Ddot{\Gamma}_X^{(k)}$ or equivalently, eigenvectors of $\wh{\Gamma}_X^{(k)}$ and $\Ddot{\Gamma}_X^{(k)}$, we conclude that there exists orthogonal matrix $J_k\in \R^{r_k\times r_k}$ such that
\[
\frac{1}{\sqrt{p_k}} \Big\| \wt\Lambda_k - \Ddot{\Lambda}_k J_k \Big\|_F = \cO_P\Big(\frac{1}{\sqrt{T\pmk}} \Big) ,
\]
as desired. This completes the proof of the lemma.
\end{proof}

\begin{lemma}\label{lemma: eigenvalue_proj_consistency}
Let all the assumptions in Lemma~\ref{lemma: eigenvalue_PCA_consistency}~(ii) hold. 
For any $k\in[K]$, recall $\wh{D}_k$ from Lemma~\ref{lemma: consistency_proj}.
Then as $\min\{T,p_1,\dots, p_K\} \to \infty$, the eigenvalues of $\wh{D}_k$ satisfy $\lambda_i(\wh{D}_k) = \lambda_i(\Sigma_{A,k}) + o_P(1)$ for any $i\leq r_k$.
\end{lemma}

\begin{proof}[Proof of Lemma~\ref{lemma: eigenvalue_proj_consistency}]
From the definition of $Y_{k,t}$ in~\eqref{eqn: est_lam}, we have
\begin{align}
    \wh{\Gamma}_Y^{(k)} &= \frac{1}{Tp_k} \sum_{t=1}^T Y_{k,t} Y_{k,t}^\trans
    = \frac{1}{Tp_k\pmk^2} \sum_{t=1}^T \mat_{k}(\cX_t) \wtLambdamk \wtLambdamk^\trans \mat_{k}(\cX_t)^\trans \notag \\
    &=
    \frac{1}{Tp\pmk} \sum_{t=1}^T (\Lambda_k G_{k,t} \Lambdamk^\trans + E_{k,t}) \wtLambdamk \wtLambdamk^\trans (\Lambda_k G_{k,t} \Lambdamk^\trans + E_{k,t})^\trans \notag \\
    &=
    \frac{1}{Tp\pmk} \sum_{t=1}^T \Lambda_k G_{k,t} \Lambdamk^\trans \wtLambdamk \wtLambdamk^\trans \Lambdamk G_{k,t}^\trans \Lambda_k^\trans
    + \frac{1}{Tp\pmk} \sum_{t=1}^T \Lambda_k G_{k,t} \Lambdamk^\trans \wtLambdamk \wtLambdamk^\trans E_{k,t}^\trans \notag \\
    &\quad
    + \frac{1}{Tp\pmk} \sum_{t=1}^T E_{k,t} \wtLambdamk \wtLambdamk^\trans \Lambdamk G_{k,t}^\trans \Lambda_k^\trans
    + \frac{1}{Tp\pmk} \sum_{t=1}^T E_{k,t} \wtLambdamk \wtLambdamk^\trans E_{k,t}^\trans
    \notag \\
    &=:
    \wt\cI + \wt\cII + \wt\cIII + \wt\cIV .
    \label{eqn: whGamma_Y_decomp}
\end{align}
We first present a preliminary result based on Lemma~\ref{lemma: consistency_PCA} and Assumption~\ref{assum: loadings}. That is, for any $\ell\in[K]$, we have
\begin{equation}
\label{eqn: wtLambda_wtLambda}
\begin{split}
    &\quad
    \frac{1}{p_{\ell}} \wt\Lambda_{\ell} \wt\Lambda_{\ell}^\trans = \frac{1}{p_{\ell}} (\wt\Lambda_{\ell} - \Lambda_{\ell} \wt{H}_{\ell}) \wt\Lambda_{\ell}^\trans + \frac{1}{p_{\ell}} \Lambda_{\ell} \wt{H}_{\ell} \wt\Lambda_{\ell}^\trans \\
    &=
    \frac{1}{p_{\ell}} (\wt\Lambda_{\ell} - \Lambda_{\ell} \wt{H}_{\ell}) \wt\Lambda_{\ell}^\trans
    + \frac{1}{p_{\ell}} \Lambda_{\ell} \wt{H}_{\ell} (\wt\Lambda_{\ell} - \Lambda_{\ell} \wt{H}_{\ell})^\trans
    + \frac{1}{p_{\ell}} \Lambda_{\ell} \wt{H}_{\ell} \wt{H}_{\ell}^\trans \Lambda_{\ell}^\trans
    = \frac{1}{p_{\ell}} \Lambda_{\ell} \Lambda_{\ell}^\trans + o_P(1) ,
\end{split}
\end{equation}
noting that
$\Lambda_{\ell} \Lambda_{\ell}^\trans /p_{\ell}$ has the first $r_k$ eigenvalues being 1's and the rest being zero. Immediately, we also have $\wtLambdamk \wtLambdamk^\trans /\pmk = \Lambdamk \Lambdamk^\trans /\pmk + o_P(1)$. Hence for $\wt\cI$, similar to~\eqref{eqn: GLambdaLambdaG},
\begin{align*}
    &\quad
    \frac{1}{T\pmk^2} \sum_{t=1}^T G_{k,t} \Lambdamk^\trans \wtLambdamk \wtLambdamk^\trans \Lambdamk G_{k,t}^\trans \\
    &=
    \frac{1}{T} \sum_{j=1}^{q+1} (\theta_j - \theta_{j-1}) A_{j,k} \Big\{ \frac{1}{\theta_j - \theta_{j-1}} \sum_{t = \theta_{j-1} + 1}^{\theta_j} F_{k,t} \Ajmk{j}^\trans \Big( \frac{\Lambdamk^\trans \wtLambdamk \wtLambdamk^\trans \Lambdamk}{\pmk^2} \Big) \Ajmk{j} F_{k,t}^\trans \Big\} A_{j,k}^\trans \\
    &\xrightarrow{P}
    \frac{1}{T} \sum_{j=1}^{q+1} (\theta_j - \theta_{j-1}) A_{j,k} \Big\{ \frac{1}{\theta_j - \theta_{j-1}} \sum_{t = \theta_{j-1} + 1}^{\theta_j} F_{k,t} \Ajmk{j}^\trans \Ajmk{j} F_{k,t}^\trans \Big\} A_{j,k}^\trans \\
    &\xrightarrow{P}
    \frac{1}{T} \sum_{j=1}^{q+1} \|\Ajmk{j}\|_F^2 \cdot (\theta_j - \theta_{j-1}) \cdot A_{j,k} A_{j,k}^\trans
    \to \Sigma_{A,k} .
\end{align*}
The remaining arguments on $\wt\cI$ is exactly the same as those on $\cI$ in Lemma~\ref{lemma: eigenvalue_PCA_consistency}, so that $\lambda_i(\wt\cI) = \lambda_i(\Sigma_{A,k}) + o_P(1)$ for all $i\leq r_k$, while $\lambda_i(\wt\cI) =0$ for $i> r_k$ by $\rank(\wt\cI)\leq r_k$.

Before working on $\wt\cII$ and $\wt\cIII$, with the notation in Lemma~\ref{lemma: population_cov_PCA_consistency}, define $\Ddot{\Lambda}_{\text{-}k} = \otimes_{\ell \in[K] \setminus\{k\}} \Ddot{\Lambda}_\ell$ and $J_{\text{-}k} = \otimes_{\ell \in[K] \setminus\{k\}} J_\ell$. Then similar to~\eqref{eqn: wtLambda_wtLambda}, it holds for any $k\in[K]$ that
\begin{equation}
\label{eqn: wtLambda_wtLambda_Ddot}
\begin{split}
    \frac{1}{\pmk} \wtLambdamk \wtLambdamk^\trans &= \frac{1}{\pmk} (\wtLambdamk - \Ddot{\Lambda}_{\text{-}k} J_{\text{-}k}) \wtLambdamk^\trans + \frac{1}{\pmk} \Ddot{\Lambda}_{\text{-}k} J_{\text{-}k} \wtLambdamk^\trans \\
    &=
    \frac{1}{\pmk} (\wtLambdamk - \Ddot{\Lambda}_{\text{-}k} J_{\text{-}k}) \wtLambdamk^\trans
    + \frac{1}{\pmk} \Ddot{\Lambda}_{\text{-}k} J_{\text{-}k} (\wtLambdamk - \Ddot{\Lambda}_{\text{-}k} J_{\text{-}k})^\trans
    + \frac{1}{\pmk} \Ddot{\Lambda}_{\text{-}k} \Ddot{\Lambda}_{\text{-}k}^\trans \\
    &= \frac{1}{\pmk} \Ddot{\Lambda}_{\text{-}k} \Ddot{\Lambda}_{\text{-}k}^\trans + \cO_P(\sqrt{\wmk}),
\end{split}
\end{equation}
where in the last equality, we used $\wmk = \sum_{\ell\in [K]\setminus\{k\}} 1/(Tp_{\text{-}\ell})$ and applied Lemma~\ref{lemma: population_cov_PCA_consistency} to
\[
\Bigg\| \frac{\wtLambdamk - \Ddot{\Lambda}_{\text{-}k} J_{\text{-}k}}{\sqrt{\pmk}} \Bigg\|_F^2
= \Bigg\| \Bigg( \otimes_{h\in[K]\setminus\{k\}} \frac{\wt\Lambda_h}{\sqrt{p_h}} \Bigg) - \Bigg( \otimes_{\ell\in[K]\setminus\{k\}} \frac{\Ddot{\Lambda}_{\ell} J_{\ell}}{\sqrt{p_\ell}} \Bigg) \Bigg\|_F^2
= \sum_{\ell\in[K]\setminus\{k\}} \frac{1}{p_\ell} \Big\| \wt\Lambda_{\ell} - \Ddot{\Lambda}_{\ell} J_{\ell} \Big\|_F^2 ,
\]
where the second equality can be easily shown and details are omitted here (see \text{e.g.} the induction argument in the proof of Lemma~6 in \citeauthor{CenLam2025_KronProd}, \citeyear{CenLam2025_KronProd}).

Then consider $\wt\cII$, for which we can also write
\begin{align*}
    \|\wt\cII\|_F^2 &= \Bigg\| \frac{1}{Tp} \sum_{t=1}^T \Lambda_k G_{k,t} \Lambdamk^\trans \frac{\wtLambdamk \wtLambdamk^\trans}{\pmk} E_{k,t}^\trans \Bigg\|_F^2 \\
    &\lesssim
    \Bigg\| \frac{1}{Tp} \sum_{t=1}^T \Lambda_k G_{k,t} \Lambdamk^\trans \frac{\Ddot{\Lambda}_{\text{-}k} \Ddot{\Lambda}_{\text{-}k}^\trans}{\pmk} E_{k,t}^\trans \Bigg\|_F^2
    + \wmk \cdot \frac{1}{(Tp)^2} \sum_{i=1}^{p_k} \sum_{z=1}^{\pmk} \Bigg\| \sum_{t=1}^T \Lambda_{k,i\cdot}^\trans G_{k,t} \Lambda_{\text{-}k,z\cdot} E_{k,t} \Bigg\|_F^2 \\
    &\lesssim_P
    \Bigg\| \frac{1}{Tp} \sum_{t=1}^T \Lambda_k G_{k,t} \Lambdamk^\trans \frac{\Ddot{\Lambda}_{\text{-}k} \Ddot{\Lambda}_{\text{-}k}^\trans}{\pmk} E_{k,t}^\trans \Bigg\|_F^2
    + \wmk \cdot \frac{1}{T}
    = \cO_P\Big( \frac{1}{T\pmk} + \frac{\wmk}{T} \Big) ,
\end{align*}
where the second line used~\eqref{eqn: wtLambda_wtLambda_Ddot} and Lemma~\ref{lemma: inequality_sandwich}~(ii), and the third used Lemma~\ref{lemma: prelim_rate}~(v) and the proof of Proposition~1.2 (in their supplement) in \cite{CenLam2025} with Hölder's inequality and Assumption~\ref{assum: loadings}~(i). Similarly, we have $\|\wt\cIII\|_F^2 =\cO_P(1/T\pmk + \wmk/T)$.

For $\wt\cIV$, using~\eqref{eqn: wtLambda_wtLambda}, we have 
\begin{align*}
    \|\wt\cIV\|_F^2 &= \Bigg\| \frac{1}{Tp} \sum_{t=1}^T E_{k,t} \frac{\wtLambdamk \wtLambdamk^\trans}{\pmk} E_{k,t}^\trans \Bigg\|_F^2 \\
    &\lesssim
    \Bigg\| \frac{1}{Tp} \sum_{t=1}^T E_{k,t} \frac{\Lambdamk \Lambdamk^\trans}{\pmk} E_{k,t}^\trans \Bigg\|_F^2
    + o_P\Bigg(\Bigg\| \frac{1}{Tp} \sum_{t=1}^T E_{k,t} E_{k,t}^\trans \Bigg\|_F^2\Bigg) \\
    &\lesssim_P
    \Bigg\| \frac{1}{Tp} \sum_{t=1}^T E_{k,t} \Big( \frac{\Lambdamk \Lambdamk^\trans}{\pmk} - I_{\rmk} \Big) E_{k,t}^\trans \Bigg\|_F^2 + \Bigg\| \frac{1}{Tp} \sum_{t=1}^T E_{k,t} E_{k,t}^\trans \Bigg\|_F^2 \\
    &\lesssim
    \Bigg\| \frac{1}{Tp} \sum_{t=1}^T E_{k,t} E_{k,t}^\trans \Bigg\|_F^2
    = \|\cIV\|_F^2 =  \cO_P\Big( \frac{1}{T\pmk} + \frac{1}{p_k} \Big) ,
\end{align*}
where the last line used Lemma~\ref{lemma: inequality_sandwich}~(i), $\|\Lambdamk \Lambdamk^\trans/\pmk - I_{\rmk}\| =\cO(1)$, and the result on $\cIV$ in the proof of Lemma~\ref{lemma: eigenvalue_PCA_consistency}. Combining Weyl's inequality and the above results on the decomposition of $\wh\Gamma_Y^{(k)}$ in~\eqref{eqn: whGamma_Y_decomp}, the proof of this lemma is complete.
\end{proof}

\subsubsection{Proof of Proposition~\ref{prop: consistency_proj}}

To show Proposition~\ref{prop: consistency_proj}, it suffices to prove the following lemma thanks to Lemma~\ref{lemma: eigenvalue_PCA_consistency}~(i); the lemma below also provides the explicit expression for $\wh{H}_k$.

\begin{lemma}\label{lemma: consistency_proj}
Let Assumptions~\ref{assum: core_factor}, \ref{assum: loadings}, \ref{assum: trans_mat} and~\ref{assum: noise} hold, with $\nu\geq 4$ and $\beta\geq 0$. 
For any $k\in[K]$, let $\wh{D}_k := \wh\Lambda_k^\trans \wh{\Gamma}_Y^{(k)} \wh\Lambda_k /p_k$ be the diagonal matrix consisting of the leading $r_k$ eigenvalues of $\wh{\Gamma}_Y^{(k)}$. 
Define
\[
\wh{H}_k :=\frac{1}{Tp\pmk} \Big(\sum_{t=1}^T \mat_k(\cG_t) \Lambdamk^\trans \wtLambdamk \wtLambdamk^\trans \Lambdamk \mat_k(\cG_t)^\trans \Big) \Lambda_k^\trans \wh\Lambda_k \wh{D}_k^{-1} .
\]
Then $\wh{H}_k$ is asymptotically invertible such that as $\min\{T,p_1,\dots, p_K\} \to \infty$, we have
\[
\wh{H}_k^\trans \wh{H}_k = I_{r_k} + \cO_P\Bigg( \Big\{ \frac{1}{T\pmk} + \frac{1}{p^2} + \Big( \sum_{\ell\in [K]\setminus\{k\}} \frac{1}{Tp_{\text{-}\ell}} \Big) \Big( \frac{1}{T} + \frac{1}{p_k} \Big) \Big\}^{1/2} \Bigg) .
\]
Moreover, the projection estimator $\wh\Lambda_k$ satisfies that
\[
\frac{1}{p_k}\Big\| \wh\Lambda_k - \Lambda_k \wh{H}_k \Big\|_F^2 = \cO_P\l( \frac{1}{T\pmk} + \frac{1}{p^2} + \Big( \sum_{\ell\in [K]\setminus\{k\}} \frac{1}{Tp_{\text{-}\ell}} \Big) \Big( \frac{1}{T} + \frac{1}{p_k} \Big) \r).
\]
\end{lemma}

\begin{proof}
Similar to the proof of Lemma~\ref{lemma: consistency_PCA}, we have $\wh\Lambda_k =\wh{\Gamma}_Y^{(k)} \wh\Lambda_k \wh{D}_k^{-1}$ from the definition, and hence $\wh\Lambda_k - \Lambda_k \wh{H}_k = (\wt\cII + \wt\cIII + \wt\cIV) \wh\Lambda_k \wh{D}_k^{-1}$ from~\eqref{eqn: whGamma_Y_decomp}. Similar to~\eqref{eqn: aux_rates},
\begin{equation}
\label{eqn: aux_rates_proj}
\begin{split}
    \max\left\{ \|\wh{D}_k\|_F ,\; \|\wh{D}_k^{-1}\|_F ,\; \|\wh{H}_k\|_F \right\} = \cO_P(1) ,
    \quad \text{and} \quad
    \|\wh\Lambda_k\|_F = \cO_P(p_k^{1/2}) &,
\end{split}
\end{equation}
where $\|\wh{H}_k\|_F$ is similar to $\|\wt{H}_k\|_F$ from Lemma~\ref{lemma: consistency_PCA}, since we have $\wtLambdamk \wtLambdamk^\trans /\pmk = \Lambdamk \Lambdamk^\trans /\pmk + o_P(1)$ by applying~\eqref{eqn: wtLambda_wtLambda} on all but the $k$-th modes. Thus,
\begin{equation}
\label{eqn: proj_decomp}
\begin{split}
    &\quad \frac{1}{p_k}\Big\| \wh\Lambda_k - \Lambda_k \wh{H}_k \Big\|_F^2 = \frac{1}{p_k} \Big\|(\wt\cII + \wt\cIII + \wt\cIV) \wh\Lambda_k \wh{D}_k^{-1} \Big\|_F^2 \\
    &\lesssim_P
    \frac{1}{p_k} \Big\|(\wt\cII + \wt\cIII + \wt\cIV) \Lambda_k \wh{H}_k \Big\|_F^2 + \frac{1}{p_k} \Big\|(\wt\cII + \wt\cIII + \wt\cIV) (\wh\Lambda_k - \Lambda_k \wh{H}_k) \Big\|_F^2 \\
    &=
    \cO_P\l( \frac{1}{p_k} \Big\|(\wt\cII + \wt\cIII + \wt\cIV) \Lambda_k \Big\|_F^2 \r) + o_P\Big( \frac{1}{p_k}\Big\| \wh\Lambda_k - \Lambda_k \wh{H}_k \Big\|_F^2 \Big) .
\end{split}
\end{equation}

With the rates derived for $\|\wt\cII\|_F^2$ and $\|\wt\cIII\|_F^2$ in the proof of Lemma~\ref{lemma: eigenvalue_proj_consistency}, and the similar arguments in the proof of Lemma~\ref{lemma: consistency_PCA}, we have $\|(\wt\cII + \wt\cIII) \Lambda_k \|_F^2 /p_k =\cO_P(1/T\pmk +\wmk/T)$. Hence for~\eqref{eqn: proj_decomp}, it remains to study $\|\wt\cIV  \cdot \Lambda_k \|_F^2 /p_k$. To this end, from the proof of Lemma~\ref{lemma: eigenvalue_proj_consistency}, recall the definition $\Ddot{\Lambda}_{\text{-}k} = \otimes_{\ell \in[K] \setminus\{k\}} \Ddot{\Lambda}_\ell$ and $J_{\text{-}k} = \otimes_{\ell \in[K] \setminus\{k\}} J_\ell$. Then we have
\begin{align*}
    \frac{1}{p_k} \|\wt\cIV  \cdot \Lambda_k \|_F^2 &= \frac{1}{p_k} \Bigg\| \frac{1}{Tp\pmk} \sum_{t=1}^T E_{k,t} \wtLambdamk \wtLambdamk^\trans E_{k,t}^\trans \Lambda_k \Bigg\|_F^2 \\
    &\lesssim
    \frac{1}{p_k} \Bigg\| \frac{1}{Tp\pmk} \sum_{t=1}^T E_{k,t} (\wtLambdamk - \Ddot{\Lambda}_{\text{-}k} J_{\text{-}k}) \wtLambdamk^\trans E_{k,t}^\trans \Lambda_k \Bigg\|_F^2 \\
    &\quad
    + \frac{1}{p_k} \Bigg\| \frac{1}{Tp\pmk} \sum_{t=1}^T E_{k,t} \Ddot{\Lambda}_{\text{-}k} J_{\text{-}k} (\wtLambdamk - \Ddot{\Lambda}_{\text{-}k} J_{\text{-}k} )^\trans E_{k,t}^\trans \Lambda_k \Bigg\|_F^2 \\
    &\quad
    + \frac{1}{p_k} \Bigg\| \frac{1}{Tp\pmk} \sum_{t=1}^T E_{k,t} \Ddot{\Lambda}_{\text{-}k} J_{\text{-}k} J_{\text{-}k}^\trans \Ddot{\Lambda}_{\text{-}k}^\trans E_{k,t}^\trans \Lambda_k \Bigg\|_F^2
    =: \wt\cIV_a + \wt\cIV_b + \wt\cIV_c .
\end{align*}
Consider $\wt\cIV_a$ first. Define $\wmk = \sum_{\ell\in [K]\setminus\{k\}} 1/Tp_{\text{-}\ell}$. It then holds that
\begin{equation}
\label{eqn: tilde_IV_a}
\begin{split}
    \wt\cIV_a &= \frac{1}{p_k} \Bigg\| \frac{1}{Tp} \sum_{t=1}^T E_{k,t} \Bigg( \frac{\wtLambdamk - \Ddot{\Lambda}_{\text{-}k} J_{\text{-}k}}{\sqrt{\pmk}} \frac{\wtLambdamk^\trans}{\sqrt{\pmk}} \Bigg) E_{k,t}^\trans \Lambda_k \Bigg\|_F^2 \\
    &\lesssim
    \Bigg\| \frac{\wtLambdamk - \Ddot{\Lambda}_{\text{-}k} J_{\text{-}k}}{\sqrt{\pmk}} \frac{\wtLambdamk^\trans}{\sqrt{\pmk}} \Bigg\|_F^2 \cdot \frac{1}{p_k} \|\Lambda_k \|_F^2 \cdot \Bigg\| \frac{1}{Tp} \sum_{t=1}^T E_{k,t} E_{k,t}^\trans \Bigg\|_F^2 \\
    &=
    \cO_P(\wmk) \cdot \cO_P(1) \cdot \cO(1) \cdot \cO_P\Big( \frac{1}{Tp} + \frac{1}{p_k} \Big)
    = \cO_P\Big\{ \wmk \Big( \frac{1}{Tp} + \frac{1}{p_k} \Big) \Big\} ,
\end{split}
\end{equation}
where Lemma~\ref{lemma: inequality_sandwich}~(i) is used in the second line in which the last term appears to be exactly $\|\cIV\|_F^2$ in the proof of Lemma~\ref{lemma: eigenvalue_PCA_consistency}, whereas the first term used $\wtLambdamk^\trans \wtLambdamk/\pmk = I_{\rmk}$ and the same argument in~\eqref{eqn: wtLambda_wtLambda_Ddot}. Similar procedures follow for $\wt\cIV_b$, by noting that $J_{\text{-}k} =\cO_P(1)$. That is, we have
\begin{equation}
\label{eqn: tilde_IV_b}
\begin{split}
    \wt\cIV_b &= \cO_P\Big\{ \wmk \Big( \frac{1}{Tp} + \frac{1}{p_k} \Big) \Big\} .
\end{split}
\end{equation}

Lastly, consider $\wt\cIV_c$. Note that Lemma~\ref{lemma: consistency_PCA} indicates $J_{\text{-}k} J_{\text{-}k}^\trans = I_{\rmk}$, and hence
\begin{equation}
\label{eqn: tilde_IV_c}
\begin{split}
    \wt\cIV_c &= \frac{1}{p_k} \Bigg\| \frac{1}{Tp\pmk} \sum_{t=1}^T E_{k,t} \Ddot{\Lambda}_{\text{-}k} \Ddot{\Lambda}_{\text{-}k}^\trans E_{k,t}^\trans \Lambda_k \Bigg\|_F^2
    = \cO_P\Big( \frac{1}{Tp\pmk} + \frac{1}{p^2} \Big),
\end{split}
\end{equation}
where the last equality used Lemma~\ref{lemma: prelim_rate}~(iv) by equating $V =\Ddot{\Lambda}_{\text{-}k}$ therein and $\|\Ddot{\Lambda}_{\text{-}k}\|_F^4 =\pmk^2$. Combining~\eqref{eqn: tilde_IV_a},~\eqref{eqn: tilde_IV_b} and~\eqref{eqn: tilde_IV_c}, we hence have
\[
\frac{1}{p_k} \|\wt\cIV  \cdot \Lambda_k \|_F^2 = \cO_P\Big\{ \wmk \Big( \frac{1}{Tp} + \frac{1}{p_k} \Big) + \frac{1}{Tp\pmk} + \frac{1}{p^2} \Big\} .
\]
Thus, we may finally conclude for~\eqref{eqn: proj_decomp} that
\begin{align*}
    \frac{1}{p_k}\Big\| \wh\Lambda_k - \Lambda_k \wh{H}_k \Big\|_F^2 &=  \cO_P\Big\{ \frac{1}{T\pmk} + \Big( \sum_{\ell\in [K]\setminus\{k\}} \frac{1}{Tp_{\text{-}\ell}} \Big) \Big( \frac{1}{T} + \frac{1}{Tp} + \frac{1}{p_k} \Big) + \frac{1}{Tp\pmk} + \frac{1}{p^2} \Big\} \\
    &=
    \cO_P\Big\{ \frac{1}{T\pmk} + \frac{1}{p^2} + \Big( \sum_{\ell\in [K]\setminus\{k\}} \frac{1}{Tp_{\text{-}\ell}} \Big) \Big( \frac{1}{T} + \frac{1}{p_k} \Big) \Big\} .
\end{align*}
It remains to show the convergence result on $\wh{H}_k^\trans \wh{H}_k$ which, given the above consistency result on $\wh\Lambda_k$, follows the similar arguments on $\wt{H}_k$ in the proof of Lemma~\ref{lemma: consistency_PCA} and hence omitted here. 
\end{proof}

\clearpage

\subsection{Consistency in change point detection}

\subsubsection{Supporting lemmas}

Let us define
\begin{align}
& \alpha_{T,p,k}^2 =\frac{1}{T\pmk} + \frac{1}{p^2} + \Big( \sum_{\ell\in [K]\setminus\{k\}} \frac{1}{Tp_{\text{-}\ell}} \Big) \Big( \frac{1}{T} + \frac{1}{p_k} \Big),
\nonumber \\
& \alpha_{T,p,\text{-}k}  = \sum_{w\ne k} \alpha_{T,p,w}
\text{ \ and \ }
\alpha_{T,p} = \sum_{k\in[K]} \alpha_{T,p,k}.
\label{eq:alpha:rate}
\end{align}

\paragraph{Lemmas on functional dependence measures.}
We define the functional dependence measures of tensor-valued time series, which extend the definitions given by \citet{zhang2021convergence} for vector-valued time series. 

\begin{definition}[Functional dependence measures]\label{def: fdc}
Let $\{\bxi_t\}_{t\in\Z}$ be a sequence of i.i.d.\ random elements with zero mean, unit variance, and uniformly bounded fourth moments. Suppose a tensor-valued time series $\{\cY_t\}_{t\in\Z}$ taking values in $\R^{p_1\times\cdots\times p_K}$, admits the causal representation $\cY_t = g_t(\bxi_t,\bxi_{t-1},\ldots)$, where $g_t$ is a (possibly time-varying) measurable map.
We further define the coupled process $\cY_t^{(s)} = g_t(\bxi_t,\ldots,\bxi_{t-s+1}, \bxi_{t-s}',\bxi_{t-s-1},\ldots)$, for $s\ge0$,
where $\bxi_{t-s}'$ is an i.i.d.\ copy of $\bxi_0$.  
For a multi-index $\bi=(i_1,\dots,i_K)\in[p_1]\times\cdots\times[p_K]$, let us denote by $\cY_{t,\bi}$ the corresponding entry of $\cY_t$. 
Then we define the element-wise functional dependence coefficients of $\{\cY_t\}_{t \in \Z}$ as
\[
\delta_{s,\nu,\bi}(\cY) = \sup_{t\in\Z}\big\|\cY_{t,\bi}-\cY_{t,\bi}^{(s)}\big\|_{\nu},
\]
for some $\nu\in[2,\infty]$, and the dependence-adjusted norms (d.a.n.) as
\[
\|\cY_{\cdot,\bi}\|_{\nu,\beta} = \sup_{m\ge0} \, (m+1)^{\beta}\Delta_{m,\nu,\bi}(\cY) \text{ \ with \ }
\Delta_{m,\nu,\bi}(\cY)\! = \!\sum_{s=m}^{\infty}\delta_{s,\nu,\bi}(\cY)
\]
for some $\beta \ge 0$.
Finally, we denote by $\Phi_{\nu,\beta}(\cY) = \max_{\bi\in[p_1]\times\cdots\times[p_K]}\|\cY_{\cdot,\bi}\|_{\nu,\beta}$ the uniform d.a.n.
\end{definition}

\begin{lemma}
\label{lemma: fdc_finite}
Let $\{\cY_t\}_{t\in\Z}$ be a $(\nu,\beta)$-general linear tensor time series as given in Definition~\ref{def: general_linear_tensor}, with $\nu>4$ and $\beta>0$. 
For every multi-index $\bi=(i_1,\ldots,i_K)\in[p_1]\times\cdots\times[p_K]$, write $\cY_{t,\bi}$ and $\cX_{y,t,\bi}$ for the $\bi$-th entries of $\cY_t$ and $\cX_{y,t}$, respectively. 
For $\{\cY_t\}_{t\in\Z}$, we define
\begin{align}\label{eqn: C_nu}
 C_\nu(\cY) = \max_{\bi \in [p_1]\times \ldots \times [p_K]} \E(\vert \cX_{y, t, \bi} \vert^\nu) \le  C_\nu <\infty,
\end{align}
with $C_{\nu}$ denoted in Definition~\ref{def: general_linear_tensor}~(i).
Then for the uniform d.a.n. of $\{\cY_t\}$, we have
\begin{align*}
\Phi_{\nu,\beta}(\cY) \le 2 (C_{\nu}(\cY))^{1/\nu}\,c_{\beta}<\infty.
\end{align*}
\end{lemma}

\begin{proof}[Proof of Lemma~\ref{lemma: fdc_finite}]
Fix a multi–index $\bi\in [p_1]\times \ldots \times [p_K]$.
By Definition~\ref{def: general_linear_tensor}, the coordinate process satisfies
\[
\cY_{t,\bi} = \sum_{h\ge0} a_{y,h} \cX_{y,t-h,\bi}, \quad t\in\Z.
\]
For some $s \ge 0$, define the coupled version $\cY_{t,\bi}^{(s)}$ by replacing $\cX_{y,t-s,\bi}$ with an i.i.d.\ copy $\cX'_{y,t-s,\bi}$. Then
\[
\cY_{t,\bi}-\cY_{t,\bi}^{(s)} = a_{y,s}\big(\cX_{y,t-s,\bi}-\cX'_{y,t-s,\bi}\big).
\]
By the triangle inequality, for some $\nu > 4$, then we have
\[
\delta_{s,\nu,\bi}(\cY) = \big\Vert \cY_{t,\bi}-\cY_{t,\bi}^{(s)} \big\Vert_\nu \le 2 (C_{\nu}(\cY))^{1/\nu}|a_{y,s}|.
\]
Hence, for any integer $m \ge 0$, by Definition~\ref{def: general_linear_tensor}~(iii), 
\begin{align*}
\Delta_{m,\nu,\bi}(\cY)&=\sum_{s = m}^{\infty} \delta_{s,\nu,\bi}(\cY) \le 2 (C_{\nu}(\cY))^{1/\nu} \sum_{s=m}^\infty |a_{y,s}| \le \frac{2 (C_{\nu}(\cY))^{1/\nu} c_{\beta}}{(m+1)^{\beta}}<\infty, \text{ \ which gives}
\\
\Vert \cY_{\cdot,\bi}\Vert_{\nu,\beta}
   &= \sup_{m\ge0}(m+1)^{\beta}\Delta_{m,\nu,\bi}(\cY)
   \le 2(C_{\nu}(\cY))^{1/\nu}\,c_{\beta}<\infty.    
\end{align*}
Taking the maximum over $\bi$, we have $\Phi_{\nu,\beta}(\cY)\le 2(C_{\nu}(\cY))^{1/\nu}\,c_{\beta}<\infty$.
\end{proof} 

\begin{lemma}
\label{lemma: fdc_Et}
Under Assumption~\ref{assum: noise} with $\nu>4$ and $\beta\geq 0$, with $C_\nu(\cdot)$ defined in~\eqref{eqn: C_nu} and some constant $C(\cE) > 0$, we have
\[
\Phi_{\nu,\alpha}(\cE)\le
C(\cE)\bigl(2(C_{\nu}(\cF_e))^{1/\nu} c_{\beta}+2(C_{\nu}(\epsilon))^{1/\nu} c_{\beta}\bigr) <\infty.
\]
\end{lemma}

\begin{proof}[Proof of Lemma~\ref{lemma: fdc_Et}]
Fix a multi–index $\bi\in[p_1]\times\cdots\times[p_K]$ and $s\ge0$.  
From Assumption~\ref{assum: noise}, 
for any $\bi$,
\[
\cE_{t,\bi}-\cE^{(s)}_{t,\bi}
= \bigl((\cF_{e,t}\times_{k=1}^K A_{e,k})_{\bi} - (\cF_{e,t}^{(s)}\times_{k=1}^K A_{e,k})_{\bi}\bigr)
 + (\Sigma_{\epsilon}\circ (\epsilon_t-\epsilon_t^{(s)}))_{\bi} .
\]
By Minkowski’s inequality,
\[
\delta_{s,\nu,\bi} (\cE) = \l\|\cE_{t,\bi}-\cE^{(s)}_{t,\bi} \r\|_\nu \le \l\|\l(\big(\cF_{e,t} - \cF_{e,t}^{(s)}\big)\times_{k=1}^K A_{e,k}\r)_{\bi}\r\|_\nu + \vert (\Sigma_{\epsilon})_{\bi} \vert \delta_{s,\nu,\bi}(\epsilon).
\]
For $\bi'\in[r_{e,1}]\times \ldots \times [r_{e,K}]$, we have
\begin{align*}
&\quad\;\l\|\l(\big(\cF_{e,t} - \cF_{e,t}^{(s)}\big)\times_{k=1}^K A_{e,k}\r)_{\bi}\r\|_\nu \\
&\le \sum_{\bi'\in[r_{e,1}]\times \ldots \times [r_{e,K}]}\delta_{s,\nu,\bi'}(\cF_e) \prod_{k=1}^K \l\vert (A_{e,k})_{i_k,i'_k} \r\vert\\
&\le \max_{\bi'\in[r_{e,1}]\times \ldots \times [r_{e,K}] }\delta_{s,\nu,\bi'}(\cF_e) \sum_{\bi'\in[r_{e,1}]\times \ldots \times [r_{e,K}]} \prod_{k=1}^K \l\vert (A_{e,k})_{i_k,i'_k} \r\vert\\
&= \max_{\bi'\in[r_{e,1}]\times \ldots \times [r_{e,K}] }\delta_{s,\nu,\bi'}(\cF_e)  \prod_{k=1}^K \sum_{i'_k =1}^{r_{e,k}} \l\vert (A_{e,k})_{i_k,i'_k} \r\vert.
\end{align*} 
Since $\sum_{i'_k =1}^{r_{e,k}} \l\vert (A_{e,k})_{i_k,i'_k} \r\vert 
\le \Vert A_{e,k} \Vert_\infty$,
we see
\begin{align*}
\sup_{t\in\Z}\l\|\l(\big(\cF_{e,t} - \cF_{e,t}^{(s)}\big)\times_{k=1}^K A_{e,k}\r)_{\bi}\r\|_\nu \le \max_{\bi'\in[r_{e,1}]\times \ldots \times [r_{e,K}] }\delta_{s,\nu,\bi'}(\cF_e)  \prod_{k=1}^K \l\Vert A_{e,k}\r\Vert_{\infty}.
\end{align*}

Let $C(\cE) = \max\l\{\prod_{k=1}^K \l\Vert A_{e,k}\r\Vert_{\infty} , \max_{\bi\in[p_1]\times\ldots \times[p_K]}\vert (\Sigma_{\epsilon})_{\bi} \vert\r\}$, by Assumption~\ref{assum: noise}, we have $C(\cE)<\infty$. Then
\begin{align*}
\delta_{s,\nu,\bi} (\cE) &\le  \l(\prod_{k=1}^K \l\Vert A_{e,k}\r\Vert_{\infty} \r)  \cdot \max_{\bi'\in[r_{e,1}]\times \ldots \times [r_{e,K}] }\delta_{s,\nu,\bi'}(\cF_e)  +  \l(\max_{\bi\in[p_1]\times\ldots \times[p_K]}\vert (\Sigma_{\epsilon})_{\bi} \vert \r)\,\delta_{s,\nu,\bi}(\epsilon)\\
&\le C(\cE) \, \l\{\max_{\bi'\in[r_{e,1}]\times \ldots \times [r_{e,K}] }\delta_{s,\nu,\bi'}(\cF_e) +  \delta_{s,\nu,\bi}(\epsilon)\r\}.
\end{align*}

Applying the arguments in the proof of Lemma~\ref{lemma: fdc_finite}, with general linear tensor series $\{\cF_{e,t}\}$ and $\{\epsilon_t\}$, for the given $\beta>0$, 
\begin{align*}
\Vert \cF_{e,\cdot,\bi'} \Vert_{\nu,\beta} &= \sup_{m\ge 0}(m+1)^{\beta}\sum_{s=m}^\infty \delta_{s,\nu,\bi'}(\cF_e)\le 2(C_{\nu}(\cF_e))^{1/\nu} c_{\beta}, \text{\ and }\\
\Vert \epsilon_{\cdot,\bi} \Vert_{\nu,\beta} &= \sup_{m\ge 0}(m+1)^{\beta}\sum_{s=m}^\infty \delta_{s,\nu,\bi}(\epsilon)\le 2(C_{\nu}(\epsilon))^{1/\nu} c_{\beta},
\end{align*}
uniformly over $\bi\in[p_{1}]\times \ldots \times [p_{K}]$ and $\bi'\in[r_{e,1}]\times \ldots \times [r_{e,K}]$. 
Hence, the d.a.n.\ of $\{\cE_t\}$ satisfies
\begin{align*}
\Vert \cE_{\cdot,\bi} \Vert_{\nu,\beta} &= \sup_{m\ge 0}(m+1)^{\beta}\sum_{s=m}^\infty \delta_{s,\nu,\bi}(\cE)\\
&\le C(\cE)\l\{\sup_{m\ge 0}(m+1)^{\beta}\sum_{s=m}^\infty \max_{\bi'\in[r_{e,1}]\times \ldots \times [r_{e,K}] }\delta_{s,\nu,\bi'}(\cF_e) + \sup_{m\ge 0}(m+1)^{\beta}\sum_{s=m}^\infty \delta_{s,\nu,\bi}(\epsilon) \r\}\\
&\le C(\cE)\l\{\max_{\bi'\in[r_{e,1}]\times \ldots \times [r_{e,K}] } 2(C_{\nu}(\cF_e))^{1/\nu} c_{\beta} + 2(C_{\nu}(\epsilon))^{1/\nu} c_{\beta} \r\}\\
&= C(\cE)\l\{ 2 (C_{\nu}(\cF_e))^{1/\nu} c_{\beta} + 2 (C_{\nu}(\epsilon))^{1/\nu} c_{\beta} \r\}.
\end{align*}
Taking the maximum over $\bi\in[p_{1}]\times \ldots \times [p_{K}]$, for the uniform d.a.n., we see
\begin{align*}
\Phi_{\nu,\beta}(\cE) \le  C(\cE)\bigl(2 (C_{\nu}(\cF_e))^{1/\nu} c_{\beta} + 2 (C_{\nu}(\epsilon))^{1/\nu} c_{\beta}\bigr) <\infty.
\end{align*}

\end{proof}

\begin{lemma}\label{lemma: fdc_common}
Let Assumptions~\ref{assum: core_factor}, \ref{assum: loadings} and~\ref{assum: trans_mat_alt}~(i) hold, with $\nu>4$ and $\beta\geq 0$. Then for the uniform d.a.n. of $\{\cC_t\}$, we have 
\begin{align*}
\Phi_{\nu,\beta}(\cC)\le 2c_{\lambda}^{K}r^{K}\,C(A) (C_{\nu}(\cF))^{1/\nu}\,c_{\beta}<\infty,
\end{align*}
where $C(A)$ is some finite positive constant, and $C_\nu(\cdot)$ is defined in~\eqref{eqn: C_nu}.
\end{lemma}

\begin{proof}[Proof of Lemma~\ref{lemma: fdc_common}]
For $\bu\in[r_1]\times\ldots\times[r_K]$, we write $g_{t,\bu} = (\cG_t)_{\bu}$.
Recalling that $\cC_t \;=\; \cG_t \times_{k=1}^K \Lambda_k$, and fixing $\bi\in[p_1]\times\cdots\times[p_K]$, by multilinearity, each entry of $\cC_t$ is a finite linear combination of entries of $\cG_t$,
\[
\cC_{t,\bi} =\sum_{\bi\in[p_1]\times\cdots\times[p_K]}\sum_{\bu\in [r_1]\times\ldots[r_K]} \lambda_{i_1,u_1}\ldots\lambda_{i_K,u_K}\  g_{t,\bu},
\]
For the coupled process $\{\cC^{(s)}_{t,\bi},s\ge 0\}$, we have
\[
\cC_{t,\bi}-\cC^{(s)}_{t,\bi} = \sum_{\bi\in[p_1]\times\cdots\times[p_K]}\sum_{\bu\in [r_1]\times\ldots[r_K]} \lambda_{i_1,u_1}\ldots\lambda_{i_K,u_K}\big(g_{t,u_1,\ldots,u_K}-g^{(s)}_{t,u_1,\ldots,u_K}\big).
\]
By triangle inequality and Assumption~\ref{assum: loadings}~(i),
\[
\delta_{s,\nu,\bi}(\cC)  = \sup_{t\in\Z}\big\|\cC_{t,\bi}-\cC^{(s)}_{t,\bi}\big\|_\nu \le  c_{\lambda}^{K} r^K \,\delta_{s,\nu}(\cG).
\]
Recalling that
\[
\cG_t = \sum_{j=1}^{q+1} \Bigl(\cF_t \times_{k=1}^K A_{j,k}\Bigr)\,\mathbb{I}_{\{\theta_{j-1}<t\le \theta_j\}},
\]
the coupling only changes $\cF_t$ to $\cF_t^{(s)}$, so for $j\in[q+1]$, we have
\[
\cG_t - \cG_t^{(s)} = \l(\cF_t - \cF_t^{(s)}\r)\times_{k=1}^K A_{j,k}.
\]
Then for each $\bu\in[r_1]\times\ldots \times[r_K]$, we apply the triangle inequality to have
\begin{align*}
\delta_{s,\nu,\bu}(\cG) &\le \sum_{\bu'\in [r_1]\times\ldots \times[r_K]} \l\vert \prod_{k=1}^K (A_{j,k})_{u_k,u'_k}\r\vert\delta_{s,\nu,\bu'}(\cF)\\
&\le \l(\max_{\bu'\in [r_1]\times\ldots \times[r_K]} \delta_{s,\nu,\bu'}(\cF) \r) \cdot \prod_{k=1}^K\sum_{u'_k=1}^{r_k} \l\vert (A_{j,k})_{u_k,u'_k} \r\vert.
\end{align*}
Since for $j\in[q+1]$, we have
\[
\sum_{u'_k=1}^{r_k} \l\vert (A_{j,k})_{u_k,u'_k} \r\vert \le \sqrt{r_k} \l\{\sum_{u'_k=1}^{r_k} (A_{j,k})_{u_k,u'_k}^2 \r\}^{1/2} \le \sqrt{r_k} \Vert A_{j,k} \Vert_F,
\]
therefore,
\[
\delta_{s,\nu,\bu}(\cG) \le \l(\prod_{k=1}^K\sqrt{r_k} \Vert A_{j,k} \Vert_F \r) \max_{\bu'\in [r_1]\times\ldots \times[r_K]}\delta_{s,\nu,\bu'}(\cF).
\]
By Assumption~\ref{assum: trans_mat_alt}~(i), 
we set $C(A)=\max_{j\in[q+1]}\prod_{k=1}^K\sqrt{r_k} \Vert A_{j,k} \Vert_F< \infty$ .
Applying Lemma~\ref{lemma: fdc_finite} to the general linear tensor series $\{\cF_t\}$ gives
\[
\Vert \cG_{\cdot,\bu} \Vert_{\nu,\beta} = \sup_{m\ge 0} (m+1)^\beta \sum_{s=m}^{\infty}
\delta_{s,\nu,\bu}(\cG)\le 2C(A)(C_{\nu}(\cF))^{1/\nu}c_{\beta}<\infty,
\]
taking the maximum over $\bu\in[r_1]\times\ldots\times[r_K]$, we have
\[
\Phi_{\nu,\beta}(\cG)\le 2C(A)(C_{\nu}(\cF))^{1/\nu}c_{\beta}<\infty.
\]
Finally we conclude
\[
\Phi_{\nu,\beta}(\cC)\le 2c_{\lambda}^K r^KC(A)(C_{\nu}(\cF))^{1/\nu}c_{\beta}<\infty.
\]
\end{proof}

\begin{lemma}\label{lemma: fdc_Xkt}
Under the assumptions of Lemmas~\ref{lemma: fdc_Et} and~\ref{lemma: fdc_common}, the uniform d.a.n. of $\{\cX_t\}$ satisfies
\begin{align*}
\Phi_{\nu,\beta}(\cX)\le 2c_{\lambda}^{K}r^{K}\,C(A) (C_{\nu}(\cF))^{1/\nu}\,c_{\beta}  + C(\cE)\bigl(2(C_{\nu}(\cF_e))^{1/\nu} c_{\beta}+2(C_{\nu}(\epsilon))^{1/\nu} c_{\beta}\bigr)<\infty,
\end{align*}
where $C_\nu(\cdot)$ is defined in~\eqref{eqn: C_nu}, $C(\cE)$ and $C(A)$ are the same as in Lemma~\ref{lemma: fdc_Et} and Lemma~\ref{lemma: fdc_common}, respectively.
\end{lemma}

\begin{proof}[Proof of Lemma~\ref{lemma: fdc_Xkt}]
The conclusion follows from noting that $\cX_t=\cC_t+\cE_t$, 
\[
\delta_{s,\nu,\bi}(\cX)=\big\|\cX_{t,\bi}-\cX^{(s)}_{t,\bi}\big\|_\nu \le\big\|\cC_{t,\bi}-\cC^{(s)}_{t,\bi}\big\|_\nu + \big\|\cE_{t,\bi}-\cE^{(s)}_{t,\bi}\big\|_\nu = \delta_{s,\nu,\bi}(\cC)+\delta_{s,\nu,\bi}(\cE).
\]
\end{proof}

\paragraph{Lemmas on $\cM$-dependent processes.}

To facilitate our proof, we show several results on weakly $\cM$-dependent processes in \cite{Berkesetal2011}, whose definition is quoted below for completeness.

\begin{definition}[Definition~1 of \cite{Berkesetal2011}]
\label{def: weakly_M_dependent}
Let $\{Y_k\}_{k\in\Z}$ be a stochastic process, let $\nu\geq 1$ and $\delta(m)\to 0$. We say that $\{Y_k\}_{k\in\Z}$ is weakly $\cM$-dependent in $L^\nu$\footnote{A random variable $X$ is said to be in $L^\nu$ if and only if $\E[|X|^\nu] <\infty$.} with rate function $\delta(\cdot)$ if:
\begin{itemize}
    \item [(i)] For any $k\in\Z$, $m\in\N$ one can find a random variable $Y_k^{(m)}$ with finite $\nu$-th moment such that $\|Y_k - Y_k^{(m)}\| \leq \delta(m)$.
    \item [(ii)] For any disjoint intervals $I_1,\dots,I_\ell$ $(\ell\in\N)$ of integers and any positive integers $m_1,\dots,m_\ell$, the vectors $\{Y_j^{(m_1)}\}_{j\in I_1}, \dots, \{Y_j^{(m_\ell)}\}_{j\in I_\ell}$ are independent provided $\inf\{|a-b|: a\in I_i, b\in I_j\} > \max\{m_i,m_j\}$ for $1\leq i<j \leq \ell$.
\end{itemize}
\end{definition}

\begin{lemma}\label{lemma: weakly_M_dependent}
(Useful results of weakly $\cM$-dependent linear processes). 
\begin{itemize}
    \item [(i)] (Linear combination).
    Consider processes $\{Y_{t,i}\}$ for $i\in[n]$ that are weakly $\cM$-dependent in $L^\nu$, each with rate function $\delta_i(m)$. Their linear combination, i.e.\ $\{\sum_{i=1}^n a_i Y_{t,i}\}$ for some given constant coefficients $a_i$, $i\in[n]$, is also weakly $\cM$-dependent in $L^\nu$, with the rate function
    \[
    \delta(m) = \sum_{i=1}^n |a_i| \delta_i(m) .
    \]
    \item [(ii)] (Product of linear processes).
    Consider two (one-sided) linear processes $Y_t = \sum_{j\geq 0} a_j \epsilon_{t-j}$ and $X_t = \sum_{i\geq 0} b_i \varepsilon_{t-j}$, with i.i.d.\ innovations {$\{\epsilon_t\}_{t \in \mathbb{Z}}$ and $\{\varepsilon_t\}_{t \in \mathbb{Z}}$ which are mutually independent at all lags.} 
    Assume that $\{Y_t\}$ satisfies
    \begin{itemize}
    \item [(a)] $\E(|\epsilon_0|^{2\nu}) <\infty$ 
    for some $\nu\geq 1$;
    \item [(b)] There exists some constant $c_\beta > 0$ depending on $\beta\geq 0$ such that
    \[
    \sup_{m \ge 0} (m + 1)^\beta \sum_{s = m}^\infty \vert a_{s} \vert \le c_{\beta} .
    \]
    \end{itemize}
    Assume analogous conditions for $\{X_t\}$ (with the same $\nu$ and $\beta$ without loss of generality). Then $\{Y_t\}$ and $\{X_t\}$ are weakly $\cM$-dependent in $L^{2\nu}$, and $\{Y_t X_t\}$ is weakly $\cM$-dependent in $L^\nu$, all with the rate function $\delta(m) = m^{-\beta}$.
    \item [(iii)] (Squared linear process).
    Consider $\{Y_t\}$ defined in (ii) above. Let $\beta>0$. Then the process $\{Y_t^2\}$ is also weakly $\cM$-dependent in $L^\nu$ with the rate function $\delta(m) = m^{-\beta}$.
\end{itemize}
\end{lemma}

\begin{proof}[Proof of Lemma~\ref{lemma: weakly_M_dependent}]
Part~(i) is straightforward from Definition~\ref{def: weakly_M_dependent}.

For part~(ii), the statement for the processes $\{Y_t\}$ and $\{X_t\}$ is direct from Section~3.3 in \cite{Berkesetal2011}. For the product process $\{Y_t X_t\}$, define $Z_t= Y_t X_t$ and
\[
Z_t^{(m)}= \sum_{j=0}^{\left\lfloor m/2 \right\rfloor} \sum_{i=0}^{\left\lfloor m/2 \right\rfloor} a_j b_i \epsilon_{t-j} \varepsilon_{t-i} .
\]
Then by Minkowski's inequality, H\"{o}lder's inequality and the independence between $\epsilon_t$ and $\varepsilon_s$ for any $t,s$, we have
\begin{align*}
    \left\| Z_t - Z_t^{(m)} \right\|_\nu &=
    \left\| \sum_{j=0}^{\infty} \sum_{i=0}^{\infty} a_j b_i \epsilon_{t-j} \varepsilon_{t-i} - \sum_{j=0}^{\left\lfloor m/2 \right\rfloor} \sum_{i=0}^{\left\lfloor m/2 \right\rfloor} a_j b_i \epsilon_{t-j} \varepsilon_{t-i} \right\|_\nu \\
    &\leq
    \left\| \left( \sum_{j\geq 0} a_j \epsilon_{t-j} \right) \left( \sum_{i> \left\lfloor m/2 \right\rfloor} b_i \varepsilon_{t-i} \right) \right\|_\nu 
    + \left\| \left( \sum_{j> \left\lfloor m/2 \right\rfloor} a_j \epsilon_{t-j} \right) \left( \sum_{i\geq 0} b_i \varepsilon_{t-i} \right) \right\|_\nu \\
    &\;\quad
    + \left\| \left( \sum_{j> \left\lfloor m/2 \right\rfloor} a_j \epsilon_{t-j} \right) \left( \sum_{i> \left\lfloor m/2 \right\rfloor} b_i \varepsilon_{t-i} \right) \right\|_\nu \\
    &\leq
    \left\| \sum_{j\geq 0} a_j \epsilon_{t-j} \right\|_{2\nu} \left\| \sum_{i> \left\lfloor m/2 \right\rfloor} b_i \varepsilon_{t-i} \right\|_{2\nu}
    + \left\| \sum_{j> \left\lfloor m/2 \right\rfloor} a_j \epsilon_{t-j} \right\|_{2\nu} \left\| \sum_{i\geq 0} b_i \varepsilon_{t-i} \right\|_{2\nu} \\
    &\;\quad
    + \left\| \sum_{j> \left\lfloor m/2 \right\rfloor} a_j \epsilon_{t-j} \right\|_{2\nu} \left\| \sum_{i> \left\lfloor m/2 \right\rfloor} b_i \varepsilon_{t-i} \right\|_{2\nu} \\
    &\lesssim
    \left( \sum_{j\geq 0} |a_j| \right) \left( \sum_{i> \left\lfloor m/2 \right\rfloor} |b_i| \right)
    + \left( \sum_{j> \left\lfloor m/2 \right\rfloor} |a_j| \right) \left( \sum_{i\geq 0} |b_i| \right) \\
    &\;\quad
    + \left( \sum_{j> \left\lfloor m/2 \right\rfloor} |a_j| \right) \left( \sum_{i> \left\lfloor m/2 \right\rfloor} |b_i| \right)
    \lesssim m^{-\beta} ,
\end{align*}
where the last two lines used similar arguments in Section~3.3 of \cite{Berkesetal2011} and condition~(b) which also implies $\sum_{j\geq 0} |a_j| + \sum_{i\geq 0} |b_i| \leq c$ for some constant $c>0$. Noting that the definition of $Z_t^{(m)}$ automatically fulfills (ii) of Definition~\ref{def: weakly_M_dependent}, the proof for part~(ii) is complete.

Lastly, for part~(iii), define $W_t= Y_t^2$, and
\[
W_t^{(m)} = \sum_{j=0}^{\left\lfloor m/2 \right\rfloor} \sum_{i=0}^{\left\lfloor m/2 \right\rfloor} a_j a_i \epsilon_{t-j} \epsilon_{t-i}
= \sum_{j=0}^{\left\lfloor m/2 \right\rfloor} a_j^2 \epsilon_{t-j}^{2} + \sum_{j=0}^{\left\lfloor m/2 \right\rfloor} \sum_{i=0; i\neq j}^{\left\lfloor m/2 \right\rfloor} a_j a_i \epsilon_{t-j} \epsilon_{t-i} .
\]
Note that we can also write
\[
W_t = \sum_{j=0}^{\infty} a_j^2 \epsilon_{t-j}^{2} + \sum_{j=0}^{\infty} \sum_{i=0; i\neq j}^{\infty} a_j a_i \epsilon_{t-j} \epsilon_{t-i} .
\]
Hence, similar to the proof for part~(ii) above,
\begin{align*}
    \left\| W_t - W_t^{(m)} \right\|_{\nu} &\leq
    \left\| \sum_{j=0}^{\infty} a_j^2 \epsilon_{t-j}^{2} - \sum_{j=0}^{\left\lfloor m/2 \right\rfloor} a_j^2 \epsilon_{t-j}^{2} \right\|_{\nu}
    \\
    & \quad    + \left\| \sum_{j=0}^{\infty} \sum_{i=0; i\neq j}^{\infty} a_j a_i \epsilon_{t-j} \epsilon_{t-i} - \sum_{j=0}^{\left\lfloor m/2 \right\rfloor} \sum_{i=0; i\neq j}^{\left\lfloor m/2 \right\rfloor} a_j a_i \epsilon_{t-j} \epsilon_{t-i} \right\|_{\nu} \\
    &\lesssim
    \sum_{j> \left\lfloor m/2 \right\rfloor} a_j^2
    + m^{-\beta} 
    \lesssim m^{-2\beta} + m^{-\beta} ,
\end{align*}
where the last line used the repetitive arguments as for part~(ii) in the proof of this lemma, together with similar arguments in Section~3.3 of \cite{Berkesetal2011}. This concludes the proof of this lemma.
\end{proof}


\paragraph{Lemmas for controlling the partial sums.}

\begin{lemma}\label{lemma: loading_prod_rate}
Suppose that all the assumptions made in Lemma~\ref{lemma: consistency_proj} hold.
\begin{enumerate}[label = (\roman*)]
\item For all $k \in [K]$, we have
\begin{align*}
\frac{1}{p_k}\bnorm{\wh{\Lambda}_k\wh{\Lambda}_k^\trans- \Lambda_k\Lambda_k^\trans}_F=\cO_P\bigl(\alpha_{T,p,k}\bigl),\text{\ and\ }
\frac{1}{\pmk}\bnorm{\whLambdamk\whLambdamk^\trans- \Lambdamk\Lambdamk^\trans}_F=\cO_P(\alpha_{T,p,\text{-}k}).
\end{align*}

\item Let us write
\begin{align*}
&H_1 = \max_{k\in[K]} \Vert \wh{H}_k \Vert_, \quad H_2 = \max_{k\in[K]}\l\Vert \wh{H}_k \wh{H}_k^\trans  \r\Vert,
\text{\ and \ }H_{3} = \max_{k\in[K]}\frac{1}{\sqrt{p_k}}\l\Vert \wh{\Lambda}_k -\Lambda_k \wh{H}_k\r\Vert_F.
\end{align*}
With some finite constant $C_H>0$ and $\alpha_{T, p}$ in~\eqref{eq:alpha:rate}, let us define 
\begin{align*}
\cH_{T, p} = \l\{ \vert H_1 - 1 \vert \ge C_H\alpha_{T,p} \r\} \cup \l\{ \vert H_2 - 1 \vert \ge C_H\alpha_{T,p} \r\} \cup \l\{ H_3 \ge C_H\alpha_{T,p} \r\}. 
\end{align*}
Then we have $\P(\cH_{T,p})\to 0$ as $\min(T,p_1,\ldots,p_K)\to \infty$.
\end{enumerate}
\end{lemma}

\begin{proof}[Proof of Lemma~\ref{lemma: loading_prod_rate}~(i)]
For any $k\in[K]$, we have
\begin{align*}
\frac{1}{p_k}\left(\wh{\Lambda}_k\wh{\Lambda}_k^\trans- \Lambda_k\Lambda_k^\trans\right)
=\frac{1}{p_k}\left(\wh{\Lambda}_k-\Lambda_k\wh{H}_k \right)\wh{\Lambda}_k^\trans 
+ \frac{1}{p_k}\Lambda_k\wh{H}_k\left(\wh{\Lambda}_k-\Lambda_k\wh{H}_k\right)^\trans + \frac{1}{p_k}\Lambda_k\left(\wh{H}_k\wh{H}_k^\trans - I_{r_k} \right)\Lambda_k^\trans.
\end{align*}
Lemma~\ref{lemma: consistency_proj} shows that $\big\|\wh{\Lambda}_k-\Lambda_k\wh{H}_k \big\|_F=\cO_P\bigl(\sqrt{p_k}\alpha_{T,p,k}\bigr)$, and since $\bnorm{\wh{\Lambda}_k}_F = \sqrt{r_kp_k}$ and $\bnorm{\wh{\Lambda}_k}=\sqrt{p_k}$, we see
\begin{align*}
\frac{1}{p_k}\Big\|\left(\wh{\Lambda}_k-\Lambda_k\wh{H}_k \right)\wh{\Lambda}_k^\trans\Big\|_F=\cO_P\bigl(\alpha_{T,p,k}\bigr),\text{\ and \ } \frac{1}{p_k}\Big\|\Lambda_k\wh{H}_k\bigl(\wh{\Lambda}_k-\Lambda_k\wh{H}_k\bigr)^\trans \Big\|_F=\cO_P\bigl(\alpha_{T,p,k}\bigr).
\end{align*}
Also, by Lemma~\ref{lemma: consistency_proj}, we have
$\wh{H}_k \wh{H}_k^\trans = I_{r_k} + \cO_P\bigl(\alpha_{T,p,k}\bigr)$, so that,
\begin{align*}
\frac{1}{p_k}\Big\|\Lambda_k\left(\wh{H}_k\wh{H}_k^\trans - I_{r_k} \right)\Lambda_k^\trans\Big\|_F\le 
\frac{1}{p_k}\bnorm{\Lambda_k}_F^2\bnorm{\wh{H}_k\wh{H}_k^\trans-I_{r_k}}_F=\cO_P\bigl(\alpha_{T,p,k}\bigr).
\end{align*}
Combining the above, we have
\begin{equation}\label{eqn: loading_prod_rate}
\frac{1}{p_k}\Big\|\wh{\Lambda}_k\wh{\Lambda}_k^\trans- \Lambda_k\Lambda_k^\trans\Big\|_F=\cO_P\bigl(\alpha_{T,p,k}\bigr). 
\end{equation}

For the second claim, it suffices to prove the result for $k=1$,  
for which
\[
\wh{\Lambda}_{-1}\wh{\Lambda}_{-1}^\trans
= \otimes_{\ell=2}^K \wh{\Lambda}_\ell\wh{\Lambda}_\ell^\trans, \text{\ and \ } \Lambda_{-1}\Lambda_{-1}^\trans
= \otimes_{\ell=2}^K \Lambda_\ell\Lambda_\ell^\trans .
\]
By successively adding and subtracting intermediate terms, we obtain
\begin{align}
\otimes_{\ell=2}^K \wh{\Lambda}_\ell\wh{\Lambda}_\ell^\trans - \otimes_{\ell=2}^K \Lambda_\ell\Lambda_\ell^\trans
= \sum_{h=2}^K \Big(\otimes_{\ell=2}^{h-1} \Lambda_\ell\Lambda_\ell^\trans\Big) \otimes \Big(\wh{\Lambda}_h\wh{\Lambda}_h^\trans-\Lambda_h\Lambda_h^\trans\Big) \otimes \Big(\otimes_{\ell=h+1}^{K} \wh{\Lambda}_\ell\wh{\Lambda}_\ell^\trans\Big).
\nonumber 
\end{align}
Using $\|A\otimes B\|_{\F}=\|A\|_{F}\|B\|_{F}$ and the triangle inequality, we have
\begin{align*}
\frac{1}{p_{-1}}\big\|\wh{\Lambda}_{-1}\wh{\Lambda}_{-1}^\trans-\Lambda_{-1}\Lambda_{-1}^\trans\big\|_{F}
&\le \frac{1}{p_{-1}}\sum_{h=2}^K \Big(\prod_{\ell=2}^{h-1}\|\Lambda_\ell\Lambda_\ell^\trans\|_{F}\Big) \, \big\|\wh{\Lambda}_h\wh{\Lambda}_h^\trans-\Lambda_h\Lambda_h^\trans\big\|_{F}\, \Big(\prod_{\ell=h+1}^{K}\|\wh{\Lambda}_\ell\wh{\Lambda}_\ell^\trans\|_{F}\Big)
\\
&=\cO_P\Big(\sum_{h=2}^K\alpha_{T,p,h}\Big)
=\cO_P(\alpha_{T,p,-1}),
\end{align*}
from~\eqref{eqn: loading_prod_rate} and that $\|\Lambda_\ell\Lambda_\ell^\trans\|_{F}=\cO(p_\ell)$ under Assumption~\ref{assum: loadings}~(ii), and $\|\wh{\Lambda}_\ell\wh{\Lambda}_\ell^\trans\|_{F}=\cO(p_\ell)$. 
\end{proof}

\begin{proof}[Proof of Lemma~\ref{lemma: loading_prod_rate}~(ii)]
By Lemma~\ref{lemma: consistency_proj}, we have
\begin{align*}
\l\Vert \wh{H}_k \wh{H}_k^\trans \r\Vert  &= 1+ \cO_P(\alpha_{T,p,k}) = 1+ \cO_P(\alpha_{T,p}), \text{ \ and }\\
\Vert \wh{H}_k \Vert &= 1+ \cO_P(\alpha_{T,p,k} ) = 1+ \cO_P(\alpha_{T,p}).
\end{align*}
With $K$ being finite, it follows that $\vert H_1-1 \vert = \cO_P(\alpha_{T,p})$, $\vert H_2-1 \vert = \cO_P(\alpha_{T,p})$, and $H_3 = \cO_P(\alpha_{T,p})$ by Lemma~\ref{lemma: consistency_proj}, so choosing $C_H>0$ large enough gives $\P(\cH_{T,p})\to 0$ as $\min(T,p_1,\ldots,p_K)\to \infty$.
\end{proof}

\begin{lemma}\label{lemma: E(xx)}
Suppose that all the assumptions made in Lemma~\ref{lemma: consistency_proj} hold.
Then we have
\begin{enumerate}[label = (\roman*)]
\item $p^{-1}\big\| \E\big(X_{k,t}X_{k,t}^\trans\big) \big\|_F =\cO(1)$. 
\item ${(p\pmk)}^{-1}\big\| \E\big(E_{k,t}\Lambdamk\Lambdamk^\trans E_{k,t}^\trans \big) \big\|_F=\cO(1)$.
\item $(p\pmk)^{-1}\big\| \E\big(X_{k,t}\Lambdamk\Lambdamk^\trans X_{k,t}^\trans\big) \big\|_F =\cO(1)$.
\item For any $0\le s<\tau<e\le T$, we have
\[
\frac{1}{p\,\pmk} \l(\frac{1}{e-\tau}\sum_{t=\tau+1}^{e} \E\big(E_{k,t}\Lambdamk\Lambdamk^{\trans}E_{k,t}^{\trans}\big) - \frac{1}{\tau-s}\sum_{t=s+1}^{\tau} \E\big(E_{k,t}\Lambdamk\Lambdamk^{\trans}E_{k,t}^{\trans}\big)\r) = \mathbf{O}.
\]
\end{enumerate}
\end{lemma}

\begin{proof}[Proof of Lemma~\ref{lemma: E(xx)}~(i)]
Using $\E(G_{k,t})=\E(E_{k,t})=0$ and the independence of $\{G_{k,t}\}$ and $\{E_{k,t}\}$ from Assumptions~\ref{assum: core_factor} and~\ref{assum: noise}, 
\[
\Big\| \E(X_{k,t}X_{k,t}^\trans) \Big\|_F \le \Big\| \Lambda_k\,\E\big(G_{k,t}\,\Lambdamk^\trans\Lambdamk\,G_{k,t}^\trans\big)\,\Lambda_k^\trans \Big\|_F
+ \Big\| \E(E_{k,t}E_{k,t}^\trans) \Big\|_F =: U_1 + U_2.
\]
By Assumption~\ref{assum: loadings}~(ii), we have $\Lambdamk^\trans\Lambdamk = \otimes_{\ell\ne k}(\Lambda_\ell^\trans\Lambda_\ell)
= \otimes_{\ell\neq k} (p_\ell I_{r_\ell})
= \pmk\, I_{\rmk}$. Therefore
\[
\Lambda_k\,\E\big(G_{k,t}\Lambdamk^\trans\Lambdamk G_{k,t}^\trans\big)\,\Lambda_k^\trans = \Lambda_k\big[\pmk\,\E(G_{k,t}G_{k,t}^\trans)\big]\Lambda_k^\trans.
\]
By sub-multiplicativity of the Frobenius norm and Jensen's inequality, and recalling that the mode-$k$ unfolding of $\cG_t$ is $G_{k,t} = \sum_{j=1}^{q+1} A_{j,k} F_{k,t} \Ajmk{j}^\trans \cdot \mathbb{I}_{\{\theta_{j-1} <t\le \theta_j\}}$, we have
\begin{align*}
\frac{1}{p}\Big\|\Lambda_k\big[\pmk\,\E(G_{k,t}G_{k,t}^\trans)\big]\Lambda_k^\trans\Big\|_F &\le \big\|\E(G_{k,t}G_{k,t}^\trans)\big\|_F \le \E\big(\|G_{k,t}G_{k,t}^\trans\|_F\big) \le  \E\big(\|G_{k,t}\|_F^2\big)
\\
&\le  \sum_{j=1}^{q+1} \Vert A_{j - 1, k} \Vert_F^2 \E(\Vert F_{k, t} \Vert_F^2) \Vert A_{j - 1, -k} \Vert_F^2 \cdot \mathbb{I}_{\{\theta_{j-1} <t\le \theta_j\}}.
\end{align*}
Then by Assumptions~\ref{assum: core_factor} and~\ref{assum: trans_mat}~(i), the LHS is of order $\cO(r)$, and hence $U_1 = \cO(p)$.
As for $U_2$, by Assumption~\ref{assum: noise}, each entry of $E_{k,t}$ has uniformly bounded variance and hence by Jensen's inequality,
\begin{align}
U_2 = \big\|\E(E_{k,t}E_{k,t}^\trans)\big\|_F \le  \E\big(\|E_{k,t}\|_F^2\big) =\sum_{i=1}^{p_k}\sum_{\ell=1}^{\pmk}\E\big((E_{k,t})_{i \ell}^2\big) = \cO(p). \label{eqn: E(EE)}
\end{align}
\end{proof}

\begin{proof}[Proof of Lemma~\ref{lemma: E(xx)}~(ii)]
By Assumption~\ref{assum: loadings}~(ii), by Jensen's inequality, we have 
\begin{align*}
\frac{1}{p\pmk} \big\Vert\E(E_{k,t}\Lambdamk\Lambdamk^{\trans}E_{k,t}^{\trans}) \big\Vert_F &\le \frac{1}{p\pmk} \E\Big(\big\Vert E_{k,t}\Lambdamk\Lambdamk^{\trans}E_{k,t}^{\trans}\big\Vert_F \Big) \le \frac{1}{p\pmk} \big\Vert \Lambdamk\Lambdamk^{\trans} \big\Vert \E(\Vert E_{k,t}\Vert_F^2) 
\\
&= \frac{1}{p} \E(\Vert E_{k,t}\Vert_F^2) = \cO(1),
\end{align*}
where the last equality follows from~\eqref{eqn: E(EE)}.
\end{proof}

\begin{proof}[Proof of Lemma~\ref{lemma: E(xx)}~(iii)]
Since $\E(G_{k,t})=\E(E_{k,t})=0$, and from the independence of $\{G_{k,t}\}$ and $\{E_{k,t}\}$, we have
\begin{align*}
&\quad\;\frac{1}{p\pmk}\E\big(X_{k,t}\Lambdamk\Lambdamk^\trans X_{k,t}^\trans\big) \\
&= \frac{1}{p_k}\Lambda_k \E\big(G_{k,t} G_{k,t}^\trans \big) \Lambda_k^\trans + \frac{1}{p}\Lambda_k \E\big(G_{k,t}\Lambdamk^\trans E_{k,t}^\trans \big) + \frac{1}{p}\E\big(E_{k,t}\Lambdamk G_{k,t}^\trans \big)\Lambda_k^\trans + \frac{1}{p\pmk}\E\big(E_{k,t}\Lambdamk\Lambdamk^\trans E_{k,t}^\trans\big)\\
&= \frac{1}{p_k}\Lambda_k \E\big(G_{k,t} G_{k,t}^\trans \big) \Lambda_k^\trans  + \frac{1}{p\pmk}\E\big(E_{k,t}\Lambdamk\Lambdamk^\trans E_{k,t}^\trans\big).
\end{align*}
Recalling the arguments in the proof of Lemma~\ref{lemma: E(xx)}~(i), using Assumption~\ref{assum: loadings} and Lemma~\ref{lemma: E(xx)}~(ii), we have
\begin{align*}
&\quad\; \frac{1}{p\pmk} \l\Vert \E\big(X_{k,t}\Lambdamk\Lambdamk^\trans X_{k,t}^\trans\big) \r\Vert_F \\
&\le \frac{1}{p_k} \l\Vert\Lambda_k \E\big(G_{k,t} G_{k,t}^\trans \big) \Lambda_k^\trans \r\Vert_F  + \frac{1}{p\pmk}\l\Vert\E\big(E_{k,t}\Lambdamk\Lambdamk^\trans E_{k,t}^\trans\big)\r\Vert_F\\
&\le \frac{1}{p_k}\Vert \Lambda_k ^\trans\Lambda_k  \Vert \l\Vert\E\big(G_{k,t} G_{k,t}^\trans \big)  \r\Vert_F  + \frac{1}{p\pmk}\l\Vert\E\big(E_{k,t}\Lambdamk\Lambdamk^\trans E_{k,t}^\trans\big)\r\Vert_F=\cO(1),
\end{align*}
which completes the proof.
\end{proof}

\begin{proof}[Proof of Lemma~\ref{lemma: E(xx)}~(iv)]
For any $i,j\in[p_k]$, we have
\begin{align*}
\l[\E\big(E_{k,t}\Lambdamk\Lambdamk^{\trans}E_{k,t}^{\trans}\big)\r]_{ij} &= \sum_{u=1}^{p_k}\sum_{v=1}^{p_k} \l(\Lambdamk\Lambdamk^{\trans} \r)_{uv} \E\big((E_{k,t})_{iu} (E_{k,t}^{\trans})_{jv}\big).
\end{align*}
Under Assumption~\ref{assum: noise}, by the stationarity of $\{E_{k,t}\}$, we have $\E\big((E_{k,t})_{iu} (E_{k,t}^{\trans})_{jv}\big)$ is constant in $t$, for all $i,j,u,v$. Therefore, for any $\tau\in(s,e)$, we have
\[
\frac{1}{p\,\pmk} \l(\frac{1}{e-\tau}\sum_{t=\tau+1}^{e} \E\big(E_{k,t}\Lambdamk\Lambdamk^{\trans}E_{k,t}^{\trans}\big) - \frac{1}{\tau-s}\sum_{t=s+1}^{\tau} \E\big(E_{k,t}\Lambdamk\Lambdamk^{\trans}E_{k,t}^{\trans}\big)\r) = \mathbf{O}.
\] 
\end{proof}

Let us define
\begin{align}
\wh{V}_t = \begin{bmatrix}
\wh{V}^{(1)}_t \\ \vdots \\ \wh{V}^{(K)}_t   
\end{bmatrix} 
= \begin{bmatrix}
\Vech\l( \wh{G}_{1,t} \wh{G}_{1,t}^\trans \r) \\
\vdots \\
\Vech\l( \wh{G}_{K,t} \wh{G}_{K,t}^\trans \r)
\end{bmatrix}\text{\ and\ }
V_t = \begin{bmatrix}
V^{(1)}_t \\ \vdots \\ V^{(K)}_t   
\end{bmatrix}
= \begin{bmatrix}
\Vech\l( \wh{H}_1^\trans \E({G}_{1,t} {G}_{1,t}^\trans) \wh{H}_1 \r) \\
\vdots \\
\Vech\l( \wh{H}_K^\trans \E({G}_{K,t} {G}_{K,t}^\trans) \wh{H}_K \r)
\end{bmatrix}.
\label{eqn: Vt_def}
\end{align}
For any $0\le s<\tau <e \le T$, define
\begin{align}
&\wh{\cV}^{(k)}_{s, \tau, e} = \sqrt{\frac{(\tau - s)(e - \tau)}{e - s}} \, \Vech\left(\wh{\Gamma}_{G,\tau,e}^{(k)} - \wh{\Gamma}_{G,s,\tau}^{(k)}\right), \nonumber \\
&{\cV}^{(k)}_{s, \tau, e} = \sqrt{\frac{(\tau - s)(e - \tau)}{e - s}} \, \Vech\left\{ \wh{H}_k^\trans\left({\Gamma}_{G,\tau,e}^{(k)}- {\Gamma}_{G,s,\tau}^{(k)}\right)\wh{H}_k\right\}, \label{eq:vk:cusum}
\end{align}
and
\begin{align*}
\wh{\mathcal{V}}_{s, \tau, e} &= \begin{bmatrix}
\wh{\cV}^{(1)}_{s, \tau, e} \\ \vdots \\ \wh{\cV}^{(K)}_{s, \tau, e}
\end{bmatrix} = \sqrt{\frac{(\tau - s)(e - \tau)}{e - s}} \l( \frac{1}{e - \tau} \sum_{t = \tau + 1}^{e} \wh{V}_t - \frac{1}{\tau - s} \sum_{t = s + 1}^\tau \wh{V}_t \r),
\\
\mathcal{V}_{s, \tau, e} &= \begin{bmatrix}
{\cV}^{(1)}_{s, \tau, e} \\ \vdots \\ {\cV}^{(K)}_{s, \tau, e}
\end{bmatrix} = \sqrt{\frac{(\tau - s)(e - \tau)}{e - s}} \l( \frac{1}{e - \tau} \sum_{t = \tau + 1}^{e} V_t - \frac{1}{\tau - s} \sum_{t = s + 1}^\tau V_t \r).
\end{align*}

\begin{lemma}\label{lemma: GG-EGG_generic_decomp}
Let Assumptions~\ref{assum: loadings} and \ref{assum: noise} hold, 
and recall that $\Theta$ denotes the set of true change points.
Then we have the following decomposition
\begin{align*}
&\quad\;\big \vert \wh{\cV}^{(k)}_{s, \tau, e} - {\cV}^{(k)}_{s, \tau, e}\big \vert_2 \\
&\le  \sqrt{\frac{(\tau - s)(e - \tau)}{e-s}}\l(\Big\Vert\cJ_{k,\tau,e}^{(1)}-\cJ_{k,s,\tau}^{(1)} \Big\Vert_F+ \Big\Vert\cJ_{k,\tau,e}^{(2)}-\cJ_{k,s,\tau}^{(2)}\Big\Vert_F\r) \\
&\quad \, +\,  2\sqrt{2}\Vert\wh{H}_k\Vert \l\Vert \frac{1}{\sqrt{p_k}}\l(\wh{\Lambda}_k-\Lambda_k\wh{H}_k \r) \r\Vert_F \cdot  \sqrt{\frac{(\tau - s)(e - \tau)}{e-s}}\cdot \l\vert  \Vech\l(\Gamma_{G,\tau,e}^{(k)} - \Gamma_{G,s, \tau}^{(k)} \r)\r\vert_2,
\end{align*}
where for any $0\le a \le b\le T$,
\begin{align*}
&\cJ_{k,a,b}^{(1)}  = \frac{1}{p\pmk}\frac{1}{b-a}\sum_{t=a+1}^{b} \left[X_{k,t}\Lambdamk\Lambdamk^\trans X_{k,t}^\trans-\E\left(X_{k,t}\Lambdamk\Lambdamk^\trans X_{k,t}^\trans \right) \right],\\
&\cJ_{k,a,b}^{(2)} =  \frac{1}{p\pmk}\frac{1}{b-a}\sum_{t=a+1}^{b} X_{k,t}\left(\whLambdamk\whLambdamk^\trans-\Lambdamk\Lambdamk^\trans \right)X_{k,t}^\trans.
\end{align*}
\end{lemma}

\begin{proof}[Proof of Lemma~\ref{lemma: GG-EGG_generic_decomp}]
Let us denote by $C_{k,t}$ the mode-$k$ unfolding of $\cC_t$. 
For any $0\le a<b\le T$, we have
\begin{align*}
{\Gamma}_{G,a,b}^{(k)} &=\frac{1}{b-a}\frac{1}{p^2}\sum_{t=a+1}^b \E\left({\Lambda}_k^\trans C_{k,t}\Lambdamk\Lambdamk^\trans C_{k,t}^\trans{\Lambda}_k \right),
\end{align*}
by Assumption~\ref{assum: loadings}, and
\begin{align*}
\wh{\Gamma}_{G,a,b}^{(k)} &=\frac{1}{b-a}\sum_{t=a+1}^b\wh{G}_{k,t}\wh{G}_{k,t}^\trans = \frac{1}{p^2}\frac{1}{b-a}\sum_{t=a+1}^b\wh{\Lambda}_k^\trans X_{k,t}\whLambdamk\whLambdamk^\trans X_{k,t}^\trans\wh{\Lambda}_k\\
&= \frac{1}{p^2}\frac{1}{b-a}\sum_{t=a+1}^b\Big\{ \wh{\Lambda}_k^\trans\left(X_{k,t}\Lambdamk\Lambdamk^\trans X_{k,t}^\trans-\E\left(X_{k,t}\Lambdamk\Lambdamk^\trans X_{k,t}^\trans \right) \right)\wh{\Lambda}_k\Big\}\\
&\quad +\frac{1}{p^2}\frac{1}{b-a}\sum_{t=a+1}^b\Big\{\wh{\Lambda}_k^\trans\left[X_{k,t}\left(\whLambdamk\whLambdamk^\trans-\Lambdamk\Lambdamk^\trans \right)X_{k,t}^\trans \right]\wh{\Lambda}_k \Big\}\\
&\quad +\frac{1}{p^2}\frac{1}{b-a}\sum_{t=a+1}^b\Big\{ \wh{\Lambda}_k^\trans\Big\{\E\left(X_{k,t}\Lambdamk\Lambdamk^\trans X_{k,t}^\trans \right) - \E\left(C_{k,t}\Lambdamk\Lambdamk^\trans C_{k,t}^\trans \right) \Big\}\wh{\Lambda}_k\Big\}\\
&\quad+ \frac{1}{p^2}\frac{1}{b-a}\sum_{t=a+1}^b\Big\{\wh{\Lambda}_k^\trans\E\left(C_{k,t}\Lambdamk\Lambdamk^\trans C_{k,t}^\trans \right)\wh{\Lambda}_k -(\Lambda_k\wh{H}_k)^\trans\E\left(C_{k,t}\Lambdamk\Lambdamk^\trans C_{k,t}^\trans \right)(\Lambda_k\wh{H}_k) \Big\}\\
&\quad+\frac{1}{p^2}\frac{1}{b-a}\sum_{t=a+1}^b\Big\{(\Lambda_k\wh{H}_k)^\trans\E\left(C_{k,t}\Lambdamk\Lambdamk^\trans C_{k,t}^\trans \right)(\Lambda_k\wh{H}_k)\Big\}\\
&= \frac{1}{p_k}\wh{\Lambda}_k^\trans\cJ_{k,a,b}^{(1)}\wh{\Lambda}_k + \frac{1}{p_k}\wh{\Lambda}_k^\trans\cJ_{k,a,b}^{(2)}\wh{\Lambda}_k  + \frac{1}{p_k}\wh{\Lambda}_k^\trans\cJ_{k,a,b}^{(3)}\wh{\Lambda}_k  + \cJ_{k,a,b}^{(4)}  + \wh{H}_k^\trans\Gamma_{G,a,b}^{(k)} \wh{H}_k,
\end{align*}
where
\begin{align*}
&\cJ_{k,a,b}^{(3)} = \frac{1}{p\pmk}\frac{1}{b-a}\sum_{t=a+1}^b  \Big\{\E\left(X_{k,t}\Lambdamk\Lambdamk^\trans X_{k,t}^\trans \right) - \E\left(C_{k,t}\Lambdamk\Lambdamk^\trans C_{k,t}^\trans \right) \Big\}, \text{\ and }\\
&\cJ_{k,a,b}^{(4)} = \frac{1}{p^2}\frac{1}{b-a}\sum_{t=a+1}^b\Big\{\wh{\Lambda}_k^\trans\E\left(C_{k,t}\Lambdamk\Lambdamk^\trans C_{k,t}^\trans \right)\wh{\Lambda}_k -(\Lambda_k\wh{H}_k)^\trans\E\left(C_{k,t}\Lambdamk\Lambdamk^\trans C_{k,t}^\trans \right)(\Lambda_k\wh{H}_k) \Big\}.
\end{align*}
Then, for any $0\le s < \tau <e\le T$, we have
\begin{align*}
&\quad\; \l\vert \wh{\cV}_{s,\tau,e}^{(k)}- {\cV}_{s,\tau,e}^{(k)} \r\vert_2\\
&= \sqrt{\frac{(\tau - s)(e - \tau)}{e-s}} \left\vert \frac{1}{e - \tau} \sum_{t = \tau + 1}^{e} \l( \wh{V}_t^{(k)} - V_t^{(k)} \r) - \frac{1}{\tau - s} \sum_{t = s + 1}^\tau \l( \wh{V}_t^{(k)} - V_t^{(k)} \r) \right\vert_2\\
&\le\left\Vert \frac{1}{e-\tau} \sum_{t= \tau + 1}^{e} \left[\wh{G}_{k,t} \wh{G}_{k,t}^\trans - \wh{H}_k^\trans {\Gamma}_{G,\tau,e}^{(k)} \wh{H}_k\right] - \frac{1}{\tau-s} \sum_{t= s + 1}^{\tau} \left[\wh{G}_{k,t} \wh{G}_{k,t}^\trans - \wh{H}_k^\trans {\Gamma}_{G,s,\tau}^{(k)}\wh{H}_k\right]\right\Vert_F\\
&= \left\Vert \left( \wh{\Gamma}_{G,\tau,e}^{(k)} - \wh{H}_k^\trans \, {\Gamma}_{G,\tau,e}^{(k)}\, \wh{H}_k \right) - \left(\wh{\Gamma}_{G,s,\tau}^{(k)} - \wh{H}_k^\trans \, {\Gamma}_{G,s,\tau}^{(k)}\, \wh{H}_k\right)\right\Vert_F\\
&\le \left\Vert \cJ_{k,\tau,e}^{(1)} -\cJ_{k,s,\tau}^{(1)}\right\Vert_F  + \left\Vert \cJ_{k,\tau,e}^{(2)} - \cJ_{k,s,\tau}^{(2)}\right\Vert_F + \left\Vert \cJ_{k,\tau,e}^{(3)} - \cJ_{k,s,\tau}^{(3)}\right\Vert_F  + \left\Vert \cJ_{k,\tau,e}^{(4)} - \cJ_{k,s,\tau}^{(4)}\right\Vert_F.
\end{align*}

Since $\{E_{k,t}\}$ and $\{C_{k,t}\}$ are independent (Assumption~\ref{assum: noise}),
we have
\begin{align}
\E\left(X_{k,t}\Lambdamk\Lambdamk^\trans X_{k,t}^\trans \right) - \E\left(C_{k,t}\Lambdamk\Lambdamk^\trans C_{k,t}^\trans \right) = \E\left(E_{k,t}\Lambdamk\Lambdamk^\trans E_{k,t}^\trans \right).\nonumber
\end{align}

Then for $\big\Vert \cJ_{k,\tau,e}^{(3)} -\cJ_{k,s,\tau}^{(3)}\big\Vert_F$, by Lemma~\ref{lemma: E(xx)}~(iv),
\begin{align*}
&\quad\;\big\Vert \cJ_{k,\tau,e}^{(3)} -\cJ_{k,s,\tau}^{(3)}\big\Vert_F   \\
&=\Bigg\Vert \frac{1}{p^2}\frac{1}{e-\tau}\sum_{t=\tau+1}^{e} \Big\{\E\left(X_{k,t}\Lambdamk\Lambdamk^\trans X_{k,t}^\trans \right) - \E\left(C_{k,t}\Lambdamk\Lambdamk^\trans C_{k,t}^\trans \right) \Big\}\\
&\quad\qquad - \frac{1}{p^2}\frac{1}{\tau-s}\sum_{t=s+1}^\tau \Big\{\E\left(X_{k,t}\Lambdamk\Lambdamk^\trans X_{k,t}^\trans \right) - \E\left(C_{k,t}\Lambdamk\Lambdamk^\trans C_{k,t}^\trans \right) \Big\}\Bigg\Vert_F \\
&=\frac{1}{p\pmk}\Bigg\|\frac{1}{e-\tau}\sum_{t=\tau+1}^{e}\E\left(E_{k,t}\Lambdamk\Lambdamk^\trans E_{k,t}^\trans \right) - \frac{1}{\tau-s}\sum_{t=s+1}^\tau\E\left(E_{k,t}\Lambdamk\Lambdamk^\trans E_{k,t}^\trans \right) \Bigg\Vert_F = 0.
\end{align*}

By Assumption~\ref{assum: loadings}, 
\begin{align*}
&\quad\;\big\Vert \cJ_{k,\tau,e}^{(4)} -\cJ_{k,s,\tau}^{(4)}\big\Vert_F\\
&\le \Bigg\Vert \frac{1}{p_k^2}\frac{1}{e-\tau}\sum_{t=\tau+1}^{e} \big(\wh{\Lambda}_k-\Lambda_k\wh{H}_k\big)^\trans\Lambda_k\E\big(G_{k,t}G_{k,t}^\trans \big)\Lambda_k^\trans\wh{\Lambda}_k\\
&\quad\quad- \frac{1}{p_k^2}\frac{1}{\tau-s}\sum_{t=s+1}^\tau \big(\wh{\Lambda}_k-\Lambda_k\wh{H}_k\big)^\trans\Lambda_k\E\big(G_{k,t}G_{k,t}^\trans \big)\Lambda_k^\trans\wh{\Lambda}_k\Bigg\Vert_F\\
&\quad + \Bigg\Vert  \frac{1}{p_k^2}\frac{1}{e-\tau}\sum_{t=\tau+1}^{e} \big(\Lambda_k\wh{H}_k\big)^\trans\Lambda_k\E\big(G_{k,t}G_{k,t}^\trans \big)\Lambda_k^\trans\big(\wh{\Lambda}_k-\Lambda_k\wh{H}_k\big)\\
&\quad\quad\quad-\frac{1}{p_k^2}\frac{1}{\tau-s}\sum_{t=s+1}^\tau \big(\Lambda_k\wh{H}_k\big)^\trans\Lambda_k\E\big(G_{k,t}G_{k,t}^\trans \big)\Lambda_k^\trans\big(\wh{\Lambda}_k-\Lambda_k\wh{H}_k\big)\Bigg\Vert_F\\
& \le 2\Vert\wh{H}_k\Vert \l\Vert \frac{1}{\sqrt{p_k}}\l(\wh{\Lambda}_k-\Lambda_k\wh{H}_k \r) \r\Vert_F \cdot \Bigg\Vert \frac{1}{e-\tau}\sum_{t=\tau+1}^{e} \E\big(G_{k,t}G_{k,t}^\trans \big) - \frac{1}{\tau-s}\sum_{t=s+1}^\tau \E\big(G_{k,t}G_{k,t}^\trans \big)\Bigg\Vert_F\\
& = 2\Vert\wh{H}_k\Vert \l\Vert \frac{1}{\sqrt{p_k}}\l(\wh{\Lambda}_k-\Lambda_k\wh{H}_k \r) \r\Vert_F \cdot \Big\Vert  \Gamma_{G,\tau,e}^{(k)} - \Gamma_{G,s, \tau}^{(k)}\Big\Vert_F\\
&\le 2\sqrt{2}\Vert\wh{H}_k\Vert \l\Vert \frac{1}{\sqrt{p_k}}\l(\wh{\Lambda}_k-\Lambda_k\wh{H}_k \r) \r\Vert_F \cdot \l\vert  \Vech\l(\Gamma_{G,\tau,e}^{(k)} - \Gamma_{G,s, \tau}^{(k)} \r)\r\vert_2.
\end{align*}

Thus, we have
\begin{align*}
&\quad\;\big \vert \wh{\cV}^{(k)}_{s, \tau, e} - {\cV}^{(k)}_{s, \tau, e}\big \vert_2 \\
&= \sqrt{\frac{(\tau - s)(e - \tau)}{e-s}} \left\vert \frac{1}{e - \tau} \sum_{t = \tau + 1}^{e} \l( \wh{V}_t^{(k)} - V_t^{(k)} \r) - \frac{1}{\tau - s} \sum_{t = s + 1}^\tau \l( \wh{V}_t^{(k)} - V_t^{(k)} \r) \right\vert_2
\\
&\le \sqrt{\frac{(\tau - s)(e - \tau)}{e-s}}\l(\Big\Vert\cJ_{k,\tau,e}^{(1)}-\cJ_{k,s,\tau}^{(1)} \Big\Vert_F+ \Big\Vert\cJ_{k,\tau,e}^{(2)}-\cJ_{k,s,\tau}^{(2)}\Big\Vert_F\r) \\
&\quad \, +\,  2\sqrt{2}\Vert\wh{H}_k\Vert \l\Vert \frac{1}{\sqrt{p_k}}\l(\wh{\Lambda}_k-\Lambda_k\wh{H}_k \r) \r\Vert_F \cdot  \sqrt{\frac{(\tau - s)(e - \tau)}{e-s}}\cdot \l\vert  \Vech\l(\Gamma_{G,\tau,e}^{(k)} - \Gamma_{G,s, \tau}^{(k)} \r)\r\vert_2.
\end{align*}
\end{proof}

\begin{lemma}\label{lemma: xx-Exx}
Let Assumptions~\ref{assum: core_factor}, \ref{assum: loadings}, \ref{assum: trans_mat_alt} and~\ref{assum: noise} hold, with $\nu>8$ and $\beta\geq 0$.
For $\M$ given in~\eqref{def: seeded_interval}, 
let us re-index its elements as $\M = \{ (s_\ell, e_\ell], \, \ell \in \vert \M \vert\}$, and define
\begin{align*}
\cI := \l\{(\sll,\tau,\el): \ell \in \vert \M\vert,\, \min(\tau - s_\ell, e_\ell - \tau) > \varpi_T\r\}.
\end{align*}
Then we have
\begin{align}
\frac{1}{p} \max_{(\sll,\tau,\el) \in \cI} \max\left\{ \frac{1}{\sqrt{\tau - s_\ell}} \left\Vert \sum_{t = s_\ell + 1}^\tau \big[ (X_{k, t}X_{k, t}^\trans - \E(X_{k, t}X_{k, t}^\trans) \big] \right\Vert_F, \right. & \nonumber
\\
\left. \frac{1}{\sqrt{e_\ell - \tau}} \left\Vert \sum_{t = \tau + 1}^{e_\ell} \big[ (X_{k, t}X_{k, t}^\trans - \E(X_{k, t}X_{k, t}^\trans) \big] \right\Vert_F \right\} &= \cO_P\left( \sqrt{\log(T)} \right).\nonumber
\end{align}
\end{lemma}

\begin{proof}[Proof of Lemma~\ref{lemma: xx-Exx}]
Recalling that the collection $\M$ contains intervals grouped into levels $h = 1, \ldots, \lfloor \mu_T \rfloor$, we write $\M = \cup_{h=1}^{\lfloor \mu_T \rfloor} \M_{h}$ where $\vert \M_{h} \vert = \lfloor T/m_{h}\rfloor - 1 \asymp 2^{h}$ and each interval therein has length $2m_{h}$, with $m_{h}= T/ 2^{h}$. 
Each interval contributes $\cO(m_{h})$ admissible split points, so the total number of admissible triples $(\sll, \tau,\el)$ across all intervals is bounded by $\sum_{h=1}^{\lfloor \mu_T \rfloor} \cO(2^{h}m_{h}) = \cO(T\log\log(T))$.

For $i,j\in[p_k]$, we have
\begin{align*}
[X_{k, t} X_{k, t}^\trans - \E(X_{k, t} X_{k, t}^\trans)]_{i j} 
= \pmk^{-2} \sum_{u=1 }^{\pmk} [ X_{k, t, i u} X_{k, t, j u} - \E(X_{k, t, i u} X_{k, t, j u})].
\end{align*}
Define $S_{k,t}(i,j;u,v)=X_{k, t, i u} X_{k, t, j v} - \E(X_{k, t, i u} X_{k, t, j v})$.
By Proposition~6.5 of \cite{zhang2021convergence} (with $B=0$ and $T=\tau-\sll$ therein) yields, for any $z>0$, $i, j \in [p_k]$ and $u \in [\pmk]$,
\begin{align*}
&\quad\; \P\Bigl(\frac{1}{\sqrt{\tau-\sll}} \Big|\sum_{t=\sll + 1}^{\tau}S_{k,t}(i,j;u,v)\Big|\ge z\Bigr) \\
&\le
C_{\nu,\beta}[\Phi_{\nu,\beta}(\cX)]^{\nu} (\tau-\sll)^{-\nu/4+1} z^{-\nu/2} + C \exp\left(-\frac{z^2}{C_\beta [\Phi_{4,\beta}(\cX)]^4}\right) ,
\end{align*}
where $C_{\nu,\beta}, C$ and $C_\beta$ are constants depending only on $\nu$ and $\beta$, and the uniform d.a.n.\ $\Phi_{\nu,\beta}(\cX)$ is finite by Lemma~\ref{lemma: fdc_Xkt}.

Let $L=\min_{\ell\in [\vert \M \vert]:\sll < \tau < \el}\{\tau-\sll, \el-\tau\}$. By construction, $L \ge \varpi_T \asymp T/\log(T)$. 
Since the number of admissible triples is $\cO(T\log(T))$, by Bonferroni correction, we have
\begin{align*}
&\quad\; \P\left(\max_{(\sll,\tau,\el) \in \cI}\Big|\sum_{t=\sll + 1}^{\tau}S_{k,t}(i,j;u,v)\Big|\ge z\right) \\
&\le
T \log(T)\left\{C_{\nu,\beta}[\Phi_{\nu,\beta}(\cX)]^{\nu} (\tau-\sll)^{-\nu/4+1} z^{-\nu/2} + C \exp\left(-\frac{z^2}{C_\beta [\Phi_{4,\beta}(\cX)]^4}\right)\right\} .
\end{align*}
Then for any $\delta > 0$,
\begin{align*}
&\quad\; \E\left( \left\{\max_{(\sll,\tau,\el) \in \cI} \Big|\sum_{t=\sll +1}^{\tau}S_{k,t}(i,j;u,v)\Big|\right\}^2 \right)\\
&= \int_{0}^{\infty} \P\left(\max_{(\sll,\tau,\el) \in \cI}\Big|\sum_{t=\sll + 1}^{\tau}S_{k,t}(i,j;u,v)\Big|\ge \sqrt{z} \right)\d z\\
&\le \delta^2 + \int_{\delta^2}^{\infty} \P\left(\max_{(\sll,\tau,\el) \in \cI}\Big|\sum_{t=\sll+1}^{\tau}S_{k,t}(i,j;u,v)\Big|\ge \sqrt{z} \right)\d z\\
&\le \delta^2 + T\log(T)\cdot C_{\nu,\beta}[\Phi_{\nu,\beta}(\cX)]^{\nu} L^{-\nu/4+1} \int_{\delta^2}^\infty z^{-\nu/4}\d z 
\\
& \qquad + T\log(T)\cdot C\int_{\delta^2}^\infty \exp\left(-\frac{z}{C_{\beta}[\Phi_{4,\beta}(\cX)]^4} \right) \d z\\
&\lesssim \delta^2 + \frac{4}{\nu-4} T\log(T) L^{-\nu/4+1} \delta^{2-\nu/2} + T\log(T) \exp\left(-\frac{\delta^2}{C_{\beta}[\Phi_{4,\beta}(\cX)]^4} \right) .
\end{align*}
Setting $\delta^2 \ge 2 C_{\beta}[\Phi_{4,\beta}(\cX)]^4\log(T)$, we have
\[
T \log(T)\cdot L^{-\nu/4+1} \delta^{2-\nu/2} \asymp  T^{-\nu/4+2}\log(T).
\]
Hence, with $\nu>8$, the expectation is bounded by
\begin{align*}
&\quad\; \E\left( \left\{\max_{(\sll,\tau,\el) \in \cI} \Big|\sum_{t=\sll+1}^{\tau}S_{k,t}(i,j;u,v)\Big|\right\}^2 \right)\\
&\lesssim \delta^2 + T^{2-\nu/4} \log(T) + T\log(T)\cdot \exp\left(-\log(T) \right) \lesssim \log(T),
\end{align*}
where the hidden constant does not depend on $i$, $j$, $u$ or $v$.
Thus we can derive by Cauchy--Schwarz inequality,
\begin{align*}
&\quad\;\max_{(\sll,\tau,\el) \in \cI} \frac{1}{\sqrt{\tau-\sll}} \bigg|\sum_{t=\sll+1}^{\tau} \left[X_{k,t}X_{k,t}^\trans - \E(X_{k,t}X_{k,t}^\trans) \right]_{i j} \bigg|\\ 
&=\max_{(\sll,\tau,\el) \in \cI} \frac{1}{\sqrt{\tau-\sll}}\Bigg\vert   \sum_{t=\sll + 1}^{\tau} \sum_{u=1}^{\pmk}  S_{k,t}(i,j;u,u)  \Bigg\vert\\
&\le \, \sum_{u=1}^{\pmk}  \max_{(\sll,\tau,\el) \in \cI} \frac{1}{\sqrt{\tau-\sll}}\Bigg\vert  \sum_{t=\sll + 1}^{\tau} S_{k,t}(i,j;u,u)  \Bigg\vert\\
&\le \, \sqrt{\pmk}\l( \sum_{u=1}^{\pmk}  \l\{\max_{(\sll,\tau,\el) \in \cI} \frac{1}{\sqrt{\tau-\sll}}\Bigg\vert  \sum_{t=\sll + 1}^{\tau} S_{k,t}(i,j;u,u)  \Bigg\vert \r\}^2 \r)^{1/2},
\end{align*}
from which we have
\begin{align*}
\E\l(\l\{\max_{(\sll,\tau,\el) \in \cI} \frac{1}{\sqrt{\tau-\sll}} \sum_{t=\sll+1}^{\tau} \left[X_{k,t}X_{k,t}^\trans - \E(X_{k,t}X_{k,t}^\trans) \right]_{ij}  \r\}^2 \r) \lesssim \pmk^2 \log(T),
\end{align*}
where the hidden constant does not depend on $i$ or $j$.
Therefore, we have
\begin{align*}
\frac{1}{p^2} \E\l(\max_{(\sll,\tau,\el) \in \cI}\left\Vert\frac{1}{\sqrt{\tau - s_\ell}} \sum_{t = s_\ell + 1}^\tau \big[ (X_{k, t}X_{k, t}^\trans - \E(X_{k, t}X_{k, t}^\trans) \big] \right\Vert_F^2 \r) \lesssim \log(T).
\end{align*}
By Markov's inequality, for any constant $C>0$, we have
\begin{align*}
&\quad\;\P\l(\frac{1}{p} \max_{(\sll,\tau,\el) \in \cI} \left\Vert \frac{1}{\sqrt{\tau - s_\ell}}  \sum_{t = s_\ell + 1}^\tau \big[ (X_{k, t}X_{k, t}^\trans - \E(X_{k, t}X_{k, t}^\trans) \big] \right\Vert_F  \ge C\sqrt{\log(T)} \r)\\
&\le {\frac{1}{p^2} \E\l(\max_{(\sll,\tau,\el) \in \cI} \left\Vert\frac{1}{\sqrt{\tau - s_\ell}} \sum_{t = s_\ell + 1}^\tau \big[ (X_{k, t}X_{k, t}^\trans - \E(X_{k, t}X_{k, t}^\trans) \big] \right\Vert_F^2 \r)} (C^2\log(T))^{-1} \lesssim\frac{1}{C^2},
\end{align*}
which shows that
\[
{\frac{1}{p} \max_{(\sll,\tau,\el) \in \cI} \left\Vert\frac{1}{\sqrt{\tau - s_\ell}} \sum_{t = s_\ell + 1}^\tau \big[ (X_{k, t}X_{k, t}^\trans - \E(X_{k, t}X_{k, t}^\trans) \big] \right\Vert_F } = \cO_P(\sqrt{\log(T)}).
\]
Analogous arguments apply to the sums over $\{\tau + 1, \ldots, e_\ell\}$ uniformly over $\cI$ so that the conclusion follows.
\end{proof}

\begin{lemma}\label{lemma: imp_maximal_deviation_mat}
Let all assumptions in Lemma~\ref{lemma: xx-Exx} hold. Define $Z_{k, t} = p_{-k}^{-1} X_{k, t} \Lambda_{-k}$ and for $\cI$ defined in Lemma~\ref{lemma: xx-Exx},
\begin{align*}
\cS_{T,p}^{(1)} = \l\{\max_{(\sll,\tau,\el) \in \cI} \max\l(\sqrt{\tau-\sll} \Big\vert \cJ_{\sll,\tau}^{(1)}\Big\vert_2,\, \sqrt{\el-\tau}\Big\vert \cJ_{\tau,\el}^{(1)}\Big\vert_2 \r) \ge C_1\sqrt{\log(T)}\r\},
\end{align*}
with some finite constant $C_1>0$, where for any $0\le a< b\le T$, 
\begin{align*}
\cJ_{a,b}^{(1)}=\begin{bmatrix}
\Vech \l(\cJ_{1,a,b}^{(1)} \r)\\
\vdots\\
\Vech \l(\cJ_{K,a,b}^{(1)} \r)
\end{bmatrix}=\begin{bmatrix}
\Vech\l(\frac{1}{pp_{-1}}\frac{1}{b-a} \sum_{t=a+1}^{b}  Z_{1, t}Z_{1, t}^\trans - \E(Z_{1, t}Z_{1, t}^\trans)\r)\\
\vdots\\
\Vech\l(\frac{1}{pp_{-K}}\frac{1}{b-a} \sum_{t=a+1}^{b} Z_{K, t}Z_{K, t}^\trans - \E(Z_{K, t}Z_{K, t}^\trans)\r)
\end{bmatrix}.
\end{align*}
Then we have
\begin{align*}
\P\l(\cS_{T,p}^{(1)} \r) \to 0  \text{\ as\ } \min(T,p_1,\ldots,p_K)\to \infty.
\end{align*}
\end{lemma}

\begin{proof}[Proof of Lemma~\ref{lemma: imp_maximal_deviation_mat}]
Note that with $\Lambdamk = [\lambda_{-k, u v}, \, u \in [\pmk], v \in [\rmk]]$, we have
\begin{align*}
[Z_{k, t} Z_{k, t}^\trans - \E(Z_{k, t} Z_{k, t}^\trans)]_{i, j} 
= p_{-k}^{-2} \sum_{u,v=1 }^{\pmk} \sum_{\ell=1}^{\rmk} \lambda_{-k, u \ell} \lambda_{-k, v \ell}\, [ X_{k, t, i u} X_{k, t, j v} - \E(X_{k, t, i u} X_{k, t, j v})].
\end{align*}
Recall $S_{k,t}(i,j;u,v)=X_{k, t, i u} X_{k, t, j v} - \E(X_{k, t, i u} X_{k, t, j v})$ defined in the proof of Lemma~\ref{lemma: xx-Exx}. 
For any $i,j\in[p_k]$, we have
\begin{align*}
&\quad\;\max_{(\sll,\tau,\el) \in \cI} \frac{1}{\sqrt{\tau-\sll}} \bigg|\sum_{t=\sll+1}^{\tau} \left[Z_{k,t}Z_{k,t}^\trans - \E(Z_{k,t}Z_{k,t}^\trans) \right]_{i,j} \bigg|\\ 
&=\max_{(\sll,\tau,\el) \in \cI} \frac{1}{\sqrt{\tau-\sll}}\Bigg\vert \pmk^{-2} \sum_{h=1}^{\rmk}\sum_{u=1}^{\pmk}\sum_{v=1}^{\pmk} \lambda_{-k,uh}\lambda_{-k,vh} \sum_{t=\sll + 1}^{\tau} S_{k,t}(i,j;u,v)  \Bigg\vert\\
&\le \pmk^{-2} \sum_{h=1}^{\rmk}\sum_{u=1}^{\pmk}\sum_{v=1}^{\pmk} \big\vert \lambda_{-k,uh}\lambda_{-k,vh} \big\vert \cdot \max_{(\sll,\tau,\el) \in \cI} \frac{1}{\sqrt{\tau-\sll}}\Bigg\vert  \sum_{t=\sll + 1}^{\tau} S_{k,t}(i,j;u,v)  \Bigg\vert.
\end{align*}
By Assumption~\ref{assum: loadings}~(i), we have
\[
\pmk^{-2} \sum_{h=1}^{\rmk}\sum_{u=1}^{\pmk}\sum_{v=1}^{\pmk} \big\vert \lambda_{-k,uh}\lambda_{-k,vh} \big\vert \le \pmk^{-2} \rmk c^2 \pmk^2 = \cO(1).
\]

Moreover, by Lemma~\ref{lemma: xx-Exx}, for any $i,j\in[p_k]$ and any $u,v\in[\pmk]$,
\begin{align*}
\E\left( \left[\max_{(\sll,\tau,\el) \in \cI} \Big|\sum_{t=\sll+1}^{\tau}S_{k,t}(i,j;u,v)\Big|\right]^2 \right) \lesssim \log(T),
\end{align*}
where the implied constant is independent of $i,j,u$ or $v$. Combining the above, by Hölder inequality, we have
\begin{align*}
&\quad\;\E\l(\l[\max_{(\sll,\tau,\el) \in \cI} \frac{1}{\sqrt{\tau-\sll}} \bigg|\sum_{t=\sll+1}^{\tau} \left[Z_{k,t}Z_{k,t}^\trans - \E(Z_{k,t}Z_{k,t}^\trans) \right]_{i j} \bigg| \r]^2 \r)\\
&\le \l(\pmk^{-2} \sum_{h=1}^{\rmk}\sum_{u=1}^{\pmk}\sum_{v=1}^{\pmk} \big\vert \lambda_{-k,uh}\lambda_{-k,vh} \big\vert \r)^2 \max_{u,v\in[\pmk]}\E\left( \left[\max_{(\sll,\tau,\el) \in \cI} \Big|\sum_{t=\sll+1}^{\tau}S_{k,t}(i,j;u,v)\Big|\right]^2 \right)\\
&\lesssim \log(T),
\end{align*}
where the hidden constant does not depend on $i$ and $j$, so that
\begin{align*}
\frac{1}{p_k^2} \E\l(\max_{(\sll,\tau,\el) \in \cI}\left\Vert\frac{1}{\sqrt{\tau - s_\ell}} \sum_{t = s_\ell + 1}^\tau \big[ (Z_{k, t}Z_{k, t}^\trans - \E(Z_{k, t}Z_{k, t}^\trans) \big] \right\Vert_F^2 \r) \lesssim \log(T).
\end{align*}
Further, by Markov's inequality, 
\[
{\frac{1}{p_k} \max_{(\sll,\tau,\el) \in \cI} \left\Vert\frac{1}{\sqrt{\tau - s_\ell}} \sum_{t = s_\ell + 1}^\tau \big[ (Z_{k, t}Z_{k, t}^\trans - \E(Z_{k, t}Z_{k, t}^\trans) \big] \right\Vert_F } = \cO_P(\sqrt{\log(T)}).
\]
The same bound holds for the sums over $\{\tau+1, \ldots, e_\ell\}$ uniformly over $\cI$, so we conclude
\begin{align}
\frac{1}{p_k} \max_{(\sll,\tau,\el) \in \cI}  \max\left\{ \frac{1}{\sqrt{\tau - s_\ell}} \left\Vert \sum_{t = s_\ell + 1}^\tau \big[ Z_{k, t}Z_{k, t}^\trans - \E(Z_{k, t}Z_{k, t}^\trans) \big] \right\Vert_F, \right. & \nonumber
\\
\left. \frac{1}{\sqrt{e_\ell - \tau}} \left\Vert \sum_{t = \tau + 1}^{e_\ell} \big[ Z_{k, t}Z_{k, t}^\trans - \E(Z_{k, t}Z_{k, t}^\trans) \big] \right\Vert_F \right\} &= \cO_P\left( \sqrt{\log(T)} \right). \label{eqn: imp_maximal_deviation_mat}
\end{align}

From the above, for $\tau-\sll > \varpi_T$, we have
\begin{align*}
&\quad\; \max_{(\sll,\tau,\el) \in \cI} \l\vert\sqrt{\tau - s_\ell}\cdot \Vech\l(\cJ_{k,\sll,\tau}^{(1)} \r)\r\vert_2 \\
&\le{\frac{1}{p_k} \max_{(\sll,\tau,\el) \in \cI} \left\Vert\frac{1}{\sqrt{\tau - s_\ell}} \sum_{t = s_\ell + 1}^\tau \big[ (Z_{k, t}Z_{k, t}^\trans - \E(Z_{k, t}Z_{k, t}^\trans) \big] \right\Vert_F } = \cO_P(\sqrt{\log(T)}),
\end{align*}
which gives
\begin{align*}
&\quad\; \max_{(\sll,\tau,\el) \in \cI} \l\vert\sqrt{\tau - s_\ell}\cdot \cJ_{\sll,\tau}^{(1)} \r\vert_2 \\
&\le \l(\sum_{k=1}^K \frac{1}{p_k} \max_{(\sll,\tau,\el) \in \cI} {\left\Vert  \frac{1}{\sqrt{\tau - s_\ell}} \sum_{t = s_\ell + 1}^\tau \big[ (Z_{k, t}Z_{k, t}^\trans - \E(Z_{k, t}Z_{k, t}^\trans) \big] \right\Vert_F^2 }\r)^{1/2} = \cO_P(\sqrt{\log(T)}).
\end{align*}
The same bound holds for the sums over $\{\tau+1, \ldots, e_\ell\}$ uniformly over $\cI$, so that we have $\P(\cS_{T,p}^{(1)}) \to 0$.
\end{proof}

\begin{lemma}\label{lemma: X(Lam-Lam)X}
Let all assumptions in Lemma~\ref{lemma: xx-Exx} hold. For some finite constant $C_2>0$, and with $\cI$ defined in Lemma~\ref{lemma: xx-Exx} and $\alpha_{T,p}$ defined in \eqref{eq:alpha:rate}, we define
\begin{align*}
&\cS_{T,p}^{(2)} = \l\{\max_{(\sll,\tau,\el) \in \cI} \l(\Big\vert\cJ_{\sll,\tau}^{(2)}\Big\vert_2,\, \Big\vert\cJ_{\tau,\el}^{(2)} \Big\vert_2\r)\ge C_2\alpha_{T,p} \r\},
\end{align*}
where for any $0\le a<b\le T$, 
\begin{align*}
\cJ_{a,b}^{(2)} = 
\begin{bmatrix}
\Vech \l(\cJ_{1,a,b}^{(2)} \r)\\
\vdots\\
\Vech \l(\cJ_{K,a,b}^{(2)} \r)
\end{bmatrix}
=\begin{bmatrix}
\Vech\l(\frac{1}{pp_{-1}}\frac{1}{b-a} \sum_{t=a+1}^{b}  X_{1,t}\left(\wh{\Lambda}_{-1}\wh{\Lambda}_{-1}^\trans -{\Lambda}_{-1}{\Lambda}_{-1}^\trans \right)  X_{1,t}^\trans\r)\\
\vdots\\
\Vech\l(\frac{1}{pp_{-K}}\frac{1}{b-a} \sum_{t=a+1}^{b}    X_{K,t}\left(\wh{\Lambda}_{-K}\wh{\Lambda}_{-K}^\trans -{\Lambda}_{-K}{\Lambda}_{-K}^\trans \right)  X_{K,t}^\trans\r)
\end{bmatrix}.
\end{align*}
then we have
\begin{align*}
\P\big(\cS_{T,p}^{(2)} \big) \to 0 \text{\ as\ } \min(T,p_1,\ldots,p_K)\to \infty.
\end{align*}
\end{lemma}

\begin{proof}[Proof of Lemma~\ref{lemma: X(Lam-Lam)X}]
For any $(\sll,\tau,\el)\in\cI$ with $\varpi_T \asymp T/\log(T)$,  
by Lemma~\ref{lemma: E(xx)}~(i) and Equation~\eqref{lemma: xx-Exx}, we have
\begin{align}
& \max_{(\sll,\tau,\el) \in \cI}
\frac{1}{p} \frac{1}{\tau-\sll} \Bigg\|\sum_{t=\sll+1}^{\tau} X_{k,t}X_{k,t}^\trans\Bigg\|_F
\nonumber \\
&\le \max_{(\sll,\tau,\el) \in \cI} \frac{1}{p} \frac{1}{\tau-\sll} \Bigg\|\sum_{t=\sll+1}^{\tau} \left( X_{k,t}X_{k,t}^\trans - \E\big(X_{k,t}X_{k,t}^\trans  \big) \right) \Bigg\|_F  
\nonumber \\
& \qquad +\max_{(\sll,\tau,\el) \in \cI} \frac{1}{p} \frac{1}{\tau-\sll} \Bigg\|\sum_{t=\sll+1}^{\tau}  \E\big(X_{k,t}X_{k,t}^\trans  \big) \Bigg\|_F
\nonumber \\
&\le  \max_{(\sll,\tau,\el) \in \cI} \frac{1}{p} \frac{1}{\tau-\sll} \Bigg\|\sum_{t=\sll+1}^{\tau} \left( X_{k,t}X_{k,t}^\trans - \E\big(X_{k,t}X_{k,t}^\trans  \big) \right) \Bigg\|_F  
\nonumber \\
& \qquad + \max_{(\sll,\tau,\el) \in \cI} \frac{1}{p} \frac{1}{\tau-\sll} \sum_{t=\sll+1}^{\tau}  \Big\|  \E\big( X_{k,t}X_{k,t}^\trans  \big) \Big\|_F
\nonumber \\
&= \cO_P\l( \sqrt{\frac{\log(T)}{\varpi_T}} \r) + \cO(1) = \cO_P(1).
\label{eq:max:xx:frob}
\end{align}
Combining the above with Lemmas~\ref{lemma: inequality_sandwich}~(i) and Lemma~\ref{lemma: loading_prod_rate}, we obtain
\begin{align*}      
& \max_{(\sll,\tau,\el) \in \cI} \frac{1}{p\pmk}\frac{1}{\tau-\sll} \Bigg\|\sum_{t=\sll+1}^{\tau}  X_{k,t}\left(\whLambdamk\whLambdamk^\trans -\Lambdamk\Lambdamk^\trans \right) X_{k,t}^\trans\Bigg\|_F\\
&\le \frac{1}{\pmk}\Big\| {\whLambdamk\whLambdamk^\trans} -\Lambdamk\Lambdamk^\trans\Big\|_F\cdot \max_{(\sll,\tau,\el) \in \cI} \frac{1}{p}\frac{1}{\tau-\sll} \Bigg\|\sum_{t=\sll+1}^{\tau}  X_{k,t}X_{k,t}^\trans\Bigg\|_F=\cO_P(\alpha_{T,p,\text{-}k}).
\end{align*}
Analogous arguments apply to the sums over $\{\tau + 1, \ldots, e_\ell\}$ uniformly over $\cI$. Therefore, we conclude
\begin{align}
\max_{(\sll,\tau,\el) \in \cI} \, &\l\{\Big\Vert \cJ_{k,\sll,\tau}^{(2)}\Big\Vert_F, \, \Big\Vert \cJ_{k,\tau, \el}^{(2)}\Big\Vert_F \r\} \nonumber\\
=\max_{(\sll,\tau,\el) \in \cI} & \l\{ \frac{1}{p\pmk}\frac{1}{\tau-\sll}\Bigg\|\sum_{t=\sll+1}^{\tau} X_{k,t}\left(\whLambdamk\whLambdamk^\trans -\Lambdamk\Lambdamk^\trans \right)X_{k,t}^\trans \Bigg\|_F \r. \nonumber\\
& \quad \l. \frac{1}{p\pmk}\frac{1}{\el-\tau}\Bigg\|\sum_{t=\tau+1}^{\el} X_{k,t}\left(\whLambdamk\whLambdamk^\trans -\Lambdamk\Lambdamk^\trans \right)X_{k,t}^\trans \Bigg\|_F \r\}=\cO_P(\alpha_{T,p,\text{-}k}). \label{eqn: X(Lam-Lam)X}
\end{align}

Moreover, we have
\begin{align*}
&\quad\;\l\vert\Vech\l(\frac{1}{p\pmk}\frac{1}{\tau-\sll} \sum_{t=\sll+1}^{\tau}  X_{k,t}\left(\whLambdamk\whLambdamk^\trans -\Lambdamk\Lambdamk^\trans \right)  X_{k,t}^\trans\r)\r\vert_2 \\
&\le \Bigg\|\frac{1}{p\pmk}\frac{1}{\tau-\sll} \sum_{t=\sll+1}^{\tau}  X_{k,t}\left(\whLambdamk\whLambdamk^\trans -\Lambdamk\Lambdamk^\trans \right) X_{k,t}^\trans\Bigg\|_F,
\end{align*}
and thus
\begin{align*}
& \max_{(\sll,\tau,\el) \in \cI} \Big\vert\cJ_{\sll,\tau}^{(2)}\Big\vert_2
\\
\le & \, \max_{(\sll,\tau,\el) \in \cI} \l(\sum_{k=1}^K \Bigg\|\frac{1}{p\pmk}\frac{1}{\tau-\sll} \sum_{t=\sll+1}^{\tau}  X_{k,t}\left(\whLambdamk\whLambdamk^\trans -\Lambdamk\Lambdamk^\trans \right) X_{k,t}^\trans\Bigg\|_F^2 \r)^{1/2}= \cO_P(\alpha_{T,p}),
\end{align*}
from which we conclude that $\P\big(\cS_{T,p}^{(2)} \big) \to 0$.
\end{proof}

\begin{lemma}\label{lemma: ZZ_2+epsilon}
Let Assumptions~\ref{assum: core_factor}, \ref{assum: loadings}, \ref{assum: trans_mat} and \ref{assum: noise} hold, with $\nu>4$ and $\beta> 1$.
Recalling that $Z_{k, t} = p_{-k}^{-1} X_{k, t} \Lambda_{-k}$, for the constant $\epsilon$ satisfying $\epsilon \in (0, \nu/2 - 2)$, we have for any $k\in[K]$, $i,j\in[p_k]$:
\begin{equation*}
\E\l( \l\vert \sum_{t = s + 1}^e \l( Z_{k, t, i \cdot}^\trans Z_{k, t, j \cdot} - \E( Z_{k, t, i \cdot}^\trans Z_{k, t, j \cdot}) \r) \r\vert^{2 + \epsilon} \r) \le c_0 (e - s)^{1 + \epsilon/2} ,
\end{equation*}
where $c_0 >0 $ is some constant that depends only on $\nu$.
\end{lemma}

\begin{proof}[Proof of Lemma~\ref{lemma: ZZ_2+epsilon}]
Using the notations in~\eqref{eqn: set_of_notation}, we have
\begin{align*}
    &\;\quad
    Z_{k, t} Z_{k, t}^\trans = \pmk^{-2} X_{k,t} \Lambdamk \Lambdamk^\trans X_{k,t}^\trans
    = \pmk^{-2} (\Lambda_k G_{k,t} \Lambdamk^\trans + E_{k,t}) \Lambdamk \Lambdamk^\trans (\Lambda_k G_{k,t} \Lambdamk^\trans + E_{k,t})^\trans \\
    &=
    \Lambda_k G_{k,t} G_{k,t}^\trans \Lambda_k^\trans
    + \pmk^{-1} \Lambda_k G_{k,t} \Lambdamk^\trans E_{k,t}^\trans
    + \pmk^{-1} E_{k,t} \Lambdamk G_{k,t}^\trans \Lambda_k^\trans
    + \pmk^{-2} E_{k,t} \Lambdamk \Lambdamk^\trans E_{k,t}^\trans ,
\end{align*}
so that this lemma holds true by Minkowski's inequality if we can show
\[
\cI_Z,\; \cII_Z,\; \cIII_Z,\; \cIV_Z \leq c_0 (e-s)^{1+\epsilon/2} ,
\]
where 
\begin{align*}
    \cI_Z &= \E\l( \l\vert \sum_{t = s + 1}^e \l( \Lambda_{k,i\cdot}^\trans G_{k,t} G_{k,t}^\trans \Lambda_{k,j\cdot} - \E( \Lambda_{k,i\cdot}^\trans G_{k,t} G_{k,t}^\trans \Lambda_{k,j\cdot} ) \r) \r\vert^{2 + \epsilon} \r) ,\\
    \cII_Z &= \E\l( \l\vert \sum_{t = s + 1}^e \l( \pmk^{-1} \Lambda_{k,i\cdot}^\trans G_{k,t} \Lambdamk^\trans E_{k,t,j\cdot} - \E( \pmk^{-1} \Lambda_{k,i\cdot}^\trans G_{k,t} \Lambdamk^\trans E_{k,t,j\cdot} ) \r) \r\vert^{2 + \epsilon} \r) ,\\
    \cIII_Z &= \E\l( \l\vert \sum_{t = s + 1}^e \l( \pmk^{-1} E_{k,t,i\cdot}^\trans \Lambdamk G_{k,t}^\trans \Lambda_{k,j\cdot} - \E( \pmk^{-1} E_{k,t,i\cdot}^\trans \Lambdamk G_{k,t}^\trans \Lambda_{k,j\cdot} ) \r) \r\vert^{2 + \epsilon} \r) ,\\
    \cIV_Z &= \E\l( \l\vert \sum_{t = s + 1}^e \l( \pmk^{-2} E_{k,t,i\cdot}^\trans \Lambdamk \Lambdamk^\trans E_{k,t,j\cdot} - \E( \pmk^{-2} E_{k,t,i\cdot}^\trans \Lambdamk \Lambdamk^\trans E_{k,t,j\cdot} ) \r) \r\vert^{2 + \epsilon} \r) .
\end{align*}

We only prove the result for $\cI_Z$, since all others can be shown in analogous arguments. Note that by Assumption~\ref{assum: core_factor}, for any $t\in[T]$,
\begin{equation}
\label{eqn: cI_Z1_component}
\begin{split}
    &\;\quad
    \Lambda_{k,i\cdot}^\trans G_{k,t} G_{k,t}^\trans \Lambda_{k,j\cdot} = \sum_{u, v \in [r_k]} \sum_{\ell \in [\rmk]} \Lambda_{k,iu} \Lambda_{k,jv} G_{k,t,u\ell} G_{k,t,v\ell} .
\end{split}
\end{equation}
Recall also by definition that
\begin{align*}
    G_{k,t} &= \sum_{j=1}^{q+1} A_{j,k} F_{k,t} \Ajmk{j}^\trans \cdot \mathbb{I}_{\{\theta_{j-1} <t\le \theta_j\}} ,
\end{align*}
which, combined with~\eqref{eqn: cI_Z1_component}, Lemma~\ref{lemma: weakly_M_dependent}
(with $\beta >1$ therein so as to satisfy the condition in Proposition~4 of \cite{Berkesetal2011}),
Assumptions~\ref{assum: core_factor}, \ref{assum: loadings}~(ii), and \ref{assum: trans_mat}~(i), we can immediately conclude $\cI_Z\leq c_0(e-s)^{1+\epsilon/2}$ for some constant $c_0$ by applying Proposition~4 in \cite{Berkesetal2011}.
For $\cII_Z$, $\cIII_Z$, and $\cIV_Z$, similar results follow by also using Lemma~\ref{lemma: weakly_M_dependent} and applying Proposition~4 in \cite{Berkesetal2011}, except that we require Assumption~\ref{assum: noise} in addition. This concludes the proof of this lemma.
\end{proof}

\begin{lemma}\label{lemma: max_ZZ_2+epsilon}
Let all assumptions in Lemma~\ref{lemma: ZZ_2+epsilon} hold, except that Assumption~\ref{assum: trans_mat} is replaced by Assumption~\ref{assum: trans_mat_alt}.
Then for the same constant $\epsilon$ in Lemma~\ref{lemma: ZZ_2+epsilon}, we have for any $k\in[K]$ and $\kappa_T\to\infty$ arbitrarily slowly,
\begin{equation}
\label{eqn: P_max_ZZ_2+epsilon}
\P\l\{ \max_{j\in[q+1]} \max_{\omega_j^{-2} \kappa_T^2 \le \ell \le \theta_j - \theta_{j - 1}} \frac{\sqrt{\omega_j^{-2}\kappa_T^2}}{p_k \ell} \l\Vert  \sum_{t = \theta_j - \ell + 1}^{\theta_j} \l( Z_{k, t} Z_{k, t}^\trans - \E( Z_{k, t} Z_{k, t}^\trans) \r) \r\Vert_F \ge \kappa_T \r\} = o(1) .
\end{equation}
\end{lemma}

\begin{proof}[Proof of Lemma~\ref{lemma: max_ZZ_2+epsilon}]

By Minkowski's inequality, we have for any $k\in[K]$, $w\in[q+1]$, and $i,j\in[p_k]$,
\begin{align*}
    &\;\quad
    \E\l[ \l( \max_{\omega_w^{-2} \kappa_T^2 \le \ell \le \theta_w - \theta_{w - 1}} \frac{\sqrt{\omega_w^{-2}\kappa_T^2}}{\ell} \l\vert  \sum_{t = \theta_w - \ell + 1}^{\theta_w} \l( Z_{k, t, i\cdot}^\trans Z_{k, t, j\cdot} - \E( Z_{k, t, i\cdot}^\trans Z_{k, t, j\cdot}) \r) \r\vert \r)^{2 + \epsilon} \r] \\
    &\leq
    \E\l[ \l( \max_{\omega_w^{-2} \kappa_T^2 \le \ell \le \theta_w - \theta_{w - 1}} \frac{\sqrt{\omega_w^{-2}\kappa_T^2}}{\ell} \l\vert  \sum_{t = \theta_w - \ell + 1}^{\theta_w - \omega_w^{-2}\kappa_T^2 + 1} \l( Z_{k, t, i\cdot}^\trans Z_{k, t, j\cdot} - \E( Z_{k, t, i\cdot}^\trans Z_{k, t, j\cdot}) \r) \r\vert  \r)^{2 + \epsilon} \r] \\
    &\;\quad
    + \E\l[ \l( \max_{\omega_w^{-2} \kappa_T^2 \le \ell \le \theta_w - \theta_{w - 1}} \frac{\sqrt{\omega_w^{-2}\kappa_T^2}}{\ell} \l\vert  \sum_{t = \theta_w - \omega_w^{-2}\kappa_T^2 + 2}^{\theta_w} \l( Z_{k, t, i\cdot}^\trans Z_{k, t, j\cdot} - \E( Z_{k, t, i\cdot}^\trans Z_{k, t, j\cdot}) \r) \r\vert  \r)^{2 + \epsilon} \r] \\
    &\lesssim
    \left( \omega_w^{-2}\kappa_T^2 \right)^{1+\epsilon/2} \sum_{\ell= \omega_w^{-2}\kappa_T^2}^{\theta_w - \theta_{w - 1}} \frac{1}{\ell^{2+\epsilon}} \cdot \ell^{\epsilon/2}
    + \left\{ \left( \omega_w^{-2}\kappa_T^2 \right)^{-1} \left( \omega_w^{-2}\kappa_T^2 -1 \right) \right\}^{1+\epsilon/2}
    = \cO(1) ,
\end{align*}
where the last line used Lemma~\ref{lemma: ZZ_2+epsilon} in applying Theorem~B.3 in \cite{Kirch2006}; the first term in the last line can be upper bounded by an integral inequality; and the second term is trivially bounded, with $\cO(1)$ that does not depend on $i$, $j$, $w$ or $k$. This shows that
\begin{equation}
\label{eqn: max_ZZ_2+epsilon_entry}
\E\l[ \l( \max_{\omega_w^{-2} \kappa_T^2 \le \ell \le \theta_w - \theta_{w - 1}} \frac{\sqrt{\omega_w^{-2}\kappa_T^2}}{\ell} \l\vert  \sum_{t = \theta_w - \ell + 1}^{\theta_w} \l( Z_{k, t, i\cdot}^\trans Z_{k, t, j\cdot} - \E( Z_{k, t, i\cdot}^\trans Z_{k, t, j\cdot}) \r) \r\vert \r)^{2 + \epsilon} \r] = \cO(1).
\end{equation}
Next, by Jensen's inequality, we have
\begin{align*}
    &\;\quad
    \E\l[ \l( \max_{\omega_w^{-2} \kappa_T^2 \le \ell \le \theta_w - \theta_{w - 1}} \frac{\sqrt{\omega_w^{-2}\kappa_T^2}}{p_k \ell} \l\Vert \sum_{t = \theta_w - \ell + 1}^{\theta_w} \l( Z_{k, t} Z_{k, t}^\trans - \E( Z_{k, t} Z_{k, t}^\trans) \r) \r\Vert_F  \r)^{2 + \epsilon} \r] \\
    &=\!
    \E\l[ \l( \max_{\omega_w^{-2} \kappa_T^2 \le \ell \le \theta_w - \theta_{w - 1}} \frac{1}{p_k^2} \sum_{i, j \in [p_k]} \l\vert \frac{\sqrt{\omega_w^{-2}\kappa_T^2}}{\ell} \sum_{t = \theta_w - \ell + 1}^{\theta_w} \l( Z_{k, t, i\cdot}^\trans Z_{k, t, j\cdot} - \E( Z_{k, t, i\cdot}^\trans Z_{k, t, j\cdot}) \r) \r\vert^2 \r)^{1 + \epsilon/2} \r] \\
    &\leq \!
    \frac{1}{p_k^2} \sum_{i, j \in [p_k]} \E\l[ \max_{\omega_w^{-2} \kappa_T^2 \le \ell \le \theta_w - \theta_{w - 1}} \l\vert \frac{\sqrt{\omega_w^{-2}\kappa_T^2}}{\ell} \sum_{t = \theta_w - \ell + 1}^{\theta_w} \l( Z_{k, t, i\cdot}^\trans Z_{k, t, j\cdot} - \E( Z_{k, t, i\cdot}^\trans Z_{k, t, j\cdot}) \r) \r\vert^{2 + \epsilon} \r]
    = \cO(1) ,
\end{align*}
where the last equality used~\eqref{eqn: max_ZZ_2+epsilon_entry}. Finally, it holds by the Markov's inequality that
\begin{align*}
    &\;\quad
    \P\l\{ \max_{j\in[q+1]} \max_{\omega_j^{-2} \kappa_T^2 \le \ell \le \theta_j - \theta_{j - 1}} \frac{\sqrt{\omega_j^{-2}\kappa_T^2}}{p_k \ell} \l\Vert  \sum_{t = \theta_j - \ell + 1}^{\theta_j} \l( Z_{k, t} Z_{k, t}^\top - \E( Z_{k, t} Z_{k, t}^\top) \r) \r\Vert_F \ge \kappa_T \r\} \\
    &\leq
    \frac{q}{(\kappa_T)^{2+\epsilon}} \max_{j\in[q+1]} \E\l\{ \l( \max_{\omega_j^{-2} \kappa_T^2 \le \ell \le \theta_j - \theta_{j - 1}} \frac{\sqrt{\omega_j^{-2}\kappa_T^2}}{p_k \ell} \l\Vert  \sum_{t = \theta_j - \ell + 1}^{\theta_j} \l( Z_{k, t} Z_{k, t}^\top - \E( Z_{k, t} Z_{k, t}^\top) \r) \r\Vert_F  \r)^{2 + \epsilon} \r\} \\
    &\lesssim
    \frac{q}{(\kappa_T)^{2+\epsilon}} = o(1),
\end{align*}
where the last equality holds as $q$ is fixed due to linearly spaced change points and $\kappa_T$ goes to infinity (arbitrarily slowly). This completes the proof of this lemma.
\end{proof}


\begin{lemma}\label{lemma: XX-EXX_zz}
Let Assumptions~\ref{assum: core_factor}, \ref{assum: loadings}, \ref{assum: trans_mat_alt} and~\ref{assum: noise} hold, with $\nu>8$ and $\beta\geq 1$. Then for $\kappa_T\to\infty$ arbitrarily slowly, with $\cJ_{a,b}^{(1)}$ defined in Lemma~\ref{lemma: imp_maximal_deviation_mat}, we have
\begin{align*}
\P\l(\wt{\cS}_{T,p}^{(1)}\r) \to 0 \text{ \ as \ } \min(T,p_1,\ldots,p_K)\to \infty,
\end{align*}
where
\begin{align*}
\wt{\cS}_{T,p}^{(1)}&=\l\{ \max_{1\le j\le q+1}\max_{\omega_j^{-2} \kappa_T^2 \le \ell \le \theta_j - \theta_{j - 1}}\sqrt{\omega_j^{-2}\kappa_T^2}\Big\vert \cJ_{\thj-\ell,\thj}^{(1)}\Big\vert_2  \ge \kappa_T  \r\}\\
& \quad\quad \bigcap \l\{ \max_{0\le j\le q}\max_{\omega_j^{-2} \kappa_T^2 \le \ell \le \theta_{j+1} - \theta_j}\sqrt{\omega_j^{-2}\kappa_T^2}\Big\vert  \cJ_{\thj,\thj+\ell}^{(1)} \Big\vert_2 \ge \kappa_T \r\}.
\end{align*}

Also, with $\cJ_{a,b}^{(2)}$ defined as in Lemma~\ref{lemma: X(Lam-Lam)X}, we have
\begin{align*}
\P\l(\wt{\cS}_{T,p}^{(2)}\r) \to 0 \text{ \ as \ } \min(T,p_1,\ldots,p_K)\to \infty,
\end{align*}
where
\begin{align*}
\wt{\cS}_{T,p}^{(2)}&=\l\{ \max_{1\le j\le q+1}\max_{\omega_j^{-2} \kappa_T^2 \le \ell \le \theta_j - \theta_{j - 1}} \Big\vert \cJ_{\thj-\ell,\thj}^{(2)} \Big\vert_2 \ge \wt{C}_2 \alpha_{T,p} \r\}\\
& \qquad\quad \bigcap \l\{ \max_{0\le j\le q}\max_{\omega_j^{-2} \kappa_T^2 \le \ell \le \theta_{j+1} - \theta_j}\Big\vert \cJ_{\thj,\thj+\ell}^{(2)} \Big\vert_2 \ge \wt{C}_2\alpha_{T,p}  \r\},
\end{align*}
with some finite constant $\wt C_2>0$.
\end{lemma}

\begin{proof}[Proof of Lemma~\ref{lemma: XX-EXX_zz}]
Following the same line of reasoning as in the proof of Lemma~\ref{lemma: ZZ_2+epsilon}, for the constant $\epsilon \in (0, \nu/2 - 2)$, we have for any $k\in[K]$ and $i,j\in[p_k]$,
\[
\E\l( \l\vert \sum_{t = s + 1}^e \l( \pmk^{-1} E_{k,t,i\cdot}^\trans  E_{k,t,j\cdot} - \E( \pmk^{-1} E_{k,t,i\cdot}^\trans E_{k,t,j\cdot} ) \r) \r\vert^{2 + \epsilon} \r) \leq c_0 (e-s)^{1+\epsilon/2}  ,
\]
with some constant $c_0$ depending on $\nu$.
Noting that
\begin{align*}
\pmk^{-1} X_{k, t} X_{k, t}^\trans =
\Lambda_k G_{k,t} G_{k,t}^\trans \Lambda_k^\trans
+  \pmk^{-1}\Lambda_k G_{k,t} \Lambdamk^\trans E_{k,t}^\trans
+ \pmk^{-1}E_{k,t} \Lambdamk G_{k,t}^\trans \Lambda_k^\trans
+  \pmk^{-1}E_{k,t} E_{k,t}^\trans ,
\end{align*}
together with $\cI_Z,\; \cII_Z,\; \cIII_Z \le c_0 (e-s)^{1+\epsilon/2} $ derived in the proof of Lemma~\ref{lemma: ZZ_2+epsilon}, we have
\begin{align*}
\E\l( \l\vert \frac{1}{\pmk} \sum_{t = s + 1}^e \l( X_{k, t, i \cdot}^\trans X_{k, t, j \cdot} - \E( X_{k, t, i \cdot}^\trans X_{k, t, j \cdot}) \r) \r\vert^{2 + \epsilon} \r) \le c_0 (e - s)^{1 + \epsilon/2} .
\end{align*}
Then, by the same argument as in the proof of Lemma~\ref{lemma: max_ZZ_2+epsilon}, for $\kappa_T\to\infty$ arbitrarily slowly, we have
\begin{align*}
\P\l\{ \max_{1\le j\le q+1} \max_{\omega_j^{-2} \kappa_T^2 \le \ell \le \theta_j - \theta_{j - 1}} \frac{\sqrt{\omega_j^{-2}\kappa_T^2}}{p\,  \ell} \l\Vert  \sum_{t = \theta_j - \ell + 1}^{\theta_j} \l( X_{k, t} X_{k, t}^\trans - \E( X_{k, t} X_{k, t}^\trans) \r) \r\Vert_F \ge \kappa_T \r\} = o(1),
\end{align*}
and 
\begin{align}
\max_{1\le j\le q+1} \max_{\omega_j^{-2} \kappa_T^2 \le \ell \le \theta_j - \theta_{j - 1}} \frac{\omega_j^{-1}}{p\, \ell} \l\Vert  \sum_{t = \theta_j - \ell + 1}^{\theta_j} \l( X_{k, t} X_{k, t}^\trans - \E( X_{k, t} X_{k, t}^\trans) \r) \r\Vert_F = \cO_P(1). \label{eqn: xx-Exx_zz}
\end{align}
Analogous arguments apply to the sums over $\{\thj + 1, \ldots, \thj+\ell\}$ and show
\begin{align*}
\P\l\{ \max_{0\le j\le q} \max_{\omega_j^{-2} \kappa_T^2 \le \ell \le \theta_{j + 1} - \theta_j} \frac{\sqrt{\omega_j^{-2}\kappa_T^2}}{p\,  \ell} \l\Vert  \sum_{t = \theta_j + 1}^{\theta_j + \ell} \l( X_{k, t} X_{k, t}^\trans - \E( X_{k, t} X_{k, t}^\trans) \r) \r\Vert_F \ge \kappa_T \r\} = o(1).
\end{align*}
Then, noting that $\cJ_{a, b}^{(1)}$ is obtained by stacking $\Vech(\cJ_{k, a, b}^{(1)})$, for $ k\in[K]$, we have
\begin{align*}
\P\l(\wt{\cS}_{T,p}^{(1)}\r) \to 0 \text{ \ as \ } \min(T,p_1,\ldots,p_K)\to \infty.
\end{align*}

We now prove the second claim. For any $j\in[q+1]$, by Lemma~\ref{lemma: E(xx)}~(i) and~\eqref{eqn: xx-Exx_zz},
\begin{align}
& \max_{\omega_j^{-2} \kappa_T^2 \le \ell \le \theta_j - \theta_{j - 1}}
\frac{1}{p\,  \ell} \Bigg\|\sum_{t=\thj -\ell+1}^{\thj} X_{k,t}X_{k,t}^\trans\Bigg\|_F
\nonumber \\
&\le \max_{\omega_j^{-2} \kappa_T^2 \le \ell \le \theta_j - \theta_{j - 1}} \frac{1}{p\,  \ell} \Bigg\|\sum_{t=\thj -\ell+1}^{\thj} \left( X_{k,t}X_{k,t}^\trans - \E\big(X_{k,t}X_{k,t}^\trans  \big) \right) \Bigg\|_F + 
\nonumber \\
& \qquad \max_{\omega_j^{-2} \kappa_T^2 \le \ell \le \theta_j - \theta_{j - 1}}  \frac{1}{p\,  \ell} \Bigg\|\sum_{t=\thj -\ell+1}^{\thj} \E\big(X_{k,t}X_{k,t}^\trans  \big) \Bigg\|_F
\nonumber \\
&\le  \max_{\omega_j^{-2} \kappa_T^2 \le \ell \le \theta_j - \theta_{j - 1}} \frac{1}{p\,  \ell} \Bigg\|\sum_{t=\thj -\ell+1}^{\thj} \left( X_{k,t}X_{k,t}^\trans - \E\big(X_{k,t}X_{k,t}^\trans  \big) \right) \Bigg\|_F + 
\nonumber \\
& \qquad \max_{\omega_j^{-2} \kappa_T^2 \le \ell \le \theta_j - \theta_{j - 1}}  \frac{1}{p\,  \ell} \sum_{t=\thj -\ell+1}^{\thj} \Big\|  \E\big( X_{k,t}X_{k,t}^\trans  \big) \Big\|_F
\nonumber \\
&= \cO_P\l( \omega_j \r) + \cO(1) = \cO_P(1),\nonumber
\end{align}
where the last equality follows from the fact that
\begin{align}
\omega_j \le 2\max_{j'\in\{j,j-1\}, \ j\in[q+1]}\l\vert
\begin{bmatrix}
\Vech\l(\Gamma_{G,(j')}^{(1)}\r)\\
\vdots\\
\Vech\l(\Gamma_{G,(j')}^{(K)}\r)
\end{bmatrix}
 \r\vert_2=\cO(1). \label{eqn: omega_j}
\end{align}

Combining the above with Lemmas~\ref{lemma: inequality_sandwich}~(i) and Lemma~\ref{lemma: loading_prod_rate}, we obtain
\begin{align*}      
& \max_{\omega_j^{-2} \kappa_T^2 \le \ell \le \theta_j - \theta_{j - 1}} \frac{1}{p \pmk\,  \ell}\Bigg\|\sum_{t=\thj-\ell+1}^{\thj}  X_{k,t}\left(\whLambdamk\whLambdamk^\trans -\Lambdamk\Lambdamk^\trans \right) X_{k,t}^\trans\Bigg\|_F\\
&\le \frac{1}{\pmk}\Big\| {\whLambdamk\whLambdamk^\trans} -\Lambdamk\Lambdamk^\trans\Big\|_F\cdot \max_{\omega_j^{-2} \kappa_T^2 \le \ell \le \theta_j - \theta_{j - 1}} \frac{1}{p \,  \ell}\Bigg\|\sum_{t=\thj-\ell+1}^{\thj} X_{k,t}X_{k,t}^\trans\Bigg\|_F=\cO_P(\alpha_{T,p,\text{-}k}), 
\end{align*}
where the LHS does not depend on $j\in[q+1]$, hence
\[
\max_{1\le j \le q+1}\max_{\omega_j^{-2} \kappa_T^2 \le \ell \le \theta_j - \theta_{j - 1}} \frac{1}{p \pmk\,  \ell}\Bigg\|\sum_{t=\thj-\ell+1}^{\thj}  X_{k,t}\left(\whLambdamk\whLambdamk^\trans -\Lambdamk\Lambdamk^\trans \right) X_{k,t}^\trans\Bigg\|_F=\cO_P(\alpha_{T,p,\text{-}k}).
\]
Moreover, we have
\begin{align*}
&\quad\;\l\vert\Vech\l(  \frac{1}{p \pmk\,  \ell}\sum_{t=\thj-\ell+1}^{\thj}   X_{k,t}\left(\whLambdamk\whLambdamk^\trans -\Lambdamk\Lambdamk^\trans \right)  X_{k,t}^\trans\r)\r\vert_2 \\
&\le \Bigg\|\frac{1}{p \pmk\,  \ell}\sum_{t=\thj-\ell+1}^{\thj}  X_{k,t}\left(\whLambdamk\whLambdamk^\trans -\Lambdamk\Lambdamk^\trans \right) X_{k,t}^\trans\Bigg\|_F,
\end{align*}
and thus
\begin{align*}
& \max_{1\le j \le q+1}\max_{\omega_j^{-2} \kappa_T^2 \le \ell \le \theta_j - \theta_{j - 1}}  \Big\vert\cJ_{\thj-\ell,\thj}^{(2)}\Big\vert_2
\\
\le & \, \max_{1\le j \le q+1}\max_{\omega_j^{-2} \kappa_T^2 \le \ell \le \theta_j - \theta_{j - 1}}  \!\!\l(\sum_{k=1}^K \l\|\frac{1}{p \pmk\,  \ell}\sum_{t=\thj-\ell+1}^{\thj}  X_{k,t}\left(\whLambdamk\whLambdamk^\trans -\Lambdamk\Lambdamk^\trans \right) X_{k,t}^\trans\r\|_F^2 \r)^{1/2} \!\! = \cO_P(\alpha_{T,p}).
\end{align*}
Analogous arguments apply to the sums over $\{\thj + 1, \ldots, \thj+\ell\}$ from which we conclude that $\P\big(\wt{\cS}_{T,p}^{(2)} \big) \to 0$.
\end{proof}

\begin{lemma}
\label{lemma: imp_ppl_T_true}
Under Model~\eqref{eqn: tfm_change_rewrite}, consider $(s, e)$ which satisfies $0\le s<e\le T$ and $\{\theta_j\}\subset\{s{+}1,\ldots,e{-}1\}\cap\Theta\subset\{\theta_j,\theta_{j+1}\}$ for some $j\in[q]$. Define
\begin{align*}
\bar{V}_t &= \begin{bmatrix}
\bar{V}^{(1)}_t \\ \vdots \\ \bar{V}^{(K)}_t   
\end{bmatrix} 
= \begin{bmatrix}
\Vech\l( \E(G_{1,t}G_{1,t}^\trans)\r) \\
\vdots \\
\Vech\l( \E(G_{K,t}G_{K,t}^\trans) \r)
\end{bmatrix}, \quad
\Omega_j = \begin{bmatrix}
\Vech\big(\Omega_j^{(1)}\big) \\ \vdots \\ \Vech\big(\Omega_j^{(K)}\big) 
\end{bmatrix}  \text{ \ for each \ } j\in[q ],
\end{align*}
and 
\begin{align*}
\bar{\cV}_{s,\tau,e} = \begin{bmatrix}
\bar{\cV}^{(1)}_{s,\tau,e} \\ \vdots \\ \bar{\cV}^{(K)}_{s,\tau,e}
\end{bmatrix} = \sqrt{\frac{(\tau - s)(e - \tau)}{e - s}}\left( \frac{1}{e - \tau} \sum_{t = \tau + 1}^{e}  \bar{V}_t  - \frac{1}{\tau - s} \sum_{t = s + 1}^\tau  \bar{V}_t\right),
\end{align*}
with which we denote 
\[
\cT^{*}_{s,\tau,e} = \begin{bmatrix}
\cT^{*(1)}_{s,\tau,e} \\ \vdots \\ \cT^{*(K)}_{s,\tau,e}
\end{bmatrix} = \l\vert \bar{\cV}_{s,\tau,e} \r\vert_2.
\]
Then we have $\argmax_{s<\tau<e}\cT^{*}_{s,\tau,e}\in\{\theta_j,\theta_{j+1}\} \cap \{s+1,\ldots,e-1 \}$, and
\begin{align*}
\max_{s<\tau<e}\cT^{*}_{s,\tau,e}
=\max\Bigg\{&
\sqrt{\frac{(\theta_j-s)(e-\theta_j)}{e-s}}\,
\Big|
\Omega_j
+\frac{e-\theta_{j+1}}{e-\theta_j}\,\Omega_{j+1}\,
\mathbb{I}_{\{\theta_{j+1} < e)\}}
\Big|_2,\\
&
\sqrt{\frac{(\theta_{j-1}-s)(e-\theta_{j-1})}{e-s}}\,
\Big|
\Omega_{j+1}
+\frac{\theta_j-s}{\theta_{j+1}-s}\,\Omega_j
\Big|_2\,
\mathbb{I}_{\{\theta_{j+1} <e)\}}
\Bigg\}.
\end{align*}
\end{lemma}

\begin{proof}[Proof of Lemma~\ref{lemma: imp_ppl_T_true}]
Under~\eqref{eqn: tfm_change_rewrite}, the sequence $\{\bar{V}_t\}_{t=1}^T$ is piecewise constant with all the change points belonging to $\Theta$.
Using the dual characterization of the Euclidean norm, we have 
\begin{align*}
\cT^{*}_{s,\tau,e} = \vert \bar{\cV}_{s,\tau,e} \vert_2 =  \sup_{\vert \bg\vert_2=1} \bg^\trans \bar{\cV}_{s,\tau,e}.
\end{align*}
Let $\tau^*\in \argmax_{s<\tau<e} \cT^{*}_{s,\tau,e}$ and choose $\bg^*$ with $\vert \bg^* \vert_2=1$ such that
\begin{align*}
\cT^{*}_{s,\tau^*,e} =  (\bg^*)^\trans \bar{\cV}_{s,\tau^*,e}.
\end{align*}
Then $\tau^*$ also maximizes $(\bg^*)^\trans \bar{\cV}_{s,\tau,e}$ over $s<\tau<e$; otherwise there exists $\tau^{**}$ with $(\bg^*)^\trans \bar{\cV}_{s,\tau^{**},e}>(\bg^*)^\trans \bar{\cV}_{s,\tau^{*},e} = \cT^{*}_{s,\tau^*,e}$, which implies $\cT^{*}_{s,\tau^{**},e} \ge (\bg^*)^\trans \bar{\cV}_{s,\tau^{**},e} > \cT^{*}_{s,\tau^{*},e} $, which is a contradiction.
Note that for fixed $\bg^*$, the scalar sequence $\{(\bg^*)^\trans \bar{V}_{t}\}_{t=1}^T$ is piecewise constant with all its change points within $\Theta$. 
Therefore, by Lemma~8 of \cite{wang2018high}, we have
\begin{align*}
\tau^*  \in\{\theta_j,\theta_{j+1}\} \cap \{s+1,\ldots,e-1 \}.
\end{align*}
The second statement follows from evaluating $\cT^{*}_{s,\tau,e}$ at $\tau=\theta_j$ and $\tau = \theta_{j+1}$ (in case $\theta_{j+1} < e)$.
\end{proof}

\begin{lemma}\label{lemma: imp_Vhat-V-k}
Let all assumptions in Lemma~\ref{lemma: xx-Exx} hold.
Recalling that $\Theta$ denotes the set of true change points, 
with the definition of $V_t$ from \eqref{eqn: Vt_def},
we have
\begin{align*}
&\quad\; \big \vert \wh{\cV}_{\sll, \tau, \el} - {\cV}_{\sll, \tau, \el}\big \vert_2\\
&=\sqrt{\frac{(\tau - s_\ell)(\el - \tau)}{\el - \sll}} \left\vert \frac{1}{e_\ell - \tau} \sum_{t = \tau + 1}^{e_\ell} \l( \wh{V}_t - V_t \r) - \frac{1}{\tau - s_\ell} \sum_{t = s_\ell + 1}^\tau \l( \wh{V}_t - V_t \r) \right\vert_2
\\
&=\l(\sqrt{\frac{(\tau - s_\ell)(\el - \tau)}{\el - \sll}} + \cT^{*}_{\sll,\tau,\el}  \r) \cdot \cO_P(\alpha_{T,p}) +  \cO_P\big(\sqrt{\log(T)}\big),
\end{align*}
where $\alpha_{T,p}$ is defined in~\eqref{eq:alpha:rate} and $\cT^*_{\sll,\tau,\el}$ in Lemma~\ref{lemma: imp_ppl_T_true}; throughout, the $\cO_P$-bounds hold uniformly over $(\sll,\tau,\el)\in\cI$, with $\cI$ as defined in Lemma~\ref{lemma: xx-Exx}.
Further, suppose that Assumption~\ref{assum: dim} holds. Then the result simplifies to
\begin{align*}
\big \vert \wh{\cV}_{\sll, \tau, \el} - {\cV}_{\sll, \tau, \el}\big \vert_2 = \cO_P\big(\sqrt{\log(T)}\big).
\end{align*}
\end{lemma}

\begin{proof}[Proof of Lemma~\ref{lemma: imp_Vhat-V-k}]
Recall from Lemma~\ref{lemma: GG-EGG_generic_decomp} that
\begin{align*}
&\quad\;\big\vert \wh{\cV}_{\sll, \tau, \el}^{(k)} - {\cV}_{\sll, \tau, \el}^{(k)}\big\vert_2 \\
&\le \sqrt{\frac{(\tau - \sll)(\el - \tau)}{\el-\sll}}\l( \Big\Vert\cJ_{k,\tau,\el}^{(1)}\Big\Vert_F + \Big\Vert \cJ_{k,\sll,\tau}^{(1)} \Big\Vert_F+ \Big\Vert\cJ_{k,\tau,\el}^{(2)}\Big\Vert_F +  \Big\Vert\cJ_{k,\sll,\tau}^{(2)}\Big\Vert_F\r) \\
&\quad \, +\,  2\sqrt{2}\Vert\wh{H}_k\Vert \l\Vert \frac{1}{\sqrt{p_k}}\l(\wh{\Lambda}_k-\Lambda_k\wh{H}_k \r) \r\Vert_F \cdot  \sqrt{\frac{(\tau - s)(e - \tau)}{e-s}}\cdot \l\vert  \Vech\l(\Gamma_{G,\tau,e}^{(k)} - \Gamma_{G,s, \tau}^{(k)} \r)\r\vert_2,
\\
&=: U_{k,1} + U_{k,2}.
\end{align*}
To control for $U_{k,1}$, first note that by \eqref{eqn: imp_maximal_deviation_mat}, 
\begin{align*}
&\quad\;\Big\Vert\cJ_{k,\tau,\el}^{(1)}\Big\Vert_F +  \Big\Vert\cJ_{k,\sll,\tau}^{(1)} \Big\Vert_F \\
&= \frac{1}{\sqrt{\el-\tau}}\Bigg\|\frac{1}{p\pmk}\frac{1}{\sqrt{\el-\tau}}\sum_{t=\tau+1}^{\el}\left[X_{k,t}\Lambdamk\Lambdamk^\trans X_{k,t}^\trans-\E\left(X_{k,t}\Lambdamk\Lambdamk^\trans X_{k,t}^\trans \right) \right] \Bigg\|_F\\
&\quad + \frac{1}{\sqrt{\tau-\sll}}\Bigg\|\frac{1}{p\pmk}\frac{1}{\sqrt{\tau-\sll}}\sum_{t=\sll+1}^\tau\left[X_{k,t}\Lambdamk\Lambdamk^\trans X_{k,t}^\trans-\E\left(X_{k,t}\Lambdamk\Lambdamk^\trans X_{k,t}^\trans \right) \right] \Bigg\|_F\\
&= \left(\frac{1}{\sqrt{\el-\tau}} +\frac{1}{\sqrt{\tau-\sll}}\right) \cdot\cO_P\big(\sqrt{\log(T)}\big),
\end{align*}
while by~\eqref{eqn: X(Lam-Lam)X}, 
\begin{align*}
\Big\Vert\cJ_{k,\tau,\el}^{(2)}\Big\Vert_F +  \Big\Vert\cJ_{k,\sll,\tau}^{(2)} \Big\Vert_F
&= \Bigg\|\frac{1}{p\pmk}\frac{1}{\el-\tau}\sum_{t=\tau+1}^{\el} X_{k,t}\left(\whLambdamk\whLambdamk^\trans-\Lambdamk\Lambdamk^\trans \right)X_{k,t}^\trans \Bigg\|_F\\
&\quad + \, \Bigg\|\frac{1}{p\pmk}\frac{1}{\tau-\sll}\sum_{t=\sll+1}^\tau X_{k,t}\left(\whLambdamk\whLambdamk^\trans-\Lambdamk\Lambdamk^\trans \right)X_{k,t}^\trans \Bigg\|_F\\
&=\cO_P(\alpha_{T,p,\text{-}k}),
\end{align*}
all uniformly over $(\sll,\tau,\el)\in\cI$.
This gives 
\begin{align*}
U_{k,1}= \cO_P\l(\sqrt{\frac{(\tau - s_\ell)(\el - \tau)}{\el - \sll}} \cdot \alpha_{T,p} +  \sqrt{\log(T)}\r). 
\end{align*} 
Also, recalling $\cT^{*(k)}_{\sll,\tau,\el}$ defined in Lemma~\ref{lemma: imp_ppl_T_true}, we can write
\begin{align*}
U_{k,2} &\le 2\sqrt{2} \Vert\wh{H}_k\Vert \l\Vert \frac{1}{\sqrt{p_k}}\l(\wh{\Lambda}_k-\Lambda_k\wh{H}_k \r) \r\Vert_F  \cdot  \sqrt{\frac{(\tau - \sll)(\el - \tau)}{\el-\sll}}\l\vert  \Vech\l(\Gamma_{G,\tau,e}^{(k)} - \Gamma_{G,s, \tau}^{(k)} \r)\r\vert_2
\\
&= \cT^{*(k)}_{\sll,\tau,\el} \cdot \cO_P(\alpha_{T,p,k}) 
\end{align*}
uniformly for $(\sll,\tau,\el)\in \cI$, by Lemma~\ref{lemma: loading_prod_rate}~(ii).
Combining the bounds on $U_{k,1}$ and $U_{k,2}$, and stacking $\wh{\cV}_{\sll, \tau, \el}^{(k)}$, for $ k \in [K]$,
\begin{align*}
&\quad\;\big \vert \wh{\cV}_{\sll, \tau, \el} - {\cV}_{\sll, \tau, \el}\big \vert_2 \\
&= \sqrt{\frac{(\tau - s_\ell)(\el - \tau)}{\el - \sll}} \cdot\left\{\sum_{k=1}^K \left\vert \frac{1}{e_\ell - \tau} \sum_{t = \tau + 1}^{e_\ell} \l( \wh{V}^{(k)}_t - V^{(k)}_t \r) - \frac{1}{\tau - s_\ell} \sum_{t = s_\ell + 1}^\tau \l( \wh{V}^{(k)}_t - V^{(k)}_t \r)\right\vert_2^2\right\}^{1/2}\\
&= \l(\sqrt{\frac{(\tau - s_\ell)(\el - \tau)}{\el - \sll}} + \cT^{*}_{\sll,\tau,\el}  \r) \cdot \cO_P(\alpha_{T,p}) +  \cO_P\big(\sqrt{\log(T)}\big) .
\end{align*}
Further, under Assumption~\ref{assum: dim}, it holds that $\sqrt{T}\alpha_{T,p} = \cO(\log(T))$, from which the simplified statement follows.
\end{proof}

\begin{lemma}\label{lemma: imp_h_pop_h_agg}
Let the assumptions of Lemma~\ref{lemma: imp_Vhat-V-k} hold. Uniformly over $(\sll,\tau,\el)\in\cI$, with $\cI$ defined in Lemma~\ref{lemma: xx-Exx} and $V_t$ defined in \eqref{eqn: Vt_def}, we have
\begin{align*}
\vert \cV_{\sll,\tau,\el} \vert_2=
\sqrt{\frac{(\tau-\sll)(\el-\tau)}{\el-\sll}}\left\vert \frac{1}{\el - \tau} \sum_{t=\tau + 1}^{\el} V_t  -\frac{1}{\tau - \sll} \sum_{t=\sll + 1}^{\tau} V_t \right\vert_2  = \cT^{*}_{\sll,\tau,\el}  \cdot (1+\cO_P(\alpha_{T,p})),
\end{align*}
where the $\cO_P$-bound depends on a quantity that does not vary with $(\sll,\tau,\el)$.

\end{lemma}

\begin{proof}[Proof of Lemma~\ref{lemma: imp_h_pop_h_agg}]
Recalling $\cV^{(k)}_{\sll,\tau,\el}$ defined in~\eqref{eq:vk:cusum}, we have 
\begin{align}
\vert \cV^{(k)}_{\sll,\tau,\el} \vert_2 &= \sqrt{\frac{(\tau-\sll)(\el-\tau)}{\el-\sll}}\bigg|\Vech\left\{\wh H_k^\trans
\Bigl(\Gamma^{(k)}_{G,\tau,\el}-\Gamma^{(k)}_{G,\sll,\tau}\Bigr)
\wh H_k\right\}\bigg|_2 \nonumber
\\
&\le\sqrt{\frac{(\tau-\sll)(\el-\tau)}{\el-\sll}}\Big\|\widehat H_k^\trans
\bigl(\Gamma^{(k)}_{G,\tau,\el}-\Gamma^{(k)}_{G,\sll,\tau}\bigr)
\widehat H_k\Big\|_F \nonumber\\
&\le \sqrt{\frac{(\tau-\sll)(\el-\tau)}{\el-\sll}}\big\|\widehat H_k\big\|^2 \Big\|
\Gamma^{(k)}_{G,\tau,\el}-\Gamma^{(k)}_{G,\sll,\tau} \Big\|_F \nonumber\\
&\le  \sqrt{\frac{(\tau-\sll)(\el-\tau)}{\el-\sll}}\sqrt{2}\Big| 
\Vech\big(\Gamma^{(k)}_{G,\tau,\el}-\Gamma^{(k)}_{G,\sll,\tau}\big) \Big|_2 \cdot \big\|\widehat H_k\big\|^2 \label{eqn: imp_h_pop_h_agg}\\
&= \cT^{*(k)}_{\sll,\tau,\el}  \cdot (1+\cO_P(\alpha_{T,p,k})) , \nonumber
\end{align}
where $\cO_P(\alpha_{T,p,k})$ is due to the bounding of $\Vert\wh{H}_k\Vert$ by Lemma~\ref{lemma: loading_prod_rate}~(ii), and thus does not depend on $(s_\ell, \tau, e_\ell)$.
This carries over with $\cV^{(k)}_{\sll,\tau,\el}$, for $ k \in [K]$, stacked over, so that uniformly over $(\sll, \tau, \el)\in\cI$, 
\begin{align*}
\vert \cV_{\sll,\tau,\el} \vert_2 = \cT^{*}_{\sll,\tau,\el}  \cdot (1+\cO_P(\alpha_{T,p})). 
\end{align*}
\end{proof}

\subsubsection{Proof of Theorem~\ref{thm: asymp_consistency_detection}}

By assumption, $W_{a, \tau, b} = W$ for all $(a, b] \in \M$ and $\tau \in (a, b)$ in~\eqref{eqn: test_stat_agg_std}.
WLOG, we may regard $W = I$ which, from the following bounds due to the condition imposed on $W$, 
\begin{align*}
\cT_{a, \tau, b} &\le C_W^{1/2} \vert \cV_{a, \tau, b} \vert_2 + C_W^{1/2} \vert \wh{\cV}_{a, \tau, b} - \cV_{a, \tau, b} \vert_2,
\\
\cT_{a, \tau, b} &\ge C_W^{-1/2} \vert \cV_{a, \tau, b} \vert_2 - C_W^{1/2} \vert \wh{\cV}_{a, \tau, b} - \cV_{a, \tau, b} \vert_2,
\end{align*}
simplifies the proof arguments.
Then, the detector statistic becomes $\cT_{s,\tau,e} = \big\vert \wh{\cV}_{s,\tau, e} \big\vert_2$.
Throughout, we regard that $\upsilon_T = \kappa_T^2$ for the sequence $\kappa_T \to \infty$ from Lemma~\ref{lemma: max_ZZ_2+epsilon}, and that $\kappa_T$ satisfies
\begin{align}
\kappa_T^2 = o\left\{ \min\left( \min_{j \in [q]} \omega_j^2 \Delta_j /\log(T), \, \log(T) \right) \right\},
\label{eq:kappa:req}
\end{align}
which is permitted as it is only required to diverge arbitrarily slowly.

For all $j\in[q]$, recall that $\Delta_j =\min(\theta_j - \theta_{j-1}, \theta_{j+1} - \theta_j)$. Supposing that $\Delta_j >12$, we define the following sets
\begin{align*}
& \underline{\cI}_j =\left\{\theta_j - \floor[\Big]{\frac{\Delta_j}{3}}, \theta_j - \floor[\Big]{\frac{\Delta_j}{3}} + 1, \dots, \theta_j - \ceil[\Big]{\frac{\Delta_j}{12}} \right\}, \\
& \overline{\cI}_j =\left\{\theta_j + \ceil[\Big]{\frac{\Delta_j}{12}}, \theta_j + \ceil[\Big]{\frac{\Delta_j}{12}} + 1, \dots, \theta_j + \floor[\Big]{\frac{\Delta_j}{3}} \right\}.
\end{align*}
Then under Assumption~\ref{assum: trans_mat_alt}~(iv), we always have the following event hold:
\begin{equation}
\label{eqn: M_{T,M}}
\cM_{T,M} = \Big\{ \text{For each $j\in[q]$, there is some $(a,b]\in\M$ such that $(a,b)\in \underline{\cI}_j \times \overline{\cI}_j$} \Big\} ,
\end{equation}
due to the reasoning sketched as follows. For simplicity, we treat floor and ceiling functions as identity transforms. By construction of $\M$ in~\eqref{def: seeded_interval}, there exists $\ell \in [\mu_T]$ 
such that $2m_\ell \in [\Delta_j/6, 2\Delta_j/3]$ with $m_\ell=T/ 2^\ell$, or equivalently, $\Delta_j \in [3m_\ell, 12m_\ell]$, noticing that $m_\ell =2 m_{\ell+1}$.
At such level $\ell$, we identify $(a,b]\in\M$ with $a\in \underline{\cI}_j$ and $b\in \overline{\cI}_j$ while $b = a + 2m_\ell$, so that
\begin{align*}
    \theta_j - \frac{\Delta_j}{3} \leq a &\leq \theta_j - \frac{\Delta_j}{12} \text{\ and\ }
    \theta_j + \frac{\Delta_j}{12} \leq b \leq \theta_j + \frac{\Delta_j}{3}.
\end{align*}
Re-arranging the above, we identify $a$ satisfying
\[
a\in \left[ \theta_j - \frac{\Delta_j}{3}, \theta_j - \frac{\Delta_j}{12} \right] \cap \left[ \theta_j + \frac{\Delta_j}{12} - 2m_\ell, \theta_j + \frac{\Delta_j}{3} - 2m_\ell \right] ,
\]
where the intersection is non-empty if
\[
\theta_j - \frac{\Delta_j}{3} \leq \theta_j + \frac{\Delta_j}{3} - 2m_\ell \text{\ and\ }
\theta_j + \frac{\Delta_j}{12} - 2m_\ell \leq \theta_j - \frac{\Delta_j}{12} ,
\]
which is fulfilled by $2m_\ell \in [\Delta_j/6, 2\Delta_j/3]$. It remains to match this $a$ to the left endpoint of intervals in $\M$. To this end, roughly speaking (up to integer rounding), we can set $i=a/m_\ell +1$ so that $( \floor[]{(i-1) m_\ell}, \ceil[]{(i+1) m_\ell} ] = (a,a+2m_\ell]$ which is in $\M$ since $i\in [T/m_\ell -1]$ due to $a+2m_\ell \in[T]$. This completes the sketch of proof that~\eqref{eqn: M_{T,M}} holds.

Throughout the proof, we regard all $\cO_P$-bounds as holding deterministically on the event $\big(\cH_{T,p} \cup \cS^{(1)}_{T,p} \cup \cS^{(2)}_{T,p} \cup \wt{\cS}^{(1)}_{T,p} \cup \wt{\cS}^{(2)}_{T,p}\big)^c$ defined in Lemmas~\ref{lemma: loading_prod_rate}~(ii), \ref{lemma: imp_maximal_deviation_mat}, \ref{lemma: X(Lam-Lam)X} and~\ref{lemma: XX-EXX_zz}. 
Let us consider some $(s,e)\subset[0,T]$ which satisfy the following conditions:
\begin{enumerate}[label=(S\arabic*)]
    \item\label{S1}
    The set 
    $\mathbb{C}_{s,e} \ne \emptyset$, where
    \begin{align*}
        \mathbb{C}_{s,e}=\left\{1 \le j \le q : \,  \exists\,(a,b]\in\mathbb{M} \text{ \ such that \ } s\le a<b\le e \text{ and } (a,b)\in\underline{I}_j\times\overline{I}_j \right\}.
    \end{align*}
    \item\label{S2}
    There exists some $j\in\{0,\ldots,q\}$ and $j'\in\{1,\ldots,q+1\}$, such that
    \begin{align*}
        \max\left\{ \omega_j^2|s-\theta_j|, \,  \omega_{j'}^2|e-\theta_{j'}| \right\} = \cO_P(\kappa_T^2).
    \end{align*}
\end{enumerate}
Then we show that for $(s,e)$ meeting \ref{S1} and \ref{S2}:
\begin{enumerate}[label=(R\arabic*)]
    \item \label{R1} There exists at least one $\ell\in \L_{s,e}$ for which $\cT_\ell > \pi_{T, p}$, where $$ \L_{s,e} = \left\{\ell\in {[\vert \M \vert]}:\ (a_\ell,b_\ell] \subseteq (s,e] \right\}. $$ 
    \item The estimation of $\widehat{\theta}=\tau_{\ell^{\circ}}$ with $\ell^{\circ}$ defined in Algorithm~\ref{alg: mad}, satisfies $\omega_j^2|\widehat{\theta}-\theta_j| = \cO_P(\kappa_T^2)$ for some $j \in \mathbb{C}_{s,e}$. \label{R2}
\end{enumerate}

At the beginning of Algorithm~\ref{alg: mad} initialized with $\widehat{\Theta} = \emptyset$, we have $(s,e)=(0,T)$ meet \ref{S1}--\ref{S2} such that by \ref{R1}--\ref{R2}, we add $\thh$ to $\widehat{\Theta}$ which, for some $j\in\mathbb{C}_{0,T} = \{1, \ldots, q\}$, consistently estimates the location~$\theta_j$. Then, we no longer have the index~$j$ in $\C_{s,e}$ for the subsequently considered $(s,e)$, since either $\thj\notin\{s+1,\ldots,e-1\}$ or, even so, it has been detected by either $s$ or $e$ such that $\min(\thj-s,e-\thj) = \cO_P(\omega_j^{-2} \kappa_T^2) = o_P(\Delta_j)$ under Assumption~\ref{assum: signal-to-noise} and the condition on $\kappa_T$ in~\eqref{eq:kappa:req}. Thus, once a change point is detected, its proximity to the boundary in subsequent interval splits prevents it from being detected twice, avoiding any duplicate estimator. 

Once $|\widehat{\Theta}|=q$, for any $(s,e)$ defined by two neighboring points in $\{0,T\}\,\cup\, \widehat{\Theta}$, we have all $\ell\in\L_{s,e}$ satisfy $|\Theta\cap\{a_{\ell}+1,\ldots,b_{\ell}-1\}|\le 2$. For any $\ell\in\mathbb{L}_{s,e}$ satisfying $\theta_j\in\{a_{\ell}+1,\ldots,b_{\ell}-1 \}$, we have either $\min_{\tau_{\ell}\in\{a_{\ell},b_{\ell}\}}\omega_j^2|\tau_{\ell}-\theta_j| = \cO_P(\kappa_T^2)$ and $\theta_{j + 1} \notin (a_\ell, b_\ell)$, or $\max\{ \omega_j^2(\theta_j - a_\ell), \omega_{j +  1}^2(b_\ell - \theta_{j + 1}) \} = \cO_P(\kappa_T^2)$ such that by Lemmas~\ref{lemma: imp_ppl_T_true}--\ref{lemma: imp_h_pop_h_agg}, we have
\begin{align}
\cT_{\ell} &= \cT_{a_{\ell},\tau,b_{\ell}}
=\big\vert \wh{\cV}_{\al, \tau, \bl} \big\vert_2 \nonumber\\
&\le \, \big\vert{\cV}_{\al,\tau, \bl} \big\vert_2 + \big\vert \wh{\cV}_{\al,\tau, \bl} - {\cV}_{\al,\tau, \bl}\big\vert_2 \nonumber\\
&\le \, \cT^*_{\al,\tau,\bl} \cdot(1+\cO_P\left(\alpha_{T,p} \right)) + \cO_P\big(\sqrt{\log(T)}\big)   \tag{Lemmas~\ref{lemma: imp_Vhat-V-k} and~\ref{lemma: imp_h_pop_h_agg}}\nonumber\\
&\le \,
\left\{ \mathbb{I}_{\{ \theta_{j + 1} \notin (a_\ell, b_\ell] \}}\cdot \cT^*_{\al,\thj,\bl} + 2 \, {\mathbb{I}_{\{ \theta_{j + 1} \in (a_\ell, b_\ell] \}}\cdot \max_{j' \in \{j, j + 1\}}} \cT^*_{\al,\theta_{j'},\blc} \right\} (1+\cO_P(\alpha_{T,p} )) \tag{Lemma~\ref{lemma: imp_ppl_T_true}} \nonumber \\
&\quad+  \cO_P\big(\sqrt{\log(T)}\big) \nonumber\\
&= \, \cO_P\big(\kappa_T ( 1+ \alpha_{T,p} ) \big)  +  \cO_P\big(\sqrt{\log(T)}\big) = o_P(\pi_{T, p}), \nonumber
\end{align}
where the last equality follows from~\eqref{eq:kappa:req} and the conditions on $\pi_{T, p}$ in Theorem~\ref{thm: asymp_consistency_detection}.
This indicates that we do not have any $T_{\ell}, \, \ell\in\mathbb{L}_{s,e}$, exceed $\pi_{T, p}$, and thus the algorithm terminates.

\paragraph{Proof of~\ref{R1}.} 
The set $\mathbb{C}_{s,e}$ is not empty by~\ref{S1}, then for every $j\in\mathbb{C}_{s,e}$, there exists $\ell=\ell(j)\in\L_{s,e}$, such that
\begin{align}
\Delta_j/12\le \min(\theta_j-a_{\ell}, b_{\ell}-\theta_j)\le \max(\theta_j-a_{\ell}, b_{\ell}-\theta_j)\le \Delta_j/3
\label{eq:balanced}
\end{align}
{on $\cM_{T,M}$ defined in~\eqref{eqn: M_{T,M}}.} 
By the construction of {$(a_\ell, b_\ell)$}, the change point $\thj$ lies in the interior of the interval in the sense that $\min(\thj - a_\ell, b_\ell - \thj) \ge \Delta_j/12 > \varpi_T$ for large enough $T$, under Assumption~\ref{assum: trans_mat_alt}~(iv).
Then from the definition of $\cT_\ell$, 
\begin{align}
\cT_\ell&\ge \cT_{a_{\ell},\theta_j,b_{\ell}}\nonumber\\
&\ge \vert \cV_{\al,\thj,\bl} \vert_2  -  \vert \wh{\cV}_{\al,\thj,\bl}- {\cV}_{\al,\thj,\bl} \vert_2  \nonumber\\
&= \cT^*_{\al,\thj,\bl}\cdot(1+\cO_P( \alpha_{T,p})) +  \cO_P(\sqrt{\log(T)}) \tag{Lemmas~\ref{lemma: imp_Vhat-V-k} and~\ref{lemma: imp_h_pop_h_agg} }\nonumber\\
&\ge \frac{1}{2\sqrt{6}}\sqrt{{\Delta_j}}\omega_j (1+\cO_P(\alpha_{T,p})) + \cO_P(\sqrt{\log(T)} ) \nonumber\\
&= {\frac{\sqrt{{\Delta_j}}\omega_j}{2\sqrt{6}}}\cdot\l\{ 1  +  \cO_P\l(\alpha_{T,p} + \frac{\sqrt{\log(T)}}{\sqrt{{\Delta_j}}\omega_j} \r)\r\} \nonumber\\
&= {\frac{\sqrt{{\Delta_j}}\omega_j}{2\sqrt{6}}}\cdot (1 + o_P(1)).\tag{Assumption~\ref{assum: signal-to-noise}}  \nonumber
\end{align}
Then, from the condition on $\pi_{T, p}$ in Theorem~\ref{thm: asymp_consistency_detection}, we conclude that $\cT_{\ell} > \pi_{T, p}$.

\paragraph{Proof of~\ref{R2}.}
It remains to show that $\vert \wh{\theta}-\theta_j\vert = \cO_P(\omega_j^{-2}\kappa_T^2)$ for some $j \in \mathbb{C}_{s,e}$.
From the arguments in the proof of \ref{R1}, \eqref{eq:balanced} in particular, we have $b_{\ell^{\circ}}-\alc\le\min_{j \in \C_{s,e}} 2\Delta_j/3$, which also indicates that $\vert (\alc, \blc) \cap \Theta\vert=1$.
Let us write $(\alc, \blc] \cap \Theta = \{\theta_j\}$, and $\delc = \min(\thj - \alc, \blc - \thj)$.
WLOG, we suppose that $\wh{\theta} \le \theta_j$; the case $\wh{\theta} > \theta_j$ is handled analogously.

First, notice that from Lemmas~\ref{lemma: imp_ppl_T_true}--\ref{lemma: imp_h_pop_h_agg},
\begin{align*}
\pi_{T, p} &< \cT_{\ell^\circ} \le \vert \cV_{\alc, \wh\theta, \blc} \vert_2 + \vert \wh{\cV}_{\alc, \wh\theta, \blc} - \cV_{\alc, \wh\theta, \blc} \vert_2
\le \vert \cV_{\alc, \theta_j, \blc} \vert_2 + \cO_P(\sqrt{\log(T)}).
\end{align*}
This, from the condition on $\pi_{T, p}$, implies that 
\begin{align}
\label{eq:delc:lb}
\pi_{T, p}(1 + o_P(1)) < \sqrt{\frac{(\theta_j - \alc)(\blc - \theta_j)}{\blc - \alc}} \omega_j (1 + \cO_P(\alpha_{T, p})) \le \sqrt{\delc} \omega_j (1 + \cO_P(\alpha_{T, p})),
\end{align}
where $o_P$- and $\cO_P$-bounds do not depend on $\alc, \blc$ or $j$.
\medskip

\underline{Step~1.} We first establish that $\thj$ remains after trimming of the boundaries by $\varpi_T$ in the maximization defining $\cT_{\lc}$.
If this is not the case, we can find another interval, say $(a_\ell, b_\ell] \in \M$, such that $\blc - \alc = b_\ell - a_\ell$ and $(a_\ell, b_\ell] \cap \Theta = \{\theta_j\}$.
Further, noting that $\min(\theta_j - \alc, \blc - \theta_j) \le \varpi_T$, we have $\min(\theta_j - a_\ell, b_\ell - \theta_j) \ge (b_\ell - a_\ell)/2 - \varpi_T > \varpi_T$, from the construction of $\M$ and the condition on $\varpi_T$.
From this, we derive that
\begin{align}
& \vert \cV_{a_\ell, \theta_j, b_\ell} \vert_2 - \vert \cV_{\alc, \theta_j, \blc} \vert_2 = \l( \sqrt{\frac{(\theta_j - a_\ell)(b_\ell - \theta_j)}{b_\ell - a_\ell}} - \sqrt{\frac{(\theta_j - \alc)(\blc - \theta_j)}{\blc - \alc}} \r) \omega_j (1 + \cO_P(\alpha_{T, p}))
\nonumber \\
\ge & \, \l( \sqrt{\frac{((\blc - \alc)/2 - \varpi_T)((\blc - \alc)/2 + \varpi_T)}{\blc - \alc}} - \sqrt{\frac{\varpi_T(\blc - \alc - \varpi_T)}{\blc - \alc}} \r) \omega_j (1 + \cO_P(\alpha_{T, p}))
\nonumber \\
\ge & \, \l( \sqrt{\frac{(\blc - \alc)/2 - \varpi_T}{2}} - \sqrt{\varpi_T} \r) \omega_j (1 + \cO_P(\alpha_{T, p}))
\gtrsim_P \sqrt{\delc} \omega_j,
\label{eq:trim:diff}
\end{align}
thanks to Lemma~\ref{lemma: imp_h_pop_h_agg} for small enough $c_\varpi$, 
where the $\cO_P$-bound depends on a quantity that does not vary with $a_\ell, b_\ell, \alc, \blc$ or $j$.
Then, recalling that $\tau_\ell = \argmax_{a_\ell + \varpi_T < \tau < b_\ell - \varpi_T} \vert \wh{\cV}_{a_\ell, \tau, b_\ell} \vert_2$, 
it follows that 
\begin{align*}
& \vert \wh{\cV}_{a_\ell, \tau_\ell, b_\ell} \vert_2 - \vert \wh{\cV}_{\alc, \wh\theta, \blc} \vert_2 \ge \vert \wh{\cV}_{a_\ell, \theta_j, b_\ell} \vert_2 - \vert \wh{\cV}_{\alc, \wh\theta, \blc} \vert_2
\\
\ge &\, \vert \cV_{a_\ell, \theta_j, b_\ell} \vert_2 - \vert \wh{\cV}_{a_\ell, \theta_j, b_\ell} - \cV_{a_\ell, \theta_j, b_\ell} \vert_2
- \l( \vert \cV_{\alc, \wh\theta, \blc} \vert_2 + \vert \wh{\cV}_{\alc, \wh\theta, \blc} - \cV_{\alc, \wh\theta, \blc} \vert_2 \r)
\\
\ge &\, \vert \cV_{a_\ell, \theta_j, b_\ell} \vert_2 - \vert \wh{\cV}_{a_\ell, \theta_j, b_\ell} - \cV_{a_\ell, \theta_j, b_\ell} \vert_2
- \l( \vert \cV_{\alc, \theta_j, \blc} \vert_2 + \vert \wh{\cV}_{\alc, \wh\theta, \blc} - \cV_{\alc, \wh\theta, \blc} \vert_2 \r)
\\
\gtrsim_{P} & \, \sqrt{\delc} \omega_j + \cO_P(\sqrt{\log(T)}) ,
\end{align*}
from Lemma~\ref{lemma: imp_Vhat-V-k} and~\eqref{eq:trim:diff}, indicating that $\cT_{\ell^\circ} < \cT_\ell$ due to~\eqref{eq:delc:lb}.
This contradicts the definition of $\ell^\circ$ in Algorithm~\ref{alg: mad}, i.e.\ we have $\min(\theta_j - \alc, \blc - \theta_j) > \varpi_T$.
Also, together with~\eqref{eq:kappa:req}, we have neither $\alc$ or $\blc$ estimate $\theta_j$ in the sense that the following does not hold: $\min_{\tau \in \{\alc, \blc\}} \omega_j^2 \vert \tau - \theta_j \vert = \cO_P(\kappa_T^2)$.
\medskip

\underline{Step~2.} Next, we show that 
\begin{equation}\label{eqn: theta_bound_1}
\big|\thh-\thj\big|\le \frac{1}{4}\Delta^{\circ}.
\end{equation}
If not, by Lemma~\ref{lemma: imp_h_pop_h_agg} and Lemma~S3.5 in \cite{Choetal2025},
\begin{align}
& \big\vert \cV_{\alc,\thj,\blc} \big\vert_2 - \big\vert \cV_{\alc,\thh,\blc} \big\vert_2 \nonumber\\
\ge& \sqrt{\frac{(\thj-\alc)(\blc-\thj)}{\blc-\alc}}\left(1-\sqrt{\frac{1-|\thh-\thj|/(\thj-\alc)}{1+|\thh-\thj|/(\blc-\thj)}} \right)\omega_j \cdot (1+\cO_P(\alpha_{T,p})) \tag{Lemma~\ref{lemma: imp_h_pop_h_agg}} \nonumber\\
\ge& \left(1-\sqrt{\frac{3}{4}} \right)\sqrt{\frac{(\thj-\alc)(\blc-\thj)}{\blc-\alc}}\omega_j\cdot (1+\cO_P(\alpha_{T,p}))
\gtrsim \sqrt{\delc}\omega_j (1+o_P(1)),
\label{eqn: 1/4delta}
\end{align}
On the other hand, since $\thj$ is admissible from Step~1, we see from that $\cT_{\lc} = \cT_{\alc, \widehat{\theta}, \blc} \ge \cT_{\alc, \theta_j \blc}$ by construction, and that $\cT^*_{\alc, \widehat{\theta}, \blc} \le \cT^*_{\alc, \theta_j, \blc}$ by Lemma~\ref{lemma: imp_ppl_T_true}. 
Then from Lemma~\ref{lemma: imp_Vhat-V-k},
\begin{align}
\big\vert {\cV}_{\alc,\thj,\blc} \big\vert_2 - \big\vert {\cV}_{\alc,\thh,\blc} \big\vert_2 
\le & \,\big\vert \wh{\cV}_{\alc,\thj,\blc} -  {\cV}_{\alc,\thj,\blc} \big\vert_2  +  \big\vert \wh{\cV}_{\alc,\thh,\blc} -  {\cV}_{\alc,\thh,\blc} \big\vert_2 \nonumber\\
= &\,  \cO_P\big(\sqrt{\log(T)}\big) 
= \,{o_P(\sqrt{\delc}\omega_j)},\nonumber
\end{align}
where the last equation follows from~\eqref{eq:delc:lb}.
This contradicts~\eqref{eqn: 1/4delta}. 
\medskip

\underline{Step~3.} We show that
\begin{align*}
(\gc)^\trans{\cV}_{\alc,\thh,\blc} \ge 0, \text{ \ where \ }
\gc &:= \frac{\wh{\cV}_{\alc,\thh,\blc}}{\big\vert \wh{\cV}_{\alc,\thh,\blc}\big\vert_2}\in\R^{d},
\end{align*}
with $\vert \gc \vert_2 = 1$ (recall that $d=\sum_{k\in[K]} r_k(r_k + 1)/2$). It trivially holds that
\[
\cT_{\lc}=\l\vert \wh{\cV}_{\alc, \thh, \blc} \r\vert_2 = (\gc)^\trans \wh{\cV}_{\alc, \thh, \blc} > 0,
\]
while for $(\gc)^\trans \wh{\cV}_{\alc, \thj, \blc}$, we have $(\gc)^\trans \wh{\cV}_{\alc, \thj, \blc}\le \vert \wh{\cV}_{\alc, \thj, \blc} \vert_2 = \cT_{\alc,\thj,\blc}$. 
By construction,
\begin{align*}
0 &\le \cT_{\lc}-\cT_{\alc,\thj,\blc}\\
&\le (\gc)^\trans \wh{\cV}_{\alc,\thh,\blc} - (\gc)^\trans \wh{\cV}_{\alc,\thj,\blc}\\
&\le (\gc)^\trans \left( \wh{\cV}_{\alc,\thh,\blc} - {\cV}_{\alc,\thh,\blc} - \wh{\cV}_{\alc,\thj,\blc} + {\cV}_{\alc,\thj,\blc}\right) -(\gc)^\trans \left( {\cV}_{\alc,\thj,\blc} -{\cV}_{\alc,\thh,\blc}\right) \\
&=: U_1 - U_2.
\end{align*}
First, for $U_2$, let us write
\begin{align*}
-U_2  &= -(\gc)^\trans \left( {\cV}_{\alc,\thj,\blc} -{\cV}_{\alc,\thh,\blc}\right)\\
&=\sqrt{\frac{(\thh-\alc)(\blc - \thh)}{\blc - \alc}}\left(\frac{1}{\blc-\thh} \sum_{t=\thh+1}^{\blc} (\gc)^\trans V_t - \frac{1}{\thh-\alc} \sum_{t=\alc+1}^{\thh} (\gc)^\trans V_t \right)\\
&\quad -\sqrt{\frac{(\thj-\alc)(\blc - \thj)}{\blc - \alc}}\left(\frac{1}{\blc-\thj} \sum_{t=\thj+1}^{\blc} (\gc)^\trans V_t - \frac{1}{\thj-\alc} \sum_{t=\alc+1}^{\thj} (\gc)^\trans V_t \right) \\
&=\sqrt{\frac{(\thh-\alc)(\blc - \thh)}{\blc - \alc}} (\gc)^\trans\left[\frac{1}{\blc-\thh} \sum_{t=\thh+1}^{\blc} V_t - \frac{1}{\thh-\alc} \sum_{t=\alc+1}^{\thh} V_t \right.\\ 
&\qquad\qquad\qquad\qquad\qquad\qquad\qquad- \left.\frac{\blc-\thj}{\blc-\thh} \left(\frac{1}{\blc-\thj} \sum_{t=\thj+1}^{\blc} V_t - \frac{1}{\thj-\alc} \sum_{t=\alc+1}^{\thj} V_t \right)\right]\\
&\quad - \left\{\sqrt{\frac{(\thj-\alc)(\blc-\thj)}{\blc-\alc}} - \sqrt{\frac{(\thh-\alc)(\blc-\thh)}{\blc-\alc}}\frac{\blc-\thj}{\blc-\thh} \right\}\\
&\qquad\qquad\qquad\qquad\qquad\qquad\qquad\qquad\cdot(\gc)^\trans\left(\frac{1}{\blc-\thj} \sum_{t=\thj+1}^{\blc} V_t - \frac{1}{\thj-\alc} \sum_{t=\alc+1}^{\thj} V_t \right)\\
&=: U_{21} - U_{22}.
\end{align*}
Firstly, note that
\begin{align*}
& \frac{1}{\blc-\thh}\! \sum_{t=\thh+1}^{\blc}\! V^{(k)}_t - \frac{1}{\thh-\alc}\! \sum_{t=\alc+1}^{\thh}\! V^{(k)}_t - \frac{\blc-\thj}{\blc-\thh} \left(\frac{1}{\blc-\thj} \!\sum_{t=\thj+1}^{\blc}\! V^{(k)}_t - \frac{1}{\thj-\alc}\! \sum_{t=\alc+1}^{\thj}\! V^{(k)}_t \right)
\\
= & \,
\Vech\l\{\!\wh{H}_k^\trans \! \l(\frac{\blc-\thj}{\blc-\thh} \Gamma_{G,(j)}^{(k)} + \frac{\thj-\thh}{\blc-\thh} \Gamma_{G,(j-1)}^{(k)} - \!\Gamma_{G,(j-1)}^{(k)}  \!- \frac{\blc-\thj}{\blc-\thh} \Gamma_{G,(j)}^{(k)} + \frac{\blc-\thj}{\blc-\thh} \Gamma_{G,(j-1)}^{(k)} \r) \!\wh{H}_k \!\r\} \\
= & \, 0,
\end{align*}
for all $k \in [K]$, which gives $U_{21} = 0$.
By Lemma~\ref{lemma: imp_Vhat-V-k}, we have $\vert \wh{\cV}_{\alc,\thh,\blc} - {\cV}_{\alc,\thh,\blc} \vert_2 =\cO_P(\sqrt{\log(T)})$.
Recall that by Lemma~\ref{lemma: imp_h_pop_h_agg}, we have $\vert \cV_{\alc,\thh,\blc} \vert_2=\cT^{*}_{\alc,\thh,\blc}  \cdot (1+\cO_P(\alpha_{T,p}))$, where $\cO_P$-bound holds uniformly over all the triplets of the indices in consideration, which gives
\begin{align*}
\l\vert \wh{\cV}_{\alc,\thh,\blc} - {\cV}_{\alc,\thh,\blc} \r\vert_2 \le \l\vert {\cV}_{\alc,\thh,\blc} \r\vert_2. 
\end{align*}
Then, by Cauchy--Schwarz inequality, we have
\begin{align*}
\wh{\cV}_{\alc,\thh,\blc}^\trans \cV_{\alc,\thh,\blc} &= \l\vert \cV_{\alc,\thh,\blc} \r\vert_2^2 + \l(\wh{\cV}_{\alc,\thh,\blc} -\cV_{\alc,\thh,\blc}\r)^\trans \cV_{\alc,\thh,\blc}\\
&\ge \l\vert \cV_{\alc,\thh,\blc} \r\vert_2^2 - \l\vert \wh{\cV}_{\alc,\thh,\blc} - {\cV}_{\alc,\thh,\blc} \r\vert_2 \l\vert {\cV}_{\alc,\thh,\blc} \r\vert_2\\
&= \l\vert {\cV}_{\alc,\thh,\blc} \r\vert_2 \l(\l\vert {\cV}_{\alc,\thh,\blc} \r\vert_2 - \l\vert \wh{\cV}_{\alc,\thh,\blc} - {\cV}_{\alc,\thh,\blc} \r\vert_2 \r)\ge0,
\end{align*}
which implies
\begin{align*}
(\gc)^\trans{\cV}_{\alc,\thh,\blc} = \frac{\wh{\cV}_{\alc,\thh,\blc}^\trans}{\l\vert \wh{\cV}_{\alc,\thh,\blc}  \r\vert_2} {\cV}_{\alc,\thh,\blc}\ge0.
\end{align*}

\underline{Step~4.} We are now ready to derive the refined rate of estimation. 
From Lemma~\ref{lemma: imp_ppl_T_true}, we know that $\vert \cV_{\alc,\tau,\blc}\vert_2$ attains its maximum at $\tau=\thj$ within $(\alc,\blc)$.
By similar arguments, $\vert (\gc)^\trans{\cV}_{\alc,\tau,\blc}\vert$ attains its maximum at $\tau = \thj$. Therefore, we have $(\gc)^\trans{\cV}_{\alc,\thj,\blc} \ge (\gc)^\trans{\cV}_{\alc,\thh,\blc} \ge 0$.
Then together with Lemma~\ref{lemma: imp_Vhat-V-k}, we have
\begin{align*}
\l\vert \wh{\cV}_{\alc, \thh,\blc}\r\vert_2 = (\gc)^\trans \wh{\cV}_{\alc, \thh,\blc} \le (\gc)^\trans {\cV}_{\alc, \thh,\blc} +\cO_P(\sqrt{\log(T)}) \le (\gc)^\trans {\cV}_{\alc, \thj,\blc}+\cO_P(\sqrt{\log(T)}).
\end{align*}
On the other hand, due to the construction of $\cT_{\lc}$, again by using Lemma~\ref{lemma: imp_Vhat-V-k}, we have
\begin{align*}
\l\vert \wh{\cV}_{\alc, \thh,\blc}\r\vert_2 \ge \l\vert \wh{\cV}_{\alc, \thj,\blc}\r\vert_2 &\gtrsim \l\vert{\cV}_{\alc, \thj,\blc}\r\vert_2 +  \cO_P(\sqrt{\log(T)}).
\end{align*}
Combining the above, by Assumption~\ref{assum: signal-to-noise} and Lemma~\ref{lemma: imp_h_pop_h_agg}, we see
\begin{align*}
(\gc)^\trans {\cV}_{\alc, \thj,\blc} &\gtrsim \l\vert{\cV}_{\alc, \thj,\blc}\r\vert_2 + \cO_P(\sqrt{\log(T)}) \\
&= \cT^*_{\alc,\thj,\blc}(1+\cO_P(\alpha_{T,p})) + \cO_P(\sqrt{\log(T)})\tag{Lemma~\ref{lemma: imp_h_pop_h_agg}}\\
&=\sqrt{\frac{(\thj-\alc)(\blc-\thj)}{\blc-\alc}} \omega_j \cdot \cO_P\l(1+
\alpha_{T,p}+ {\frac{\sqrt{\log(T)}}{\sqrt{\delc}\omega_j}}\r)\\
&=\sqrt{\frac{(\thj-\alc)(\blc-\thj)}{\blc-\alc}} \omega_j \cdot (1+o_P(1)). \tag{from~\eqref{eq:delc:lb}}
\end{align*}

Then for $U_{22}$, by Lemma~7 of \cite{wang2018high}, 
\begin{align}
U_{22} &= \left\{\sqrt{\frac{(\thj-\alc)(\blc-\thj)}{\blc-\alc}} - \sqrt{\frac{(\thh-\alc)(\blc-\thh)}{\blc-\alc}}\frac{\blc-\thj}{\blc-\thh} \right\} \nonumber\\
&\qquad\qquad\qquad\qquad\qquad\qquad\cdot(\gc)^\trans\left(\frac{1}{\blc-\thj} \sum_{t=\thj+1}^{\blc} V_t - \frac{1}{\thj-\alc} \sum_{t=\alc+1}^{\thj} V_t \right)\nonumber\\
&\ge  \left\{\sqrt{\frac{(\thj-\alc)(\blc-\thj)}{\blc-\alc}} - \sqrt{\frac{(\thh-\alc)(\blc-\thh)}{\blc-\alc}}\frac{\blc-\thj}{\blc-\thh} \right\}\nonumber\\
&\quad\qquad\qquad\qquad \cdot {\left(\sqrt{\frac{(\thj-\alc)(\blc-\thj)}{\blc-\alc}}\right)^{-1}} \sqrt{\frac{(\thj-\alc)(\blc-\thj)}{\blc-\alc}}\omega_j \cdot(1+o_P(1)) \nonumber\\
&=  \left\{\sqrt{\frac{(\thj-\alc)(\blc-\thj)}{\blc-\alc}} - \sqrt{\frac{(\thh-\alc)(\blc-\thh)}{\blc-\alc}}\frac{\blc-\thj}{\blc-\thh} \right\}\, \omega_j \cdot(1+o_P(1)) \label{eqn: U2 : intermediate} \\
&\ge \, \frac{2\big\vert \thh - \thj \big\vert}{3\sqrt{6}\sqrt{\delc}} \;\omega_j \cdot(1+o_P(1)). \tag{Lemma~7 of \cite{wang2018high}} \nonumber
\end{align}
Altogether, we have
\begin{align}
U_2 \ge \frac{2\big\vert \thh - \thj \big\vert}{3\sqrt{6}\sqrt{\delc}} \;\omega_j \cdot(1+o_P(1)).
\label{eqn: U2}
\end{align}
As for $U_1$, by Lemmas~\ref{lemma: loading_prod_rate}~(ii) and \ref{lemma: GG-EGG_generic_decomp}, and Equation~(C.9) in \cite{Choetal2025}, we have
\begin{align}
\vert U_1 \vert_2 &\le\Big\vert  \left( \wh{\cV}_{\alc,\thh,\blc} - {\cV}_{\alc,\thh,\blc} \right) -\left(\wh{\cV}_{\alc,\thj,\blc} - {\cV}_{\alc,\thj,\blc}\right) \Big\vert_2
\nonumber \\
&= \l\vert \sqrt{\frac{(\blc-\thh)(\thh-\alc)}{\blc-\alc}}\l\{\frac{1}{\blc-\thh}\sum_{t= \thh +1}^{\blc} (\wh{V}_t-V_t) - \frac{1}{\thh-\alc}\sum_{t=\alc+1}^{\thh}(\wh{V}_t-V_t)\r\}  \r. 
\nonumber \\
&\quad\quad - \l. \sqrt{\frac{(\blc-\thj)(\thj-\alc)}{\blc-\alc}}\l\{\frac{1}{\blc-\thj}\sum_{t= \thj +1}^{\blc} (\wh{V}_t-V_t) - \frac{1}{\thj-\alc}\sum_{t=\alc+1}^{\thj}(\wh{V}_t-V_t)\r\} \r\vert_2
\nonumber \\
&\le \l\vert \sqrt{\frac{(\blc-\thh)(\thh-\alc)}{\blc-\alc}}\l( \cJ_{\thh,\blc}^{(1)}-\cJ_{\alc,\thh}^{(1)} + \cJ_{\thh,\blc}^{(2)}-\cJ_{\alc,\thh}^{(2)} \r)  \r.
\nonumber \\
&\quad\quad - \l. \sqrt{\frac{(\blc-\thj)(\thj-\alc)}{\blc-\alc}}\l( \cJ_{\thj,\blc}^{(1)}-\cJ_{\alc,\thj}^{(1)} + \cJ_{\thj,\blc}^{(2)}-\cJ_{\alc,\thj}^{(2)}\r) \r\vert_2 
\nonumber \\
&\quad\quad  + \l\vert \sqrt{\frac{(\blc-\thh)(\thh-\alc)}{\blc-\alc}} \cdot \frac{\blc-\thj}{\blc-\thh} - \sqrt{\frac{(\blc-\thj)(\thj-\alc)}{\blc-\alc}} \r\vert \vert {\Omega_j} \vert_2 H_1 H_3 
\nonumber \\
&\le \frac{\blc-\thj}{\sqrt{\blc-\alc}}\l\vert \sqrt{\frac{\wh\theta - \alc}{\blc - \wh\theta}} - \sqrt{\frac{\theta_j - \alc}{\blc - \theta_j}} \r\vert \l(\big\vert\cJ_{\thj,\blc}^{(1)}\big\vert_2 + \big\vert\cJ_{\thj,\blc}^{(2)} \big\vert_2\r) 
\nonumber \\
&\quad\quad + \frac{\thj-\alc}{\sqrt{\blc-\alc}}\l\vert  \sqrt{\frac{\blc - \wh\theta}{\wh\theta - \alc}} - \sqrt{\frac{\blc - \theta_j}{\theta_j - \alc}} \r\vert \l(\big\vert\cJ_{\alc,\thj}^{(1)}\big\vert_2 + \big\vert\cJ_{\alc,\thj}^{(2)}\big\vert_2\r) 
\nonumber \\
&\quad\quad + \frac{\thj-\thh}{\sqrt{\blc-\alc}} \cdot \sqrt{\frac{\blc - \wh\theta}{\wh\theta - \alc}} \l(\big\vert\cJ_{\thh,\thj}^{(1)}\big\vert_2 + \big\vert\cJ_{\thh,\thj}^{(2)}\big\vert_2 \r) 
\nonumber \\
&\quad\quad + \frac{\thj-\thh}{\sqrt{\blc-\alc}} \cdot \sqrt{\frac{\thh -\alc}{\blc-\thh}} \l(\big\vert\cJ_{\thh,\thj}^{(1)}\big\vert_2 + \big\vert\cJ_{\thh,\thj}^{(2)} \big\vert_2 \r)
\nonumber \\
&\quad\quad + \l\vert\sqrt{\frac{(\blc-\thh)(\thh-\alc)}{\blc-\alc}}\frac{\blc-\thj}{\blc-\thh} - \sqrt{\frac{(\blc-\thj)(\thj-\alc)}{\blc-\alc}}\r\vert\omega_j H_1 H_3  
\nonumber \\
&=: U_{11} + U_{12} + U_{13} + U_{14}+ U_{15}.
\label{eq:U1:decomp}
\end{align}

Let $\wt{\cT}_{\alc,\tau,\blc} = \cT_{\alc,\tau,\blc} - \cT_{\alc,\thj\,\blc}$, then $\wt{\cT}_{\alc,\thj,\blc}=0 $. Since $\thh = \argmax_{\alc\le \tau \le \blc} \cT_{\alc,\tau,\blc}<\thj$ by construction and from~\eqref{eqn: theta_bound_1}, it follows that for some fixed constant $C>0$ and $\kappa_T \to \infty$ arbitrarily slow, 
\begin{align*}
&\quad\;\l\{\omega_j^2 ({\thh - \thj}) \le -{C}\kappa_T^2 \r\} \\
&\subset \l\{\max_{\thj-\delc/4 +1 \le \thh\le \thj - C\omega_j^{-2}\kappa_T^2 } \wt{\cT}_{\alc,\thh,\blc} \ge \max_{\thj - C\omega_j^{-2}\kappa_T^2 +1 \le \thh\le \thj+\delc/4 } \wt{\cT}_{\alc,\thh,\blc} \r\}\\
&\subset \l\{\max_{\thj-\delc/4 +1 \le \thh\le \thj - C\omega_j^{-2}\kappa_T^2 } \wt{\cT}_{\alc,\thh,\blc} \ge 0 \r\} 
{\subset \l\{ \max_{\thj-\delc/4 +1 \le \thh\le \thj - C\omega_j^{-2}\kappa_T^2 } \vert U_1 \vert - U_2 \ge 0  \r\}.}
\end{align*}
Then from~\eqref{eqn: U2} and~\eqref{eq:U1:decomp},
\begin{align*}
&\quad\; \P\l\{\max_{\thj-\delc/4 + 1\le \thh\le \thj - C\omega_j^{-2}\kappa_T^2} {\vert U_{1} \vert}-{U_2}\ge 0 \r\}\\
&=\P\l\{\max_{\thj-\delc/4 + 1\le \thh\le \thj - C\omega_j^{-2}\kappa_T^2} \frac{U_{11}+U_{12}+U_{13}+U_{14}+U_{15}}{U_2}\ge 1 \r\}\\
&\le \P\l\{\max_{\thj-\delc/4 + 1\le \thh\le \thj - C\omega_j^{-2}\kappa_T^2} \frac{U_{11}}{U_2}\ge \frac{1}{5} \r\} + \P\l\{\max_{\thj-\delc/4 + 1\le \thh\le \thj - C\omega_j^{-2}\kappa_T^2} \frac{U_{12}}{U_2}\ge \frac{1}{5} \r\}\\
&\quad+\P\l\{\max_{\thj-\delc/4 + 1\le \thh\le \thj - C\omega_j^{-2}\kappa_T^2} \frac{U_{13}}{U_2}\ge \frac{1}{5}\r\} +\P\l\{\max_{\thj-\delc/4 + 1\le \thh\le \thj - C\omega_j^{-2}\kappa_T^2} \frac{U_{14}}{U_2}\ge \frac{1}{5}\r\} \\
&\quad + \P\l\{\max_{\thj-\delc/4 + 1\le \thh\le \thj - C\omega_j^{-2}\kappa_T^2} \frac{U_{15}}{U_2}\ge \frac{1}{5}\r\} =: W_1 + W_2 + W_3 + W_4 + W_5.
\end{align*}
Firstly, for $W_1$, by Lemma~S3.5 of \cite{Choetal2025}, and by~\eqref{eqn: theta_bound_1}, 
\begin{align*}
U_{11} &= \frac{\blc-\thj}{\sqrt{\blc-\alc}}\l\vert \sqrt{\frac{\wh\theta - \alc}{\blc - \wh\theta}} - \sqrt{\frac{\theta_j - \alc}{\blc - \theta_j}} \r\vert \l(\big\vert\cJ_{\thj,\blc}^{(1)}\big\vert_2 + \big\vert\cJ_{\thj,\blc}^{(2)}\big\vert_2 \r) \\
&\le \frac{\blc-\thj}{\sqrt{\blc-\alc}}\l\vert \sqrt{\frac{\thj-\alc}{\blc-\thj}} \frac{(\thj-\thh)(\blc-\alc)}{(\thj-\alc)(\blc-\thj)} \r\vert\l(\big\vert\cJ_{\thj,\blc}^{(1)}\big\vert_2 + \big\vert\cJ_{\thj,\blc}^{(2)}\big\vert_2 \r)  \\
&\le \sqrt{\frac{\blc-\alc}{(\thj-\alc)(\blc-\thj)}}\cdot \vert \thj-\thh \vert \l(\big\vert\cJ_{\thj,\blc}^{(1)}\big\vert_2 + \big\vert\cJ_{\thj,\blc}^{(2)} \big\vert_2 \r).
\end{align*}
In addition, we have
\begin{align}
\sqrt{\frac{\thj-\alc}{\blc-\alc}} \cdot \frac{\blc-\thj}{\sqrt{\delc}} &\ge {\sqrt{\frac{\delc}{2}}}, \label{eq:delc:1}
\\
\sqrt{\frac{(\thj-\alc)(\blc-\thj)}{\blc-\alc}} \cdot \frac{1}{\sqrt{\delc}} &\ge {\sqrt{\frac{1}{2}}}. \label{eq:delc:2}
\end{align}
Then, by Assumption~\ref{assum: signal-to-noise}, Lemmas~\ref{lemma: loading_prod_rate}~(ii), \ref{lemma: imp_maximal_deviation_mat}, and~\ref{lemma: X(Lam-Lam)X}, and~\eqref{eqn: U2},
\begin{align*}
W_1
&\le \P\l\{\max_{\thj-\delc/4+1\le \thh \le \thj+\delc/4} \frac{U_{11}}{U_2}\ge \frac{1}{5} ,\, \cH_{T,p}^c \r\} + \P(\cH_{T,p}) \\
&\le \P\l\{ \sqrt{\frac{\blc-\alc}{\thj-\alc}}\frac{1}{\blc-\thj} \cdot \sqrt{\blc-\thj}  \big\vert \cJ_{\thj,\blc}^{(1)} \big\vert_2 \ge \frac{3\sqrt{6} \omega_j}{10 \sqrt{\delc}} \big(1+C_H\alpha_{T,p} \big) \r\} \\
&\quad + \P\l\{ \sqrt{\frac{\blc-\alc}{(\thj-\alc)(\blc-\thj)}} \big\vert \cJ_{\thj,\blc}^{(2)}\big\vert_2 \ge \frac{3\sqrt{6} \omega_j}{10 \sqrt{\delc}} \big(1+C_H\alpha_{T,p} \big) \r\}  + o(1)  \tag{Lemma~\ref{lemma: loading_prod_rate}~(ii)}\\
&\le \P\l\{ \sqrt{\blc-\thj} \big\vert \cJ_{\thj,\blc}^{(1)} \big\vert_2 \ge {\frac{3\sqrt{3}}{10}  \sqrt{\delc} \omega_j} \cdot \big(1+C_H\alpha_{T,p} \big) \r\} \tag{from~\eqref{eq:delc:1}}\\
&\quad + \P\l\{ \big\vert \cJ_{\thj,\blc}^{(2)}\big\vert_2 \ge {\frac{3 \sqrt{3} }{10}\omega_j } \cdot\big(1+C_H\alpha_{T,p} \big) \r\}  + o(1) \tag{from~\eqref{eq:delc:2}}\\
&= o(1).\tag{Lemmas~\ref{lemma: imp_maximal_deviation_mat}, \ref{lemma: X(Lam-Lam)X} and~\eqref{eq:delc:lb}} 
\end{align*}
We handle $W_2$ analogously. 
As for $W_3$, by~\eqref{eqn: theta_bound_1}, we have
\begin{align*}
\sqrt{\frac{\blc - \wh\theta}{(\wh\theta - \alc)(\blc-\alc)}} \le 
{\sqrt{\frac{4}{3\delc}}},
\end{align*}
which gives
\begin{align*}
U_{13} \le {\frac{2}{\sqrt{3}}}\frac{ \vert\thj-\thh \vert}{\sqrt{\delc}}  \l(\big\vert \cJ_{\thh,\thj}^{(1)} \big\vert_2 +\big\vert \cJ_{\thh,\thj}^{(2)} \big\vert_2 \r).
\end{align*}
Then by Assumption~\ref{assum: signal-to-noise}, Lemmas~\ref{lemma: loading_prod_rate}~(ii), \ref{lemma: XX-EXX_zz} and~\eqref{eqn: U2}, with $\kappa_T$ satisfying~\eqref{eq:kappa:req}, we have
\begin{align*}
W_3 &\le \P\l\{\max_{\thj-\delc/4 + 1\le \thh\le \thj - C\omega_j^{-2}\kappa_T^2} {\frac{2}{\sqrt{3}}}\frac{1}{\sqrt{\delc}} \l(\big\vert \cJ_{\thh,\thj}^{(1)} \big\vert_2 +\big\vert \cJ_{\thh,\thj}^{(2)} \big\vert_2 \r) \ge \frac{3\sqrt{6} \omega_j}{10\sqrt{\delc}} \big(1+C_H\alpha_{T,p} \big)\r\} \\
&\quad\,+ \P(\cH_{T,p}) \tag{Lemma~\ref{lemma: loading_prod_rate}~(ii)}\\
&\le \P\l\{\max_{\thj-\delc/4 + 1\le \thh\le \thj - C\omega_j^{-2}\kappa_T^2} \sqrt{\omega_j^{-2}\kappa_T^2} \big\vert \cJ_{\thh,\thj}^{(1)} \big\vert_2 \ge {\frac{9\sqrt{2}}{40}}\kappa_T\cdot  \big(1+C_H\alpha_{T,p} \big)\r\} 
\\
&\quad + {\P\l\{\max_{\thj-\delc/4 + 1\le \thh\le \thj - C\omega_j^{-2}\kappa_T^2} \big\vert \cJ_{\thh,\thj}^{(2)} \big\vert_2 \ge {\frac{9\sqrt{2}}{40}} \omega_j\cdot \big(1+C_H\alpha_{T,p} \big)\r\}}
+ o(1)  \tag{Lemma~\ref{lemma: loading_prod_rate}~(ii)}\\
&= o(1).\tag{Lemma~\ref{lemma: XX-EXX_zz} and Assumption~\ref{assum: signal-to-noise}}
\end{align*}
We can handle $W_4$ analogously. 
Finally, for $W_5$, noting that
\begin{align*}
U_{15} &= \l\vert\sqrt{\frac{(\blc-\thj)(\thj-\alc)}{\blc-\alc}} - \sqrt{\frac{(\blc-\thh)(\thh-\alc)}{\blc-\alc}}\frac{\blc-\thj}{\blc-\thh}\r\vert\omega_j \cdot H_1H_3,
\end{align*}
we invoke~\eqref{eqn: U2 : intermediate} and the arguments used in deriving~\eqref{eqn: U2} to have
\begin{align*}
W_5 &\le 
\P\l\{\max_{\thj-\delc/4 + 1\le \thh\le \thj - C\omega_j^{-2}\kappa_T^2} H_1H_3 \ge \frac{1}{5}(1+o(1)) \r\} + \P\l( \cS^{(1)}_{T,p} \cup \cS^{(2)}_{T,p} \cup \wt{\cS}^{(1)}_{T,p} \cup \wt{\cS}^{(2)}_{T,p}\r)\\
&\le \P(\cH_{T, p}) + \P\l( \cS^{(1)}_{T,p} \cup \cS^{(2)}_{T,p} \cup \wt{\cS}^{(1)}_{T,p} \cup \wt{\cS}^{(2)}_{T,p}\r) = o(1),
\end{align*}
by Lemma~\ref{lemma: loading_prod_rate}~(ii).
Combining the bounds on $W_1$--$W_5$, we have
\begin{align*}
\P\l\{\omega_j^2 (\thj-\thh) \le - C \kappa_T^2 \r\} \le \P\l\{\max_{\thj-\delc/4 +1 \le \thh\le \thj - C\omega_j^{-2}\kappa_T^2 } \wt{\cT}_{\alc,\thh,\blc} \ge 0 \r\} =o(1).
\end{align*}
Analogous arguments apply to the case where $\thh>\thj $. Hence, concluding \ref{R1}--\ref{R2}, the proof is complete.

\clearpage

\subsection{Consistency in mode-identification}

\subsubsection{Supporting lemmas}

Let us define for $j \in [q]$,
\begin{align}
\begin{split}
\Theta_j^{-} &= \left\{a_\ell : \, (a_\ell, b_\ell] \in \M^*, \, b_\ell \in [\theta_j - 3T/2^{\mu_T + 1}, \theta_j + T/2^{\mu_T}] \right\},
\\
\Theta_j^{+} &= \left\{b_\ell : \, (a_\ell, b_\ell] \in \M^*, \, a_\ell \in [\theta_j - T/2^{\mu_T}, \theta_j + 3 T/2^{\mu_T + 1} ] \right\},
\end{split}
\label{eq:Theta:pm}
\end{align}
with $\Theta_0^\pm = \{0\}$ and $\Theta_{q + 1}^\pm = \{T\}$.
By construction of $\M^*$, all the intervals therein have length $T/2^{\mu_T}$ such that under Assumption~\ref{assum: trans_mat_alt}~(iv), for $T$ large enough,
\begin{enumerate}[label = (\alph*)]
\item \label{Theta:one} $\vert \Theta_j^\pm \vert \le 4$ for all $j$, 
\item \label{Theta:two} for any $a \in \Theta_j^-$ (resp.\ $b \in \Theta_j^+$), we have $a \in (\theta_{j - 1}, \theta_j]$ (resp.\ $b \in [\theta_j, \theta_{j + 1})$), and
\item \label{Theta:three} for any $j \in [q + 1]$, $a \in \Theta^+_{j - 1}$ and $b \in \Theta^-_j$, we have $b - a \ge c_\Delta T - 5 T/2^{\mu_T} \ge c_\Delta T/2$.
\end{enumerate}

\begin{lemma}[]\label{lemma: gen_sample_Hpopu_error}
Let Assumptions~\ref{assum: core_factor}, \ref{assum: loadings}, \ref{assum: trans_mat_alt} and \ref{assum: noise} hold, with $\nu>8$ and $\beta\geq 0$.
\begin{itemize}
\item [(i)] For any $a<b$ and $c<d$, with the the definitions in \eqref{eqn: gamma_G} and the notations in Lemma~\ref{lemma: GG-EGG_generic_decomp}, we have
\begin{align*}
&\;\quad
\left\| \wh{\Gamma}_{G,a, b}^{(k)} - \wh{\Gamma}_{G,c, d}^{(k)}
- \left( \wh{H}_k^\trans  {\Gamma}_{G,a,b}^{(k)} \wh{H}_k - \wh{H}_k^\trans {\Gamma}_{G,c,d}^{(k)} \wh{H}_k \right) \right\|  \\
&\leq
2\, \Big\Vert\cJ_{k, a, b}^{(1)}-\cJ_{k, c, d}^{(1)} \Big\Vert_F
+ 2\, \Big\Vert\cJ_{k, a, b}^{(2)}-\cJ_{k, c, d}^{(2)}\Big\Vert_F  \\
&\;\quad
+ 4\sqrt{2} \cdot \Vert\wh{H}_k\Vert \l\Vert \frac{1}{\sqrt{p_k}}\l(\wh{\Lambda}_k-\Lambda_k\wh{H}_k \r) \r\Vert_F \cdot \l\vert  \Vech\l(\Gamma_{G, a, b}^{(k)} - \Gamma_{G, c, d}^{(k)} \r)\r\vert_2 .
\end{align*}
\item [(ii)] We have
\[
\max_{j \in [q]} \max_{\substack{a_j \in \Theta_j^+ \\ b_j \in \Theta_j^-}} \max_{k\in[K]} 
\left\| \wh{\Gamma}_{G, a_j, b_{j + 1}}^{(k)} \!\! - \wh{\Gamma}_{G, a_{j - 1}, b_j}^{(k)}
\!\! - \wh{H}_k^\trans \left(  {\Gamma}_{G, a_j, b_{j + 1}}^{(k)}\!\! - {\Gamma}_{G, a_{j - 1}, b_j}^{(k)} \right) \wh{H}_k  \right\|
= \cO_P\left( \frac{1}{\sqrt{T}} + \frac{1}{p} \right) .
\]
\end{itemize}
\end{lemma}

\begin{proof}[Proof of Lemma~\ref{lemma: gen_sample_Hpopu_error}]

Part~(i) follows directly from carefully checking the proof of Lemma~\ref{lemma: GG-EGG_generic_decomp}.

For part~(ii), fix any $k\in[K]$. 
By~\ref{Theta:two}--\ref{Theta:three}, we have all $\Gamma^{(k)}_{G, a_j, b_{j + 1}}$ for $a_j \in \Theta^+_j$ and $b_{j + 1} \in \Theta^-_{j + 1}$ well-defined and
$\Gamma^{(k)}_{G, a_j, b_{j + 1}} = \Gamma^{(k)}_{G, j + 1}$, for all $j = 0, \ldots, q$.

Note that $\|X\|_F \leq 2 \cdot \vert \Vech(X)\vert_2$ for any symmetric matrix~$X$.
By part~(i), we have
\begin{align}
    &\;\quad
    \left\| \wh{\Gamma}_{G, a_j, b_{j + 1}}^{(k)} - \wh{\Gamma}_{G, a_{j - 1}, b_j}^{(k)}
    - \left( \wh{H}_k^\trans  {\Gamma}_{G, a_j, b_{j + 1}}^{(k)} \wh{H}_k - \wh{H}_k^\trans {\Gamma}_{G, a_{j - 1}, b_j}^{(k)} \wh{H}_k \right) \right\| \notag \\
    &\leq
    2\, \Big\Vert\cJ_{k, a_j, b_{j + 1}}^{(1)}-\cJ_{k, a_{j - 1}, b_j}^{(1)} \Big\Vert_F
    + 2\, \Big\Vert\cJ_{k, a_j, b_{j + 1}}^{(2)}-\cJ_{k, a_{j - 1}, b_j}^{(2)}\Big\Vert_F \notag \\
    &\;\quad
    + 4\sqrt{2} \cdot \Vert\wh{H}_k\Vert \l\Vert \frac{1}{\sqrt{p_k}}\l(\wh{\Lambda}_k-\Lambda_k\wh{H}_k \r) \r\Vert_F \cdot \l\vert  \Vech\l(\Gamma_{G, a_j, b_{j + 1}}^{(k)} - \Gamma_{G, a_{j - 1}, b_j}^{(k)} \r)\r\vert_2 .
    \label{eqn: hat_Gamma_diff_step1}
\end{align}

Following the arguments in the proof of Lemmas~\ref{lemma: xx-Exx} and \ref{lemma: imp_maximal_deviation_mat}, while noticing that by~\ref{Theta:one}, there are only finitely many tuples of $(a_j, b_{j + 1})$ in consideration, and that $b_{j + 1} - a_j \ge c_\Delta T/2$ by~\ref{Theta:three}, we can analogously show that as $\min(T,p_1,\ldots,p_K)\to \infty$,
\begin{equation}
\label{eqn: hat_Gamma_diff_step2}
    \P
    \l\{ \max_{0 \le j \le q} \max_{a_j \in \Theta_j^+} \max_{b_j \in \Theta_j^-} \max_{k\in[K]} \sqrt{b_{j + 1} - a_j} \Big\vert \cJ_{a_j, b_{j + 1}}^{(1)} \Big\vert_2 \geq C_1' \r\}
     \to 0 ,
\end{equation}
for some finite constant $C_1'>0$.
Similarly, we have from the proof of Lemma~\ref{lemma: X(Lam-Lam)X} that as $\min(T,p_1,\ldots,p_K)\to \infty$, 
\begin{equation}
\label{eqn: hat_Gamma_diff_step3}
    \P 
    \l\{ \max_{0 \le j \le q} \max_{a_j \in \Theta_j^+} \max_{b_j \in \Theta_j^-} \max_{k\in[K]} \Big\vert \cJ_{a_j, b_{j + 1}}^{(2)} \Big\vert_2 \geq C_2' \alpha_{T,p} \r\}
    \to 0 ,
\end{equation}
for some finite constant $C_2'>0$, and $\alpha_{T,p}$ is defined in \eqref{eq:alpha:rate}.
Further note that by Assumptions~\ref{assum: core_factor} and \ref{assum: trans_mat_alt}, Lemma~\ref{lemma: consistency_proj} and \eqref{eq:alpha:rate}, we have
\begin{align}
    & \max_{k\in[K]} \Vert\wh{H}_k\Vert \l\Vert \frac{1}{\sqrt{p_k}}\l(\wh{\Lambda}_k-\Lambda_k\wh{H}_k \r) \r\Vert_F \times \nonumber \\
    & \qquad \max_{j \in [q]} \max_{a_j \in \Theta_j^+} \max_{b_j \in \Theta_j^-} \l\vert  \Vech\l(\Gamma_{G, a_j, b_{j+1}}^{(k)} - \Gamma_{G, a_{j-1}, b_j}^{(k)} \r)\r\vert_2 
    = \cO_P(\alpha_{T,p}) . \label{eqn: hat_Gamma_diff_step4}
\end{align}
Altogether, we conclude the proof of part~(ii) by
combining \eqref{eqn: hat_Gamma_diff_step1}, \eqref{eqn: hat_Gamma_diff_step2}, \eqref{eqn: hat_Gamma_diff_step3} and \eqref{eqn: hat_Gamma_diff_step4} with~\ref{Theta:three}.
\end{proof}

\begin{lemma}\label{lemma: rate_mode_change}
Let Assumptions~\ref{assum: core_factor}, \ref{assum: loadings}, \ref{assum: trans_mat_alt} and \ref{assum: noise} hold, with $\nu>8$ and $\beta\geq 0$.
For each $j \in [q]$, we define
\begin{align*}
\wh\Xi^{(k)}_{a_{j - 1}, b_j, a_j, b_{j + 1}} = \tr(\wh{\Gamma}_{G, a_j, b_{j + 1}}^{(k)}) \wh{\Gamma}_{G, a_j, b_{j + 1}}^{(k)} - \tr(\wh{\Gamma}_{G, a_{j-1}, b_j}^{(k)}) \wh{\Gamma}_{G, a_{j-1}, b_j}^{(k)}.
\end{align*}
Then, we have
\begin{align*}
    \max_{j \in [q]} \max_{a_j \in \Theta^+_j} \max_{b_j \in \Theta^-_j} \max_{k \in [K]}
     \left\Vert \wh\Xi^{(k)}_{a_{j - 1}, b_j, a_j, b_{j + 1}} - \wh{H}_k^\trans \Xi_j^{(k)} \wh{H}_k \right\Vert = \cO_P\left( \frac{1}{\sqrt{T}} + \frac{1}{p} \right) .
\end{align*}
\end{lemma}

\begin{proof}[Proof of Lemma~\ref{lemma: rate_mode_change}]
As an auxiliary result, for any symmetric and positive semi-definite matrices $B$ and $C$ of fixed dimensions and with $\tr(B), \tr(C) \asymp_P 1$,
we have
\begin{align*}
    \left\| \frac{B}{\tr(B)}- \frac{C}{\tr(C)} \right\| = \left\| \frac{B-C}{\tr(B)} + C\left(\frac{1}{\tr(B)} - \frac{1}{\tr(C)}\right) \right\| 
    \lesssim_P \|B-C\| .
\end{align*}
Similarly, for any symmetric and positive semi-definite matrices $B_1$, $B_2$, $C_1$ and $C_2$ of fixed dimensions and with $\tr(B_1) \asymp_P \tr(B_2) \asymp_P 1$ and $\tr(C_1)$ and $\tr(C_2)$ are both positive with probability approaching one:
\begin{align}
    &\;\quad
    \left\| \frac{B_1}{\tr(B_1)}- \frac{C_1}{\tr(C_1)} + \frac{B_2}{\tr(B_2)}- \frac{C_2}{\tr(C_2)} \right\| \notag \\
    &=
    \left\| \frac{B_1-C_1}{\tr(B_1)} + C_1\left(\frac{1}{\tr(B_1)} - \frac{1}{\tr(C_1)}\right) + \frac{B_2-C_2}{\tr(B_2)} + C_2\left(\frac{1}{\tr(B_2)} - \frac{1}{\tr(C_2)}\right) \right\| \notag \\
    &\lesssim_P
    \|B_1 - C_1 + B_2 - C_2\| .
    \label{eqn: trace_diff_aux}
\end{align}

Next, fix any $k \in [K]$.
Notice that by~\ref{Theta:two}--\ref{Theta:three}, Definition~\ref{def: jump},~\eqref{eqn: cov_G_signal} and~\eqref{eqn: gamma_G}, we have, for all $j \in [q + 1]$, $a_j \in \Theta^{-}_j$ and $b_j \in \Theta^+_j$,
\begin{align}
\Gamma_{G, a_{j-1}, b_j}^{(k)} &= \Gamma_{G,(j)}^{(k)} 
= \E\Big\{ \mat_k(\cG_{\theta_j}) \mat_k(\cG_{\theta_j})^\trans \Big\}
\nonumber \\
&= \tr\Big\{ \Big( \otimes_{i = K, i \ne k}^1 A_{j,i} \Big)^\trans \Big( \otimes_{i = K, i \ne k}^1 A_{j,i} \Big) \Big\} \cdot A_{j,k} A_{j,k}^\trans .
\nonumber 
\end{align}
Hence from~\eqref{eqn: scaled_mode_change}, we may write, for $j\in[q]$,
\begin{align*}
    \Xi_j^{(k)} 
    &=
    \tr\Big( {\Gamma}_{G, a_j, b_{j+1}}^{(k)} \Big)^{-1} {\Gamma}_{G, a_j, b_{j+1}}^{(k)} - \tr\Big( {\Gamma}_{G, a_{j - 1}, b_j}^{(k)} \Big)^{-1} {\Gamma}_{G, a_{j - 1}, b_j}^{(k)} .
\end{align*}
Then we have
\begin{align*}
    &\; \quad \max_{j \in [q]} \max_{a_j \in \Theta^+_j} \max_{b_j \in \Theta^-_j} \max_{k \in [K]}
    \left\| \wh\Xi^{(k)}_{a_{j - 1}, b_j, a_j, b_{j + 1}} -\wh{H}_k^\trans \Xi_j^{(k)} \wh{H}_k \right\| \\
    &= \max_{j \in [q]} \max_{a_j \in \Theta^+_j} \max_{b_j \in \Theta^-_j} \max_{k \in [K]}
    \Bigg\| \left\{ \tr\Big( \wh{\Gamma}_{G,a_j, b_{j + 1}}^{(k)} \Big)^{-1} \wh{\Gamma}_{G,a_j, b_{j + 1}}^{(k)} - \tr\Big( \wh{\Gamma}_{G, a_{j - 1}, b_j}^{(k)} \Big)^{-1} \wh{\Gamma}_{G, a_{j - 1}, b_j}^{(k)} \right\} \\
    &\;\quad 
    - \wh{H}_k^\trans \left\{ \tr\Big( {\Gamma}_{G,a_j, b_{j + 1}}^{(k)} \Big)^{-1} {\Gamma}_{G,a_j, b_{j + 1}}^{(k)} - \tr\Big( {\Gamma}_{G, a_{j - 1}, b_j}^{(k)} \Big)^{-1} {\Gamma}_{G, a_{j - 1}, b_j}^{(k)} \right\} \wh{H}_k \Bigg\| \\
    &\lesssim_P \max_{j \in [q]} \max_{a_j \in \Theta^+_j} \max_{b_j \in \Theta^-_j} \max_{k \in [K]} \left\| \wh{\Gamma}_{G, a_j, b_{j+1}}^{(k)} - \wh{\Gamma}_{G, a_{j-1}, b_j}^{(k)}
    - \left( \wh{H}_k^\trans  {\Gamma}_{G, a_j, b_{j+1}}^{(k)} \wh{H}_k - \wh{H}_k^\trans {\Gamma}_{G, a_{j-1}, b_j}^{(k)} \wh{H}_k \right) \right\| \notag \\
    &=
    \cO_P\left( \frac{1}{\sqrt{T}} + \frac{1}{p} \right),  
\end{align*}
where we make use of~\eqref{eqn: trace_diff_aux}, Assumption~\ref{assum: trans_mat_alt} and the fact that $\wh{H}_k$ is asymptotically orthogonal according to Lemma~\ref{lemma: consistency_proj}, for the first inequality, and the last equality follows from Lemma~\ref{lemma: gen_sample_Hpopu_error}~(ii), which completes the proof.
\end{proof}

\subsubsection{Proof of Theorem~\ref{thm: mode_identify}}

\begin{proof}[Proof of Theorem~\ref{thm: mode_identify}]
Throughout this proof, we use that for $T$ large enough, we have
$\upsilon_T = \kappa_T^2 \le \min_{j \in [q]} \omega_j^2 \Delta_j / 2^{\mu_T}$, which is in line with~\eqref{eq:kappa:req}.
By Theorem~\ref{thm: asymp_consistency_detection}, we have
\begin{align*}
    &\;\quad
    \P\l(\wh q = q, \, \max_{j\in[q]}\, \vert \thh-\thj\vert \le \frac{T}{2^{\mu_T}} \r) 
    \geq \P\l(\wh q = q, \, \max_{j\in[q]}\, \omega_j^{2} \vert \thh-\thj\vert \le \min_{j \in [q]} \omega_j^2 \cdot \frac{T}{2^{\mu_T}} \r) \\
    &\geq
    \P\l(\wh q = q, \, \max_{j\in[q]}\, \omega_j^{2} \vert \thh-\thj\vert \le \kappa_T^2 \r) \to 1 
\end{align*}
as $\min(T, p_1, \ldots, p_K) \to \infty$.
Then, it follows that on the event $\{\wh q = q, \, \max_{j \in [q]} \omega_j^2 \vert \wh\theta_j - \theta_j \vert \le \kappa_T^2 \}$, we have 
$\wh\theta_j^{-1} \in \Theta_j^{-}$ and $\wh\theta_j^+ \in \Theta_j^+$ for all $j \in [q]$, by their construction in~\eqref{eqn: seeded_adj_cp} and~\eqref{eq:Theta:pm}. Therefore,
\begin{align*}
    &\;\quad
    \frac{\zeta_{T, p}}{\min_{j \in [q]} \min_{k \in \cK_j} \|\wh\Xi_j^{(k)} \| } \\
    &\leq 
    \frac{\zeta_{T, p}}{\min_{j \in [q]} \min_{k \in \cK_j} \|\wh{H}_k^\trans \Xi_j^{(k)} \wh{H}_k \| - \max_{j \in [q]} \max_{k \in \cK_j} \|\wh\Xi_j^{(k)} -\wh{H}_k^\trans \Xi_j^{(k)} \wh{H}_k\|} \\
    &=
    \frac{\zeta_{T, p}}{\min_{j \in [q]} \min_{k \in \cK_j} \|\wh{H}_k^\trans \Xi_j^{(k)} \wh{H}_k \| + \cO_P(T^{-1/2} + p^{-1})} \\
    &\asymp_P 
    \frac{\zeta_{T, p}}{\min_{j \in [q]} \min_{k \in \cK_j} \|\Xi_j^{(k)} \| } = o(1) ,
\end{align*}
where the first equality holds thanks to Lemma~\ref{lemma: rate_mode_change}, the second under Assumption~\ref{assum: mode-id}, and the final under the condition imposed on $\zeta_{T, p}$. 
Similarly, it holds that
\begin{align*}
    \max_{j \in [q]} \max_{k \in [K] \setminus \cK_j} \|\wh\Xi_j^{(k)} \| 
    &\leq 
    \max_{j \in [q]} \max_{k \in [K]} \|\wh\Xi_j^{(k)} -\wh{H}_k^\trans \Xi_j^{(k)} \wh{H}_k\| = o_P(\zeta_{T, p}),
\end{align*}
which completes the proof.
\end{proof}

\clearpage

\section{Numerical experiments for threshold selection}\label{app:thd}

\subsection{Threshold \texorpdfstring{$\pi_{T,p}$}{pi(T,p)} for change point detection}
\label{app: num_det_oracle}

We simulate tensor time series under the null model~\ref{s:null} with no change points in Appendix~\ref{app: num_setting}.
Varying the dimensions and core factor tensor ranks as
\begin{align*}
T&\in\{400, 1200, 2000, 2800, 3200, 4000, 4800, 5600\},\\
(p_1,p_2,p_3)&\in\{(10,10,10), (10,10,100), (10,20,40), (20,20,20)\}, \\
(r_1,r_2,r_3) &\in \{(2,2,2), (2,3,2), (2,3,3), (3,3,3)\},
\end{align*}
while fixing $(\rho_f, \rho_e) = (0.7, 0.3)$, we generate $500$ realizations for each combination.
Next, for each realization, we take $\max_{(s, e] \in \M} \max_{\tau \in (s + \varpi_T, e - \varpi_T]} \cT_{s, \tau, e}$, the maximum detector statistic over the collection of seeded intervals, and record the $0.9$-empirical quantiles of these maxima for each setting.
Then, we regress them (contained in the \verb+R+ object \verb+y+) onto the corresponding on functions of $T$ and $d = \sum_{k = 1}^K r_k(r_k + 1) / 2$, whose choices are motivated by the approximation of the asymptotic null distribution made in \citet{cho2026multivariate}, 
using the \verb+R+ function \verb+lm+:
\begin{verbatim}
lm(y ~ sqrt(d) + sqrt(log(T)) + log(log(T))/sqrt(log(T)) + 1/sqrt(log(T)))
\end{verbatim}
The fitted model has $R^2_{\text{adj}} = 0.9305$, and Figure~\ref{fig:det_reg} displays diagnostic plots which show that the model provides a good fit to the data. 
As the model depends only on $T$ and $d$, we use it for determining the threshold $\pi_{T, p}$.

\begin{figure}[h!t!b!]
\centering
\includegraphics[width=0.9\linewidth]{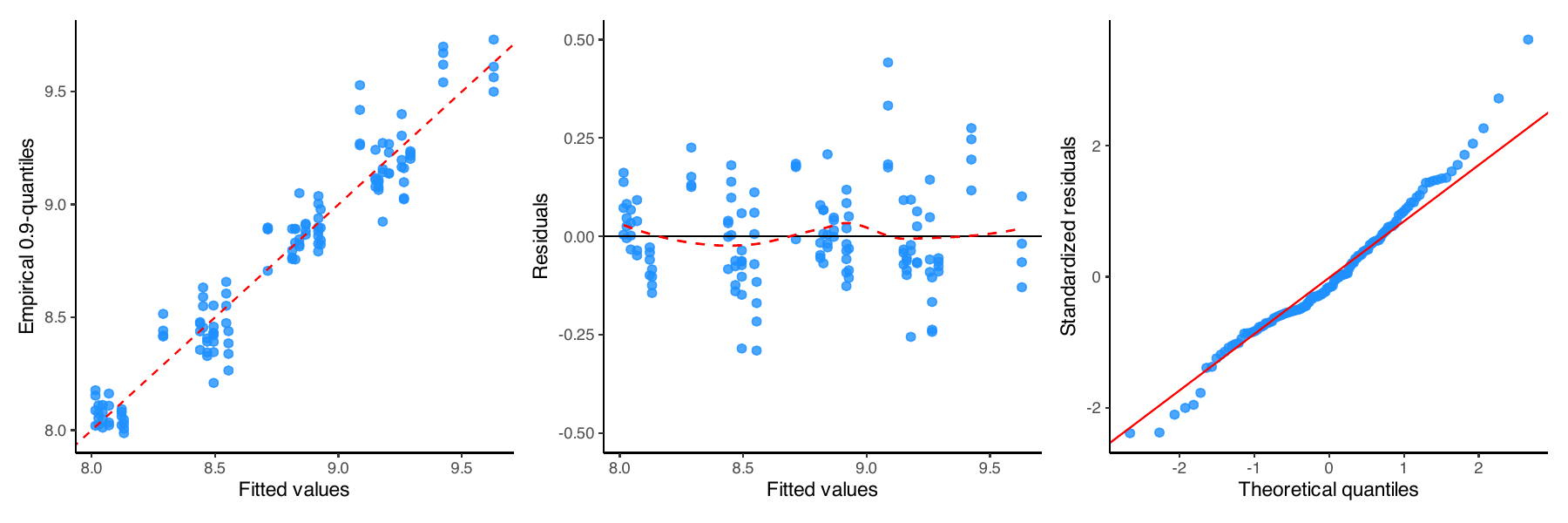}
\caption{Diagnostic plots for the model regressing the empirical 0.9-quantiles of the maximal detector statistics on the functions of $T$ and $d$. 
Empirical 0.9-quantiles vs.\ the fitted values (left), residuals vs.\ the fitted values (middle), and the normal Q-Q plot of the standardized residuals.}
\label{fig:det_reg}
\end{figure}


\subsection{Threshold \texorpdfstring{$\zeta_{T,p}$}{zeta(T,p)} for mode-identification}
\label{app: num_idt_oracle}

We calibrate the threshold for mode-identification through numerical experiments using the true change point locations; see Figures~\ref{fig: idt_trueCP_s1}--\ref{fig: idt_trueCP_s2}. Specifically, under~\ref{s:one} and~\ref{s:two}, we examine the values of $\wh\Xi_j^{(k)}/(T^{-1/2}+p^{-1})$ computed with the knowledge of the true change point locations (i.e.\ with $\wh\theta^\pm_j = \theta_j$), over varying $T$ and $(p_1, p_2, p_3)$. The empirical $0.99$-quantile of the pooled statistics corresponding to the modes for which the changes are not mode-identifiable, is approximately $3.773$ under~\ref{s:one} and $3.81$ under~\ref{s:two}. 
This motivates use to set $\zeta_{T,p} = 3.8 (T^{-1/2}+p^{-1}) $ and declare mode $k$ to be identified whenever $\Vert \wh\Xi_j^{(k)} \Vert \ge \zeta_{T,p}$.

\begin{figure}[h!t!b!]
\centering
\includegraphics[width=1\linewidth]{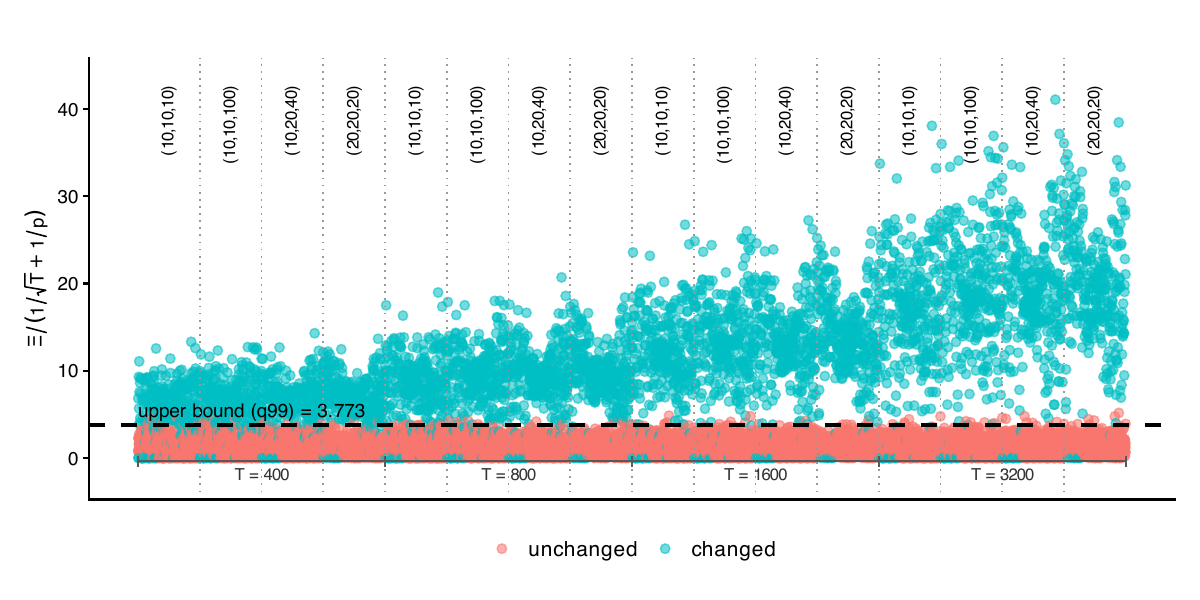}
\caption{\ref{s:one} Pooled mode-identification results using the true change point locations. For each true change point and each mode, the quantity $\Vert \wh\Xi_j^{(k)}\Vert / (T^{-1/2} + p^{-1})$ is plotted after pooling across all three change points, sample sizes $T \in \{400, 800, 1600, 3200\}$, and tensor dimensions $(p_1,p_2,p_3) \in \{(10,10,10), (10,10,100), (10,20,40), (20,20,20)\}$. Points are colored according to whether the corresponding mode $k$ is mode-identifiable at the change point. The dashed horizontal line denotes the empirical $0.99$-quantile of the pooled statistics from the modes for which the change is not mode-identifiable.}
\label{fig: idt_trueCP_s1}
\end{figure}

\begin{figure}[h!t!b!]
\centering
\includegraphics[width=1\linewidth]{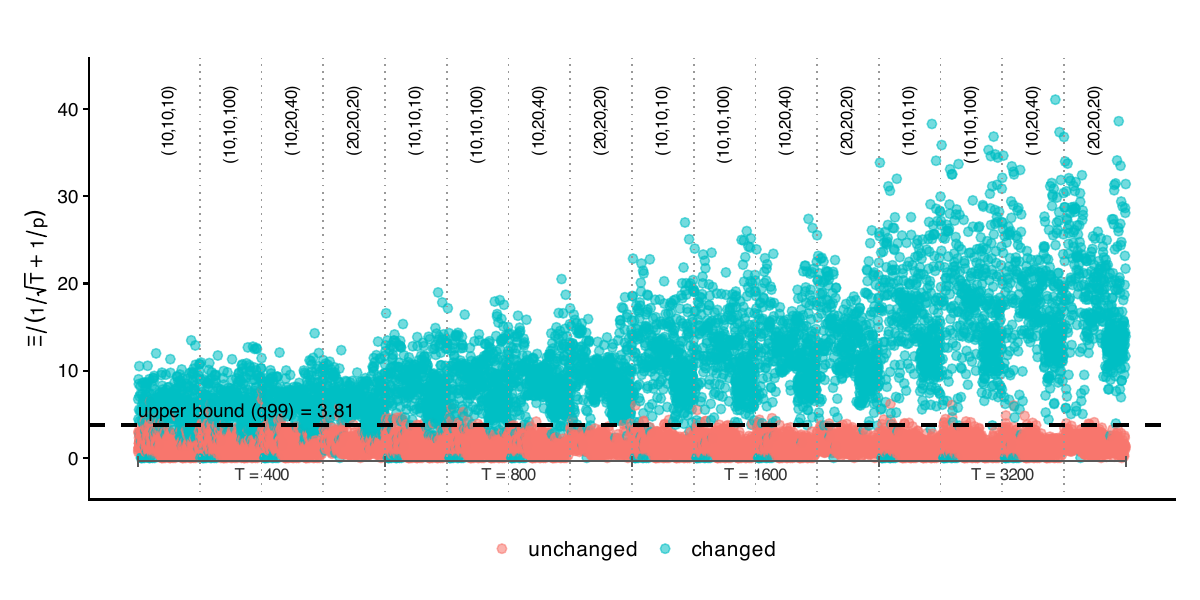}
\caption{\ref{s:two} Pooled mode-identification results using the true change point locations. For each true change point and each mode, the quantity $\Vert \wh\Xi_j^{(k)}\Vert / (T^{-1/2} + p^{-1})$ is plotted after pooling across all three change points, sample sizes $T \in \{400, 800, 1600, 3200\}$, and tensor dimensions $(p_1,p_2,p_3) \in \{(10,10,10), (10,10,100), (10,20,40), (20,20,20)\}$. Points are colored according to whether the corresponding mode $k$ is mode-identifiable at the change point. The dashed horizontal line denotes the empirical $0.99$-quantile of the pooled statistics from the modes for which the change is not mode-identifiable.}
\label{fig: idt_trueCP_s2}
\end{figure}

\clearpage

\section{Complete numerical experiments}\label{app: num}

\subsection{Simulation scenarios}\label{app: num_setting}

Throughout, we consider order-$3$ tensor time series $\{ \cX_t \}_{t = 1}^T$ generated under~\eqref{eqn: tfm_change}--\eqref{eqn: tfm_change_rewrite} with $(r_1,r_2,r_3) = (3,3,3)$.
Also, $\Lambda_{1, k}, \, k \in [3]$, have their entries generated independently from the uniform distribution $\cU(-1, 1)$.
The factor process is generated as
$\vec(\cF_t) =  \rho_f\,\vec(\cF_{t-1}) + \varepsilon_t$ with $\varepsilon_t \stackrel{\mathrm{i.i.d.}}{\sim}
\cN_r({0},\, (1- \rho_f^{2})I_r)$, 
and the vectorized idiosyncratic component follows $\vec(\cE_t) =  \rho_e\,\vec(\cE_{t-1}) + \eta_t$ with $\eta_t \stackrel{\mathrm{i.i.d.}}{\sim} \cN_r({0},\, (1- \rho_e^{2})\Sigma_e)$, where $\rho_e = 1-\rho_f$ and $(\Sigma_e)_{ij} = 0.3^{\vert i-j\vert}$, with $(\rho_f, \rho_e) \in \{(0, 0), (0.7, 0.3)\}$ in scenarios~\ref{s:null}--\ref{s:two} below, while we focus on the case where $(\rho_f, \rho_e) = (0, 0)$ under~\ref{s:three}.

\begin{enumerate}[label = (S\arabic*)]
\setcounter{enumi}{-1}
\item \label{s:null} \textbf{Null scenario.}
We consider the tensor factor model as in~\eqref{eqn: tfm_change} with $\Theta = \emptyset$, i.e.\ the loading matrix is time-invariant.

\item \label{s:one} \textbf{Single-mode change at each change point.}
Extending the data generating processes from \cite{Duanetal2023} and \cite{barigozzi2025moving} to the tensor setting, we consider order-$3$ tensor time series with $q = 3$ change points and the following transformation matrices (see~\eqref{eqn: tfm_change_rewrite}):
\begin{equation*}
A_{1,1}=\begin{bmatrix}
0.5&0&0\\
a_{1,21}&1 &0\\
a_{1,31}&a_{1,32}&1.5
\end{bmatrix},\quad A_{2,2}=\begin{bmatrix}
1&0&0\\
0&1&0\\
0&0&0
\end{bmatrix}, \text{\ and\ } A_{3,3}= \begin{bmatrix} a_{3, 11} & a_{3, 12} & a_{3, 13} \\
a_{3, 21} & a_{3, 22} & a_{3, 23} \\
a_{3, 31} & a_{3, 32} & a_{3, 33} 
\end{bmatrix},
\end{equation*}
where $a_{1,\ell h} \stackrel{\mathrm{i.i.d.}}{\sim} \cN(0,\, 1)$ for $(\ell, h) \in \{(2,1), (3,1), (3,2)\}$ and $a_{3,\ell h} \stackrel{\mathrm{i.i.d.}}{\sim} \cN(0,\, 1/3)$ for $\ell ,h \in[3]$. Then at each $j \in [3]$,  we set $\Lambda_{j + 1, k} = \Lambda_{j, k} A_{j, k}$ for $k = j$ and $\Lambda_{j + 1, k} = \Lambda_{j, k}, \, k \ne j$, i.e.\ $\cK_j = \{j\}$ for all $j \in [3]$.
We present the results when $\Theta = \{\lfloor 0.25T\rfloor,\lfloor 0.5T\rfloor,\lfloor 0.625T\rfloor\}$ with serial correlation $(\rho_f, \rho_e) = (0.7, 0.3)$ in the main text (Section~\ref{sec: num_detect}).

\item \label{s:two} \textbf{Multi-mode change at each change point.}
We modify the scenario in~\ref{s:one} by allowing multiple modes to change at $\theta_1$ and $\theta_3$, including mode-(un)identifiable ones. For this, we additionally set $\Lambda_{2,3} = \Lambda_{1,3} \cdot 3I_3$ and $\Lambda_{4, 2} = \Lambda_{3,2}A_{3,2}$ with $A_{3,2} =\diag(1, 0.6, 0.2)$, while all other loading matrices are generated as in~\ref{s:one}. 
Then, we have $\cK_1 = \{1\}, \, \cK_2 = \{2\}$ and $\cK_3 = \{2, 3\}$. 

\item \label{s:three} \textbf{Single change point detection.}
To ease the investigation of the downstream task of mode-wise loading space estimation after change point detection and mode-identification in Appendix~\ref{app: num_reest}, we consider the situation where there is a single change point. We apply the same rank-deficient transformation matrix as in~\ref{s:one} to the first mode only, so that $\Lambda_{2,1} = \Lambda_{1,1} A_{2,2}$ while $\Lambda_{2,k} = \Lambda_{1,k} I_3$ for $k\in \{2,3\}$, with $\cK_1 = \{1\}$. 
\end{enumerate}

With varying $T \in \{400, 800, 1600, 3200\}$, for~\ref{s:null}--\ref{s:two}, we consider dimensions $$(p_1,p_2,p_3) \in \{(10,10,10),(10,10,100),(10,20,40),(20,20,20)\},$$ while for~\ref{s:three}, we consider $$(p_1,p_2,p_3) \in \{(10,10,10),(10,10,100),(10,100,10),(100,10,10)\}.$$
We consider both the cases when the change points are equally ($\Theta = \{ \lfloor 0.25T \rfloor,\lfloor 0.5T \rfloor,\lfloor 0.75T \rfloor \}$) and unequally ($\Theta =\{ \lfloor 0.25T \rfloor,\lfloor 0.5T \rfloor,\lfloor 0.625T \rfloor \}$) spaced in~\ref{s:one} and~\ref{s:two}.
In~\ref{s:three}, we set $\Theta = \{ \lfloor 0.5T \rfloor \}$.

Throughout, we generate $N = 100$ realizations for each setting.

\subsection{Results for change point detection}\label{app: num_det}

For each $n \in [N]$, denote by $\wh{q}^{(n)}$ and $\wh{\theta}_j^{(n)}$ the number of change point estimators and their locations in the $n$-th replication. We report the distribution of $\wh{q}^{(n)} - q$, as well as the accuracy in estimating each $\theta_j$ measured as the proportion of replications over which it has an estimator within the distance $2\log(T)$, namely
\begin{align}
\text{Accuracy}_j = \frac{1}{N}\sum_{n=1}^N \mathbb{I}\left\{ \min_{1\le \ell \le \wh{q}^{(n)}}\Big|\widehat{\theta}_{\ell}^{(n)}-\theta_j \Big| \le 2\log(T) \right\} \nonumber,
\end{align}
following \cite{li2023detection}. This criterion is defined point-wise for each true change point and does not rely on the estimated number of change points being the same as the true number.
We additionally report the accuracy of the change point estimators conditional on the successful detection of all three change points, by considering the realizations where $\wh q^{(n)} = \wh q$, and $\wh\theta^{(n)}_j \in ((\theta_{j-1} + \theta_j)/2, (\theta_j + \theta_{j + 1})/2]$ for all $j \in [q]$; we denote the index set of those realizations by $\wh{\cD} \subset [N]$.
We also report the runtime in the format of ``average $
\pm$ standard deviation'' over $100$ realizations per configuration.
As in Section~\ref{sec: simulation}, we report the results from applying the binary segmentation procedure of \cite{Baietal2024}, referred to as \enquote{LR}, which recursively performs likelihood-ratio test, to the vectorized time series data.

\subsubsection{\ref{s:null} Null scenarios}\label{app: num_det_null}

Tables~\ref{tab: CPnull_T400}--\ref{tab: CPnull_T3200} report the summary of (spuriously) detected change points under the null scenario where there is no change point.
We observe that while its performance is influenced by the degree of serial correlations, generally TFMseg returns few false positives across all the settings in consideration. 

\begin{table}[h!t!b!]
\caption{\ref{s:null} Summary of change point estimators returned by TFMseg and LR when $T = 400$ with varying $(p_1,p_2,p_3)$, based on $100$ realizations.}
\label{tab: CPnull_T400}
\centering
\setlength{\tabcolsep}{4.5pt}
\begin{tabular}[t]{clcccccccc}
\toprule
\multicolumn{2}{c}{ } & \multicolumn{4}{c}{$(\rho_f, \rho_e) = (0.7, 0.3)$} & \multicolumn{4}{c}{$(\rho_f, \rho_e) = (0, 0)$} \\
\cmidrule(l{3pt}r{3pt}){3-6} \cmidrule(l{3pt}r{3pt}){7-10}
\multicolumn{2}{c}{ } & \multicolumn{3}{c}{$\wh q - q$} &  &\multicolumn{3}{c}{$\wh q - q$} \\
Dimensions & Method & $0$ & $1$ & $\geq 2$ & Runtime (s) & $0$ & $1$ & $\geq 2$ & Runtime (s)\\
\cmidrule(lr){1-2} \cmidrule(lr){3-5} \cmidrule(lr){6-6} \cmidrule(lr){7-9} \cmidrule(lr){10-10}
(10,10,10) & TFMseg & 0.87 & 0.11 & 0.02 & 1.21 $\pm$ 0.11 & 1 & 0 & 0 & 1.24 $\pm$ 0.10\\
 & LR & 0.98 & 0 & 0.02 & 50.32 $\pm$ 4.12 & 1 & 0 & 0 & 51.10 $\pm$ 3.79\\
(10,10,100) & TFMseg & 0.90 & 0.10 & 0 & 4.39 $\pm$ 0.56 & 1 & 0 & 0 & 4.60 $\pm$ 0.45\\
 & LR & 0.99 & 0 & 0.01 & 52.04 $\pm$ 3.36 & 1 & 0 & 0 & 54.74 $\pm$ 3.41\\
(10,20,40) & TFMseg & 0.86 & 0.14 & 0 & 3.58 $\pm$ 0.30 & 1 & 0 & 0 & 3.53 $\pm$ 0.18\\
 & LR & 1 & 0 & 0 & 50.96 $\pm$ 3.61 & 1 & 0 & 0 & 53.10 $\pm$ 3.19\\
(20,20,20) & TFMseg & 0.89 & 0.11 & 0 & 3.70 $\pm$ 0.29 & 1 & 0 & 0 & 3.74 $\pm$ 0.21\\
 & LR & 1 & 0 & 0 & 51.00 $\pm$ 3.65 & 0.99 & 0.01 & 0 & 52.60 $\pm$ 2.85\\
\bottomrule
\end{tabular}
\end{table}

\begin{table}[h!t!b!]
\caption{\ref{s:null} Summary of change point estimators returned by TFMseg and LR when $T = 800$ with varying $(p_1,p_2,p_3)$, based on $100$ realizations.}
\label{tab: CPnull_T800}
\centering
\setlength{\tabcolsep}{4.5pt}
\begin{tabular}[t]{clcccccccc}
\toprule
\multicolumn{2}{c}{ } & \multicolumn{4}{c}{$(\rho_f, \rho_e) = (0.7, 0.3)$} & \multicolumn{4}{c}{$(\rho_f, \rho_e) = (0, 0)$} \\
\cmidrule(l{3pt}r{3pt}){3-6} \cmidrule(l{3pt}r{3pt}){7-10}
\multicolumn{2}{c}{ } & \multicolumn{3}{c}{$\wh q - q$} &  &\multicolumn{3}{c}{$\wh q - q$} \\
Dimensions & Method & $0$ & $1$ & $\geq 2$ & Runtime (s) & $0$ & $1$ & $\geq 2$ & Runtime (s)\\
\cmidrule(lr){1-2} \cmidrule(lr){3-5} \cmidrule(lr){6-6} \cmidrule(lr){7-9} \cmidrule(lr){10-10}
(10,10,10) & TFMseg & 0.93 & 0.07 & 0 & 1.99 $\pm$ 0.17 & 1 & 0 & 0 & 2.05 $\pm$ 0.15\\
 & LR & 0.99 & 0.01 & 0 & 51.83 $\pm$ 4.07 & 1 & 0 & 0 & 53.39 $\pm$ 2.95\\
(10,10,100) & TFMseg & 0.94 & 0.06 & 0 & 7.72 $\pm$ 0.54 & 1 & 0 & 0 & 8.17 $\pm$ 0.55\\
 & LR & 0.97 & 0.01 & 0.02 & 57.37 $\pm$ 3.59 & 1 & 0 & 0 & 60.82 $\pm$ 5.11\\
(10,20,40) & TFMseg & 0.90 & 0.10 & 0 & 6.56 $\pm$ 0.78 & 1 & 0 & 0 & 6.72 $\pm$ 0.57\\
 & LR & 0.99 & 0 & 0.01 & 56.45 $\pm$ 5.28 & 1 & 0 & 0 & 58.38 $\pm$ 4.23\\
(20,20,20) & TFMseg & 0.93 & 0.07 & 0 & 6.65 $\pm$ 0.64 & 1 & 0 & 0 & 6.73 $\pm$ 0.43\\
 & LR & 0.98 & 0.01 & 0.01 & 55.94 $\pm$ 4.25 & 1 & 0 & 0 & 57.94 $\pm$ 3.69\\
\bottomrule
\end{tabular}
\end{table}

\begin{table}[h!t!b!]
\caption{\ref{s:null} Summary of change point estimators returned by TFMseg and LR when $T = 1600$ with varying $(p_1,p_2,p_3)$, based on $100$ realizations.}
\label{tab: CPnull_T1600} 
\centering
\setlength{\tabcolsep}{4.5pt}
\begin{tabular}[t]{clcccccccc}
\toprule
\multicolumn{2}{c}{ } & \multicolumn{4}{c}{$(\rho_f, \rho_e) = (0.7, 0.3)$} & \multicolumn{4}{c}{$(\rho_f, \rho_e) = (0, 0)$} \\
\cmidrule(l{3pt}r{3pt}){3-6} \cmidrule(l{3pt}r{3pt}){7-10}
\multicolumn{2}{c}{ } & \multicolumn{3}{c}{$\wh q - q$} &  &\multicolumn{3}{c}{$\wh q - q$} \\
Dimensions & Method & $0$ & $1$ & $\geq 2$ & Runtime (s) & $0$ & $1$ & $\geq 2$ & Runtime (s)\\
\cmidrule(lr){1-2} \cmidrule(lr){3-5} \cmidrule(lr){6-6} \cmidrule(lr){7-9} \cmidrule(lr){10-10}
(10,10,10) & TFMseg & 0.95 & 0.05 & 0 & 3.52 $\pm$ 0.28 & 1 & 0 & 0 & 3.54 $\pm$ 0.21\\
 & LR & 1 & 0 & 0 & 57.14 $\pm$ 4.10 & 1 & 0 & 0 & 58.62 $\pm$ 3.16\\
(10,10,100) & TFMseg & 0.97 & 0.03 & 0 & 15.27 $\pm$ 1.36 & 1 & 0 & 0 & 16.15 $\pm$ 1.20\\
 & LR & 0.99 & 0.01 & 0 & 79.93 $\pm$ 6.30 & 1 & 0 & 0 & 85.11 $\pm$ 7.59\\
(10,20,40) & TFMseg & 0.98 & 0.02 & 0 & 12.43 $\pm$ 0.78 & 1 & 0 & 0 & 13.08 $\pm$ 1.03\\
 & LR & 0.99 & 0.01 & 0 & 76.27 $\pm$ 5.79 & 1 & 0 & 0 & 81.98 $\pm$ 6.34\\
(20,20,20) & TFMseg & 0.95 & 0.05 & 0 & 12.76 $\pm$ 0.94 & 1 & 0 & 0 & 13.22 $\pm$ 1.15\\
 & LR & 0.99 & 0.01 & 0 & 77.25 $\pm$ 6.02 & 1 & 0 & 0 & 81.69 $\pm$ 8.84\\
\bottomrule
\end{tabular}
\end{table}

\begin{table}[h!t!b!]
\caption{\ref{s:null} Summary of change point estimators returned by TFMseg and LR when $T = 3200$ with varying $(p_1,p_2,p_3)$, based on $100$ realizations.}
\label{tab: CPnull_T3200} 
\centering
\setlength{\tabcolsep}{4.5pt}
\begin{tabular}[t]{clcccccccc}
\toprule
\multicolumn{2}{c}{ } & \multicolumn{4}{c}{$(\rho_f, \rho_e) = (0.7, 0.3)$} & \multicolumn{4}{c}{$(\rho_f, \rho_e) = (0, 0)$} \\
\cmidrule(l{3pt}r{3pt}){3-6} \cmidrule(l{3pt}r{3pt}){7-10}
\multicolumn{2}{c}{ } & \multicolumn{3}{c}{$\wh q - q$} &  &\multicolumn{3}{c}{$\wh q - q$} \\
Dimensions & Method & $0$ & $1$ & $\geq 2$ & Runtime (s) & $0$ & $1$ & $\geq 2$ & Runtime (s)\\
\cmidrule(lr){1-2} \cmidrule(lr){3-5} \cmidrule(lr){6-6} \cmidrule(lr){7-9} \cmidrule(lr){10-10}
(10,10,10) & TFMseg & 0.94 & 0.06 & 0 & 5.91 $\pm$ 0.46 & 0.99 & 0.01 & 0 & 6.14 $\pm$ 0.39\\
 & LR & 1 & 0 & 0 & 77.48 $\pm$ 6.03 & 1 & 0 & 0 & 79.89 $\pm$ 5.00\\
(10,10,100) & TFMseg & 0.96 & 0.04 & 0 & 31.33 $\pm$ 3.97 & 1 & 0 & 0 & 32.05 $\pm$ 2.57\\
 & LR & 0.99 & 0.01 & 0 & 204.50 $\pm$ 17.89 & 0.99 & 0.01 & 0 & 218.28 $\pm$ 17.33\\
(10,20,40) & TFMseg & 0.93 & 0.07 & 0 & 24.09 $\pm$ 1.58 & 1 & 0 & 0 & 25.27 $\pm$ 1.72\\
 & LR & 1 & 0 & 0 & 197.05 $\pm$ 17.36 & 1 & 0 & 0 & 219.50 $\pm$ 18.25\\
(20,20,20) & TFMseg & 0.91 & 0.09 & 0 & 25.05 $\pm$ 2.27 & 1 & 0 & 0 & 25.88 $\pm$ 1.95\\
 & LR & 1 & 0 & 0 & 200.44 $\pm$ 18.39 & 1 & 0 & 0 & 214.03 $\pm$ 18.78\\
\bottomrule
\end{tabular}
\end{table}

\clearpage 

\subsubsection{\ref{s:one} Single-mode change with equal-spaced change points}\label{app: num_det_bal}

We consider when the change points are equally spaced with $\Theta = \big\{\lfloor 0.25T \rfloor, \lfloor 0.5T \rfloor, \lfloor 0.75T \rfloor\big\}$ under~\ref{s:one}, see Tables~\ref{tab: bal_rnull_T400}--\ref{tab: bal_rnull_T3200} and Figure~\ref{fig: subplot_rnull} for the case of serially dependent data with $(\rho_f, \rho_e) = (0.7, 0.3)$, and Tables~\ref{tab: bal_rnull_indep_T400}--\ref{tab: bal_rnull_indep_T3200} and Figure~\ref{fig: subplot_rnull_indep} for the case when the data are serially uncorrelated.

Figure~\ref{fig:runtime} compares the runtime of TFMseg and LR. To facilitate visual comparison across settings with substantially different computational costs, the $y$-axis is displayed on a logarithmic scale, where it is evident that the runtime of TFMseg is a fraction of the runtime of LR.


\begin{table}[h!t!b!]
\caption{\ref{s:one} Summary of change point estimators returned by TFMseg, oraclePZW and LR when $\Theta  = \big\{\lfloor 0.25T \rfloor, \lfloor 0.5T \rfloor, \lfloor 0.75T \rfloor\big\}$, $(\rho_f, \rho_e) = (0.7, 0.3)$,  $T=400$ and varying $(p_1,p_2,p_3)$, based on 100 realizations.}
\label{tab: bal_rnull_T400} 
\centering
\setlength{\tabcolsep}{4.5pt}
\begin{tabular}[t]{clccccccccc}
\toprule
\multicolumn{2}{c}{ } & \multicolumn{5}{c}{$\wh{q}-q$} & \multicolumn{3}{c}{Accuracy} & \multicolumn{1}{c}{ } \\
\cmidrule(lr){3-7} \cmidrule(lr){8-10} 
Dimensions & Method & $\leq -2$ & $-1$ & $0$ & $1$ & $\geq 2$ & $j=1$ & $j=2$ & $j=3$ & Runtime (s)\\
\cmidrule(lr){1-2} \cmidrule(lr){3-7} \cmidrule(lr){8-10} \cmidrule(lr){11-11} 
(10,10,10) & TFMseg & 0.09 & 0.40 & 0.51 & 0 & 0 & 0.91 & 0.62 & 0.58 & 1.28 $\pm$ 0.15\\
 & oraclePZW & 0 & 0 & 1 & 0 & 0 & 0.38 & 0.27 & 0.32 & 1.43 $\pm$ 0.26\\
 & LR & 0.69 & 0.29 & 0.02 & 0 & 0 & 0.15 & 0.76 & 0.39 & 52.24 $\pm$ 5.19\\
(10,10,100) & TFMseg & 0.02 & 0.38 & 0.60 & 0 & 0 & 0.91 & 0.73 & 0.65 & 4.49 $\pm$ 0.33\\
 & oraclePZW & 0 & 0 & 1 & 0 & 0 & 0.42 & 0.28 & 0.25 & 4.68 $\pm$ 0.33\\
 & LR & 0.65 & 0.28 & 0.06 & 0 & 0.01 & 0.26 & 0.58 & 0.53 & 53.51 $\pm$ 3.04\\
(10,20,40) & TFMseg & 0.05 & 0.31 & 0.60 & 0.04 & 0 & 0.94 & 0.65 & 0.66 & 3.61 $\pm$ 0.32\\
 & oraclePZW & 0 & 0 & 1 & 0 & 0 & 0.32 & 0.18 & 0.32 & 3.80 $\pm$ 0.29\\
 & LR & 0.67 & 0.28 & 0.05 & 0 & 0 & 0.16 & 0.76 & 0.41 & 53.37 $\pm$ 3.56\\
(20,20,20) & TFMseg & 0.02 & 0.39 & 0.58 & 0.01 & 0 & 0.92 & 0.68 & 0.68 & 4.03 $\pm$ 0.58\\
 & oraclePZW & 0 & 0 & 1 & 0 & 0 & 0.50 & 0.17 & 0.25 & 4.08 $\pm$ 0.48\\
 & LR & 0.66 & 0.33 & 0.01 & 0 & 0 & 0.17 & 0.76 & 0.40 & 53.30 $\pm$ 3.24\\
\bottomrule
\end{tabular}
\end{table}

\begin{table}[h!t!b!]
\caption{\ref{s:one} Summary of change point estimators returned by TFMseg, oraclePZW and LR when $\Theta  = \big\{\lfloor 0.25T \rfloor, \lfloor 0.5T \rfloor, \lfloor 0.75T \rfloor\big\}$, $(\rho_f, \rho_e) = (0.7, 0.3)$,  $T=800$ and varying $(p_1,p_2,p_3)$, based on 100 realizations.}
\label{tab: bal_rnull_T800}
\centering
\setlength{\tabcolsep}{4.5pt}
\begin{tabular}[t]{clccccccccc}
\toprule
\multicolumn{2}{c}{ } & \multicolumn{5}{c}{$\wh{q}-q$} & \multicolumn{3}{c}{Accuracy} & \multicolumn{1}{c}{ } \\
\cmidrule(lr){3-7} \cmidrule(lr){8-10} 
Dimensions & Method & $\leq -2$ & $-1$ & $0$ & $1$ & $\geq 2$ & $j=1$ & $j=2$ & $j=3$ & Runtime (s)\\
\cmidrule(lr){1-2} \cmidrule(lr){3-7} \cmidrule(lr){8-10} \cmidrule(lr){11-11} 
(10,10,10) & TFMseg & 0.02 & 0.15 & 0.78 & 0.05 & 0 & 0.94 & 0.83 & 0.72 & 2.15 $\pm$ 0.29\\
 & oraclePZW & 0 & 0 & 1 & 0 & 0 & 0.60 & 0.45 & 0.50 & 2.28 $\pm$ 0.35\\
 & LR & 0.22 & 0.57 & 0.21 & 0 & 0 & 0.55 & 0.94 & 0.50 & 54.39 $\pm$ 4.33\\
(10,10,100) & TFMseg & 0 & 0.11 & 0.83 & 0.06 & 0 & 0.94 & 0.86 & 0.78 & 7.90 $\pm$ 0.46\\
 & oraclePZW & 0 & 0 & 1 & 0 & 0 & 0.57 & 0.40 & 0.45 & 8.14 $\pm$ 0.66\\
 & LR & 0.30 & 0.50 & 0.20 & 0 & 0 & 0.54 & 0.77 & 0.59 & 60.72 $\pm$ 3.70\\
(10,20,40) & TFMseg & 0 & 0.10 & 0.85 & 0.05 & 0 & 0.95 & 0.93 & 0.80 & 6.63 $\pm$ 0.74\\
 & oraclePZW & 0 & 0 & 1 & 0 & 0 & 0.61 & 0.34 & 0.50 & 6.87 $\pm$ 0.63\\
 & LR & 0.22 & 0.51 & 0.26 & 0 & 0.01 & 0.56 & 0.88 & 0.61 & 59.36 $\pm$ 3.89\\
(20,20,20) & TFMseg & 0 & 0.10 & 0.85 & 0.04 & 0.01 & 0.98 & 0.88 & 0.82 & 7.24 $\pm$ 1.08\\
 & oraclePZW & 0 & 0 & 1 & 0 & 0 & 0.75 & 0.29 & 0.47 & 7.31 $\pm$ 0.58\\
 & LR & 0.11 & 0.62 & 0.27 & 0 & 0 & 0.61 & 0.91 & 0.63 & 60.21 $\pm$ 5.33\\
\bottomrule
\end{tabular}
\end{table}

\begin{table}[h!t!b!]
\caption{\ref{s:one} Summary of change point estimators returned by TFMseg, oraclePZW and LR when $\Theta  = \big\{\lfloor 0.25T \rfloor, \lfloor 0.5T \rfloor, \lfloor 0.75T \rfloor\big\}$, $(\rho_f, \rho_e) = (0.7, 0.3)$,  $T=1600$ and varying $(p_1,p_2,p_3)$, based on 100 realizations.}
\label{tab: bal_rnull_T1600}
\centering
\setlength{\tabcolsep}{4.5pt}
\begin{tabular}[t]{clccccccccc}
\toprule
\multicolumn{2}{c}{ } & \multicolumn{5}{c}{$\wh{q}-q$} & \multicolumn{3}{c}{Accuracy} & \multicolumn{1}{c}{ } \\
\cmidrule(lr){3-7} \cmidrule(lr){8-10} 
Dimensions & Method & $\leq -2$ & $-1$ & $0$ & $1$ & $\geq 2$ & $j=1$ & $j=2$ & $j=3$ & Runtime (s)\\
\cmidrule(lr){1-2} \cmidrule(lr){3-7} \cmidrule(lr){8-10} \cmidrule(lr){11-11} 
(10,10,10) & TFMseg & 0 & 0.05 & 0.87 & 0.08 & 0 & 0.96 & 0.79 & 0.75 & 3.73 $\pm$ 0.38\\
 & oraclePZW & 0 & 0 & 1 & 0 & 0 & 0.79 & 0.75 & 0.75 & 4.04 $\pm$ 0.57\\
 & LR & 0.05 & 0.41 & 0.54 & 0 & 0 & 0.83 & 0.93 & 0.71 & 61.09 $\pm$ 5.94\\
(10,10,100) & TFMseg & 0 & 0.03 & 0.84 & 0.12 & 0.01 & 1 & 0.91 & 0.80 & 15.52 $\pm$ 0.89\\
 & oraclePZW & 0 & 0 & 1 & 0 & 0 & 0.87 & 0.72 & 0.75 & 15.80 $\pm$ 1.07\\
 & LR & 0.09 & 0.45 & 0.46 & 0 & 0 & 0.83 & 0.82 & 0.72 & 85.10 $\pm$ 6.63\\
(10,20,40) & TFMseg & 0 & 0.02 & 0.89 & 0.09 & 0 & 0.97 & 0.87 & 0.82 & 12.81 $\pm$ 0.81\\
 & oraclePZW & 0 & 0 & 1 & 0 & 0 & 0.88 & 0.79 & 0.77 & 13.58 $\pm$ 1.03\\
 & LR & 0.06 & 0.47 & 0.47 & 0 & 0 & 0.78 & 0.92 & 0.71 & 82.56 $\pm$ 6.37\\
(20,20,20) & TFMseg & 0 & 0.03 & 0.89 & 0.07 & 0.01 & 0.99 & 0.90 & 0.85 & 13.86 $\pm$ 1.76\\
 & oraclePZW & 0 & 0 & 1 & 0 & 0 & 0.99 & 0.74 & 0.83 & 14.15 $\pm$ 1.13\\
 & LR & 0.01 & 0.45 & 0.54 & 0 & 0 & 0.83 & 0.94 & 0.76 & 83.32 $\pm$ 9.30\\
\bottomrule
\end{tabular}
\end{table}

\begin{table}[h!t!b!]
\caption{\ref{s:one} Summary of change point estimators returned by TFMseg, oraclePZW and LR when $\Theta  = \big\{\lfloor 0.25T \rfloor, \lfloor 0.5T \rfloor, \lfloor 0.75T \rfloor\big\}$, $(\rho_f, \rho_e) = (0.7, 0.3)$,  $T=3200$ and varying $(p_1,p_2,p_3)$, based on 100 realizations.}
\label{tab: bal_rnull_T3200}
\centering
\setlength{\tabcolsep}{4.5pt}
\begin{tabular}[t]{clccccccccc}
\toprule
\multicolumn{2}{c}{ } & \multicolumn{5}{c}{$\wh{q}-q$} & \multicolumn{3}{c}{Accuracy} & \multicolumn{1}{c}{ } \\
\cmidrule(lr){3-7} \cmidrule(lr){8-10} 
Dimensions & Method & $\leq -2$ & $-1$ & $0$ & $1$ & $\geq 2$ & $j=1$ & $j=2$ & $j=3$ & Runtime (s)\\
\cmidrule(lr){1-2} \cmidrule(lr){3-7} \cmidrule(lr){8-10} \cmidrule(lr){11-11} 
(10,10,10) & TFMseg & 0 & 0.02 & 0.85 & 0.11 & 0.02 & 1 & 0.95 & 0.90 & 6.32 $\pm$ 0.65\\
 & oraclePZW & 0 & 0 & 1 & 0 & 0 & 0.96 & 0.90 & 0.89 & 7.16 $\pm$ 1.40\\
 & LR & 0 & 0.26 & 0.74 & 0 & 0 & 0.94 & 0.97 & 0.83 & 95.08 $\pm$ 14.65\\
(10,10,100) & TFMseg & 0 & 0 & 0.86 & 0.14 & 0 & 1 & 0.97 & 0.91 & 30.90 $\pm$ 1.98\\
 & oraclePZW & 0 & 0 & 1 & 0 & 0 & 0.95 & 0.94 & 0.89 & 32.30 $\pm$ 2.80\\
 & LR & 0.02 & 0.30 & 0.68 & 0 & 0 & 0.91 & 0.85 & 0.87 & 229.23 $\pm$ 19.68\\
(10,20,40) & TFMseg & 0 & 0.01 & 0.80 & 0.19 & 0 & 1 & 1 & 0.90 & 25.03 $\pm$ 1.85\\
 & oraclePZW & 0 & 0 & 1 & 0 & 0 & 0.98 & 0.95 & 0.90 & 26.09 $\pm$ 2.17\\
 & LR & 0 & 0.29 & 0.71 & 0 & 0 & 0.93 & 0.93 & 0.85 & 220.84 $\pm$ 19.62\\
(20,20,20) & TFMseg & 0 & 0.01 & 0.88 & 0.08 & 0.03 & 1 & 0.98 & 0.92 & 25.98 $\pm$ 2.76\\
 & oraclePZW & 0 & 0 & 1 & 0 & 0 & 1 & 0.98 & 0.93 & 27.61 $\pm$ 2.70\\
 & LR & 0.01 & 0.25 & 0.74 & 0 & 0 & 0.94 & 0.94 & 0.84 & 222.27 $\pm$ 17.98\\
\bottomrule
\end{tabular}
\end{table}

\begin{figure}[h!t!b!p!]
\centering
\includegraphics[width=0.85\linewidth]{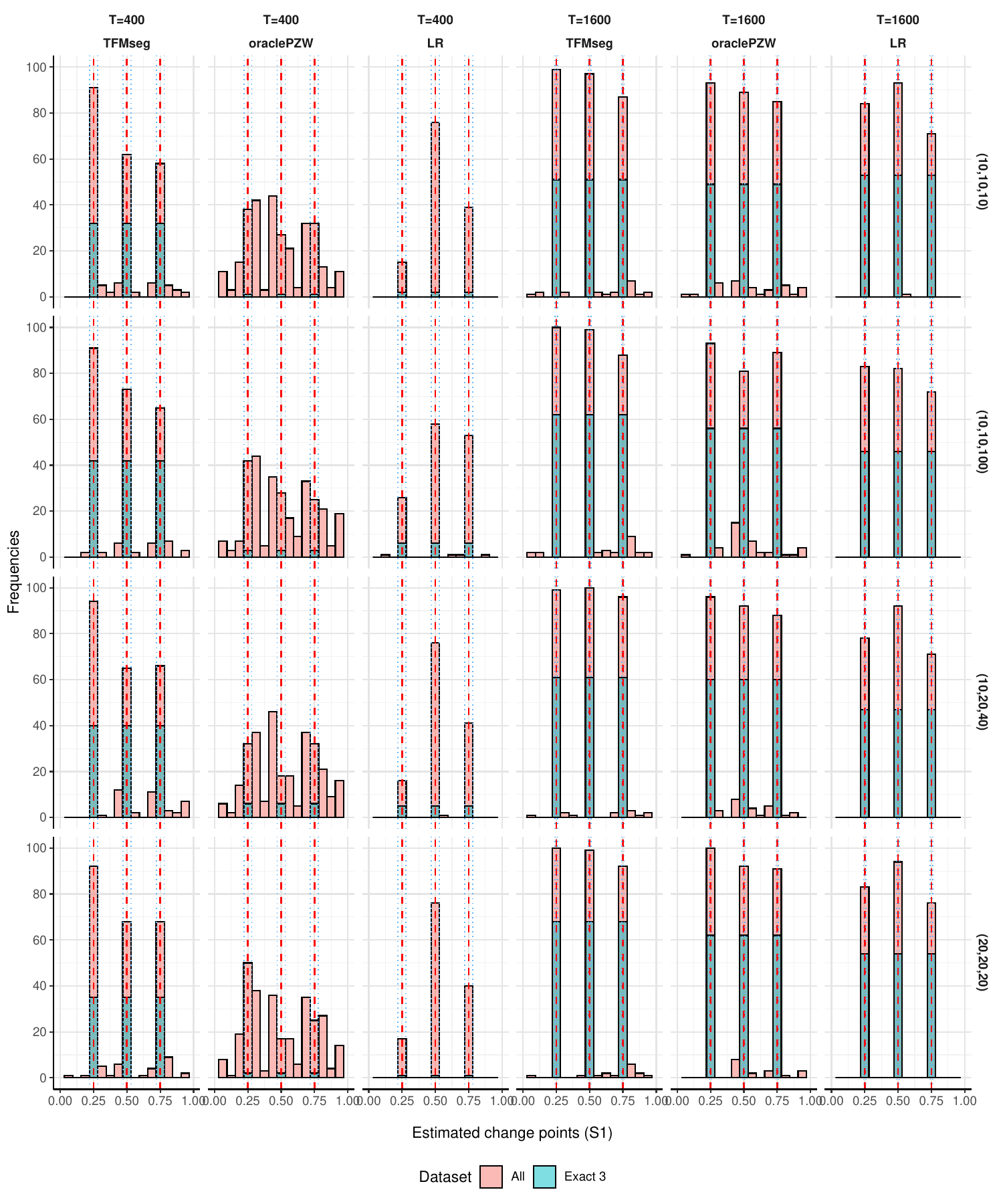}
\caption{\ref{s:one} Barplots of the scaled change point estimators $\{ \wh\theta^{(n)}_j/T, \, j \in [\wh q^{(n)}]\}$ returned by TFMseg, oraclePZW and LR  when $\Theta = \big\{\lfloor 0.25T \rfloor, \lfloor 0.5T \rfloor, \lfloor 0.75T \rfloor\big\}$, $(\rho_f, \rho_e) = (0.7, 0.3)$, $T \in \{400, 1600\}$ and varying $(p_1,p_2,p_3)$ (top to bottom) over $100$ realizations. The red bars give the total frequency of estimated change points for all $n \in [N]$, while the blue bars give the frequency from the subset of realizations $n \in \wh{\cD}$.}
\label{fig: subplot_rnull}
\end{figure}

\clearpage

\begin{table}[h!t!b!p!]
\caption{\ref{s:one} Summary of change point estimators returned by TFMseg, oraclePZW and LR when $\Theta  = \big\{\lfloor 0.25T \rfloor, \lfloor 0.5T \rfloor, \lfloor 0.75T \rfloor\big\}$, $(\rho_f, \rho_e) = (0, 0)$,  $T=400$ and varying $(p_1,p_2,p_3)$, based on 100 realizations.}
\label{tab: bal_rnull_indep_T400}
\centering
\setlength{\tabcolsep}{4.5pt}
\begin{tabular}[t]{clccccccccc}
\toprule
\multicolumn{2}{c}{ } & \multicolumn{5}{c}{$\wh{q}-q$} & \multicolumn{3}{c}{Accuracy} & \multicolumn{1}{c}{ } \\
\cmidrule(lr){3-7} \cmidrule(lr){8-10} 
Dimensions & Method & $\leq -2$ & $-1$ & $0$ & $1$ & $\geq 2$ & $j=1$ & $j=2$ & $j=3$ & Runtime (s)\\
\cmidrule(lr){1-2} \cmidrule(lr){3-7} \cmidrule(lr){8-10} \cmidrule(lr){11-11} 
(10,10,10) & TFMseg & 0.04 & 0.33 & 0.63 & 0 & 0 & 0.96 & 0.80 & 0.76 & 1.20 $\pm$ 0.06\\
 & oraclePZW & 0 & 0 & 1 & 0 & 0 & 0.84 & 0.74 & 0.74 & 1.38 $\pm$ 0.22\\
 & LR & 0.33 & 0.49 & 0.18 & 0 & 0 & 0.48 & 0.86 & 0.51 & 50.89 $\pm$ 5.01\\
(10,10,100) & TFMseg & 0.02 & 0.21 & 0.77 & 0 & 0 & 0.97 & 0.89 & 0.82 & 4.72 $\pm$ 0.52\\
 & oraclePZW & 0 & 0 & 1 & 0 & 0 & 0.85 & 0.69 & 0.77 & 4.81 $\pm$ 0.53\\
 & LR & 0.37 & 0.48 & 0.15 & 0 & 0 & 0.43 & 0.75 & 0.59 & 53.52 $\pm$ 4.18\\
(10,20,40) & TFMseg & 0.01 & 0.29 & 0.70 & 0 & 0 & 0.97 & 0.84 & 0.84 & 3.60 $\pm$ 0.26\\
 & oraclePZW & 0 & 0 & 1 & 0 & 0 & 0.87 & 0.61 & 0.68 & 3.89 $\pm$ 0.41\\
 & LR & 0.26 & 0.52 & 0.22 & 0 & 0 & 0.45 & 0.88 & 0.61 & 53.34 $\pm$ 4.14\\
(20,20,20) & TFMseg & 0 & 0.24 & 0.75 & 0.01 & 0 & 0.99 & 0.89 & 0.85 & 3.74 $\pm$ 0.22\\
 & oraclePZW & 0 & 0 & 1 & 0 & 0 & 0.94 & 0.52 & 0.67 & 4.02 $\pm$ 0.31\\
 & LR & 0.24 & 0.59 & 0.17 & 0 & 0 & 0.47 & 0.89 & 0.57 & 52.35 $\pm$ 2.74\\
\bottomrule
\end{tabular}
\end{table}

\begin{table}[h!t!b!p!]
\caption{\ref{s:one} Summary of change point estimators returned by TFMseg, oraclePZW and LR when $\Theta  = \big\{\lfloor 0.25T \rfloor, \lfloor 0.5T \rfloor, \lfloor 0.75T \rfloor\big\}$, $(\rho_f, \rho_e) = (0, 0)$,  $T=800$ and varying $(p_1,p_2,p_3)$, based on 100 realizations.}
\label{tab: bal_rnull_indep_T800}
\centering
\setlength{\tabcolsep}{4.5pt}
\begin{tabular}[t]{clccccccccc}
\toprule
\multicolumn{2}{c}{ } & \multicolumn{5}{c}{$\wh{q}-q$} & \multicolumn{3}{c}{Accuracy} & \multicolumn{1}{c}{ } \\
\cmidrule(lr){3-7} \cmidrule(lr){8-10} 
Dimensions & Method & $\leq -2$ & $-1$ & $0$ & $1$ & $\geq 2$ & $j=1$ & $j=2$ & $j=3$ & Runtime (s)\\
\cmidrule(lr){1-2} \cmidrule(lr){3-7} \cmidrule(lr){8-10} \cmidrule(lr){11-11} 
(10,10,10) & TFMseg & 0 & 0.12 & 0.87 & 0.01 & 0 & 0.99 & 0.94 & 0.88 & 2.01 $\pm$ 0.13\\
 & oraclePZW & 0 & 0 & 1 & 0 & 0 & 0.97 & 0.92 & 0.93 & 2.20 $\pm$ 0.28\\
 & LR & 0.11 & 0.49 & 0.40 & 0 & 0 & 0.71 & 0.93 & 0.65 & 52.15 $\pm$ 2.97\\
(10,10,100) & TFMseg & 0 & 0.06 & 0.94 & 0 & 0 & 0.99 & 0.96 & 0.94 & 8.10 $\pm$ 0.86\\
 & oraclePZW & 0 & 0 & 1 & 0 & 0 & 0.95 & 0.91 & 0.90 & 8.21 $\pm$ 0.78\\
 & LR & 0.14 & 0.48 & 0.38 & 0 & 0 & 0.72 & 0.84 & 0.68 & 61.10 $\pm$ 4.39\\
(10,20,40) & TFMseg & 0 & 0.03 & 0.96 & 0.01 & 0 & 1 & 0.99 & 0.93 & 6.75 $\pm$ 0.68\\
 & oraclePZW & 0 & 0 & 1 & 0 & 0 & 0.97 & 0.93 & 0.88 & 7.03 $\pm$ 0.70\\
 & LR & 0.08 & 0.49 & 0.43 & 0 & 0 & 0.75 & 0.92 & 0.68 & 59.37 $\pm$ 4.37\\
(20,20,20) & TFMseg & 0 & 0.03 & 0.95 & 0.02 & 0 & 1 & 0.99 & 0.92 & 6.87 $\pm$ 0.56\\
 & oraclePZW & 0 & 0 & 1 & 0 & 0 & 0.99 & 0.92 & 0.86 & 7.22 $\pm$ 0.80\\
 & LR & 0.05 & 0.54 & 0.41 & 0 & 0 & 0.74 & 0.93 & 0.69 & 58.09 $\pm$ 2.94\\
\bottomrule
\end{tabular}
\end{table}

\begin{table}[h!t!b!p!]
\caption{\ref{s:one} Summary of change point estimators returned by TFMseg, oraclePZW and LR when $\Theta  = \big\{\lfloor 0.25T \rfloor, \lfloor 0.5T \rfloor, \lfloor 0.75T \rfloor\big\}$, $(\rho_f, \rho_e) = (0, 0)$,  $T=1600$ and varying $(p_1,p_2,p_3)$, based on 100 realizations.}
\label{tab: bal_rnull_indep_T1600}
\centering
\setlength{\tabcolsep}{4.5pt}
\begin{tabular}[t]{clccccccccc}
\toprule
\multicolumn{2}{c}{ } & \multicolumn{5}{c}{$\wh{q}-q$} & \multicolumn{3}{c}{Accuracy} & \multicolumn{1}{c}{ } \\
\cmidrule(lr){3-7} \cmidrule(lr){8-10} 
Dimensions & Method & $\leq -2$ & $-1$ & $0$ & $1$ & $\geq 2$ & $j=1$ & $j=2$ & $j=3$ & Runtime (s)\\
\cmidrule(lr){1-2} \cmidrule(lr){3-7} \cmidrule(lr){8-10} \cmidrule(lr){11-11} 
(10,10,10) & TFMseg & 0 & 0.03 & 0.96 & 0.01 & 0 & 1 & 0.96 & 0.92 & 3.50 $\pm$ 0.17\\
 & oraclePZW & 0 & 0 & 1 & 0 & 0 & 0.98 & 0.97 & 0.96 & 3.85 $\pm$ 0.43\\
 & LR & 0.02 & 0.38 & 0.60 & 0 & 0 & 0.84 & 0.95 & 0.79 & 58.59 $\pm$ 3.25\\
(10,10,100) & TFMseg & 0 & 0.02 & 0.97 & 0.01 & 0 & 0.99 & 0.98 & 0.97 & 15.83 $\pm$ 1.70\\
 & oraclePZW & 0 & 0 & 1 & 0 & 0 & 0.99 & 0.98 & 0.98 & 15.71 $\pm$ 1.31\\
 & LR & 0.05 & 0.38 & 0.57 & 0 & 0 & 0.86 & 0.86 & 0.80 & 89.68 $\pm$ 7.46\\
(10,20,40) & TFMseg & 0 & 0.01 & 0.97 & 0.02 & 0 & 0.99 & 1 & 0.98 & 13.35 $\pm$ 1.34\\
 & oraclePZW & 0 & 0 & 1 & 0 & 0 & 0.98 & 1 & 0.97 & 13.75 $\pm$ 1.41\\
 & LR & 0.03 & 0.42 & 0.55 & 0 & 0 & 0.85 & 0.93 & 0.74 & 84.12 $\pm$ 6.87\\
(20,20,20) & TFMseg & 0 & 0.02 & 0.94 & 0.04 & 0 & 1 & 0.99 & 0.97 & 13.35 $\pm$ 1.07\\
 & oraclePZW & 0 & 0 & 1 & 0 & 0 & 1 & 1 & 0.95 & 14.05 $\pm$ 1.27\\
 & LR & 0.02 & 0.36 & 0.62 & 0 & 0 & 0.88 & 0.93 & 0.79 & 83.91 $\pm$ 5.28\\
\bottomrule
\end{tabular}
\end{table}

\begin{table}[h!t!b!p!]
\caption{\ref{s:one} Summary of change point estimators returned by TFMseg, oraclePZW and LR when $\Theta  = \big\{\lfloor 0.25T \rfloor, \lfloor 0.5T \rfloor, \lfloor 0.75T \rfloor\big\}$, $(\rho_f, \rho_e) = (0, 0)$,  $T=3200$ and varying $(p_1,p_2,p_3)$, based on 100 realizations.}
\label{tab: bal_rnull_indep_T3200}
\centering
\setlength{\tabcolsep}{4.5pt}
\begin{tabular}[t]{clccccccccc}
\toprule
\multicolumn{2}{c}{ } & \multicolumn{5}{c}{$\wh{q}-q$} & \multicolumn{3}{c}{Accuracy} & \multicolumn{1}{c}{ } \\
\cmidrule(lr){3-7} \cmidrule(lr){8-10} 
Dimensions & Method & $\leq -2$ & $-1$ & $0$ & $1$ & $\geq 2$ & $j=1$ & $j=2$ & $j=3$ & Runtime (s)\\
\cmidrule(lr){1-2} \cmidrule(lr){3-7} \cmidrule(lr){8-10} \cmidrule(lr){11-11} 
(10,10,10) & TFMseg & 0 & 0.03 & 0.95 & 0.02 & 0 & 1 & 0.97 & 0.95 & 5.93 $\pm$ 0.32\\
 & oraclePZW & 0 & 0 & 1 & 0 & 0 & 1 & 0.97 & 0.98 & 6.89 $\pm$ 1.01\\
 & LR & 0.01 & 0.23 & 0.76 & 0 & 0 & 0.97 & 0.95 & 0.83 & 89.85 $\pm$ 7.11\\
(10,10,100) & TFMseg & 0 & 0.01 & 0.98 & 0.01 & 0 & 1 & 0.98 & 0.98 & 32.66 $\pm$ 4.26\\
 & oraclePZW & 0 & 0 & 1 & 0 & 0 & 1 & 0.98 & 0.99 & 33.95 $\pm$ 4.74\\
 & LR & 0.04 & 0.28 & 0.68 & 0 & 0 & 0.93 & 0.86 & 0.85 & 236.01 $\pm$ 28.22\\
(10,20,40) & TFMseg & 0 & 0 & 0.97 & 0.03 & 0 & 1 & 1 & 0.98 & 25.56 $\pm$ 2.87\\
 & oraclePZW & 0 & 0 & 1 & 0 & 0 & 1 & 1 & 0.99 & 26.70 $\pm$ 3.03\\
 & LR & 0 & 0.29 & 0.71 & 0 & 0 & 0.93 & 0.94 & 0.84 & 224.99 $\pm$ 23.54\\
(20,20,20) & TFMseg & 0 & 0 & 0.99 & 0.01 & 0 & 1 & 1 & 0.97 & 26.72 $\pm$ 3.35\\
 & oraclePZW & 0 & 0 & 1 & 0 & 0 & 1 & 1 & 0.96 & 28.61 $\pm$ 3.11\\
 & LR & 0.01 & 0.20 & 0.79 & 0 & 0 & 0.97 & 0.95 & 0.86 & 238.00 $\pm$ 23.41\\
\bottomrule
\end{tabular}
\end{table}

\clearpage
\begin{figure}[h!t!b!p!]
\centering
\includegraphics[width=0.85\linewidth]{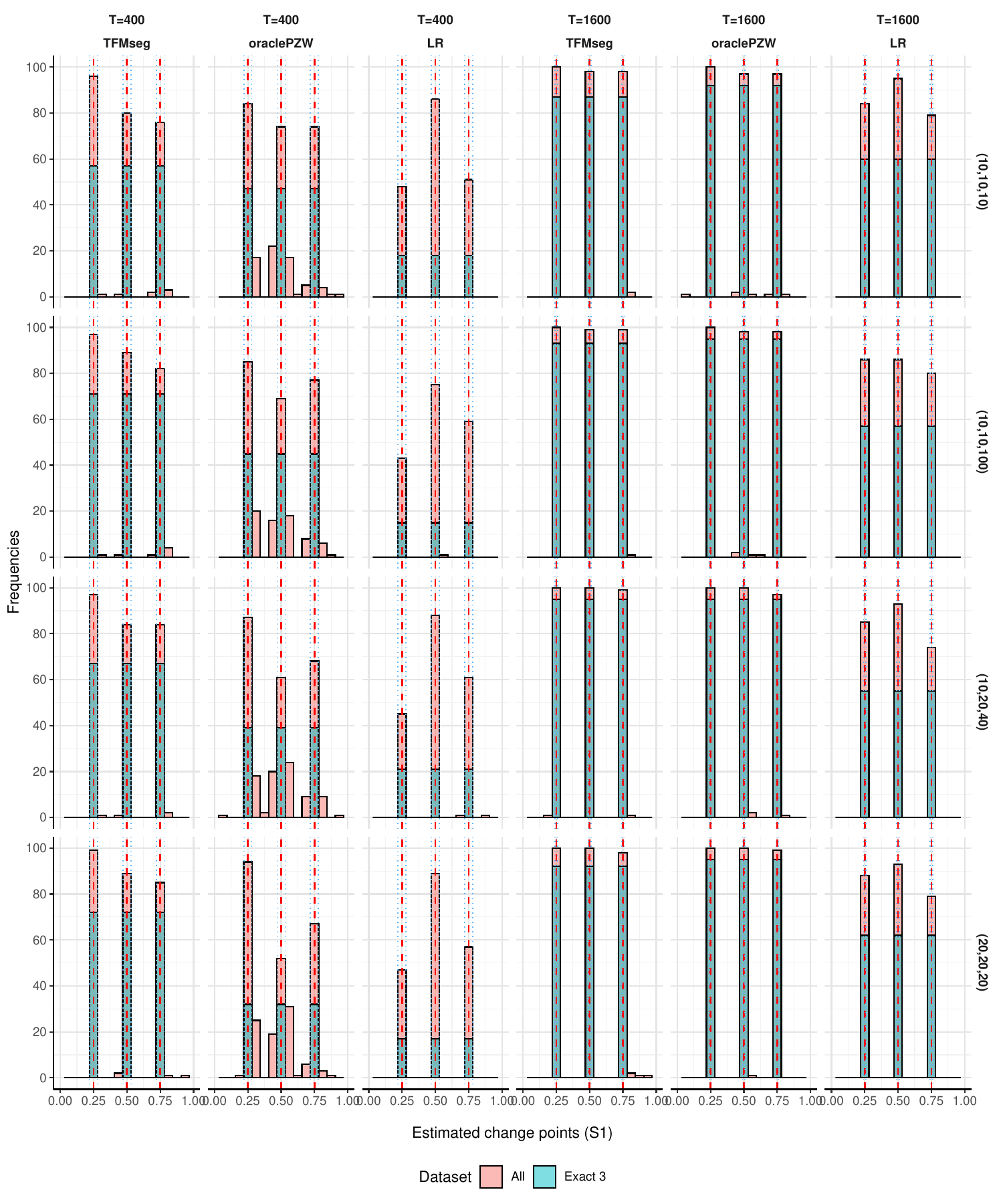}
\caption{\ref{s:one} Barplots of the scaled change point estimators $\{ \wh\theta^{(n)}_j/T, \, j \in [\wh q^{(n)}]\}$ returned by TFMseg, oraclePZW and LR  when $\Theta = \big\{\lfloor 0.25T \rfloor, \lfloor 0.5T \rfloor, \lfloor 0.75T \rfloor\big\}$, $(\rho_f, \rho_e) = (0, 0)$, $T \in \{400, 1600\}$ and varying $(p_1,p_2,p_3)$ (top to bottom) over $100$ realizations. The red bars give the total frequency of estimated change points for all $n \in [N]$, while the blue bars give the frequency from the subset of realizations $n \in \wh{\cD}$.}
\label{fig: subplot_rnull_indep}
\end{figure}

\begin{figure}[h!t!b!p!]
\centering
\includegraphics[width=1\linewidth]{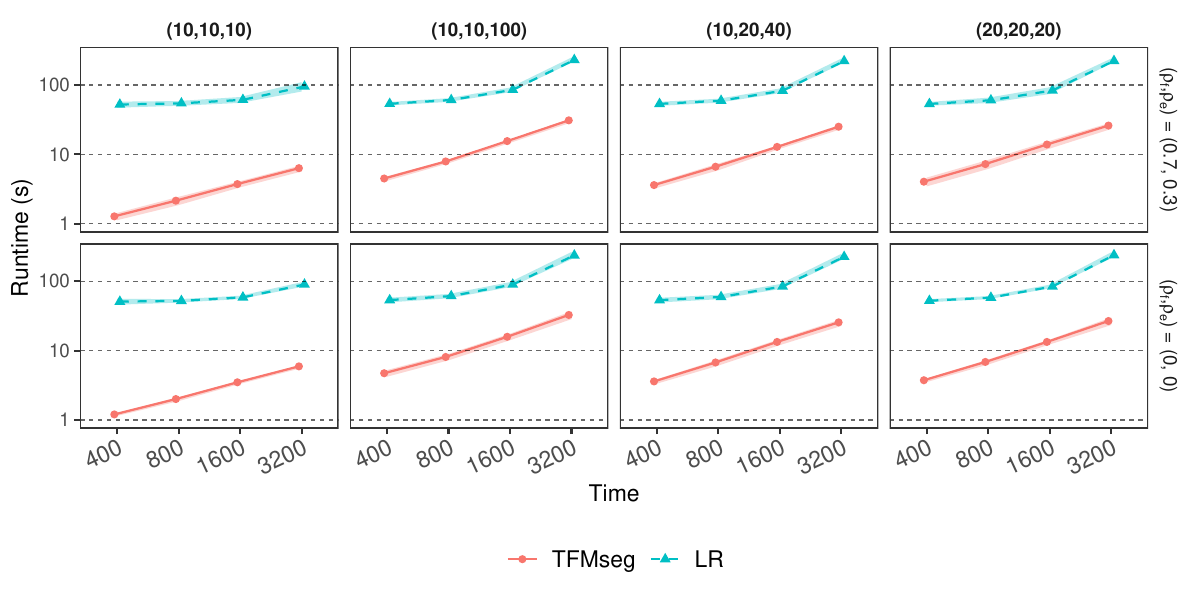}
\caption{\ref{s:one}
Average runtime of TFMseg and LR in seconds, computed from $100$ repetitions, for varying $T$ and $(p_1,p_2,p_3)$. The two rows correspond to $(\rho_f, \rho_e) = (0.7, 0.3)$ and $(\rho_f, \rho_e) = (0, 0)$, respectively. The runtime is shown on a logarithmic scale. Shaded bands indicate the empirical standard deviation.}
\label{fig:runtime}
\end{figure}

\clearpage 

\subsubsection{\ref{s:one} Single-mode change with unequally spaced change points}\label{app: num_det_unbal}

We consider when the change points are unequally spaced with $\Theta = \big\{\lfloor 0.25T \rfloor, \lfloor 0.5T \rfloor, \lfloor 0.625T \rfloor\big\}$  under~\ref{s:one}, see Tables~\ref{tab: unbal_rnull_T400}--\ref{tab: unbal_rnull_T3200} and Figure~\ref{fig: subplot_unbal_rnull} for the case of serially dependent data with $(\rho_f, \rho_e) = (0.7, 0.3)$, and Tables~\ref{tab: unbal_rnull_indep_T400}--\ref{tab: unbal_rnull_indep_T3200} and Figure~\ref{fig: subplot_unbal_rnull_indep} for the case when the data are serially uncorrelated.

\begin{table}[h!t!b!p!]
\caption{\ref{s:one} Summary of change point estimators returned by TFMseg, oraclePZW and LR when $\Theta  = \big\{\lfloor 0.25T \rfloor, \lfloor 0.5T \rfloor, \lfloor 0.625T \rfloor\big\}$, $(\rho_f, \rho_e) = (0.7, 0.3)$,  $T=400$ and varying $(p_1,p_2,p_3)$, based on 100 realizations.}
\label{tab: unbal_rnull_T400}
\centering
\setlength{\tabcolsep}{4.5pt}
\begin{tabular}[t]{clccccccccc}
\toprule
\multicolumn{2}{c}{ } & \multicolumn{5}{c}{$\wh{q}-q$} & \multicolumn{3}{c}{Accuracy} & \multicolumn{1}{c}{ } \\
\cmidrule(lr){3-7} \cmidrule(lr){8-10} 
Dimensions & Method & $\leq -2$ & $-1$ & $0$ & $1$ & $\geq 2$ & $j=1$ & $j=2$ & $j=3$ & Runtime (s)\\
\cmidrule(lr){1-2} \cmidrule(lr){3-7} \cmidrule(lr){8-10} \cmidrule(lr){11-11} 
(10,10,10) & TFMseg & 0.11 & 0.48 & 0.41 & 0 & 0 & 0.91 & 0.50 & 0.63 & 1.24 $\pm$ 0.14\\
 & oraclePZW & 0 & 0 & 1 & 0 & 0 & 0.46 & 0.27 & 0.31 & 1.35 $\pm$ 0.22\\
 & LR & 0.82 & 0.16 & 0.02 & 0 & 0 & 0.17 & 0.49 & 0.47 & 50.79 $\pm$ 5.93\\
(10,10,100) & TFMseg & 0.04 & 0.52 & 0.41 & 0.03 & 0 & 0.91 & 0.66 & 0.61 & 4.62 $\pm$ 0.47\\
 & oraclePZW & 0 & 0 & 1 & 0 & 0 & 0.46 & 0.22 & 0.41 & 4.80 $\pm$ 0.84\\
 & LR & 0.86 & 0.12 & 0.02 & 0 & 0 & 0.15 & 0.37 & 0.60 & 52.99 $\pm$ 3.75\\
(10,20,40) & TFMseg & 0.09 & 0.44 & 0.45 & 0.02 & 0 & 0.95 & 0.51 & 0.68 & 3.62 $\pm$ 0.48\\
 & oraclePZW & 0 & 0 & 1 & 0 & 0 & 0.44 & 0.18 & 0.36 & 3.78 $\pm$ 0.40\\
 & LR & 0.84 & 0.15 & 0.01 & 0 & 0 & 0.15 & 0.41 & 0.57 & 50.35 $\pm$ 4.52\\
(20,20,20) & TFMseg & 0.02 & 0.60 & 0.35 & 0.03 & 0 & 0.93 & 0.59 & 0.59 & 3.82 $\pm$ 0.44\\
 & oraclePZW & 0 & 0 & 1 & 0 & 0 & 0.58 & 0.11 & 0.26 & 3.95 $\pm$ 0.36\\
 & LR & 0.86 & 0.14 & 0 & 0 & 0 & 0.12 & 0.44 & 0.53 & 51.20 $\pm$ 3.37\\
\bottomrule
\end{tabular}
\end{table}

\begin{table}[h!t!b!p!]
\caption{\ref{s:one} Summary of change point estimators returned by TFMseg, oraclePZW and LR when $\Theta  = \big\{\lfloor 0.25T \rfloor, \lfloor 0.5T \rfloor, \lfloor 0.625T \rfloor\big\}$, $(\rho_f, \rho_e) = (0.7, 0.3)$,  $T=800$ and varying $(p_1,p_2,p_3)$, based on 100 realizations.}
\label{tab: unbal_rnull_T800}
\centering
\setlength{\tabcolsep}{4.5pt}
\begin{tabular}[t]{clccccccccc}
\toprule
\multicolumn{2}{c}{ } & \multicolumn{5}{c}{$\wh{q}-q$} & \multicolumn{3}{c}{Accuracy} & \multicolumn{1}{c}{ } \\
\cmidrule(lr){3-7} \cmidrule(lr){8-10} 
Dimensions & Method & $\leq -2$ & $-1$ & $0$ & $1$ & $\geq 2$ & $j=1$ & $j=2$ & $j=3$ & Runtime (s)\\
\cmidrule(lr){1-2} \cmidrule(lr){3-7} \cmidrule(lr){8-10} \cmidrule(lr){11-11} 
(10,10,10) & TFMseg & 0.02 & 0.37 & 0.57 & 0.04 & 0 & 0.92 & 0.66 & 0.69 & 1.94 $\pm$ 0.19\\
 & oraclePZW & 0 & 0 & 1 & 0 & 0 & 0.62 & 0.46 & 0.44 & 2.24 $\pm$ 0.34\\
 & LR & 0.57 & 0.37 & 0.06 & 0 & 0 & 0.38 & 0.61 & 0.48 & 53.30 $\pm$ 5.03\\
(10,10,100) & TFMseg & 0 & 0.22 & 0.69 & 0.08 & 0.01 & 0.92 & 0.80 & 0.72 & 7.58 $\pm$ 0.82\\
 & oraclePZW & 0 & 0 & 1 & 0 & 0 & 0.54 & 0.35 & 0.46 & 8.03 $\pm$ 0.79\\
 & LR & 0.56 & 0.35 & 0.09 & 0 & 0 & 0.39 & 0.54 & 0.56 & 57.00 $\pm$ 5.77\\
(10,20,40) & TFMseg & 0 & 0.26 & 0.68 & 0.06 & 0 & 0.96 & 0.78 & 0.74 & 6.50 $\pm$ 0.80\\
 & oraclePZW & 0 & 0 & 1 & 0 & 0 & 0.69 & 0.39 & 0.52 & 6.76 $\pm$ 0.75\\
 & LR & 0.51 & 0.39 & 0.10 & 0 & 0 & 0.33 & 0.66 & 0.59 & 57.21 $\pm$ 4.08\\
(20,20,20) & TFMseg & 0 & 0.26 & 0.66 & 0.07 & 0.01 & 0.98 & 0.79 & 0.73 & 6.63 $\pm$ 0.64\\
 & oraclePZW & 0 & 0 & 1 & 0 & 0 & 0.74 & 0.37 & 0.50 & 7.10 $\pm$ 0.72\\
 & LR & 0.41 & 0.47 & 0.12 & 0 & 0 & 0.43 & 0.64 & 0.60 & 57.10 $\pm$ 5.08\\
\bottomrule
\end{tabular}
\end{table}

\begin{table}[h!t!b!p!]
\caption{\ref{s:one} Summary of change point estimators returned by TFMseg, oraclePZW and LR when $\Theta  = \big\{\lfloor 0.25T \rfloor, \lfloor 0.5T \rfloor, \lfloor 0.625T \rfloor\big\}$, $(\rho_f, \rho_e) = (0.7, 0.3)$,  $T=1600$ and varying $(p_1,p_2,p_3)$, based on 100 realizations.}
\label{tab: unbal_rnull_T1600}
\centering
\setlength{\tabcolsep}{4.5pt}
\begin{tabular}[t]{clccccccccc}
\toprule
\multicolumn{2}{c}{ } & \multicolumn{5}{c}{$\wh{q}-q$} & \multicolumn{3}{c}{Accuracy} & \multicolumn{1}{c}{ } \\
\cmidrule(lr){3-7} \cmidrule(lr){8-10} 
Dimensions & Method & $\leq -2$ & $-1$ & $0$ & $1$ & $\geq 2$ & $j=1$ & $j=2$ & $j=3$ & Runtime (s)\\
\cmidrule(lr){1-2} \cmidrule(lr){3-7} \cmidrule(lr){8-10} \cmidrule(lr){11-11} 
(10,10,10) & TFMseg & 0 & 0.09 & 0.80 & 0.11 & 0 & 0.95 & 0.77 & 0.82 & 3.59 $\pm$ 0.40\\
 & oraclePZW & 0 & 0 & 1 & 0 & 0 & 0.78 & 0.77 & 0.71 & 3.99 $\pm$ 0.64\\
 & LR & 0.18 & 0.53 & 0.29 & 0 & 0 & 0.70 & 0.77 & 0.61 & 59.86 $\pm$ 6.63\\
(10,10,100) & TFMseg & 0 & 0.05 & 0.84 & 0.11 & 0 & 1 & 0.87 & 0.83 & 14.88 $\pm$ 1.34\\
 & oraclePZW & 0 & 0 & 1 & 0 & 0 & 0.91 & 0.76 & 0.76 & 15.45 $\pm$ 1.98\\
 & LR & 0.15 & 0.51 & 0.34 & 0 & 0 & 0.74 & 0.74 & 0.69 & 82.58 $\pm$ 8.60\\
(10,20,40) & TFMseg & 0 & 0.08 & 0.81 & 0.11 & 0 & 0.98 & 0.86 & 0.82 & 12.71 $\pm$ 1.16\\
 & oraclePZW & 0 & 0 & 1 & 0 & 0 & 0.87 & 0.75 & 0.77 & 13.18 $\pm$ 1.20\\
 & LR & 0.06 & 0.60 & 0.34 & 0 & 0 & 0.69 & 0.88 & 0.69 & 81.34 $\pm$ 8.40\\
(20,20,20) & TFMseg & 0 & 0.08 & 0.80 & 0.11 & 0.01 & 0.99 & 0.88 & 0.83 & 12.88 $\pm$ 1.60\\
 & oraclePZW & 0 & 0 & 1 & 0 & 0 & 0.98 & 0.78 & 0.81 & 13.74 $\pm$ 1.42\\
 & LR & 0.04 & 0.61 & 0.35 & 0 & 0 & 0.74 & 0.89 & 0.65 & 82.12 $\pm$ 6.66\\
\bottomrule
\end{tabular}
\end{table}

\begin{table}[h!t!b!p!]
\caption{\ref{s:one} Summary of change point estimators returned by TFMseg, oraclePZW and LR when $\Theta  = \big\{\lfloor 0.25T \rfloor, \lfloor 0.5T \rfloor, \lfloor 0.625T \rfloor\big\}$, $(\rho_f, \rho_e) = (0.7, 0.3)$,  $T=3200$ and varying $(p_1,p_2,p_3)$, based on 100 realizations.}
\label{tab: unbal_rnull_T3200}
\centering
\setlength{\tabcolsep}{4.5pt}
\begin{tabular}[t]{clccccccccc}
\toprule
\multicolumn{2}{c}{ } & \multicolumn{5}{c}{$\wh{q}-q$} & \multicolumn{3}{c}{Accuracy} & \multicolumn{1}{c}{ } \\
\cmidrule(lr){3-7} \cmidrule(lr){8-10} 
Dimensions & Method & $\leq -2$ & $-1$ & $0$ & $1$ & $\geq 2$ & $j=1$ & $j=2$ & $j=3$ & Runtime (s)\\
\cmidrule(lr){1-2} \cmidrule(lr){3-7} \cmidrule(lr){8-10} \cmidrule(lr){11-11} 
(10,10,10) & TFMseg & 0 & 0.04 & 0.82 & 0.12 & 0.02 & 0.99 & 0.95 & 0.94 & 6.16 $\pm$ 0.72\\
 & oraclePZW & 0 & 0 & 1 & 0 & 0 & 0.95 & 0.90 & 0.91 & 6.97 $\pm$ 1.20\\
 & LR & 0.04 & 0.30 & 0.66 & 0 & 0 & 0.92 & 0.92 & 0.78 & 87.91 $\pm$ 9.11\\
(10,10,100) & TFMseg & 0 & 0 & 0.84 & 0.16 & 0 & 1 & 0.97 & 0.95 & 29.99 $\pm$ 3.58\\
 & oraclePZW & 0 & 0 & 1 & 0 & 0 & 0.96 & 0.95 & 0.90 & 30.15 $\pm$ 3.28\\
 & LR & 0.04 & 0.42 & 0.54 & 0 & 0 & 0.88 & 0.83 & 0.78 & 215.30 $\pm$ 29.50\\
(10,20,40) & TFMseg & 0 & 0.01 & 0.78 & 0.19 & 0.02 & 1 & 1 & 0.94 & 24.95 $\pm$ 2.72\\
 & oraclePZW & 0 & 0 & 1 & 0 & 0 & 0.97 & 0.94 & 0.89 & 25.85 $\pm$ 2.98\\
 & LR & 0.02 & 0.38 & 0.60 & 0 & 0 & 0.88 & 0.93 & 0.77 & 220.30 $\pm$ 32.67\\
(20,20,20) & TFMseg & 0 & 0.01 & 0.81 & 0.16 & 0.02 & 1 & 0.98 & 0.93 & 25.26 $\pm$ 2.28\\
 & oraclePZW & 0 & 0 & 1 & 0 & 0 & 0.99 & 0.96 & 0.89 & 26.65 $\pm$ 3.30\\
 & LR & 0.01 & 0.31 & 0.68 & 0 & 0 & 0.88 & 0.96 & 0.82 & 215.01 $\pm$ 23.79\\
\bottomrule
\end{tabular}
\end{table}


\clearpage
\begin{table}[h!t!b!p!]
\caption{\ref{s:one} Summary of change point estimators returned by TFMseg, oraclePZW and LR when $\Theta  = \big\{\lfloor 0.25T \rfloor, \lfloor 0.5T \rfloor, \lfloor 0.625T \rfloor\big\}$, $(\rho_f, \rho_e) = (0, 0)$,  $T=400$ and varying $(p_1,p_2,p_3)$, based on 100 realizations.}
\label{tab: unbal_rnull_indep_T400}
\centering
\setlength{\tabcolsep}{4.5pt}
\begin{tabular}[t]{clccccccccc}
\toprule
\multicolumn{2}{c}{ } & \multicolumn{5}{c}{$\wh{q}-q$} & \multicolumn{3}{c}{Accuracy} & \multicolumn{1}{c}{ } \\
\cmidrule(lr){3-7} \cmidrule(lr){8-10} 
Dimensions & Method & $\leq -2$ & $-1$ & $0$ & $1$ & $\geq 2$ & $j=1$ & $j=2$ & $j=3$ & Runtime (s)\\
\cmidrule(lr){1-2} \cmidrule(lr){3-7} \cmidrule(lr){8-10} \cmidrule(lr){11-11} 
(10,10,10) & TFMseg & 0.05 & 0.56 & 0.39 & 0 & 0 & 0.98 & 0.57 & 0.77 & 1.22 $\pm$ 0.09\\
 & oraclePZW & 0 & 0 & 1 & 0 & 0 & 0.85 & 0.72 & 0.74 & 1.39 $\pm$ 0.22\\
 & LR & 0.66 & 0.26 & 0.08 & 0 & 0 & 0.31 & 0.60 & 0.50 & 50.39 $\pm$ 5.34\\
(10,10,100) & TFMseg & 0.02 & 0.43 & 0.55 & 0 & 0 & 0.97 & 0.72 & 0.80 & 5.08 $\pm$ 0.65\\
 & oraclePZW & 0 & 0 & 1 & 0 & 0 & 0.85 & 0.64 & 0.65 & 5.27 $\pm$ 0.66\\
 & LR & 0.60 & 0.31 & 0.09 & 0 & 0 & 0.33 & 0.53 & 0.62 & 54.10 $\pm$ 4.11\\
(10,20,40) & TFMseg & 0.01 & 0.51 & 0.48 & 0 & 0 & 0.99 & 0.64 & 0.80 & 3.69 $\pm$ 0.28\\
 & oraclePZW & 0 & 0 & 1 & 0 & 0 & 0.86 & 0.62 & 0.66 & 4.01 $\pm$ 0.48\\
 & LR & 0.55 & 0.37 & 0.08 & 0 & 0 & 0.28 & 0.63 & 0.61 & 52.95 $\pm$ 3.99\\
(20,20,20) & TFMseg & 0.01 & 0.55 & 0.43 & 0.01 & 0 & 1 & 0.63 & 0.78 & 3.97 $\pm$ 0.53\\
 & oraclePZW & 0 & 0 & 1 & 0 & 0 & 0.95 & 0.59 & 0.59 & 4.05 $\pm$ 0.36\\
 & LR & 0.60 & 0.33 & 0.07 & 0 & 0 & 0.30 & 0.58 & 0.59 & 52.21 $\pm$ 4.14\\
\bottomrule
\end{tabular}
\end{table}

\begin{table}[h!t!b!p!]
\caption{\ref{s:one} Summary of change point estimators returned by TFMseg, oraclePZW and LR when $\Theta  = \big\{\lfloor 0.25T \rfloor, \lfloor 0.5T \rfloor, \lfloor 0.625T \rfloor\big\}$, $(\rho_f, \rho_e) = (0, 0)$,  $T=800$ and varying $(p_1,p_2,p_3)$, based on 100 realizations.}
\label{tab: unbal_rnull_indep_T800}
\centering
\setlength{\tabcolsep}{4.5pt}
\begin{tabular}[t]{clccccccccc}
\toprule
\multicolumn{2}{c}{ } & \multicolumn{5}{c}{$\wh{q}-q$} & \multicolumn{3}{c}{Accuracy} & \multicolumn{1}{c}{ } \\
\cmidrule(lr){3-7} \cmidrule(lr){8-10} 
Dimensions & Method & $\leq -2$ & $-1$ & $0$ & $1$ & $\geq 2$ & $j=1$ & $j=2$ & $j=3$ & Runtime (s)\\
\cmidrule(lr){1-2} \cmidrule(lr){3-7} \cmidrule(lr){8-10} \cmidrule(lr){11-11} 
(10,10,10) & TFMseg & 0.01 & 0.21 & 0.78 & 0 & 0 & 0.99 & 0.92 & 0.80 & 2.03 $\pm$ 0.25\\
 & oraclePZW & 0 & 0 & 1 & 0 & 0 & 0.97 & 0.93 & 0.89 & 2.19 $\pm$ 0.30\\
 & LR & 0.31 & 0.51 & 0.18 & 0 & 0 & 0.60 & 0.71 & 0.55 & 52.56 $\pm$ 4.40\\
(10,10,100) & TFMseg & 0.01 & 0.19 & 0.79 & 0.01 & 0 & 0.99 & 0.90 & 0.87 & 8.58 $\pm$ 1.02\\
 & oraclePZW & 0 & 0 & 1 & 0 & 0 & 0.97 & 0.95 & 0.90 & 8.64 $\pm$ 0.91\\
 & LR & 0.43 & 0.38 & 0.19 & 0 & 0 & 0.47 & 0.65 & 0.63 & 60.70 $\pm$ 4.50\\
(10,20,40) & TFMseg & 0 & 0.21 & 0.78 & 0.01 & 0 & 1 & 0.85 & 0.90 & 6.93 $\pm$ 0.83\\
 & oraclePZW & 0 & 0 & 1 & 0 & 0 & 0.98 & 0.88 & 0.86 & 7.75 $\pm$ 1.14\\
 & LR & 0.28 & 0.52 & 0.20 & 0 & 0 & 0.48 & 0.79 & 0.65 & 59.97 $\pm$ 4.20\\
(20,20,20) & TFMseg & 0 & 0.19 & 0.80 & 0.01 & 0 & 1 & 0.88 & 0.90 & 7.24 $\pm$ 1.08\\
 & oraclePZW & 0 & 0 & 1 & 0 & 0 & 1 & 0.90 & 0.83 & 7.42 $\pm$ 0.84\\
 & LR & 0.17 & 0.58 & 0.25 & 0 & 0 & 0.60 & 0.86 & 0.62 & 58.57 $\pm$ 4.65\\
\bottomrule
\end{tabular}
\end{table}

\begin{table}[h!t!b!p!]
\caption{\ref{s:one} Summary of change point estimators returned by TFMseg, oraclePZW and LR when $\Theta  = \big\{\lfloor 0.25T \rfloor, \lfloor 0.5T \rfloor, \lfloor 0.625T \rfloor\big\}$, $(\rho_f, \rho_e) = (0, 0)$,  $T=1600$ and varying $(p_1,p_2,p_3)$, based on 100 realizations.}
\label{tab: unbal_rnull_indep_T1600}
\centering
\setlength{\tabcolsep}{4.5pt}
\begin{tabular}[t]{clccccccccc}
\toprule
\multicolumn{2}{c}{ } & \multicolumn{5}{c}{$\wh{q}-q$} & \multicolumn{3}{c}{Accuracy} & \multicolumn{1}{c}{ } \\
\cmidrule(lr){3-7} \cmidrule(lr){8-10} 
Dimensions & Method & $\leq -2$ & $-1$ & $0$ & $1$ & $\geq 2$ & $j=1$ & $j=2$ & $j=3$ & Runtime (s)\\
\cmidrule(lr){1-2} \cmidrule(lr){3-7} \cmidrule(lr){8-10} \cmidrule(lr){11-11} 
(10,10,10) & TFMseg & 0 & 0.06 & 0.92 & 0.02 & 0 & 1 & 0.93 & 0.97 & 3.58 $\pm$ 0.29\\
 & oraclePZW & 0 & 0 & 1 & 0 & 0 & 0.98 & 0.98 & 0.94 & 3.98 $\pm$ 0.53\\
 & LR & 0.11 & 0.50 & 0.39 & 0 & 0 & 0.77 & 0.85 & 0.65 & 61.20 $\pm$ 9.31\\
(10,10,100) & TFMseg & 0 & 0.05 & 0.93 & 0.02 & 0 & 0.99 & 0.95 & 0.97 & 16.31 $\pm$ 1.67\\
 & oraclePZW & 0 & 0 & 1 & 0 & 0 & 0.98 & 0.96 & 0.96 & 16.62 $\pm$ 1.70\\
 & LR & 0.14 & 0.43 & 0.43 & 0 & 0 & 0.75 & 0.81 & 0.73 & 87.05 $\pm$ 8.13\\
(10,20,40) & TFMseg & 0 & 0.01 & 0.97 & 0.02 & 0 & 0.99 & 1 & 0.97 & 14.68 $\pm$ 1.60\\
 & oraclePZW & 0 & 0 & 1 & 0 & 0 & 0.98 & 0.99 & 0.97 & 14.90 $\pm$ 1.67\\
 & LR & 0.07 & 0.52 & 0.41 & 0 & 0 & 0.73 & 0.92 & 0.69 & 85.86 $\pm$ 5.26\\
(20,20,20) & TFMseg & 0 & 0.04 & 0.92 & 0.04 & 0 & 1 & 1 & 0.96 & 14.08 $\pm$ 1.77\\
 & oraclePZW & 0 & 0 & 1 & 0 & 0 & 1 & 0.97 & 0.96 & 14.62 $\pm$ 1.46\\
 & LR & 0.02 & 0.56 & 0.42 & 0 & 0 & 0.79 & 0.92 & 0.68 & 81.82 $\pm$ 6.95\\
\bottomrule
\end{tabular}
\end{table}

\begin{table}[h!t!b!p!]
\caption{\ref{s:one} Summary of change point estimators returned by TFMseg, oraclePZW and LR when $\Theta  = \big\{\lfloor 0.25T \rfloor, \lfloor 0.5T \rfloor, \lfloor 0.625T \rfloor\big\}$, $(\rho_f, \rho_e) = (0, 0)$,  $T=3200$ and varying $(p_1,p_2,p_3)$, based on 100 realizations.}
\label{tab: unbal_rnull_indep_T3200}
\centering
\setlength{\tabcolsep}{4.5pt}
\begin{tabular}[t]{clccccccccc}
\toprule
\multicolumn{2}{c}{ } & \multicolumn{5}{c}{$\wh{q}-q$} & \multicolumn{3}{c}{Accuracy} & \multicolumn{1}{c}{ } \\
\cmidrule(lr){3-7} \cmidrule(lr){8-10} 
Dimensions & Method & $\leq -2$ & $-1$ & $0$ & $1$ & $\geq 2$ & $j=1$ & $j=2$ & $j=3$ & Runtime (s)\\
\cmidrule(lr){1-2} \cmidrule(lr){3-7} \cmidrule(lr){8-10} \cmidrule(lr){11-11} 
(10,10,10) & TFMseg & 0 & 0.03 & 0.95 & 0.02 & 0 & 1 & 0.97 & 0.95 & 6.05 $\pm$ 0.51\\
 & oraclePZW & 0 & 0 & 1 & 0 & 0 & 1 & 0.97 & 0.96 & 7.03 $\pm$ 1.11\\
 & LR & 0.04 & 0.30 & 0.66 & 0 & 0 & 0.93 & 0.92 & 0.77 & 93.80 $\pm$ 22.85\\
(10,10,100) & TFMseg & 0 & 0.02 & 0.97 & 0.01 & 0 & 1 & 0.98 & 0.98 & 33.24 $\pm$ 3.54\\
 & oraclePZW & 0 & 0 & 1 & 0 & 0 & 1 & 0.98 & 0.97 & 33.73 $\pm$ 3.33\\
 & LR & 0.07 & 0.33 & 0.60 & 0 & 0 & 0.88 & 0.85 & 0.79 & 227.07 $\pm$ 21.50\\
(10,20,40) & TFMseg & 0 & 0.01 & 0.96 & 0.03 & 0 & 1 & 1 & 0.97 & 28.20 $\pm$ 2.91\\
 & oraclePZW & 0 & 0 & 1 & 0 & 0 & 1 & 1 & 0.98 & 28.86 $\pm$ 3.05\\
 & LR & 0.01 & 0.34 & 0.65 & 0 & 0 & 0.90 & 0.96 & 0.78 & 220.05 $\pm$ 19.34\\
(20,20,20) & TFMseg & 0 & 0.01 & 0.96 & 0.03 & 0 & 1 & 1 & 0.97 & 27.62 $\pm$ 3.32\\
 & oraclePZW & 0 & 0 & 1 & 0 & 0 & 1 & 1 & 0.98 & 29.61 $\pm$ 3.57\\
 & LR & 0.01 & 0.31 & 0.68 & 0 & 0 & 0.91 & 0.97 & 0.79 & 219.11 $\pm$ 21.13\\
\bottomrule
\end{tabular}
\end{table}

\clearpage
\begin{figure}[h!t!b!p!]
\centering
\includegraphics[width=0.85\linewidth]{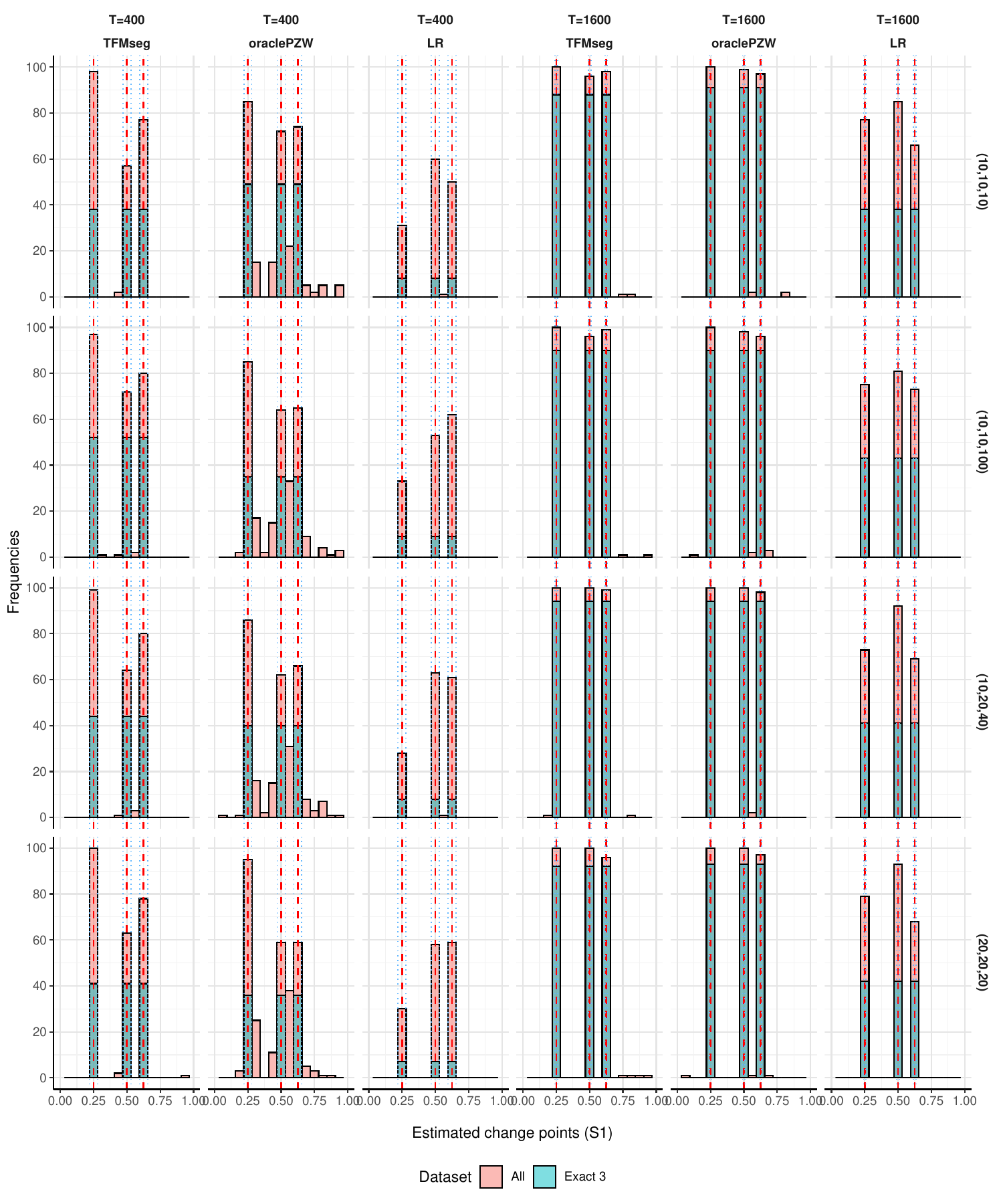}
\caption{\ref{s:one} Barplots of the scaled change point estimators $\{ \wh\theta^{(n)}_j/T, \, j \in [\wh q^{(n)}]\}$ returned by TFMseg, oraclePZW and LR  when $\Theta = \big\{\lfloor 0.25T \rfloor, \lfloor 0.5T \rfloor, \lfloor 0.625T \rfloor\big\}$, $(\rho_f, \rho_e) = (0, 0)$, $T \in \{400, 1600\}$ and varying $(p_1,p_2,p_3)$ (top to bottom) over $100$ realizations. The red bars give the total frequency of estimated change points for all $n \in [N]$, while the blue bars give the frequency from the subset of realizations $n \in \wh{\cD}$.}
\label{fig: subplot_unbal_rnull_indep}
\end{figure}

\clearpage
\subsubsection{\ref{s:two} Multi-mode change with equal-spaced change points}\label{app: num_det_bal_si}

We consider when the change points are equally spaced with $\Theta = \big\{\lfloor 0.25T \rfloor, \lfloor 0.5T \rfloor, \lfloor 0.75T \rfloor\big\}$  under~\ref{s:two}, see Tables~\ref{tab: bal_rnull_si_T400}--\ref{tab: bal_rnull_si_T3200} together with Figure~\ref{fig: subplot_rnull_si} for the case of serially dependent data with $(\rho_f, \rho_e) = (0.7, 0.3)$, and Tables~\ref{tab: bal_rnull_indep_si_T400}--\ref{tab: bal_rnull_indep_si_T3200} and Figure~\ref{fig: subplot_rnull_indep_si} for the case when the data are serially uncorrelated.
Comparing the results to those obtained under~\ref{s:one}, we note that the increased size of change as measured by $\omega_j$ in Definition~\ref{def: jump}, thanks to the additional mode undergoing a shift under~\ref{s:two}, benefits TFMseg and leads to its improved performance; this is not the case for LR as vectorization of the data does not preserve the increased change size.

\begin{table}[h!t!b!p!]
\caption{\ref{s:two} Summary of change point estimators returned by TFMseg, oraclePZW and LR when $\Theta  = \big\{\lfloor 0.25T \rfloor, \lfloor 0.5T \rfloor, \lfloor 0.75T \rfloor\big\}$, $(\rho_f, \rho_e) = (0.7, 0.3)$,  $T=400$ and varying $(p_1,p_2,p_3)$, based on 100 realizations.}
\label{tab: bal_rnull_si_T400}
\centering
\setlength{\tabcolsep}{4.5pt}
\begin{tabular}[t]{clccccccccc}
\toprule
\multicolumn{2}{c}{ } & \multicolumn{5}{c}{$\wh{q}-q$} & \multicolumn{3}{c}{Accuracy} & \multicolumn{1}{c}{ } \\
\cmidrule(lr){3-7} 
\cmidrule(lr){8-10} 
Dimensions & Method & $\leq -2$ & $-1$ & $0$ & $1$ & $\geq 2$ & $j=1$ & $j=2$ & $j=3$ & Runtime (s)\\
\cmidrule(lr){1-2} \cmidrule(lr){3-7} \cmidrule(lr){8-10} \cmidrule(lr){11-11} 
(10,10,10) & TFMseg & 0.01 & 0.25 & 0.71 & 0.03 & 0 & 1 & 0.72 & 0.81 & 1.21 $\pm$ 0.21\\
 & oraclePZW & 0 & 0 & 1 & 0 & 0 & 0.51 & 0.12 & 0.27 & 1.40 $\pm$ 0.20\\
 & LR & 0.06 & 0.91 & 0.03 & 0 & 0 & 1 & 0.60 & 0.37 & 50.36 $\pm$ 3.91\\
(10,10,100) & TFMseg & 0.01 & 0.20 & 0.77 & 0.02 & 0 & 1 & 0.75 & 0.80 & 4.52 $\pm$ 0.37\\
 & oraclePZW & 0 & 0 & 1 & 0 & 0 & 0.50 & 0.13 & 0.32 & 4.76 $\pm$ 0.53\\
 & LR & 0.15 & 0.84 & 0.01 & 0 & 0 & 1 & 0.41 & 0.45 & 52.92 $\pm$ 3.65\\
(10,20,40) & TFMseg & 0.01 & 0.31 & 0.61 & 0.07 & 0 & 1 & 0.67 & 0.82 & 3.44 $\pm$ 0.39\\
 & oraclePZW & 0 & 0 & 1 & 0 & 0 & 0.47 & 0.08 & 0.31 & 3.61 $\pm$ 0.31\\
 & LR & 0.09 & 0.86 & 0.05 & 0 & 0 & 1 & 0.52 & 0.43 & 50.86 $\pm$ 4.20\\
(20,20,20) & TFMseg & 0.02 & 0.29 & 0.65 & 0.04 & 0 & 1 & 0.62 & 0.87 & 3.50 $\pm$ 0.35\\
 & oraclePZW & 0 & 0 & 1 & 0 & 0 & 0.34 & 0.08 & 0.26 & 3.84 $\pm$ 0.55\\
 & LR & 0.11 & 0.84 & 0.05 & 0 & 0 & 1 & 0.55 & 0.39 & 48.63 $\pm$ 4.33\\
\bottomrule
\end{tabular}
\end{table}

\begin{table}[h!t!b!p!]
\caption{\ref{s:two} Summary of change point estimators returned by TFMseg, oraclePZW and LR when $\Theta  = \big\{\lfloor 0.25T \rfloor, \lfloor 0.5T \rfloor, \lfloor 0.75T \rfloor\big\}$, $(\rho_f, \rho_e) = (0.7, 0.3)$,  $T=800$ and varying $(p_1,p_2,p_3)$, based on 100 realizations.}
\label{tab: bal_rnull_si_T800}
\centering
\setlength{\tabcolsep}{4.5pt}
\begin{tabular}[t]{clccccccccc}
\toprule
\multicolumn{2}{c}{ } & \multicolumn{5}{c}{$\wh{q}-q$} & \multicolumn{3}{c}{Accuracy} & \multicolumn{1}{c}{ } \\
\cmidrule(lr){3-7} 
\cmidrule(lr){8-10} 
Dimensions & Method & $\leq -2$ & $-1$ & $0$ & $1$ & $\geq 2$ & $j=1$ & $j=2$ & $j=3$ & Runtime (s)\\
\cmidrule(lr){1-2} \cmidrule(lr){3-7} \cmidrule(lr){8-10} \cmidrule(lr){11-11} 
(10,10,10) & TFMseg & 0 & 0.02 & 0.89 & 0.09 & 0 & 0.99 & 0.88 & 0.92 & 1.93 $\pm$ 0.15\\
 & oraclePZW & 0 & 0 & 1 & 0 & 0 & 0.39 & 0.33 & 0.50 & 2.26 $\pm$ 0.39\\
 & LR & 0.01 & 0.82 & 0.17 & 0 & 0 & 1 & 0.70 & 0.46 & 51.06 $\pm$ 4.73\\
(10,10,100) & TFMseg & 0 & 0.01 & 0.88 & 0.10 & 0.01 & 0.99 & 0.90 & 0.89 & 7.91 $\pm$ 0.68\\
 & oraclePZW & 0 & 0 & 1 & 0 & 0 & 0.38 & 0.30 & 0.44 & 7.96 $\pm$ 0.57\\
 & LR & 0.05 & 0.78 & 0.17 & 0 & 0 & 1 & 0.57 & 0.54 & 59.07 $\pm$ 4.01\\
(10,20,40) & TFMseg & 0 & 0.04 & 0.87 & 0.09 & 0 & 1 & 0.92 & 0.88 & 6.33 $\pm$ 0.64\\
 & oraclePZW & 0 & 0 & 1 & 0 & 0 & 0.40 & 0.25 & 0.39 & 6.47 $\pm$ 0.68\\
 & LR & 0.01 & 0.75 & 0.24 & 0 & 0 & 1 & 0.68 & 0.55 & 55.47 $\pm$ 3.92\\
(20,20,20) & TFMseg & 0 & 0.05 & 0.85 & 0.10 & 0 & 1 & 0.87 & 0.93 & 6.33 $\pm$ 0.89\\
 & oraclePZW & 0 & 0 & 1 & 0 & 0 & 0.36 & 0.20 & 0.40 & 6.68 $\pm$ 0.76\\
 & LR & 0.03 & 0.67 & 0.30 & 0 & 0 & 1 & 0.72 & 0.54 & 53.13 $\pm$ 4.00\\
\bottomrule
\end{tabular}
\end{table}

\begin{table}[h!t!b!p!]
\caption{\ref{s:two} Summary of change point estimators returned by TFMseg, oraclePZW and LR when $\Theta  = \big\{\lfloor 0.25T \rfloor, \lfloor 0.5T \rfloor, \lfloor 0.75T \rfloor\big\}$, $(\rho_f, \rho_e) = (0.7, 0.3)$,  $T=1600$ and varying $(p_1,p_2,p_3)$, based on 100 realizations.}
\label{tab: bal_rnull_si_T1600}
\centering
\setlength{\tabcolsep}{4.5pt}
\begin{tabular}[t]{clccccccccc}
\toprule
\multicolumn{2}{c}{ } & \multicolumn{5}{c}{$\wh{q}-q$} & \multicolumn{3}{c}{Accuracy} & \multicolumn{1}{c}{ } \\
\cmidrule(lr){3-7} 
\cmidrule(lr){8-10} 
Dimensions & Method & $\leq -2$ & $-1$ & $0$ & $1$ & $\geq 2$ & $j=1$ & $j=2$ & $j=3$ & Runtime (s)\\
\cmidrule(lr){1-2} \cmidrule(lr){3-7} \cmidrule(lr){8-10} \cmidrule(lr){11-11} 
(10,10,10) & TFMseg & 0 & 0 & 0.88 & 0.11 & 0.01 & 1 & 0.89 & 0.90 & 3.37 $\pm$ 0.35\\
 & oraclePZW & 0 & 0 & 1 & 0 & 0 & 0.86 & 0.76 & 0.87 & 3.74 $\pm$ 0.50\\
 & LR & 0 & 0.47 & 0.53 & 0 & 0 & 1 & 0.88 & 0.65 & 55.39 $\pm$ 4.67\\
(10,10,100) & TFMseg & 0 & 0 & 0.90 & 0.07 & 0.03 & 1 & 0.93 & 0.90 & 15.37 $\pm$ 1.12\\
 & oraclePZW & 0 & 0 & 1 & 0 & 0 & 0.83 & 0.75 & 0.82 & 16.03 $\pm$ 1.57\\
 & LR & 0.02 & 0.56 & 0.42 & 0 & 0 & 1 & 0.74 & 0.66 & 85.41 $\pm$ 4.60\\
(10,20,40) & TFMseg & 0 & 0 & 0.87 & 0.13 & 0 & 1 & 0.88 & 0.92 & 12.00 $\pm$ 1.10\\
 & oraclePZW & 0 & 0 & 1 & 0 & 0 & 0.81 & 0.75 & 0.80 & 12.75 $\pm$ 1.53\\
 & LR & 0 & 0.52 & 0.48 & 0 & 0 & 1 & 0.82 & 0.66 & 78.36 $\pm$ 5.40\\
(20,20,20) & TFMseg & 0 & 0 & 0.93 & 0.06 & 0.01 & 1 & 0.92 & 0.95 & 12.21 $\pm$ 1.25\\
 & oraclePZW & 0 & 0 & 1 & 0 & 0 & 0.78 & 0.73 & 0.79 & 12.85 $\pm$ 1.52\\
 & LR & 0 & 0.45 & 0.55 & 0 & 0 & 1 & 0.85 & 0.69 & 74.16 $\pm$ 7.18\\
\bottomrule
\end{tabular}
\end{table}

\begin{table}[h!t!b!p!]
\caption{\ref{s:two} Summary of change point estimators returned by TFMseg, oraclePZW and LR when $\Theta  = \big\{\lfloor 0.25T \rfloor, \lfloor 0.5T \rfloor, \lfloor 0.75T \rfloor\big\}$, $(\rho_f, \rho_e) = (0.7, 0.3)$,  $T=3200$ and varying $(p_1,p_2,p_3)$, based on 100 realizations.}
\label{tab: bal_rnull_si_T3200}
\centering
\setlength{\tabcolsep}{4.5pt}
\begin{tabular}[t]{clccccccccc}
\toprule
\multicolumn{2}{c}{ } & \multicolumn{5}{c}{$\wh{q}-q$} & \multicolumn{3}{c}{Accuracy} & \multicolumn{1}{c}{ } \\
\cmidrule(lr){3-7} 
\cmidrule(lr){8-10} 
Dimensions & Method & $\leq -2$ & $-1$ & $0$ & $1$ & $\geq 2$ & $j=1$ & $j=2$ & $j=3$ & Runtime (s)\\
\cmidrule(lr){1-2} \cmidrule(lr){3-7} \cmidrule(lr){8-10} \cmidrule(lr){11-11} 
(10,10,10) & TFMseg & 0 & 0 & 0.80 & 0.19 & 0.01 & 1 & 0.98 & 1 & 5.59 $\pm$ 0.54\\
 & oraclePZW & 0 & 0 & 1 & 0 & 0 & 0.99 & 0.98 & 0.97 & 6.83 $\pm$ 1.06\\
 & LR & 0 & 0.30 & 0.70 & 0 & 0 & 1 & 0.93 & 0.77 & 82.24 $\pm$ 12.38\\
(10,10,100) & TFMseg & 0 & 0 & 0.82 & 0.16 & 0.02 & 1 & 1 & 0.99 & 31.78 $\pm$ 3.41\\
 & oraclePZW & 0 & 0 & 1 & 0 & 0 & 0.98 & 0.99 & 0.97 & 32.69 $\pm$ 2.62\\
 & LR & 0 & 0.40 & 0.60 & 0 & 0 & 1 & 0.82 & 0.78 & 222.88 $\pm$ 16.83\\
(10,20,40) & TFMseg & 0 & 0 & 0.71 & 0.28 & 0.01 & 1 & 0.99 & 0.99 & 24.03 $\pm$ 2.58\\
 & oraclePZW & 0 & 0 & 1 & 0 & 0 & 1 & 0.95 & 0.98 & 25.15 $\pm$ 2.55\\
 & LR & 0 & 0.34 & 0.66 & 0 & 0 & 1 & 0.90 & 0.76 & 211.57 $\pm$ 18.26\\
(20,20,20) & TFMseg & 0 & 0 & 0.87 & 0.12 & 0.01 & 1 & 1 & 0.97 & 24.41 $\pm$ 3.26\\
 & oraclePZW & 0 & 0 & 1 & 0 & 0 & 0.99 & 0.99 & 0.95 & 26.78 $\pm$ 3.42\\
 & LR & 0 & 0.26 & 0.74 & 0 & 0 & 1 & 0.93 & 0.81 & 202.50 $\pm$ 23.09\\
\bottomrule
\end{tabular}
\end{table}

\clearpage
\begin{figure}[h!t!b!p!]
\centering
\includegraphics[width=0.85\linewidth]{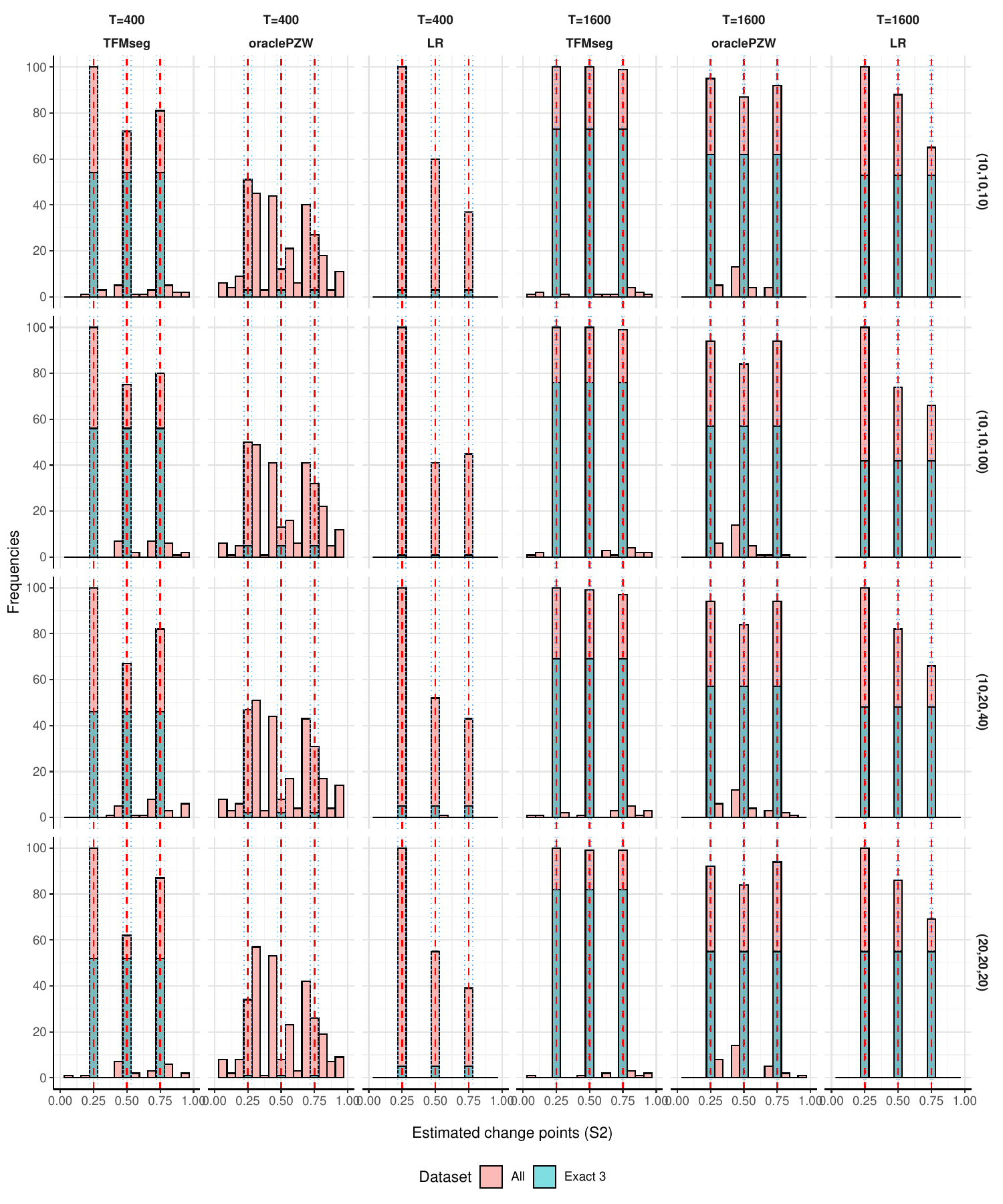}
\caption{\ref{s:two} Barplots of the scaled change point estimators $\{ \wh\theta^{(n)}_j/T, \, j \in [\wh q^{(n)}]\}$ returned by TFMseg, oraclePZW and LR  when $\Theta = \big\{\lfloor 0.25T \rfloor, \lfloor 0.5T \rfloor, \lfloor 0.75T \rfloor\big\}$, $(\rho_f, \rho_e) = (0.7, 0.3)$, $T \in \{400, 1600\}$ and varying $(p_1,p_2,p_3)$ (top to bottom) over $100$ realizations. The red bars give the total frequency of estimated change points for all $n \in [N]$, while the blue bars give the frequency from the subset of realizations $n \in \wh{\cD}$.}
\label{fig: subplot_rnull_si}
\end{figure}

\clearpage
\begin{table}[h!t!b!p!]
\caption{\ref{s:two} Summary of change point estimators returned by TFMseg, oraclePZW and LR when $\Theta  = \big\{\lfloor 0.25T \rfloor, \lfloor 0.5T \rfloor, \lfloor 0.75T \rfloor\big\}$, $(\rho_f, \rho_e) = (0, 0)$,  $T=400$ and varying $(p_1,p_2,p_3)$, based on 100 realizations.}
\label{tab: bal_rnull_indep_si_T400}
\centering
\setlength{\tabcolsep}{4.5pt}
\begin{tabular}[t]{clccccccccc}
\toprule
\multicolumn{2}{c}{ } & \multicolumn{5}{c}{$\wh{q}-q$} & \multicolumn{3}{c}{Accuracy} & \multicolumn{1}{c}{ } \\
\cmidrule(lr){3-7} 
\cmidrule(lr){8-10} 
Dimensions & Method & $\leq -2$ & $-1$ & $0$ & $1$ & $\geq 2$ & $j=1$ & $j=2$ & $j=3$ & Runtime (s)\\
\cmidrule(lr){1-2} \cmidrule(lr){3-7} \cmidrule(lr){8-10} \cmidrule(lr){11-11} 
(10,10,10) & TFMseg & 0 & 0.20 & 0.80 & 0 & 0 & 1 & 0.84 & 0.94 & 1.24 $\pm$ 0.11\\
 & oraclePZW & 0 & 0 & 1 & 0 & 0 & 0.75 & 0.53 & 0.76 & 1.46 $\pm$ 0.25\\
 & LR & 0.03 & 0.90 & 0.07 & 0 & 0 & 1 & 0.68 & 0.36 & 51.70 $\pm$ 2.81\\
(10,10,100) & TFMseg & 0 & 0.14 & 0.85 & 0.01 & 0 & 1 & 0.89 & 0.97 & 4.60 $\pm$ 0.24\\
 & oraclePZW & 0 & 0 & 1 & 0 & 0 & 0.88 & 0.55 & 0.74 & 4.76 $\pm$ 0.26\\
 & LR & 0.09 & 0.84 & 0.07 & 0 & 0 & 1 & 0.48 & 0.50 & 53.74 $\pm$ 2.86\\
(10,20,40) & TFMseg & 0 & 0.16 & 0.82 & 0.02 & 0 & 1 & 0.86 & 0.95 & 3.58 $\pm$ 0.26\\
 & oraclePZW & 0 & 0 & 1 & 0 & 0 & 0.75 & 0.40 & 0.68 & 3.92 $\pm$ 0.57\\
 & LR & 0.04 & 0.83 & 0.13 & 0 & 0 & 1 & 0.66 & 0.43 & 53.18 $\pm$ 3.58\\
(20,20,20) & TFMseg & 0 & 0.17 & 0.82 & 0.01 & 0 & 1 & 0.86 & 0.95 & 3.76 $\pm$ 0.41\\
 & oraclePZW & 0 & 0 & 1 & 0 & 0 & 0.63 & 0.32 & 0.73 & 3.94 $\pm$ 0.37\\
 & LR & 0.06 & 0.80 & 0.14 & 0 & 0 & 1 & 0.65 & 0.43 & 52.23 $\pm$ 2.68\\
\bottomrule
\end{tabular}
\end{table}

\begin{table}[h!t!b!p!]
\caption{\ref{s:two} Summary of change point estimators returned by TFMseg, oraclePZW and LR when $\Theta  = \big\{\lfloor 0.25T \rfloor, \lfloor 0.5T \rfloor, \lfloor 0.75T \rfloor\big\}$, $(\rho_f, \rho_e) = (0, 0)$,  $T=800$ and varying $(p_1,p_2,p_3)$, based on 100 realizations.}
\label{tab: bal_rnull_indep_si_T800}
\centering
\setlength{\tabcolsep}{4.5pt}
\begin{tabular}[t]{clccccccccc}
\toprule
\multicolumn{2}{c}{ } & \multicolumn{5}{c}{$\wh{q}-q$} & \multicolumn{3}{c}{Accuracy} & \multicolumn{1}{c}{ } \\
\cmidrule(lr){3-7} 
\cmidrule(lr){8-10} 
Dimensions & Method & $\leq -2$ & $-1$ & $0$ & $1$ & $\geq 2$ & $j=1$ & $j=2$ & $j=3$ & Runtime (s)\\
\cmidrule(lr){1-2} \cmidrule(lr){3-7} \cmidrule(lr){8-10} \cmidrule(lr){11-11} 
(10,10,10) & TFMseg & 0 & 0.01 & 0.96 & 0.03 & 0 & 1 & 0.98 & 0.97 & 2.01 $\pm$ 0.19\\
 & oraclePZW & 0 & 0 & 1 & 0 & 0 & 0.94 & 0.95 & 0.94 & 2.29 $\pm$ 0.25\\
 & LR & 0 & 0.66 & 0.34 & 0 & 0 & 1 & 0.84 & 0.50 & 53.13 $\pm$ 3.04\\
(10,10,100) & TFMseg & 0 & 0.01 & 0.97 & 0.02 & 0 & 1 & 0.99 & 0.97 & 7.85 $\pm$ 0.43\\
 & oraclePZW & 0 & 0 & 1 & 0 & 0 & 0.97 & 0.87 & 0.90 & 8.12 $\pm$ 0.50\\
 & LR & 0.03 & 0.71 & 0.26 & 0 & 0 & 1 & 0.66 & 0.57 & 60.91 $\pm$ 3.32\\
(10,20,40) & TFMseg & 0 & 0 & 0.97 & 0.03 & 0 & 1 & 0.99 & 0.98 & 6.91 $\pm$ 1.17\\
 & oraclePZW & 0 & 0 & 1 & 0 & 0 & 0.93 & 0.86 & 0.89 & 6.89 $\pm$ 0.68\\
 & LR & 0.01 & 0.63 & 0.36 & 0 & 0 & 1 & 0.82 & 0.53 & 59.77 $\pm$ 3.49\\
(20,20,20) & TFMseg & 0 & 0 & 0.99 & 0.01 & 0 & 1 & 1 & 0.97 & 6.77 $\pm$ 0.54\\
 & oraclePZW & 0 & 0 & 1 & 0 & 0 & 0.95 & 0.86 & 0.93 & 7.05 $\pm$ 0.50\\
 & LR & 0.01 & 0.62 & 0.37 & 0 & 0 & 1 & 0.81 & 0.55 & 58.64 $\pm$ 2.87\\
\bottomrule
\end{tabular}
\end{table}

\begin{table}[h!t!b!p!]
\caption{\ref{s:two} Summary of change point estimators returned by TFMseg, oraclePZW and LR when $\Theta  = \big\{\lfloor 0.25T \rfloor, \lfloor 0.5T \rfloor, \lfloor 0.75T \rfloor\big\}$, $(\rho_f, \rho_e) = (0, 0)$,  $T=1600$ and varying $(p_1,p_2,p_3)$, based on 100 realizations.}
\label{tab: bal_rnull_indep_si_T1600}
\centering
\setlength{\tabcolsep}{4.5pt}
\begin{tabular}[t]{clccccccccc}
\toprule
\multicolumn{2}{c}{ } & \multicolumn{5}{c}{$\wh{q}-q$} & \multicolumn{3}{c}{Accuracy} & \multicolumn{1}{c}{ } \\
\cmidrule(lr){3-7} 
\cmidrule(lr){8-10} 
Dimensions & Method & $\leq -2$ & $-1$ & $0$ & $1$ & $\geq 2$ & $j=1$ & $j=2$ & $j=3$ & Runtime (s)\\
\cmidrule(lr){1-2} \cmidrule(lr){3-7} \cmidrule(lr){8-10} \cmidrule(lr){11-11} 
(10,10,10) & TFMseg & 0 & 0 & 0.98 & 0.02 & 0 & 1 & 1 & 0.97 & 3.65 $\pm$ 0.36\\
 & oraclePZW & 0 & 0 & 1 & 0 & 0 & 1 & 1 & 0.97 & 4.01 $\pm$ 0.41\\
 & LR & 0 & 0.44 & 0.56 & 0 & 0 & 1 & 0.91 & 0.65 & 60.62 $\pm$ 5.08\\
(10,10,100) & TFMseg & 0 & 0 & 0.98 & 0.02 & 0 & 1 & 1 & 0.98 & 15.45 $\pm$ 0.91\\
 & oraclePZW & 0 & 0 & 1 & 0 & 0 & 1 & 1 & 0.98 & 16.07 $\pm$ 1.14\\
 & LR & 0.02 & 0.51 & 0.47 & 0 & 0 & 1 & 0.78 & 0.67 & 84.92 $\pm$ 5.53\\
(10,20,40) & TFMseg & 0 & 0 & 0.98 & 0.02 & 0 & 1 & 1 & 1 & 13.62 $\pm$ 1.73\\
 & oraclePZW & 0 & 0 & 1 & 0 & 0 & 1 & 0.98 & 0.99 & 13.68 $\pm$ 1.18\\
 & LR & 0 & 0.46 & 0.54 & 0 & 0 & 1 & 0.90 & 0.64 & 84.92 $\pm$ 5.17\\
(20,20,20) & TFMseg & 0 & 0 & 0.97 & 0.03 & 0 & 1 & 1 & 1 & 13.60 $\pm$ 1.74\\
 & oraclePZW & 0 & 0 & 1 & 0 & 0 & 1 & 1 & 1 & 14.37 $\pm$ 1.59\\
 & LR & 0 & 0.41 & 0.59 & 0 & 0 & 1 & 0.89 & 0.70 & 83.08 $\pm$ 5.57\\
\bottomrule
\end{tabular}
\end{table}

\begin{table}[h!t!b!p!]
\caption{\ref{s:two} Summary of change point estimators returned by TFMseg, oraclePZW and LR when $\Theta  = \big\{\lfloor 0.25T \rfloor, \lfloor 0.5T \rfloor, \lfloor 0.75T \rfloor\big\}$, $(\rho_f, \rho_e) = (0, 0)$,  $T=3200$ and varying $(p_1,p_2,p_3)$, based on 100 realizations.}
\label{tab: bal_rnull_indep_si_T3200}
\centering
\setlength{\tabcolsep}{4.5pt}
\begin{tabular}[t]{clccccccccc}
\toprule
\multicolumn{2}{c}{ } & \multicolumn{5}{c}{$\wh{q}-q$} & \multicolumn{3}{c}{Accuracy} & \multicolumn{1}{c}{ } \\
\cmidrule(lr){3-7} 
\cmidrule(lr){8-10} 
Dimensions & Method & $\leq -2$ & $-1$ & $0$ & $1$ & $\geq 2$ & $j=1$ & $j=2$ & $j=3$ & Runtime (s)\\
\cmidrule(lr){1-2} \cmidrule(lr){3-7} \cmidrule(lr){8-10} \cmidrule(lr){11-11} 
(10,10,10) & TFMseg & 0 & 0 & 0.95 & 0.04 & 0.01 & 1 & 1 & 1 & 6.11 $\pm$ 0.60\\
 & oraclePZW & 0 & 0 & 1 & 0 & 0 & 1 & 1 & 1 & 7.26 $\pm$ 1.04\\
 & LR & 0 & 0.31 & 0.69 & 0 & 0 & 1 & 0.94 & 0.75 & 96.90 $\pm$ 19.33\\
(10,10,100) & TFMseg & 0 & 0 & 0.96 & 0.04 & 0 & 1 & 1 & 1 & 30.58 $\pm$ 2.19\\
 & oraclePZW & 0 & 0 & 1 & 0 & 0 & 1 & 1 & 1 & 31.79 $\pm$ 2.28\\
 & LR & 0 & 0.41 & 0.59 & 0 & 0 & 1 & 0.83 & 0.76 & 231.81 $\pm$ 17.59\\
(10,20,40) & TFMseg & 0 & 0 & 0.97 & 0.03 & 0 & 1 & 1 & 1 & 26.44 $\pm$ 3.67\\
 & oraclePZW & 0 & 0 & 1 & 0 & 0 & 1 & 1 & 1 & 26.70 $\pm$ 2.09\\
 & LR & 0 & 0.31 & 0.69 & 0 & 0 & 1 & 0.91 & 0.78 & 225.79 $\pm$ 16.28\\
(20,20,20) & TFMseg & 0 & 0 & 0.95 & 0.05 & 0 & 1 & 1 & 1 & 27.38 $\pm$ 4.13\\
 & oraclePZW & 0 & 0 & 1 & 0 & 0 & 1 & 1 & 1 & 28.14 $\pm$ 3.54\\
 & LR & 0 & 0.27 & 0.73 & 0 & 0 & 1 & 0.93 & 0.80 & 221.77 $\pm$ 18.85\\
\bottomrule
\end{tabular}
\end{table}

\clearpage
\begin{figure}[h!t!b!p!]
\centering
\includegraphics[width=0.85\linewidth]{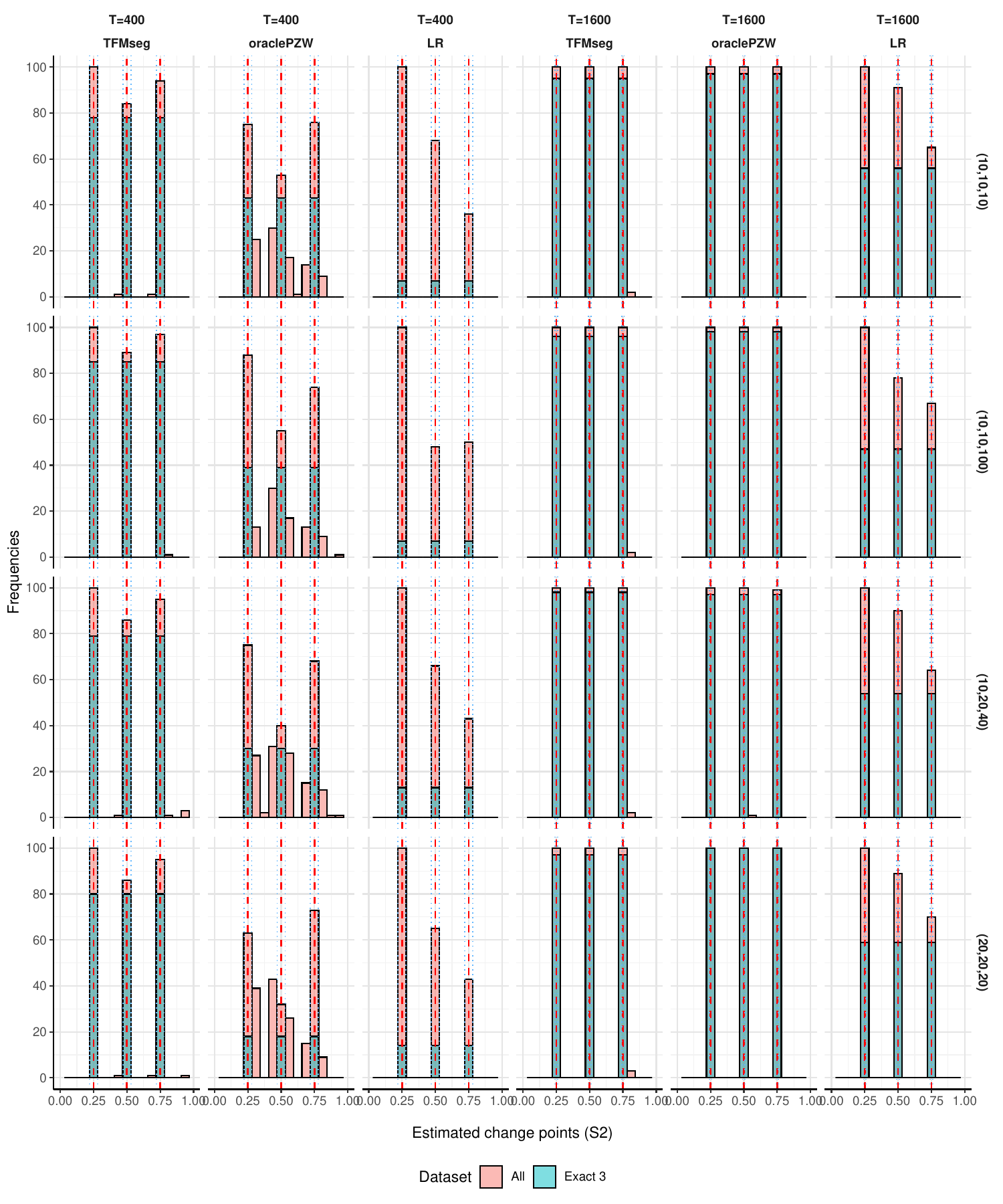}
\caption{\ref{s:two} Barplots of the scaled change point estimators $\{ \wh\theta^{(n)}_j/T, \, j \in [\wh q^{(n)}]\}$ returned by TFMseg, oraclePZW and LR  when $\Theta = \big\{\lfloor 0.25T \rfloor, \lfloor 0.5T \rfloor, \lfloor 0.75T \rfloor\big\}$, $(\rho_f, \rho_e) = (0, 0)$, $T \in \{400, 1600\}$ and varying $(p_1,p_2,p_3)$ (top to bottom) over $100$ realizations. The red bars give the total frequency of estimated change points for all $n \in [N]$, while the blue bars give the frequency from the subset of realizations $n \in \wh{\cD}$.}
\label{fig: subplot_rnull_indep_si}
\end{figure}

\clearpage 

\subsubsection{\ref{s:two} Multi-mode change with unequally spaced change points}\label{app: num_det_unbal_si}

We consider when the change points are unequally spaced with $\Theta = \big\{\lfloor 0.25T \rfloor, \lfloor 0.5T \rfloor, \lfloor 0.625T \rfloor\big\}$  under~\ref{s:two}, see Tables~\ref{tab: unbal_rnull_si_T400}--\ref{tab: unbal_rnull_si_T3200} together with Figure~\ref{fig: subplot_unbal_rnull_si} for the case of serially dependent data with $(\rho_f, \rho_e) = (0.7, 0.3)$, and Tables~\ref{tab: unbal_rnull_indep_si_T400}--\ref{tab: unbal_rnull_indep_si_T3200} and Figure~\ref{fig: subplot_unbal_rnull_indep_si} for the case when the data are serially uncorrelated.

\begin{table}[h!t!b!p!]
\caption{\ref{s:two} Summary of change point estimators returned by TFMseg, oraclePZW and LR when $\Theta  = \big\{\lfloor 0.25T \rfloor, \lfloor 0.5T \rfloor, \lfloor 0.625T \rfloor\big\}$, $(\rho_f, \rho_e) = (0.7, 0.3)$,  $T=400$ and varying $(p_1,p_2,p_3)$, based on 100 realizations.}
\label{tab: unbal_rnull_si_T400}
\centering
\setlength{\tabcolsep}{4.5pt}
\begin{tabular}[t]{clccccccccc}
\toprule
\multicolumn{2}{c}{ } & \multicolumn{5}{c}{$\wh{q}-q$} & \multicolumn{3}{c}{Accuracy} & \multicolumn{1}{c}{ } \\
\cmidrule(lr){3-7} 
\cmidrule(lr){8-10} 
Dimensions & Method & $\leq -2$ & $-1$ & $0$ & $1$ & $\geq 2$ & $j=1$ & $j=2$ & $j=3$ & Runtime (s)\\
\cmidrule(lr){1-2} \cmidrule(lr){3-7} \cmidrule(lr){8-10} \cmidrule(lr){11-11} 
(10,10,10) & TFMseg & 0 & 0.55 & 0.45 & 0 & 0 & 1 & 0.45 & 0.77 & 1.15 $\pm$ 0.05\\
 & oraclePZW & 0 & 0 & 1 & 0 & 0 & 0.56 & 0.13 & 0.31 & 1.37 $\pm$ 0.19\\
 & LR & 0.07 & 0.93 & 0 & 0 & 0 & 0.97 & 0.39 & 0.56 & 50.57 $\pm$ 2.22\\
(10,10,100) & TFMseg & 0 & 0.52 & 0.45 & 0.03 & 0 & 1 & 0.54 & 0.78 & 4.60 $\pm$ 0.23\\
 & oraclePZW & 0 & 0 & 1 & 0 & 0 & 0.49 & 0.14 & 0.32 & 4.82 $\pm$ 0.29\\
 & LR & 0.15 & 0.85 & 0 & 0 & 0 & 0.96 & 0.29 & 0.59 & 54.71 $\pm$ 3.11\\
(10,20,40) & TFMseg & 0.01 & 0.59 & 0.37 & 0.03 & 0 & 1 & 0.42 & 0.82 & 3.59 $\pm$ 0.23\\
 & oraclePZW & 0 & 0 & 1 & 0 & 0 & 0.56 & 0.09 & 0.42 & 3.85 $\pm$ 0.31\\
 & LR & 0.07 & 0.93 & 0 & 0 & 0 & 0.99 & 0.30 & 0.60 & 52.19 $\pm$ 3.30\\
(20,20,20) & TFMseg & 0.02 & 0.63 & 0.32 & 0.03 & 0 & 1 & 0.40 & 0.75 & 3.64 $\pm$ 0.18\\
 & oraclePZW & 0 & 0 & 1 & 0 & 0 & 0.37 & 0.07 & 0.28 & 3.89 $\pm$ 0.20\\
 & LR & 0.07 & 0.93 & 0 & 0 & 0 & 0.99 & 0.35 & 0.59 & 51.82 $\pm$ 2.38\\
\bottomrule
\end{tabular}
\end{table}

\begin{table}[h!t!b!p!]
\caption{\ref{s:two} Summary of change point estimators returned by TFMseg, oraclePZW and LR when $\Theta  = \big\{\lfloor 0.25T \rfloor, \lfloor 0.5T \rfloor, \lfloor 0.625T \rfloor\big\}$, $(\rho_f, \rho_e) = (0.7, 0.3)$,  $T=800$ and varying $(p_1,p_2,p_3)$, based on 100 realizations.}
\label{tab: unbal_rnull_si_T800}
\centering
\setlength{\tabcolsep}{4.5pt}
\begin{tabular}[t]{clccccccccc}
\toprule
\multicolumn{2}{c}{ } & \multicolumn{5}{c}{$\wh{q}-q$} & \multicolumn{3}{c}{Accuracy} & \multicolumn{1}{c}{ } \\
\cmidrule(lr){3-7} 
\cmidrule(lr){8-10} 
Dimensions & Method & $\leq -2$ & $-1$ & $0$ & $1$ & $\geq 2$ & $j=1$ & $j=2$ & $j=3$ & Runtime (s)\\
\cmidrule(lr){1-2} \cmidrule(lr){3-7} \cmidrule(lr){8-10} \cmidrule(lr){11-11} 
(10,10,10) & TFMseg & 0 & 0.20 & 0.71 & 0.09 & 0 & 0.99 & 0.70 & 0.91 & 1.96 $\pm$ 0.11\\
 & oraclePZW & 0 & 0 & 1 & 0 & 0 & 0.48 & 0.38 & 0.46 & 2.25 $\pm$ 0.25\\
 & LR & 0 & 0.99 & 0.01 & 0 & 0 & 1 & 0.41 & 0.57 & 51.91 $\pm$ 2.33\\
(10,10,100) & TFMseg & 0 & 0.25 & 0.62 & 0.13 & 0 & 0.99 & 0.63 & 0.91 & 8.09 $\pm$ 0.59\\
 & oraclePZW & 0 & 0 & 1 & 0 & 0 & 0.53 & 0.33 & 0.44 & 8.29 $\pm$ 0.57\\
 & LR & 0.03 & 0.95 & 0.02 & 0 & 0 & 1 & 0.31 & 0.68 & 61.54 $\pm$ 3.20\\
(10,20,40) & TFMseg & 0 & 0.30 & 0.65 & 0.05 & 0 & 1 & 0.67 & 0.89 & 6.64 $\pm$ 0.40\\
 & oraclePZW & 0 & 0 & 1 & 0 & 0 & 0.51 & 0.31 & 0.41 & 6.92 $\pm$ 0.49\\
 & LR & 0.01 & 0.97 & 0.02 & 0 & 0 & 1 & 0.38 & 0.61 & 58.95 $\pm$ 2.84\\
(20,20,20) & TFMseg & 0 & 0.36 & 0.60 & 0.04 & 0 & 1 & 0.61 & 0.85 & 6.61 $\pm$ 0.38\\
 & oraclePZW & 0 & 0 & 1 & 0 & 0 & 0.42 & 0.20 & 0.37 & 6.98 $\pm$ 0.41\\
 & LR & 0.02 & 0.96 & 0.02 & 0 & 0 & 1 & 0.32 & 0.67 & 57.25 $\pm$ 3.12\\
\bottomrule
\end{tabular}
\end{table}

\begin{table}[h!t!b!p!]
\caption{\ref{s:two} Summary of change point estimators returned by TFMseg, oraclePZW and LR when $\Theta  = \big\{\lfloor 0.25T \rfloor, \lfloor 0.5T \rfloor, \lfloor 0.625T \rfloor\big\}$, $(\rho_f, \rho_e) = (0.7, 0.3)$,  $T=1600$ and varying $(p_1,p_2,p_3)$, based on 100 realizations.}
\label{tab: unbal_rnull_si_T1600}
\centering
\setlength{\tabcolsep}{4.5pt}
\begin{tabular}[t]{clccccccccc}
\toprule
\multicolumn{2}{c}{ } & \multicolumn{5}{c}{$\wh{q}-q$} & \multicolumn{3}{c}{Accuracy} & \multicolumn{1}{c}{ } \\
\cmidrule(lr){3-7} 
\cmidrule(lr){8-10} 
Dimensions & Method & $\leq -2$ & $-1$ & $0$ & $1$ & $\geq 2$ & $j=1$ & $j=2$ & $j=3$ & Runtime (s)\\
\cmidrule(lr){1-2} \cmidrule(lr){3-7} \cmidrule(lr){8-10} \cmidrule(lr){11-11} 
(10,10,10) & TFMseg & 0 & 0.05 & 0.83 & 0.10 & 0.02 & 0.99 & 0.83 & 0.94 & 3.47 $\pm$ 0.16\\
 & oraclePZW & 0 & 0 & 1 & 0 & 0 & 0.91 & 0.77 & 0.83 & 3.89 $\pm$ 0.36\\
 & LR & 0 & 0.83 & 0.17 & 0 & 0 & 1 & 0.57 & 0.60 & 57.70 $\pm$ 2.89\\
(10,10,100) & TFMseg & 0 & 0.01 & 0.86 & 0.12 & 0.01 & 1 & 0.93 & 0.95 & 15.91 $\pm$ 1.08\\
 & oraclePZW & 0 & 0 & 1 & 0 & 0 & 0.92 & 0.81 & 0.86 & 16.38 $\pm$ 1.36\\
 & LR & 0 & 0.81 & 0.19 & 0 & 0 & 1 & 0.49 & 0.69 & 87.32 $\pm$ 5.22\\
(10,20,40) & TFMseg & 0 & 0.03 & 0.83 & 0.12 & 0.02 & 1 & 0.88 & 0.93 & 12.97 $\pm$ 0.73\\
 & oraclePZW & 0 & 0 & 1 & 0 & 0 & 0.86 & 0.78 & 0.83 & 13.65 $\pm$ 1.01\\
 & LR & 0 & 0.78 & 0.22 & 0 & 0 & 1 & 0.55 & 0.65 & 82.56 $\pm$ 4.86\\
(20,20,20) & TFMseg & 0 & 0.05 & 0.84 & 0.10 & 0.01 & 1 & 0.88 & 0.96 & 12.94 $\pm$ 0.70\\
 & oraclePZW & 0 & 0 & 1 & 0 & 0 & 0.81 & 0.67 & 0.81 & 13.87 $\pm$ 0.83\\
 & LR & 0 & 0.79 & 0.21 & 0 & 0 & 1 & 0.55 & 0.66 & 81.26 $\pm$ 4.61\\
\bottomrule
\end{tabular}
\end{table}

\begin{table}[h!t!b!p!]
\caption{\ref{s:two} Summary of change point estimators returned by TFMseg, oraclePZW and LR when $\Theta  = \big\{\lfloor 0.25T \rfloor, \lfloor 0.5T \rfloor, \lfloor 0.625T \rfloor\big\}$, $(\rho_f, \rho_e) = (0.7, 0.3)$,  $T=3200$ and varying $(p_1,p_2,p_3)$, based on 100 realizations.}
\label{tab: unbal_rnull_si_T3200}
\centering
\setlength{\tabcolsep}{4.5pt}
\begin{tabular}[t]{clccccccccc}
\toprule
\multicolumn{2}{c}{ } & \multicolumn{5}{c}{$\wh{q}-q$} & \multicolumn{3}{c}{Accuracy} & \multicolumn{1}{c}{ } \\
\cmidrule(lr){3-7} 
\cmidrule(lr){8-10} 
Dimensions & Method & $\leq -2$ & $-1$ & $0$ & $1$ & $\geq 2$ & $j=1$ & $j=2$ & $j=3$ & Runtime (s)\\
\cmidrule(lr){1-2} \cmidrule(lr){3-7} \cmidrule(lr){8-10} \cmidrule(lr){11-11} 
(10,10,10) & TFMseg & 0 & 0 & 0.77 & 0.22 & 0.01 & 1 & 0.98 & 0.99 & 5.86 $\pm$ 0.27\\
 & oraclePZW & 0 & 0 & 1 & 0 & 0 & 0.98 & 0.97 & 0.96 & 7.01 $\pm$ 0.85\\
 & LR & 0 & 0.48 & 0.52 & 0 & 0 & 1 & 0.82 & 0.70 & 85.84 $\pm$ 5.35\\
(10,10,100) & TFMseg & 0 & 0 & 0.79 & 0.20 & 0.01 & 1 & 1 & 0.98 & 31.72 $\pm$ 1.93\\
 & oraclePZW & 0 & 0 & 1 & 0 & 0 & 0.97 & 0.96 & 0.94 & 33.09 $\pm$ 2.89\\
 & LR & 0 & 0.59 & 0.41 & 0 & 0 & 1 & 0.66 & 0.74 & 226.83 $\pm$ 18.71\\
(10,20,40) & TFMseg & 0 & 0 & 0.76 & 0.22 & 0.02 & 1 & 0.99 & 0.96 & 25.07 $\pm$ 1.93\\
 & oraclePZW & 0 & 0 & 1 & 0 & 0 & 0.97 & 0.94 & 0.94 & 26.98 $\pm$ 2.20\\
 & LR & 0 & 0.49 & 0.51 & 0 & 0 & 1 & 0.82 & 0.69 & 227.70 $\pm$ 21.95\\
(20,20,20) & TFMseg & 0 & 0 & 0.81 & 0.19 & 0 & 1 & 1 & 0.99 & 25.86 $\pm$ 1.68\\
 & oraclePZW & 0 & 0 & 1 & 0 & 0 & 0.93 & 0.91 & 0.98 & 27.63 $\pm$ 1.73\\
 & LR & 0 & 0.45 & 0.55 & 0 & 0 & 1 & 0.81 & 0.74 & 221.57 $\pm$ 18.96\\
\bottomrule
\end{tabular}
\end{table}

\begin{figure}[h!t!b!p!]
\centering
\includegraphics[width=0.85\linewidth]{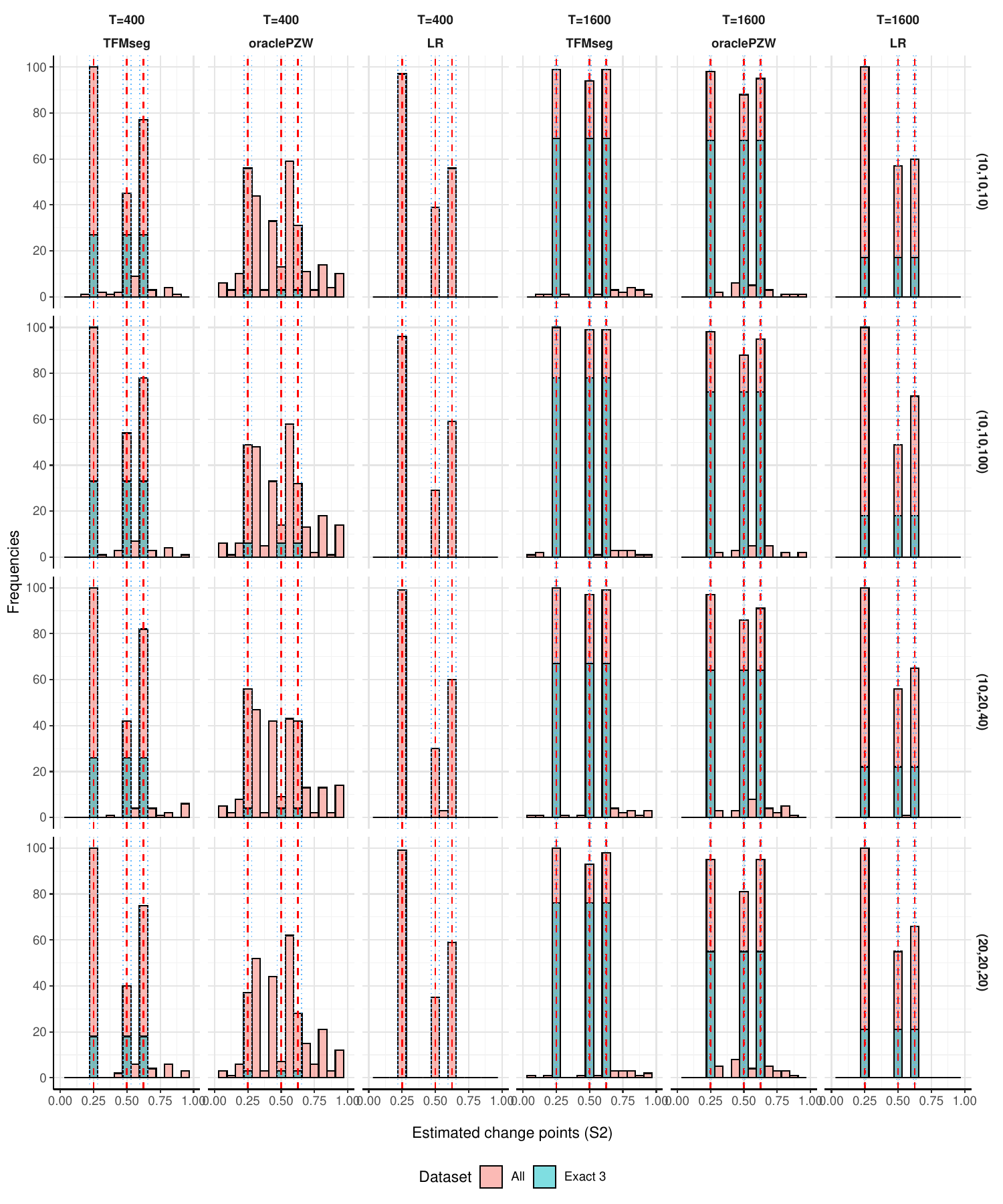}
\caption{\ref{s:two} Barplots of the scaled change point estimators $\{ \wh\theta^{(n)}_j/T, \, j \in [\wh q^{(n)}]\}$ returned by TFMseg, oraclePZW and LR  when $\Theta = \big\{\lfloor 0.25T \rfloor, \lfloor 0.5T \rfloor, \lfloor 0.625T \rfloor\big\}$, $(\rho_f, \rho_e) = (0.7, 0.3)$, $T \in \{400, 1600\}$ and varying $(p_1,p_2,p_3)$ (top to bottom) over $100$ realizations. The red bars give the total frequency of estimated change points for all $n \in [N]$, while the blue bars give the frequency from the subset of realizations $n \in \wh{\cD}$.}
\label{fig: subplot_unbal_rnull_si}
\end{figure}

\clearpage
\begin{table}[h!t!b!p!]
\caption{\ref{s:two} Summary of change point estimators returned by TFMseg, oraclePZW and LR when $\Theta  = \big\{\lfloor 0.25T \rfloor, \lfloor 0.5T \rfloor, \lfloor 0.625T \rfloor\big\}$, $(\rho_f, \rho_e) = (0, 0)$,  $T=400$ and varying $(p_1,p_2,p_3)$, based on 100 realizations.}
\label{tab: unbal_rnull_indep_si_T400}
\centering
\setlength{\tabcolsep}{4.5pt}
\begin{tabular}[t]{clccccccccc}
\toprule
\multicolumn{2}{c}{ } & \multicolumn{5}{c}{$\wh{q}-q$} & \multicolumn{3}{c}{Accuracy} & \multicolumn{1}{c}{ } \\
\cmidrule(lr){3-7} 
\cmidrule(lr){8-10} 
Dimensions & Method & $\leq -2$ & $-1$ & $0$ & $1$ & $\geq 2$ & $j=1$ & $j=2$ & $j=3$ & Runtime (s)\\
\cmidrule(lr){1-2} \cmidrule(lr){3-7} \cmidrule(lr){8-10} \cmidrule(lr){11-11} 
(10,10,10) & TFMseg & 0 & 0.59 & 0.41 & 0 & 0 & 1 & 0.48 & 0.92 & 1.21 $\pm$ 0.07\\
 & oraclePZW & 0 & 0 & 1 & 0 & 0 & 0.86 & 0.65 & 0.66 & 1.43 $\pm$ 0.21\\
 & LR & 0.05 & 0.95 & 0 & 0 & 0 & 0.98 & 0.39 & 0.57 & 51.07 $\pm$ 2.42\\
(10,10,100) & TFMseg & 0 & 0.51 & 0.49 & 0 & 0 & 1 & 0.54 & 0.93 & 4.70 $\pm$ 0.26\\
 & oraclePZW & 0 & 0 & 1 & 0 & 0 & 0.94 & 0.71 & 0.70 & 4.87 $\pm$ 0.28\\
 & LR & 0.08 & 0.91 & 0.01 & 0 & 0 & 1 & 0.33 & 0.60 & 54.59 $\pm$ 2.25\\
(10,20,40) & TFMseg & 0 & 0.58 & 0.41 & 0.01 & 0 & 1 & 0.47 & 0.92 & 3.72 $\pm$ 0.47\\
 & oraclePZW & 0 & 0 & 1 & 0 & 0 & 0.86 & 0.51 & 0.59 & 3.82 $\pm$ 0.33\\
 & LR & 0.04 & 0.96 & 0 & 0 & 0 & 1 & 0.38 & 0.58 & 51.19 $\pm$ 6.22\\
(20,20,20) & TFMseg & 0 & 0.58 & 0.42 & 0 & 0 & 1 & 0.50 & 0.86 & 3.76 $\pm$ 0.25\\
 & oraclePZW & 0 & 0 & 1 & 0 & 0 & 0.75 & 0.49 & 0.54 & 4.00 $\pm$ 0.24\\
 & LR & 0.05 & 0.95 & 0 & 0 & 0 & 0.99 & 0.34 & 0.61 & 53.35 $\pm$ 3.25\\
\bottomrule
\end{tabular}
\end{table}

\begin{table}[h!t!b!p!]
\caption{\ref{s:two} Summary of change point estimators returned by TFMseg, oraclePZW and LR when $\Theta  = \big\{\lfloor 0.25T \rfloor, \lfloor 0.5T \rfloor, \lfloor 0.625T \rfloor\big\}$, $(\rho_f, \rho_e) = (0, 0)$,  $T=800$ and varying $(p_1,p_2,p_3)$, based on 100 realizations.}
\label{tab: unbal_rnull_indep_si_T800}
\centering
\setlength{\tabcolsep}{4.5pt}
\begin{tabular}[t]{clccccccccc}
\toprule
\multicolumn{2}{c}{ } & \multicolumn{5}{c}{$\wh{q}-q$} & \multicolumn{3}{c}{Accuracy} & \multicolumn{1}{c}{ } \\
\cmidrule(lr){3-7} 
\cmidrule(lr){8-10} 
Dimensions & Method & $\leq -2$ & $-1$ & $0$ & $1$ & $\geq 2$ & $j=1$ & $j=2$ & $j=3$ & Runtime (s)\\
\cmidrule(lr){1-2} \cmidrule(lr){3-7} \cmidrule(lr){8-10} \cmidrule(lr){11-11} 
(10,10,10) & TFMseg & 0 & 0.16 & 0.84 & 0 & 0 & 1 & 0.85 & 0.96 & 1.98 $\pm$ 0.13\\
 & oraclePZW & 0 & 0 & 1 & 0 & 0 & 0.96 & 0.95 & 0.91 & 2.31 $\pm$ 0.27\\
 & LR & 0 & 0.98 & 0.02 & 0 & 0 & 1 & 0.45 & 0.57 & 52.61 $\pm$ 2.68\\
(10,10,100) & TFMseg & 0 & 0.13 & 0.86 & 0.01 & 0 & 1 & 0.88 & 0.97 & 8.14 $\pm$ 0.59\\
 & oraclePZW & 0 & 0 & 1 & 0 & 0 & 0.97 & 0.88 & 0.88 & 8.10 $\pm$ 0.50\\
 & LR & 0.02 & 0.93 & 0.05 & 0 & 0 & 1 & 0.38 & 0.65 & 60.21 $\pm$ 3.33\\
(10,20,40) & TFMseg & 0 & 0.12 & 0.88 & 0 & 0 & 1 & 0.86 & 0.99 & 6.67 $\pm$ 0.73\\
 & oraclePZW & 0 & 0 & 1 & 0 & 0 & 0.92 & 0.88 & 0.88 & 7.02 $\pm$ 0.89\\
 & LR & 0.02 & 0.96 & 0.02 & 0 & 0 & 1 & 0.37 & 0.63 & 58.32 $\pm$ 5.28\\
(20,20,20) & TFMseg & 0 & 0.25 & 0.74 & 0.01 & 0 & 1 & 0.77 & 0.97 & 6.94 $\pm$ 0.50\\
 & oraclePZW & 0 & 0 & 1 & 0 & 0 & 0.93 & 0.83 & 0.90 & 7.23 $\pm$ 0.58\\
 & LR & 0.01 & 0.91 & 0.08 & 0 & 0 & 1 & 0.42 & 0.65 & 60.12 $\pm$ 3.14\\
\bottomrule
\end{tabular}
\end{table}

\begin{table}[h!t!b!p!]
\caption{\ref{s:two} Summary of change point estimators returned by TFMseg, oraclePZW and LR when $\Theta  = \big\{\lfloor 0.25T \rfloor, \lfloor 0.5T \rfloor, \lfloor 0.625T \rfloor\big\}$, $(\rho_f, \rho_e) = (0, 0)$,  $T=1600$ and varying $(p_1,p_2,p_3)$, based on 100 realizations.}
\label{tab: unbal_rnull_indep_si_T1600}
\centering
\setlength{\tabcolsep}{4.5pt}
\begin{tabular}[t]{clccccccccc}
\toprule
\multicolumn{2}{c}{ } & \multicolumn{5}{c}{$\wh{q}-q$} & \multicolumn{3}{c}{Accuracy} & \multicolumn{1}{c}{ } \\
\cmidrule(lr){3-7} 
\cmidrule(lr){8-10} 
Dimensions & Method & $\leq -2$ & $-1$ & $0$ & $1$ & $\geq 2$ & $j=1$ & $j=2$ & $j=3$ & Runtime (s)\\
\cmidrule(lr){1-2} \cmidrule(lr){3-7} \cmidrule(lr){8-10} \cmidrule(lr){11-11} 
(10,10,10) & TFMseg & 0 & 0 & 0.98 & 0.02 & 0 & 1 & 1 & 1 & 3.55 $\pm$ 0.25\\
 & oraclePZW & 0 & 0 & 1 & 0 & 0 & 0.99 & 0.99 & 0.99 & 4.02 $\pm$ 0.43\\
 & LR & 0 & 0.80 & 0.20 & 0 & 0 & 1 & 0.63 & 0.57 & 59.10 $\pm$ 3.50\\
(10,10,100) & TFMseg & 0 & 0.01 & 0.96 & 0.03 & 0 & 1 & 0.99 & 1 & 15.30 $\pm$ 1.24\\
 & oraclePZW & 0 & 0 & 1 & 0 & 0 & 0.99 & 0.98 & 0.99 & 16.02 $\pm$ 1.24\\
 & LR & 0.01 & 0.78 & 0.21 & 0 & 0 & 1 & 0.54 & 0.66 & 87.15 $\pm$ 5.42\\
(10,20,40) & TFMseg & 0 & 0 & 0.97 & 0.03 & 0 & 1 & 1 & 1 & 12.99 $\pm$ 1.19\\
 & oraclePZW & 0 & 0 & 1 & 0 & 0 & 0.98 & 0.95 & 0.99 & 13.52 $\pm$ 1.22\\
 & LR & 0 & 0.66 & 0.34 & 0 & 0 & 1 & 0.70 & 0.64 & 83.47 $\pm$ 6.94\\
(20,20,20) & TFMseg & 0 & 0 & 0.97 & 0.03 & 0 & 1 & 1 & 0.99 & 13.54 $\pm$ 0.94\\
 & oraclePZW & 0 & 0 & 1 & 0 & 0 & 0.99 & 0.93 & 1 & 14.04 $\pm$ 1.46\\
 & LR & 0 & 0.67 & 0.33 & 0 & 0 & 1 & 0.68 & 0.65 & 85.42 $\pm$ 6.06\\
\bottomrule
\end{tabular}
\end{table}

\begin{table}[h!t!b!p!]
\caption{\ref{s:two} Summary of change point estimators returned by TFMseg, oraclePZW and LR when $\Theta  = \big\{\lfloor 0.25T \rfloor, \lfloor 0.5T \rfloor, \lfloor 0.625T \rfloor\big\}$, $(\rho_f, \rho_e) = (0, 0)$,  $T=3200$ and varying $(p_1,p_2,p_3)$, based on 100 realizations.}
\label{tab: unbal_rnull_indep_si_T3200}
\centering
\setlength{\tabcolsep}{4.5pt}
\begin{tabular}[t]{clccccccccc}
\toprule
\multicolumn{2}{c}{ } & \multicolumn{5}{c}{$\wh{q}-q$} & \multicolumn{3}{c}{Accuracy} & \multicolumn{1}{c}{ } \\
\cmidrule(lr){3-7} 
\cmidrule(lr){8-10} 
Dimensions & Method & $\leq -2$ & $-1$ & $0$ & $1$ & $\geq 2$ & $j=1$ & $j=2$ & $j=3$ & Runtime (s)\\
\cmidrule(lr){1-2} \cmidrule(lr){3-7} \cmidrule(lr){8-10} \cmidrule(lr){11-11} 
(10,10,10) & TFMseg & 0 & 0 & 0.95 & 0.04 & 0.01 & 1 & 1 & 1 & 6.05 $\pm$ 0.34\\
 & oraclePZW & 0 & 0 & 1 & 0 & 0 & 1 & 1 & 1 & 7.34 $\pm$ 0.99\\
 & LR & 0 & 0.51 & 0.49 & 0 & 0 & 1 & 0.80 & 0.69 & 90.21 $\pm$ 8.04\\
(10,10,100) & TFMseg & 0 & 0 & 0.96 & 0.04 & 0 & 1 & 1 & 1 & 30.03 $\pm$ 2.22\\
 & oraclePZW & 0 & 0 & 1 & 0 & 0 & 1 & 1 & 1 & 30.96 $\pm$ 2.44\\
 & LR & 0 & 0.53 & 0.47 & 0 & 0 & 1 & 0.73 & 0.74 & 242.11 $\pm$ 29.26\\
(10,20,40) & TFMseg & 0 & 0 & 0.95 & 0.05 & 0 & 1 & 1 & 1 & 25.72 $\pm$ 3.00\\
 & oraclePZW & 0 & 0 & 1 & 0 & 0 & 1 & 1 & 1 & 27.22 $\pm$ 2.95\\
 & LR & 0 & 0.45 & 0.55 & 0 & 0 & 1 & 0.86 & 0.69 & 226.57 $\pm$ 21.82\\
(20,20,20) & TFMseg & 0 & 0 & 0.98 & 0.02 & 0 & 1 & 1 & 1 & 26.38 $\pm$ 1.79\\
 & oraclePZW & 0 & 0 & 1 & 0 & 0 & 1 & 1 & 1 & 28.83 $\pm$ 4.22\\
 & LR & 0 & 0.39 & 0.61 & 0 & 0 & 1 & 0.88 & 0.73 & 225.62 $\pm$ 29.21\\
\bottomrule
\end{tabular}
\end{table}

\clearpage
\begin{figure}[h!t!b!p!]
\centering
\includegraphics[width=0.85\linewidth]{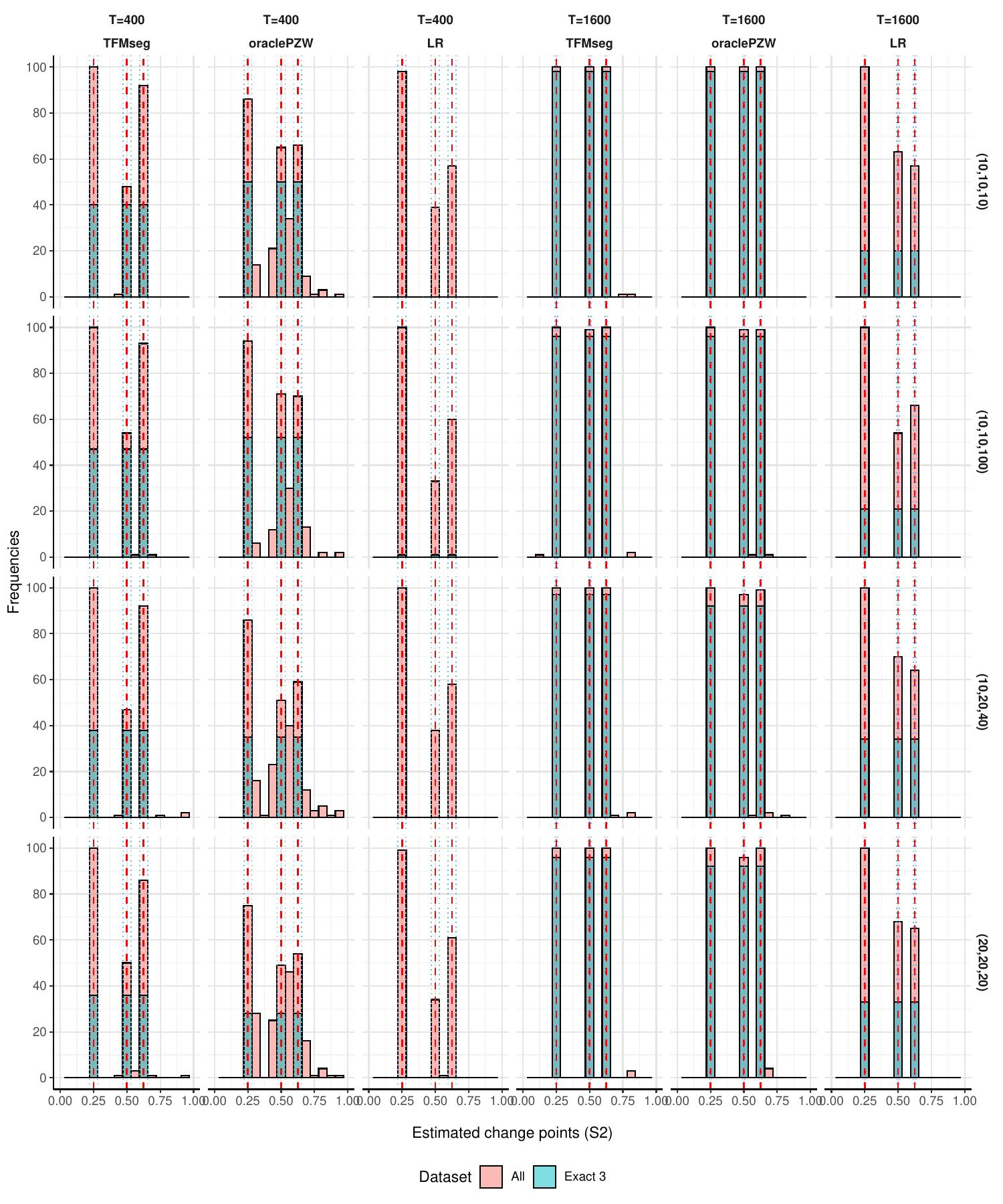}
\caption{\ref{s:two} Barplots of the scaled change point estimators $\{ \wh\theta^{(n)}_j/T, \, j \in [\wh q^{(n)}]\}$ returned by TFMseg, oraclePZW and LR  when $\Theta = \big\{\lfloor 0.25T \rfloor, \lfloor 0.5T \rfloor, \lfloor 0.625T \rfloor\big\}$, $(\rho_f, \rho_e) = (0, 0)$, $T \in \{400, 1600\}$ and varying $(p_1,p_2,p_3)$ (top to bottom) over $100$ realizations. The red bars give the total frequency of estimated change points for all $n \in [N]$, while the blue bars give the frequency from the subset of realizations $n \in \wh{\cD}$.}
\label{fig: subplot_unbal_rnull_indep_si}
\end{figure}



\clearpage

\subsection{Results for mode-identification}\label{app: num_idt}


Recalling the definition of $\wh{\cD}$ in Appendix~\ref{app: num_det}, we report the following performance metrics:
With $\wh\cK_j^{(n)} = \{ k \in [K]: \, \Vert \wh\Xi_j^{(k), (n)} \Vert > \zeta_{T,p} \}$,
\[
\text{TPR}_j = \frac{1}{\vert \wh{\cD} \vert}\sum_{n \in \wh{\cD}} \frac{\vert \wh{\cK}_j^{(n)} \cap \cK_j \vert}{ \max(\vert \cK_j \vert, \, 1) } \text{ \ and \ } \text{FPR}_j = \frac{1}{\vert \wh{\cD} \vert}\sum_{n \in \wh{\cD}} \frac{\vert \wh\cK_j^{(n)} \setminus \cK_j \vert}{\max(K - \vert \cK_j \vert, \, 1)}.
\]

\subsubsection{\ref{s:one} Single-mode change with equal-spaced change points}\label{app: num_idt_bal}

We consider when the change points are equally spaced with $\Theta = \big\{\lfloor 0.25T \rfloor, \lfloor 0.5T \rfloor, \lfloor 0.75T \rfloor\big\}$ under~\ref{s:one}, see Tables~\ref{tab: Mode_rnull_est}--\ref{tab: Mode_rnull_indep_est}.

\begin{table}[h!t!b!p!]
\centering
\caption{\ref{s:one} Summary of mode-identification results when $\Theta  = \big\{\lfloor 0.25T \rfloor, \lfloor 0.5T \rfloor, \lfloor 0.75T \rfloor\big\}$ and $(\rho_f, \rho_e) = (0.7, 0.3)$, for varying $T$ and $(p_1,p_2,p_3)$, over the subset of realizations $n \in \wh{\cD}$.}
\label{tab: Mode_rnull_est}
\centering
\resizebox{\ifdim\width>\linewidth\linewidth\else\width\fi}{!}{
\begin{tabular}[t]{ccccccccc}
\toprule
\multicolumn{2}{c}{ } & \multicolumn{2}{c}{$j=1$} & \multicolumn{2}{c}{$j=2$} & \multicolumn{2}{c}{$j=3$} & \multicolumn{1}{c}{  } \\
\cmidrule(l{3pt}r{3pt}){3-4} \cmidrule(l{3pt}r{3pt}){5-6} \cmidrule(l{3pt}r{3pt}){7-8}
$T$ & $(p_1,p_2,p_3)$ & TPR & FPR & TPR & FPR & TPR & FPR & $\vert \wh{\cD} \vert$\\
\cmidrule(lr){1-2} \cmidrule(lr){3-4} \cmidrule(lr){5-6} \cmidrule(lr){7-8} \cmidrule(lr){9-9}
 & (10,10,10) & 0.75 & 0 & 0.97 & 0.03 & 0.91 & 0 & 32\\
 & (10,10,100) & 0.76 & 0.01 & 1 & 0 & 0.93 & 0.01 & 42\\
 & (10,20,40) & 0.80 & 0 & 0.97 & 0 & 0.95 & 0 & 40\\
\multirow{-4}{*}{\centering\arraybackslash 400} & (20,20,20) & 0.89 & 0 & 1 & 0 & 1 & 0.01 & 35\\
\cmidrule(lr){1-2} \cmidrule(lr){3-4} \cmidrule(lr){5-6} \cmidrule(lr){7-8} \cmidrule(lr){9-9}
 & (10,10,10) & 0.77 & 0 & 1 & 0.04 & 0.94 & 0.01 & 53\\
 & (10,10,100) & 0.78 & 0.01 & 1 & 0 & 0.97 & 0 & 60\\
 & (10,20,40) & 0.80 & 0 & 1 & 0.01 & 0.98 & 0.01 & 64\\
\multirow{-4}{*}{\centering\arraybackslash 800} & (20,20,20) & 0.88 & 0 & 1 & 0 & 0.99 & 0.02 & 68\\
\cmidrule(lr){1-2} \cmidrule(lr){3-4} \cmidrule(lr){5-6} \cmidrule(lr){7-8} \cmidrule(lr){9-9}
 & (10,10,10) & 0.82 & 0 & 1 & 0.01 & 1 & 0 & 51\\
 & (10,10,100) & 0.79 & 0 & 1 & 0.01 & 1 & 0.02 & 62\\
 & (10,20,40) & 0.80 & 0 & 1 & 0 & 0.98 & 0.02 & 61\\
\multirow{-4}{*}{\centering\arraybackslash 1600} & (20,20,20) & 0.94 & 0 & 1 & 0.02 & 0.99 & 0.01 & 68\\
\cmidrule(lr){1-2} \cmidrule(lr){3-4} \cmidrule(lr){5-6} \cmidrule(lr){7-8} \cmidrule(lr){9-9}
 & (10,10,10) & 0.73 & 0 & 1 & 0.02 & 1 & 0.01 & 74\\
 & (10,10,100) & 0.80 & 0 & 1 & 0.01 & 1 & 0.01 & 75\\
 & (10,20,40) & 0.77 & 0.01 & 1 & 0.01 & 1 & 0.01 & 74\\
\multirow{-4}{*}{\centering\arraybackslash 3200} & (20,20,20) & 0.94 & 0.01 & 1 & 0.01 & 0.99 & 0.01 & 80\\
\bottomrule
\end{tabular}}
\end{table}

\begin{table}[h!t!b!p!]
\centering
\caption{\ref{s:one} Summary of mode-identification results when $\Theta  = \big\{\lfloor 0.25T \rfloor, \lfloor 0.5T \rfloor, \lfloor 0.75T \rfloor\big\}$ and $(\rho_f, \rho_e) = (0, 0)$, for varying $T$ and $(p_1,p_2,p_3)$, over the subset of realizations $n \in \wh{\cD}$.} 
\label{tab: Mode_rnull_indep_est}
\centering
\resizebox{\ifdim\width>\linewidth\linewidth\else\width\fi}{!}{
\begin{tabular}[t]{ccccccccc}
\toprule
\multicolumn{2}{c}{ } & \multicolumn{2}{c}{$j=1$} & \multicolumn{2}{c}{$j=2$} & \multicolumn{2}{c}{$j=3$} & \multicolumn{1}{c}{  } \\
\cmidrule(l{3pt}r{3pt}){3-4} \cmidrule(l{3pt}r{3pt}){5-6} \cmidrule(l{3pt}r{3pt}){7-8}
$T$ & $(p_1,p_2,p_3)$ & TPR & FPR & TPR & FPR & TPR & FPR & $\vert \wh{\cD} \vert$\\
\cmidrule(lr){1-2} \cmidrule(lr){3-4} \cmidrule(lr){5-6} \cmidrule(lr){7-8} \cmidrule(lr){9-9}
 & (10,10,10) & 0.70 & 0 & 0.98 & 0 & 0.91 & 0 & 57\\
 & (10,10,100) & 0.70 & 0 & 0.97 & 0 & 0.89 & 0 & 71\\
 & (10,20,40) & 0.72 & 0 & 1 & 0 & 0.90 & 0 & 67\\
\multirow{-4}{*}{\centering\arraybackslash 400} & (20,20,20) & 0.93 & 0 & 1 & 0 & 0.94 & 0 & 72\\
\cmidrule(lr){1-2} \cmidrule(lr){3-4} \cmidrule(lr){5-6} \cmidrule(lr){7-8} \cmidrule(lr){9-9}
 & (10,10,10) & 0.71 & 0 & 1 & 0 & 0.95 & 0 & 80\\
 & (10,10,100) & 0.78 & 0 & 1 & 0 & 0.98 & 0 & 89\\
 & (10,20,40) & 0.76 & 0 & 1 & 0 & 0.98 & 0 & 91\\
\multirow{-4}{*}{\centering\arraybackslash 800} & (20,20,20) & 0.89 & 0 & 1 & 0 & 0.94 & 0 & 89\\
\cmidrule(lr){1-2} \cmidrule(lr){3-4} \cmidrule(lr){5-6} \cmidrule(lr){7-8} \cmidrule(lr){9-9}
 & (10,10,10) & 0.70 & 0 & 0.98 & 0 & 0.99 & 0 & 87\\
 & (10,10,100) & 0.77 & 0 & 0.99 & 0 & 0.99 & 0 & 93\\
 & (10,20,40) & 0.78 & 0 & 1 & 0 & 1 & 0 & 95\\
\multirow{-4}{*}{\centering\arraybackslash 1600} & (20,20,20) & 0.91 & 0 & 1 & 0 & 0.99 & 0 & 92\\
\cmidrule(lr){1-2} \cmidrule(lr){3-4} \cmidrule(lr){5-6} \cmidrule(lr){7-8} \cmidrule(lr){9-9}
 & (10,10,10) & 0.73 & 0 & 0.99 & 0 & 1 & 0 & 90\\
 & (10,10,100) & 0.77 & 0 & 0.99 & 0 & 1 & 0 & 95\\
 & (10,20,40) & 0.78 & 0 & 1 & 0 & 1 & 0 & 95\\
\multirow{-4}{*}{\centering\arraybackslash 3200} & (20,20,20) & 0.93 & 0 & 1 & 0 & 1 & 0.01 & 96\\
\bottomrule
\end{tabular}}
\end{table}

\clearpage
\subsubsection{\ref{s:one} Single-mode change with unequally spaced change points}\label{app: num_idt_unbal}

We consider when the change points are unequally spaced with $\Theta = \big\{\lfloor 0.25T \rfloor, \lfloor 0.5T \rfloor, \lfloor 0.625T \rfloor\big\}$ under~\ref{s:one}, see Tables~\ref{tab: Mode_unbal_rnull_est} and~\ref{tab: Mode_unbal_rnull_indep_est}.

\begin{table}[h!t!b!p!]
\centering
\caption{\ref{s:one} Summary of mode-identification results when $\Theta  = \big\{\lfloor 0.25T \rfloor, \lfloor 0.5T \rfloor, \lfloor 0.625T \rfloor\big\}$ and $(\rho_f, \rho_e) = (0, 0)$, for varying $T$ and $(p_1,p_2,p_3)$, over the subset of realizations $n \in \wh{\cD}$.} 
\label{tab: Mode_unbal_rnull_indep_est}
\centering
\resizebox{\ifdim\width>\linewidth\linewidth\else\width\fi}{!}{
\begin{tabular}[t]{ccccccccc}
\toprule
\multicolumn{2}{c}{ } & \multicolumn{2}{c}{$j=1$} & \multicolumn{2}{c}{$j=2$} & \multicolumn{2}{c}{$j=3$} & \multicolumn{1}{c}{  } \\
\cmidrule(l{3pt}r{3pt}){3-4} \cmidrule(l{3pt}r{3pt}){5-6} \cmidrule(l{3pt}r{3pt}){7-8}
$T$ & $(p_1,p_2,p_3)$ & TPR & FPR & TPR & FPR & TPR & FPR & $\vert \wh{\cD} \vert$\\
\cmidrule(lr){1-2} \cmidrule(lr){3-4} \cmidrule(lr){5-6} \cmidrule(lr){7-8} \cmidrule(lr){9-9}
 & (10,10,10) & 0.63 & 0 & 0.97 & 0 & 0.92 & 0 & 38\\
 & (10,10,100) & 0.65 & 0 & 0.98 & 0 & 0.94 & 0 & 52\\
 & (10,20,40) & 0.75 & 0 & 1 & 0 & 0.93 & 0 & 44\\
\multirow{-4}{*}{\centering\arraybackslash 400} & (20,20,20) & 0.93 & 0 & 1 & 0 & 0.98 & 0 & 41\\
\cmidrule(lr){1-2} \cmidrule(lr){3-4} \cmidrule(lr){5-6} \cmidrule(lr){7-8} \cmidrule(lr){9-9}
 & (10,10,10) & 0.72 & 0 & 1 & 0.01 & 0.93 & 0 & 72\\
 & (10,10,100) & 0.82 & 0 & 1 & 0 & 0.95 & 0 & 76\\
 & (10,20,40) & 0.80 & 0 & 1 & 0 & 0.97 & 0 & 75\\
\multirow{-4}{*}{\centering\arraybackslash 800} & (20,20,20) & 0.91 & 0 & 1 & 0 & 0.97 & 0 & 77\\
\cmidrule(lr){1-2} \cmidrule(lr){3-4} \cmidrule(lr){5-6} \cmidrule(lr){7-8} \cmidrule(lr){9-9}
 & (10,10,10) & 0.70 & 0 & 0.99 & 0 & 0.99 & 0 & 88\\
 & (10,10,100) & 0.79 & 0 & 1 & 0 & 1 & 0 & 90\\
 & (10,20,40) & 0.78 & 0 & 1 & 0 & 1 & 0 & 94\\
\multirow{-4}{*}{\centering\arraybackslash 1600} & (20,20,20) & 0.89 & 0 & 1 & 0 & 0.99 & 0 & 92\\
\cmidrule(lr){1-2} \cmidrule(lr){3-4} \cmidrule(lr){5-6} \cmidrule(lr){7-8} \cmidrule(lr){9-9}
 & (10,10,10) & 0.72 & 0 & 0.99 & 0 & 0.99 & 0 & 90\\
 & (10,10,100) & 0.78 & 0 & 0.99 & 0 & 1 & 0 & 95\\
 & (10,20,40) & 0.81 & 0 & 1 & 0 & 1 & 0 & 94\\
\multirow{-4}{*}{\centering\arraybackslash 3200} & (20,20,20) & 0.91 & 0 & 1 & 0 & 1 & 0 & 94\\
\bottomrule
\end{tabular}}
\end{table}

\clearpage
\subsubsection{\ref{s:two} Multi-mode change with equal-spaced change points}\label{app: num_idt_bal_si}

We consider when the change points are equally spaced with $\Theta = \big\{\lfloor 0.25T \rfloor, \lfloor 0.5T \rfloor, \lfloor 0.75T \rfloor\big\}$ under~\ref{s:two},
see Tables~\ref{tab: Mode_rnull_si_est}--\ref{tab: Mode_rnull_indep_si_est}.

\begin{table}[h!t!b!p!]
\centering
\caption{\ref{s:two} Summary of mode-identification results when $\Theta  = \big\{\lfloor 0.25T \rfloor, \lfloor 0.5T \rfloor, \lfloor 0.75T \rfloor\big\}$ and $(\rho_f, \rho_e) = (0.7, 0.3)$, for varying $T$ and $(p_1,p_2,p_3)$, over the subset of realizations $n \in \wh{\cD}$.} 
\label{tab: Mode_rnull_si_est}
\centering
\resizebox{\ifdim\width>\linewidth\linewidth\else\width\fi}{!}{
\begin{tabular}[t]{ccccccccc}
\toprule
\multicolumn{2}{c}{ } & \multicolumn{2}{c}{$j=1$} & \multicolumn{2}{c}{$j=2$} & \multicolumn{2}{c}{$j=3$} & \multicolumn{1}{c}{  } \\
\cmidrule(l{3pt}r{3pt}){3-4} \cmidrule(l{3pt}r{3pt}){5-6} \cmidrule(l{3pt}r{3pt}){7-8}
$T$ & $(p_1,p_2,p_3)$ & TPR & FPR & TPR & FPR & TPR & FPR & $\vert \wh{\cD} \vert$\\
\cmidrule(lr){1-2} \cmidrule(lr){3-4} \cmidrule(lr){5-6} \cmidrule(lr){7-8} \cmidrule(lr){9-9}
 & (10,10,10) & 0.74 & 0.05 & 0.98 & 0 & 0.78 & 0 & 54\\
 & (10,10,100) & 0.75 & 0.02 & 0.96 & 0 & 0.76 & 0 & 56\\
 & (10,20,40) & 0.78 & 0.02 & 1 & 0 & 0.84 & 0 & 46\\
\multirow{-4}{*}{\centering\arraybackslash 400} & (20,20,20) & 0.81 & 0.03 & 1 & 0 & 0.78 & 0 & 52\\
\cmidrule(lr){1-2} \cmidrule(lr){3-4} \cmidrule(lr){5-6} \cmidrule(lr){7-8} \cmidrule(lr){9-9}
 & (10,10,10) & 0.81 & 0.05 & 0.99 & 0.02 & 0.91 & 0.01 & 74\\
 & (10,10,100) & 0.83 & 0.06 & 1 & 0.01 & 0.94 & 0 & 72\\
 & (10,20,40) & 0.88 & 0 & 1 & 0.01 & 0.96 & 0 & 72\\
\multirow{-4}{*}{\centering\arraybackslash 800} & (20,20,20) & 0.93 & 0.01 & 1 & 0 & 0.99 & 0.03 & 71\\
\cmidrule(lr){1-2} \cmidrule(lr){3-4} \cmidrule(lr){5-6} \cmidrule(lr){7-8} \cmidrule(lr){9-9}
 & (10,10,10) & 0.86 & 0.01 & 1 & 0.01 & 1 & 0 & 73\\
 & (10,10,100) & 0.87 & 0.03 & 1 & 0.01 & 0.99 & 0.01 & 76\\
 & (10,20,40) & 0.86 & 0.01 & 1 & 0 & 0.99 & 0 & 69\\
\multirow{-4}{*}{\centering\arraybackslash 1600} & (20,20,20) & 0.96 & 0.01 & 1 & 0.01 & 0.98 & 0 & 82\\
\cmidrule(lr){1-2} \cmidrule(lr){3-4} \cmidrule(lr){5-6} \cmidrule(lr){7-8} \cmidrule(lr){9-9}
 & (10,10,10) & 0.87 & 0.03 & 1 & 0.01 & 0.99 & 0 & 78\\
 & (10,10,100) & 0.91 & 0.01 & 1 & 0.01 & 1 & 0 & 82\\
 & (10,20,40) & 0.88 & 0.01 & 1 & 0.01 & 1 & 0 & 69\\
\multirow{-4}{*}{\centering\arraybackslash 3200} & (20,20,20) & 0.99 & 0.01 & 1 & 0.02 & 0.99 & 0 & 85\\
\bottomrule
\end{tabular}}
\end{table}

\begin{table}[h!t!b!p!]
\centering
\caption{\ref{s:two} Summary of mode-identification results when $\Theta  = \big\{\lfloor 0.25T \rfloor, \lfloor 0.5T \rfloor, \lfloor 0.75T \rfloor\big\}$ and $(\rho_f, \rho_e) = (0, 0)$, for varying $T$ and $(p_1,p_2,p_3)$, over the subset of realizations $n \in \wh{\cD}$.} 
\label{tab: Mode_rnull_indep_si_est}
\centering
\resizebox{\ifdim\width>\linewidth\linewidth\else\width\fi}{!}{
\begin{tabular}[t]{ccccccccc}
\toprule
\multicolumn{2}{c}{ } & \multicolumn{2}{c}{$j=1$} & \multicolumn{2}{c}{$j=2$} & \multicolumn{2}{c}{$j=3$} & \multicolumn{1}{c}{  } \\
\cmidrule(l{3pt}r{3pt}){3-4} \cmidrule(l{3pt}r{3pt}){5-6} \cmidrule(l{3pt}r{3pt}){7-8}
$T$ & $(p_1,p_2,p_3)$ & TPR & FPR & TPR & FPR & TPR & FPR & $\vert \wh{\cD} \vert$\\
\cmidrule(lr){1-2} \cmidrule(lr){3-4} \cmidrule(lr){5-6} \cmidrule(lr){7-8} \cmidrule(lr){9-9}
 & (10,10,10) & 0.73 & 0 & 0.96 & 0 & 0.75 & 0 & 78\\
 & (10,10,100) & 0.73 & 0 & 0.95 & 0 & 0.82 & 0 & 85\\
 & (10,20,40) & 0.77 & 0 & 1 & 0 & 0.85 & 0 & 79\\
\multirow{-4}{*}{\centering\arraybackslash 400} & (20,20,20) & 0.86 & 0 & 1 & 0 & 0.88 & 0 & 80\\
\cmidrule(lr){1-2} \cmidrule(lr){3-4} \cmidrule(lr){5-6} \cmidrule(lr){7-8} \cmidrule(lr){9-9}
 & (10,10,10) & 0.80 & 0 & 1 & 0 & 0.95 & 0 & 92\\
 & (10,10,100) & 0.85 & 0 & 1 & 0 & 0.97 & 0 & 94\\
 & (10,20,40) & 0.85 & 0 & 1 & 0 & 0.97 & 0 & 94\\
\multirow{-4}{*}{\centering\arraybackslash 800} & (20,20,20) & 0.91 & 0 & 1 & 0 & 0.97 & 0 & 96\\
\cmidrule(lr){1-2} \cmidrule(lr){3-4} \cmidrule(lr){5-6} \cmidrule(lr){7-8} \cmidrule(lr){9-9}
 & (10,10,10) & 0.87 & 0 & 1 & 0 & 0.99 & 0 & 95\\
 & (10,10,100) & 0.88 & 0 & 1 & 0 & 1 & 0 & 96\\
 & (10,20,40) & 0.87 & 0 & 1 & 0 & 0.98 & 0 & 98\\
\multirow{-4}{*}{\centering\arraybackslash 1600} & (20,20,20) & 0.95 & 0 & 1 & 0 & 0.99 & 0 & 97\\
\cmidrule(lr){1-2} \cmidrule(lr){3-4} \cmidrule(lr){5-6} \cmidrule(lr){7-8} \cmidrule(lr){9-9}
 & (10,10,10) & 0.86 & 0 & 1 & 0 & 0.99 & 0 & 95\\
 & (10,10,100) & 0.88 & 0 & 1 & 0 & 1 & 0 & 96\\
 & (10,20,40) & 0.89 & 0 & 1 & 0 & 1 & 0 & 97\\
\multirow{-4}{*}{\centering\arraybackslash 3200} & (20,20,20) & 0.99 & 0 & 1 & 0 & 1 & 0 & 95\\
\bottomrule
\end{tabular}}
\end{table}

\clearpage
\subsubsection{\ref{s:two} Multi-mode change with unequally spaced change points}\label{app: num_idt_unbal_si}

We consider when the change points are unequally spaced with $\Theta = \big\{\lfloor 0.25T \rfloor, \lfloor 0.5T \rfloor, \lfloor 0.625T \rfloor\big\}$ under~\ref{s:two},
see Tables~\ref{tab: Mode_unbal_rnull_si_est}--\ref{tab: Mode_unbal_rnull_indep_si_est}.

\begin{table}[h!t!b!p!]
\centering
\caption{\ref{s:two} Summary of mode-identification results when $\Theta  = \big\{\lfloor 0.25T \rfloor, \lfloor 0.5T \rfloor, \lfloor 0.625T \rfloor\big\}$ and $(\rho_f, \rho_e) = (0.7, 0.3)$, for varying $T$ and $(p_1,p_2,p_3)$, over the subset of realizations $n \in \wh{\cD}$.} 
\label{tab: Mode_unbal_rnull_si_est}
\centering
\resizebox{\ifdim\width>\linewidth\linewidth\else\width\fi}{!}{
\begin{tabular}[t]{ccccccccc}
\toprule
\multicolumn{2}{c}{ } & \multicolumn{2}{c}{$j=1$} & \multicolumn{2}{c}{$j=2$} & \multicolumn{2}{c}{$j=3$} & \multicolumn{1}{c}{  } \\
\cmidrule(l{3pt}r{3pt}){3-4} \cmidrule(l{3pt}r{3pt}){5-6} \cmidrule(l{3pt}r{3pt}){7-8}
$T$ & $(p_1,p_2,p_3)$ & TPR & FPR & TPR & FPR & TPR & FPR & $\vert \wh{\cD} \vert$\\
\cmidrule(lr){1-2} \cmidrule(lr){3-4} \cmidrule(lr){5-6} \cmidrule(lr){7-8} \cmidrule(lr){9-9}
 & (10,10,10) & 0.63 & 0.04 & 1 & 0.02 & 0.83 & 0 & 27\\
 & (10,10,100) & 0.64 & 0.06 & 0.94 & 0.03 & 0.82 & 0.03 & 33\\
 & (10,20,40) & 0.88 & 0 & 1 & 0.04 & 0.83 & 0 & 26\\
\multirow{-4}{*}{\centering\arraybackslash 400} & (20,20,20) & 0.89 & 0.03 & 1 & 0.03 & 0.81 & 0 & 18\\
\cmidrule(lr){1-2} \cmidrule(lr){3-4} \cmidrule(lr){5-6} \cmidrule(lr){7-8} \cmidrule(lr){9-9}
 & (10,10,10) & 0.80 & 0.06 & 1 & 0.06 & 0.94 & 0 & 54\\
 & (10,10,100) & 0.81 & 0.06 & 1 & 0.04 & 0.95 & 0.02 & 48\\
 & (10,20,40) & 0.85 & 0.01 & 1 & 0.03 & 0.97 & 0 & 52\\
\multirow{-4}{*}{\centering\arraybackslash 800} & (20,20,20) & 0.83 & 0.02 & 1 & 0.02 & 0.98 & 0.02 & 46\\
\cmidrule(lr){1-2} \cmidrule(lr){3-4} \cmidrule(lr){5-6} \cmidrule(lr){7-8} \cmidrule(lr){9-9}
 & (10,10,10) & 0.87 & 0.01 & 1 & 0.04 & 1 & 0.01 & 69\\
 & (10,10,100) & 0.88 & 0.03 & 1 & 0.05 & 1 & 0 & 78\\
 & (10,20,40) & 0.87 & 0 & 1 & 0.04 & 0.99 & 0 & 67\\
\multirow{-4}{*}{\centering\arraybackslash 1600} & (20,20,20) & 0.95 & 0 & 1 & 0.06 & 0.99 & 0 & 76\\
\cmidrule(lr){1-2} \cmidrule(lr){3-4} \cmidrule(lr){5-6} \cmidrule(lr){7-8} \cmidrule(lr){9-9}
 & (10,10,10) & 0.88 & 0.01 & 1 & 0.05 & 0.99 & 0 & 75\\
 & (10,10,100) & 0.90 & 0 & 1 & 0.07 & 0.99 & 0.01 & 77\\
 & (10,20,40) & 0.90 & 0 & 1 & 0.04 & 1 & 0 & 71\\
\multirow{-4}{*}{\centering\arraybackslash 3200} & (20,20,20) & 0.99 & 0 & 1 & 0.04 & 0.99 & 0.01 & 80\\
\bottomrule
\end{tabular}}
\end{table}

\begin{table}[h!t!b!p!]
\centering
\caption{\ref{s:two} Summary of mode-identification results when $\Theta = \big\{\lfloor 0.25T \rfloor, \lfloor 0.5T \rfloor, \lfloor 0.625T \rfloor\big\}$ and $(\rho_f, \rho_e) = (0, 0)$, for varying $T$ and $(p_1,p_2,p_3)$, over the subset of realizations $n \in \wh{\cD}$.} 
\label{tab: Mode_unbal_rnull_indep_si_est}
\centering
\resizebox{\ifdim\width>\linewidth\linewidth\else\width\fi}{!}{
\begin{tabular}[t]{ccccccccc}
\toprule
\multicolumn{2}{c}{ } & \multicolumn{2}{c}{$j=1$} & \multicolumn{2}{c}{$j=2$} & \multicolumn{2}{c}{$j=3$} & \multicolumn{1}{c}{  } \\
\cmidrule(l{3pt}r{3pt}){3-4} \cmidrule(l{3pt}r{3pt}){5-6} \cmidrule(l{3pt}r{3pt}){7-8}
$T$ & $(p_1,p_2,p_3)$ & TPR & FPR & TPR & FPR & TPR & FPR & $\vert \wh{\cD} \vert$\\
\cmidrule(lr){1-2} \cmidrule(lr){3-4} \cmidrule(lr){5-6} \cmidrule(lr){7-8} \cmidrule(lr){9-9}
 & (10,10,10) & 0.70 & 0 & 0.97 & 0 & 0.69 & 0 & 40\\
 & (10,10,100) & 0.66 & 0 & 1 & 0 & 0.80 & 0 & 47\\
 & (10,20,40) & 0.82 & 0 & 1 & 0 & 0.82 & 0 & 38\\
\multirow{-4}{*}{\centering\arraybackslash 400} & (20,20,20) & 0.92 & 0 & 1 & 0 & 0.82 & 0 & 36\\
\cmidrule(lr){1-2} \cmidrule(lr){3-4} \cmidrule(lr){5-6} \cmidrule(lr){7-8} \cmidrule(lr){9-9}
 & (10,10,10) & 0.84 & 0 & 1 & 0 & 0.95 & 0 & 81\\
 & (10,10,100) & 0.86 & 0 & 1 & 0 & 0.98 & 0 & 84\\
 & (10,20,40) & 0.87 & 0 & 1 & 0 & 0.98 & 0 & 85\\
\multirow{-4}{*}{\centering\arraybackslash 800} & (20,20,20) & 0.89 & 0 & 1 & 0 & 0.97 & 0 & 73\\
\cmidrule(lr){1-2} \cmidrule(lr){3-4} \cmidrule(lr){5-6} \cmidrule(lr){7-8} \cmidrule(lr){9-9}
 & (10,10,10) & 0.86 & 0 & 1 & 0 & 0.99 & 0 & 98\\
 & (10,10,100) & 0.85 & 0 & 1 & 0 & 1 & 0 & 96\\
 & (10,20,40) & 0.87 & 0 & 1 & 0 & 0.99 & 0 & 97\\
\multirow{-4}{*}{\centering\arraybackslash 1600} & (20,20,20) & 0.95 & 0 & 1 & 0 & 0.99 & 0 & 96\\
\cmidrule(lr){1-2} \cmidrule(lr){3-4} \cmidrule(lr){5-6} \cmidrule(lr){7-8} \cmidrule(lr){9-9}
 & (10,10,10) & 0.87 & 0 & 1 & 0 & 0.99 & 0 & 95\\
 & (10,10,100) & 0.88 & 0 & 1 & 0 & 1 & 0 & 96\\
 & (10,20,40) & 0.88 & 0 & 1 & 0 & 1 & 0 & 95\\
\multirow{-4}{*}{\centering\arraybackslash 3200} & (20,20,20) & 0.99 & 0 & 1 & 0 & 1 & 0 & 98\\
\bottomrule
\end{tabular}}
\end{table}

\clearpage

\subsection{Results for mode-wise loading space estimation}
\label{app: num_reest}

For the implementation of mode-wise loading estimation, we focus on the single change scenario~\ref{s:three} in which a rank-deficient transformation matrix $A_{2,2}$ is applied to the first mode with $\cK_1 = \{ 1 \}$. 
Consequently, we have $\col(\Lambda_{1, k}) = \col(\Lambda_{2, k}) = \col(\Lambda_{k})$ for $k \in \{2, 3\}$.
Upon consistently detecting and locating $\theta_1 = \lfloor T/2\rfloor$ by TFMseg and identifying $\cK_1$ as described in Section~\ref{sec: mode_identify}, we compare two approaches to estimating $\col(\Lambda_{j, k}), \, j \in [2], \, k \in [3]$:
\begin{enumerate}[label=(M\arabic*)]
\item \label{M:one} A segment-wise estimator that applies the projection-based estimation method described in Section~\ref{sec: detection}, to each $(0, \wh\theta_1]$ and $(\wh\theta_1, T]$, separately.
This requires the estimation of $r_{j, k}$, the mode-$k$ factor number on the $j$-th segment, which is achieved as described in Section~\ref{sec: tuning}.

\item \label{M:two} A mode-informed estimator following Section~\ref{subsec: mode_informed}, where $\Lambda_{1, 1}$ and $\Lambda_{2, 1}$ are estimated as in~\ref{M:one} (denoted by $\wh\Lambda_{1, 1}$ and $\wh\Lambda_{2, 1}$, respectively), while $\Lambda_2$ is estimated as $\sqrt{p_2}$ times the $\wh{r}_2$ leading eigenvectors of the matrix
\begin{align*}
\frac{1}{Tpp_1p_2}&\l\{\sum_{t=1}^{\wh\theta_1} \l( \mat_2(\cX_t)(\wt{\Lambda}_{3} \otimes \wh{\Lambda}_{1,1})\r) \l( \mat_2(\cX_t)(\wt{\Lambda}_{3} \otimes \wh{\Lambda}_{1,1})\r)^\trans \r.\\
& \l.\quad + \sum_{t=\wh\theta_1 + 1}^{T} \l( \mat_2(\cX_t)(\wt{\Lambda}_{3} \otimes \wh{\Lambda}_{2,1})\r) \l( \mat_2(\cX_t)(\wt{\Lambda}_{3} \otimes \wh{\Lambda}_{2,1})\r)^\trans\r\},
\end{align*}
where $\wt{\Lambda}_3$ is the preliminary estimator given in Section~\ref{sec: detection}; we estimate $\Lambda_3$ analogously. 
\end{enumerate}

After applying TFMseg and the mode-identification procedure to $N = 100$ realizations, we denote by $\wt{\cD} \subset \wh{\cD} \subset [N]$ the index set of realizations where the single change is detected (i.e.\ $\wh{\cD}$) and $\wh{\cK}_1 = \cK_1$ is returned.
Then for the realizations $n \in \wt{\cD}$, we measure the distance between the projection matrices onto the corresponding column spaces of the target matrix, say $A$, and its estimator $\wh A$, as
\begin{align}
d(\wh A, A) = \l\Vert \wh A(\wh A^\trans \wh A)^{-1} \wh A^\trans - A(A^\trans A)^{-1} A^\trans \r\Vert, \label{eq:fle_dist}
\end{align}
see Figure~\ref{fig: plot_reest_rnull_indep}.
We truncate anomalously large estimation errors from the boxplots for the ease of comparing the results across different settings, noting that the large errors are solely attributed to that the factor number estimator occasionally under-estimates $r_{j, k}$ over shorter segments post change point detection when $T$ is small. 

As expected, for modes~2 and~3, pooling the pre- and post-change point segments improves the estimation of the loading spaces, and the improvement is observable across all settings.
In line with Proposition~\ref{prop: consistency_proj}, we note that the mode~$k$ loading estimation error decreases with $\pmk$, which is noticeable when the dimensions are highly unbalanced.
For mode~1, the post-change estimation error is smaller as $r_{j, 1} = \text{rank}(\Lambda_{j, 1})$ reduces from $3$ to $2$ after the change point.
This has an adverse effect on the other two modes, when observing the estimation errors form~\ref{M:one}, which is accounted by that the projection involved in~\eqref{eqn: est_lam} post-change point is of reduced rank.

\begin{figure}[h!t!b!p!]
\centering
\includegraphics[width=0.7\linewidth]{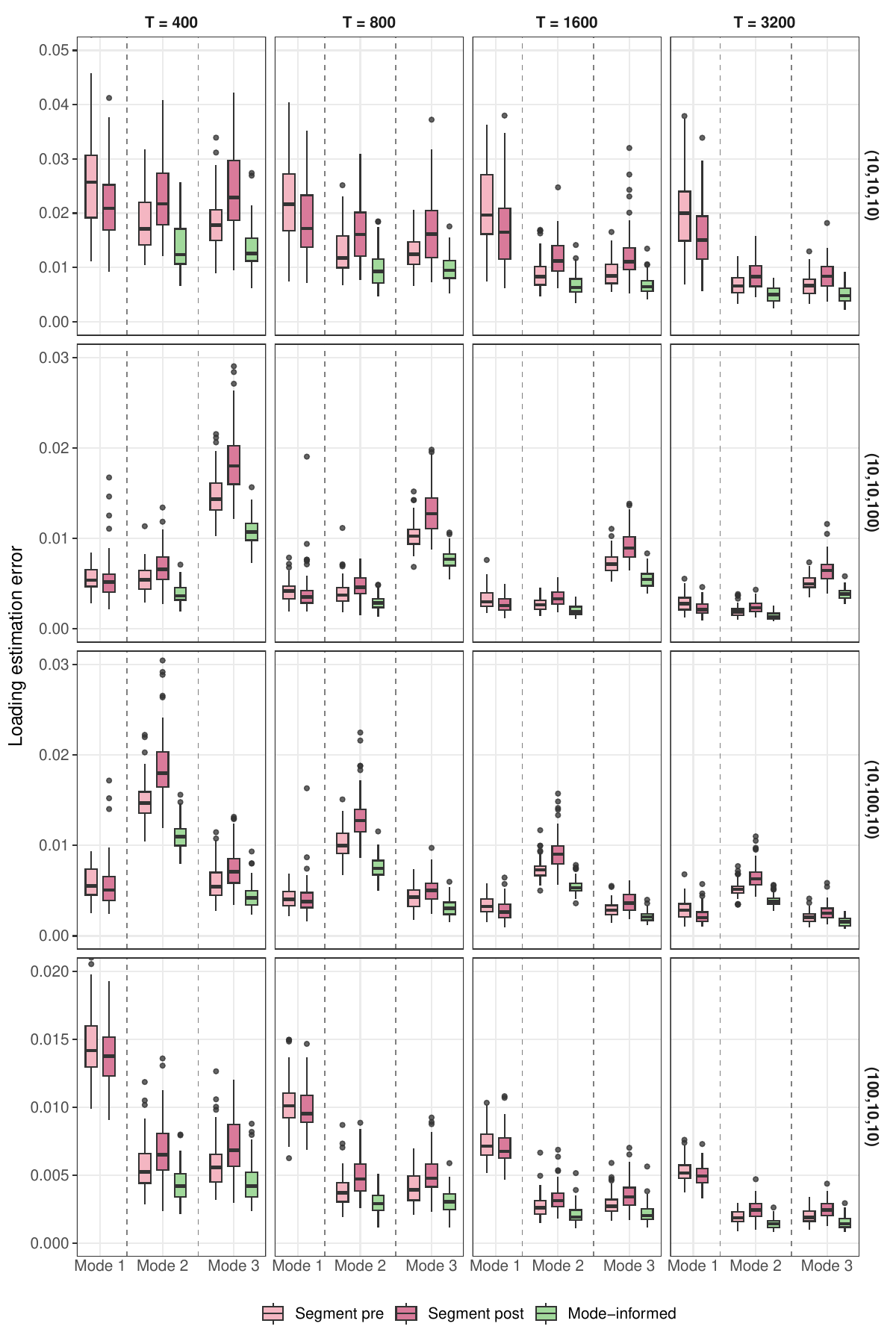}
\caption{\ref{s:three} Boxplots of mode-wise loading estimation errors with varying $T$ (left to right) and $(p_1,p_2,p_3)$ (top to bottom) over subset of realizations $n \in \wt{\cD}$. Segment pre (resp.\ Segment post) refers to segment-wise estimation errors using the data before (resp.\ after) the change point estimator $\wh\theta_1$ by \ref{M:one}.}
\label{fig: plot_reest_rnull_indep}
\end{figure}


\clearpage 

\subsection{Results for change point detection in the presence of missingness}\label{app: num_miss}

To investigate the method described in Remark~\ref{rem: miss} for handling missing observations in the data, we introduce missingness to $\{\cX_t\}_{t = 1}^T$ generated under~\ref{s:one} by setting an element $\cX_{i_1,\ldots,_K,t}$ to be missing if $t > \lfloor T/2 \rfloor$ and $i_k > \lfloor p_k/2 \rfloor$ for $k\in[K]$.
This is a structured missingness pattern explored in \cite{CenLam2025}.

Overall, we observe that the proposed modification to TFMseg handles missingness well; the detection of $\theta_1$ is little influenced by the missingness, while detecting $\theta_2$ and $\theta_3$ becomes more challenging as the missingness is concentrated in the second half of the data. It suffers more from the missingness in the presence of serial correlations, which is considered to be more challenging even in the absence of missingness.

\subsubsection{\ref{s:one} Single-mode change with equal-spaced change points}\label{app: num_miss_bal}

We consider when the change points are equally spaced with $\Theta = \big\{\lfloor 0.25T \rfloor, \lfloor 0.5T \rfloor, \lfloor 0.75T \rfloor\big\}$, see Tables~\ref{tab: miss_bal_rnull}--\ref{tab: miss_bal_rnull_indep} and Figures~\ref{fig: subplot_miss_rnull_method}--\ref{fig: subplot_miss_rnull_indep_method}; in the figures, we present the results obtained from the complete observations for comparison, which are taken from Appendix~\ref{app: num_det_bal}.

\begin{table}[h!t!b!]
\centering
\caption{\ref{s:one} Summary of change point estimators returned by TFMseg in the presence of missingness, when $\Theta  = \big\{\lfloor 0.25T \rfloor, \lfloor 0.5T \rfloor, \lfloor 0.75T \rfloor\big\}$ and $(\rho_f, \rho_e) = (0.7, 0.3)$, for varying $T$ and $(p_1,p_2,p_3)$, based on 100 realizations.}
\label{tab: miss_bal_rnull}
\centering
\resizebox{\ifdim\width>\linewidth\linewidth\else\width\fi}{!}{
\begin{tabular}[t]{cccccccccc}
\toprule
\multicolumn{2}{c}{ } & \multicolumn{5}{c}{$\wh{q}-q$} & \multicolumn{3}{c}{Accuracy} \\
\cmidrule(l{3pt}r{3pt}){3-7} \cmidrule(l{3pt}r{3pt}){8-10}
$T$ & $(p_1,p_2,p_3)$ & $\leq -2$ & $-1$ & $0$ & $1$ & $\geq 2$ & $j=1$ & $j=2$ & $j=3$\\
\cmidrule(lr){1-2} 
\cmidrule(lr){3-7} 
\cmidrule(lr){8-10} 
 & (10,10,10) & 0.31 & 0.47 & 0.22 & 0 & 0 & 0.77 & 0.32 & 0.38\\
 & (10,10,100) & 0.13 & 0.53 & 0.34 & 0 & 0 & 0.86 & 0.44 & 0.51\\
 & (10,20,40) & 0.13 & 0.38 & 0.48 & 0.01 & 0 & 0.89 & 0.52 & 0.58\\
\multirow{-4}{*}{\centering\arraybackslash 400} & (20,20,20) & 0.08 & 0.42 & 0.47 & 0.03 & 0 & 0.89 & 0.56 & 0.58\\
\cmidrule(lr){1-2} 
\cmidrule(lr){3-7} 
\cmidrule(lr){8-10} 
 & (10,10,10) & 0.11 & 0.42 & 0.44 & 0.03 & 0 & 0.81 & 0.49 & 0.60\\
 & (10,10,100) & 0.04 & 0.39 & 0.52 & 0.05 & 0 & 0.85 & 0.54 & 0.73\\
 & (10,20,40) & 0.03 & 0.18 & 0.73 & 0.06 & 0 & 0.90 & 0.78 & 0.68\\
\multirow{-4}{*}{\centering\arraybackslash 800} & (20,20,20) & 0.02 & 0.18 & 0.73 & 0.07 & 0 & 0.93 & 0.75 & 0.70\\
\cmidrule(lr){1-2} 
\cmidrule(lr){3-7} 
\cmidrule(lr){8-10} 
 & (10,10,10) & 0.02 & 0.35 & 0.61 & 0.02 & 0 & 0.88 & 0.49 & 0.68\\
 & (10,10,100) & 0 & 0.30 & 0.67 & 0.03 & 0 & 0.94 & 0.62 & 0.75\\
 & (10,20,40) & 0 & 0.06 & 0.86 & 0.08 & 0 & 0.96 & 0.79 & 0.76\\
\multirow{-4}{*}{\centering\arraybackslash 1600} & (20,20,20) & 0 & 0.05 & 0.88 & 0.07 & 0 & 0.97 & 0.85 & 0.82\\
\cmidrule(lr){1-2} 
\cmidrule(lr){3-7} 
\cmidrule(lr){8-10} 
 & (10,10,10) & 0.02 & 0.21 & 0.68 & 0.09 & 0 & 0.94 & 0.64 & 0.80\\
 & (10,10,100) & 0 & 0.15 & 0.77 & 0.08 & 0 & 0.98 & 0.72 & 0.86\\
 & (10,20,40) & 0 & 0.05 & 0.77 & 0.17 & 0.01 & 0.99 & 0.91 & 0.90\\
\multirow{-4}{*}{\centering\arraybackslash 3200} & (20,20,20) & 0 & 0.04 & 0.86 & 0.09 & 0.01 & 1 & 0.92 & 0.87\\
\bottomrule
\end{tabular}}
\end{table}

\begin{figure}[h!t!b!p!]
\centering
\includegraphics[width=0.68\linewidth]{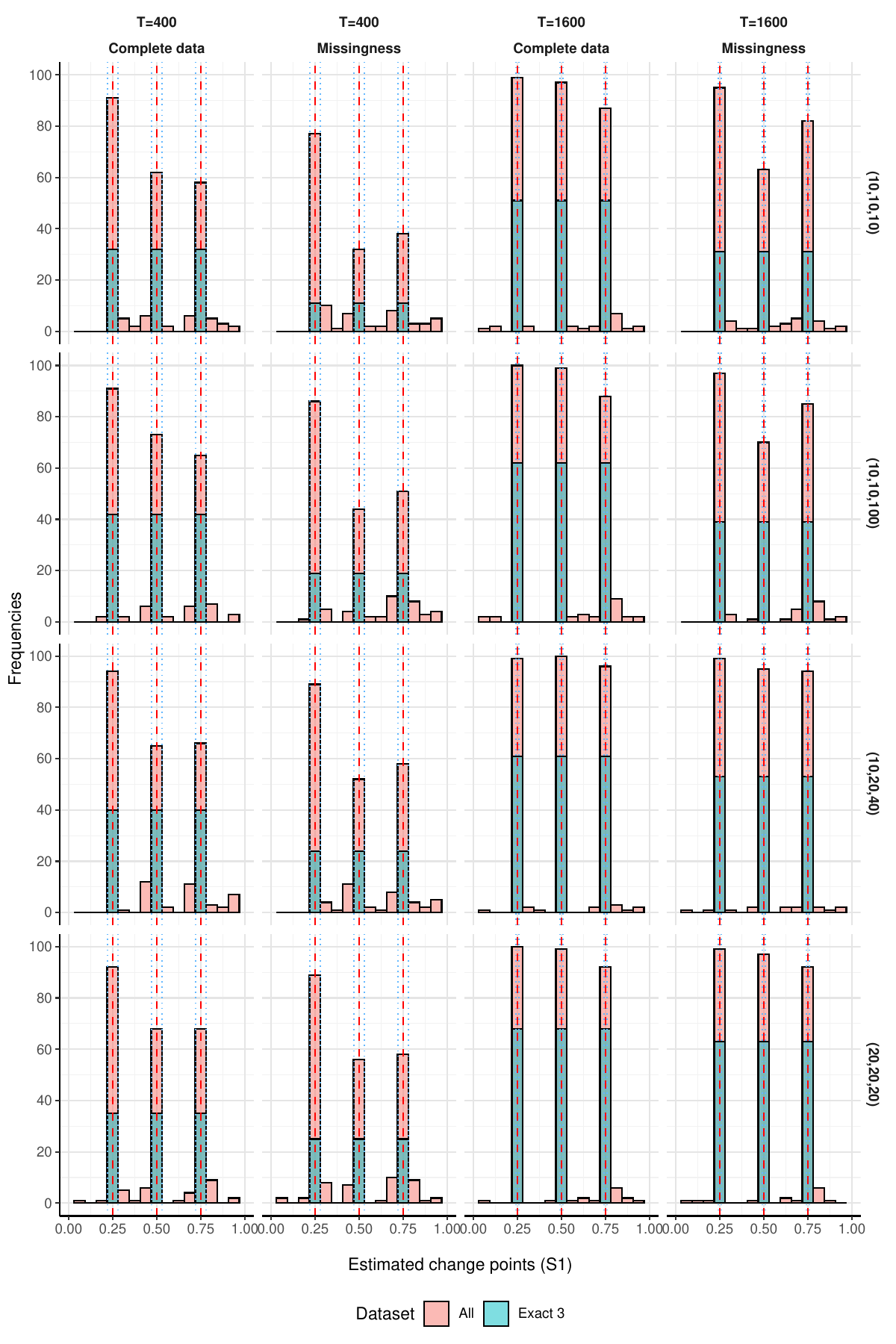}
\caption{\ref{s:one} Barplots of the scaled change point estimators $\{ \wh\theta^{(n)}_j/T, \, j \in [\wh q^{(n)}]\}$ returned by TFMseg for the complete and missing data, when $\Theta = \big\{\lfloor 0.25T \rfloor, \lfloor 0.5T \rfloor, \lfloor 0.75T \rfloor\big\}$, $(\rho_f, \rho_e) = (0.7, 0.3)$, $T \in \{400, 1600\}$ and varying $(p_1,p_2,p_3)$ (top to bottom) over $100$ realizations. The red bars give the total frequency of estimated change points for all $n \in [N]$, while the blue bars give the frequency from the subset of realizations $n \in \wh{\cD}$.}
\label{fig: subplot_miss_rnull_method}
\end{figure}

\begin{table}[h!t!b!p!]
\centering
\caption{\ref{s:one} Summary of change point estimators returned by TFMseg in the presence of missingness, when $\Theta  = \big\{\lfloor 0.25T \rfloor, \lfloor 0.5T \rfloor, \lfloor 0.75T \rfloor\big\}$ and $(\rho_f, \rho_e) = (0, 0)$, for varying $T$ and $(p_1,p_2,p_3)$, based on 100 realizations.}
\label{tab: miss_bal_rnull_indep}
\centering
\resizebox{\ifdim\width>\linewidth\linewidth\else\width\fi}{!}{
\begin{tabular}[t]{cccccccccc}
\toprule
\multicolumn{2}{c}{ } & \multicolumn{5}{c}{$\wh{q}-q$} & \multicolumn{3}{c}{Accuracy} \\
\cmidrule(l{3pt}r{3pt}){3-7} \cmidrule(l{3pt}r{3pt}){8-10}
$T$ & $(p_1,p_2,p_3)$ & $\leq -2$ & $-1$ & $0$ & $1$ & $\geq 2$ & $j=1$ & $j=2$ & $j=3$\\
\cmidrule(lr){1-2} 
\cmidrule(lr){3-7} 
\cmidrule(lr){8-10} 
 & (10,10,10) & 0.21 & 0.50 & 0.29 & 0 & 0 & 0.86 & 0.47 & 0.61\\
 & (10,10,100) & 0.08 & 0.44 & 0.48 & 0 & 0 & 0.96 & 0.59 & 0.76\\
 & (10,20,40) & 0.07 & 0.35 & 0.57 & 0.01 & 0 & 0.95 & 0.70 & 0.83\\
\multirow{-4}{*}{\centering\arraybackslash 400} & (20,20,20) & 0.02 & 0.41 & 0.56 & 0.01 & 0 & 0.99 & 0.70 & 0.77\\
\cmidrule(lr){1-2} 
\cmidrule(lr){3-7} 
\cmidrule(lr){8-10} 
 & (10,10,10) & 0.05 & 0.39 & 0.55 & 0.01 & 0 & 0.95 & 0.64 & 0.75\\
 & (10,10,100) & 0.03 & 0.34 & 0.63 & 0 & 0 & 0.97 & 0.68 & 0.86\\
 & (10,20,40) & 0.01 & 0.09 & 0.90 & 0 & 0 & 1 & 0.92 & 0.88\\
\multirow{-4}{*}{\centering\arraybackslash 800} & (20,20,20) & 0 & 0.08 & 0.90 & 0.02 & 0 & 1 & 0.95 & 0.90\\
\cmidrule(lr){1-2} 
\cmidrule(lr){3-7} 
\cmidrule(lr){8-10} 
 & (10,10,10) & 0.01 & 0.26 & 0.72 & 0.01 & 0 & 0.98 & 0.71 & 0.84\\
 & (10,10,100) & 0 & 0.23 & 0.77 & 0 & 0 & 0.97 & 0.74 & 0.90\\
 & (10,20,40) & 0 & 0.06 & 0.91 & 0.03 & 0 & 0.98 & 0.91 & 0.96\\
\multirow{-4}{*}{\centering\arraybackslash 1600} & (20,20,20) & 0 & 0.05 & 0.91 & 0.04 & 0 & 0.99 & 0.95 & 0.94\\
\cmidrule(lr){1-2} 
\cmidrule(lr){3-7} 
\cmidrule(lr){8-10} 
 & (10,10,10) & 0 & 0.21 & 0.76 & 0.03 & 0 & 1 & 0.75 & 0.87\\
 & (10,10,100) & 0 & 0.14 & 0.82 & 0.04 & 0 & 1 & 0.80 & 0.96\\
 & (10,20,40) & 0 & 0.03 & 0.93 & 0.04 & 0 & 1 & 0.94 & 0.97\\
\multirow{-4}{*}{\centering\arraybackslash 3200} & (20,20,20) & 0 & 0.01 & 0.98 & 0.01 & 0 & 1 & 0.98 & 0.97\\
\bottomrule
\end{tabular}}
\end{table}

\begin{figure}[h!t!b!p!]
\centering
\includegraphics[width=0.68\linewidth]{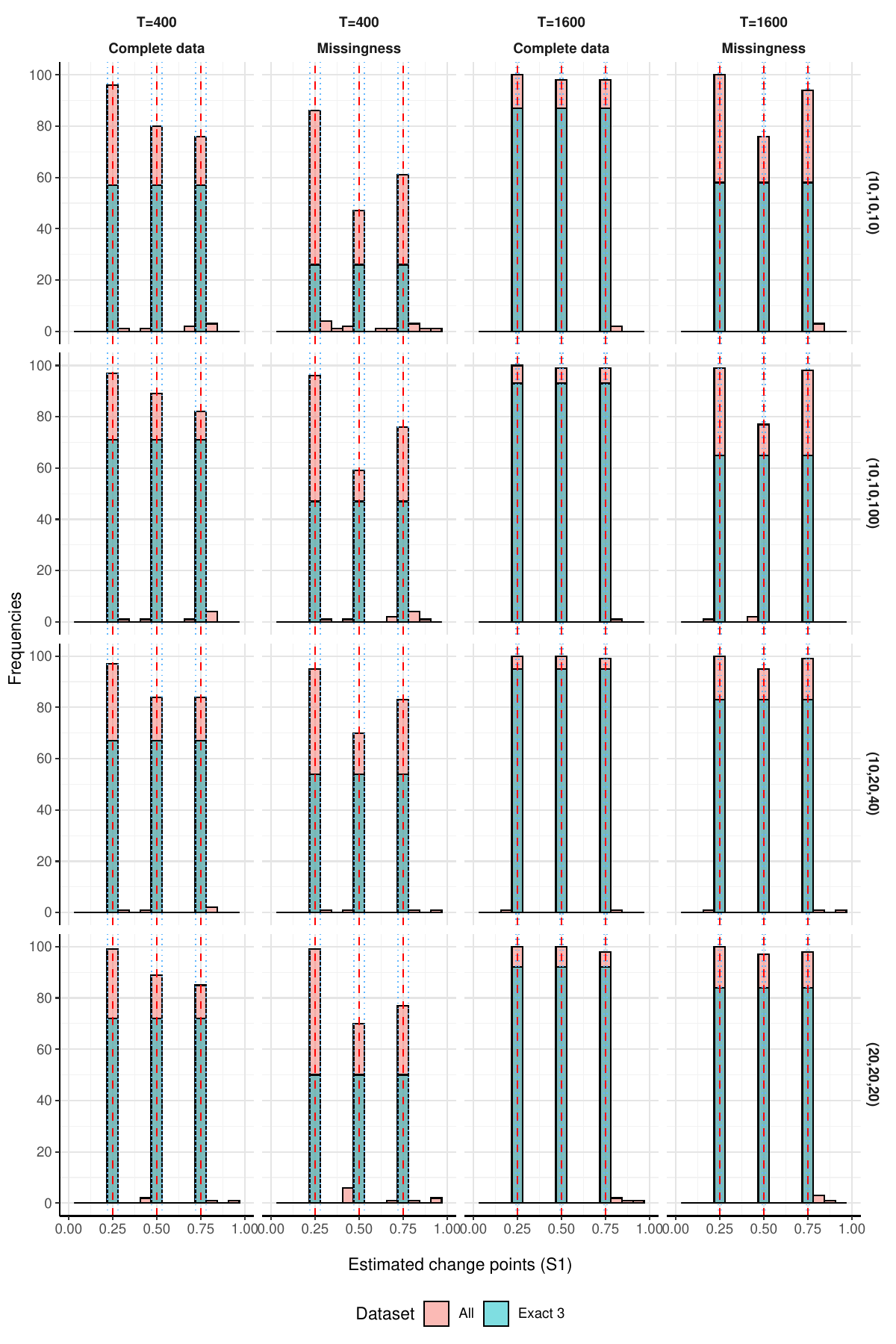}
\caption{\ref{s:one} Barplots of the scaled change point estimators $\{ \wh\theta^{(n)}_j/T, \, j \in [\wh q^{(n)}]\}$ returned by TFMseg for the complete and missing data, when $\Theta = \big\{\lfloor 0.25T \rfloor, \lfloor 0.5T \rfloor, \lfloor 0.75T \rfloor\big\}$, $(\rho_f, \rho_e) = (0, 0)$, $T \in \{400, 1600\}$ and varying $(p_1,p_2,p_3)$ (top to bottom) over $100$ realizations. The red bars give the total frequency of estimated change points for all $n \in [N]$, while the blue bars give the frequency from the subset of realizations $n \in \wh{\cD}$.}
\label{fig: subplot_miss_rnull_indep_method}
\end{figure}

\clearpage

\subsubsection{\ref{s:one} Single-mode change with unequally spaced change points}\label{app: num_miss_unbal}

We consider when the change points are unequally spaced with $\Theta = \big\{\lfloor 0.25T \rfloor, \lfloor 0.5T \rfloor, \lfloor 0.625T \rfloor\big\}$, see Tables~\ref{tab: miss_unbal_rnull}--\ref{tab: miss_unbal_rnull_indep} and Figures~\ref{fig: subplot_miss_unbal_rnull_method}--\ref{fig: subplot_miss_unbal_rnull_indep_method} ; in the figures, we present the results obtained from the complete observations for comparison, which are taken from Appendix~\ref{app: num_det_unbal}.

\begin{table}[h!t!b!]
\centering
\caption{\ref{s:one} Summary of change point estimators returned by TFMseg in the presence of missingness, when $\Theta  = \big\{\lfloor 0.25T \rfloor, \lfloor 0.5T \rfloor, \lfloor 0.625T \rfloor\big\}$ and $(\rho_f, \rho_e) = (0.7, 0.3)$, for varying $T$ and $(p_1,p_2,p_3)$, based on 100 realizations.}
\label{tab: miss_unbal_rnull}
\centering
\resizebox{\ifdim\width>\linewidth\linewidth\else\width\fi}{!}{
\begin{tabular}[t]{cccccccccc}
\toprule
\multicolumn{2}{c}{ } & \multicolumn{5}{c}{$\wh{q}-q$} & \multicolumn{3}{c}{Accuracy} \\
\cmidrule(l{3pt}r{3pt}){3-7} \cmidrule(l{3pt}r{3pt}){8-10}
$T$ & $(p_1,p_2,p_3)$ & $\leq -2$ & $-1$ & $0$ & $1$ & $\geq 2$ & $j=1$ & $j=2$ & $j=3$\\
\cmidrule(lr){1-2} 
\cmidrule(lr){3-7} 
\cmidrule(lr){8-10} 
 & (10,10,10) & 0.32 & 0.50 & 0.18 & 0 & 0 & 0.77 & 0.32 & 0.42\\
 & (10,10,100) & 0.17 & 0.57 & 0.26 & 0 & 0 & 0.86 & 0.35 & 0.50\\
 & (10,20,40) & 0.12 & 0.53 & 0.33 & 0.02 & 0 & 0.91 & 0.42 & 0.60\\
\multirow{-4}{*}{\centering\arraybackslash 400} & (20,20,20) & 0.12 & 0.57 & 0.29 & 0.02 & 0 & 0.89 & 0.43 & 0.55\\
\cmidrule(lr){1-2} 
\cmidrule(lr){3-7} 
\cmidrule(lr){8-10} 
 & (10,10,10) & 0.15 & 0.51 & 0.32 & 0.02 & 0 & 0.77 & 0.36 & 0.57\\
 & (10,10,100) & 0.03 & 0.52 & 0.39 & 0.05 & 0.01 & 0.84 & 0.49 & 0.61\\
 & (10,20,40) & 0.03 & 0.32 & 0.58 & 0.07 & 0 & 0.90 & 0.67 & 0.68\\
\multirow{-4}{*}{\centering\arraybackslash 800} & (20,20,20) & 0.01 & 0.50 & 0.42 & 0.07 & 0 & 0.91 & 0.55 & 0.61\\
\cmidrule(lr){1-2} 
\cmidrule(lr){3-7} 
\cmidrule(lr){8-10} 
 & (10,10,10) & 0.03 & 0.45 & 0.48 & 0.04 & 0 & 0.88 & 0.46 & 0.66\\
 & (10,10,100) & 0.01 & 0.36 & 0.57 & 0.06 & 0 & 0.96 & 0.59 & 0.75\\
 & (10,20,40) & 0 & 0.12 & 0.78 & 0.09 & 0.01 & 0.95 & 0.75 & 0.78\\
\multirow{-4}{*}{\centering\arraybackslash 1600} & (20,20,20) & 0 & 0.15 & 0.73 & 0.12 & 0 & 0.95 & 0.79 & 0.74\\
\cmidrule(lr){1-2} 
\cmidrule(lr){3-7} 
\cmidrule(lr){8-10} 
 & (10,10,10) & 0.04 & 0.22 & 0.62 & 0.12 & 0 & 0.94 & 0.62 & 0.77\\
 & (10,10,100) & 0 & 0.25 & 0.68 & 0.07 & 0 & 0.97 & 0.72 & 0.89\\
 & (10,20,40) & 0 & 0.07 & 0.76 & 0.15 & 0.02 & 0.99 & 0.91 & 0.91\\
\multirow{-4}{*}{\centering\arraybackslash 3200} & (20,20,20) & 0 & 0.04 & 0.76 & 0.17 & 0.03 & 1 & 0.91 & 0.89\\
\bottomrule
\end{tabular}}
\end{table}

\begin{figure}[h!t!b!p!]
\centering
\includegraphics[width=0.68\linewidth]{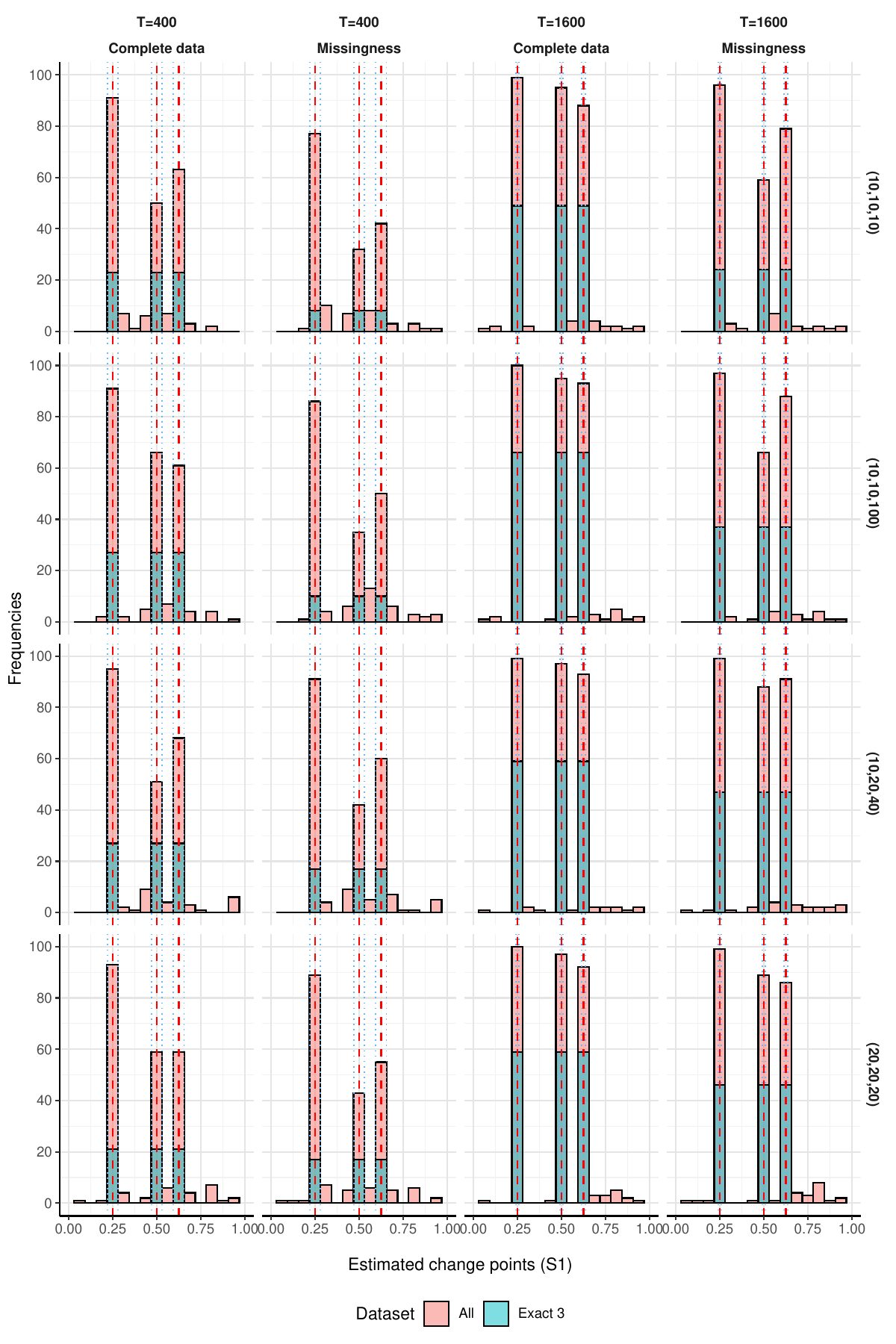}
\caption{\ref{s:one} Barplots of the scaled change point estimators $\{ \wh\theta^{(n)}_j/T, \, j \in [\wh q^{(n)}]\}$ returned by TFMseg for the complete and missing data, when $\Theta = \big\{\lfloor 0.25T \rfloor, \lfloor 0.5T \rfloor, \lfloor 0.625T \rfloor\big\}$, $(\rho_f, \rho_e) = (0.7, 0.3)$, $T \in \{400, 1600\}$ and varying $(p_1,p_2,p_3)$ (top to bottom) over $100$ realizations. The red bars give the total frequency of estimated change points for all $n \in [N]$, while the blue bars give the frequency from the subset of realizations $n \in \wh{\cD}$.}
\label{fig: subplot_miss_unbal_rnull_method}
\end{figure}

\clearpage

\begin{table}[h!t!b!]
\centering
\caption{\ref{s:one} Summary of change point estimators returned by TFMseg in the presence of missingness, when $\Theta  = \big\{\lfloor 0.25T \rfloor, \lfloor 0.5T \rfloor, \lfloor 0.625T \rfloor\big\}$ and $(\rho_f, \rho_e) = (0, 0)$, for varying $T$ and $(p_1,p_2,p_3)$, based on 100 realizations.}
\label{tab: miss_unbal_rnull_indep}
\centering
\resizebox{\ifdim\width>\linewidth\linewidth\else\width\fi}{!}{
\begin{tabular}[t]{cccccccccc}
\toprule
\multicolumn{2}{c}{ } & \multicolumn{5}{c}{$\wh{q}-q$} & \multicolumn{3}{c}{Accuracy} \\
\cmidrule(l{3pt}r{3pt}){3-7} \cmidrule(l{3pt}r{3pt}){8-10}
$T$ & $(p_1,p_2,p_3)$ & $\leq -2$ & $-1$ & $0$ & $1$ & $\geq 2$ & $j=1$ & $j=2$ & $j=3$\\
\cmidrule(lr){1-2} 
\cmidrule(lr){3-7} 
\cmidrule(lr){8-10} 
 & (10,10,10) & 0.23 & 0.56 & 0.21 & 0 & 0 & 0.85 & 0.38 & 0.59\\
 & (10,10,100) & 0.10 & 0.54 & 0.36 & 0 & 0 & 0.96 & 0.49 & 0.74\\
 & (10,20,40) & 0.07 & 0.52 & 0.39 & 0.02 & 0 & 0.96 & 0.54 & 0.77\\
\multirow{-4}{*}{\centering\arraybackslash 400} & (20,20,20) & 0.03 & 0.60 & 0.35 & 0.02 & 0 & 0.99 & 0.60 & 0.70\\
\cmidrule(lr){1-2} 
\cmidrule(lr){3-7} 
\cmidrule(lr){8-10} 
 & (10,10,10) & 0.11 & 0.49 & 0.40 & 0 & 0 & 0.94 & 0.54 & 0.67\\
 & (10,10,100) & 0.02 & 0.46 & 0.51 & 0.01 & 0 & 0.96 & 0.66 & 0.84\\
 & (10,20,40) & 0.02 & 0.31 & 0.67 & 0 & 0 & 1 & 0.73 & 0.90\\
\multirow{-4}{*}{\centering\arraybackslash 800} & (20,20,20) & 0 & 0.32 & 0.67 & 0.01 & 0 & 1 & 0.77 & 0.86\\
\cmidrule(lr){1-2} 
\cmidrule(lr){3-7} 
\cmidrule(lr){8-10} 
 & (10,10,10) & 0.04 & 0.32 & 0.62 & 0.02 & 0 & 0.96 & 0.65 & 0.83\\
 & (10,10,100) & 0 & 0.33 & 0.66 & 0.01 & 0 & 0.97 & 0.70 & 0.90\\
 & (10,20,40) & 0 & 0.09 & 0.89 & 0.02 & 0 & 0.97 & 0.89 & 0.96\\
\multirow{-4}{*}{\centering\arraybackslash 1600} & (20,20,20) & 0 & 0.11 & 0.85 & 0.04 & 0 & 0.99 & 0.92 & 0.92\\
\cmidrule(lr){1-2} 
\cmidrule(lr){3-7} 
\cmidrule(lr){8-10} 
 & (10,10,10) & 0.01 & 0.26 & 0.72 & 0.01 & 0 & 1 & 0.72 & 0.85\\
 & (10,10,100) & 0 & 0.23 & 0.75 & 0.02 & 0 & 1 & 0.75 & 0.98\\
 & (10,20,40) & 0 & 0.04 & 0.91 & 0.05 & 0 & 1 & 0.94 & 0.95\\
\multirow{-4}{*}{\centering\arraybackslash 3200} & (20,20,20) & 0 & 0.07 & 0.91 & 0.02 & 0 & 1 & 0.96 & 0.94\\
\bottomrule
\end{tabular}}
\end{table}

\begin{figure}[h!t!b!p!]
\centering
\includegraphics[width=0.68\linewidth]{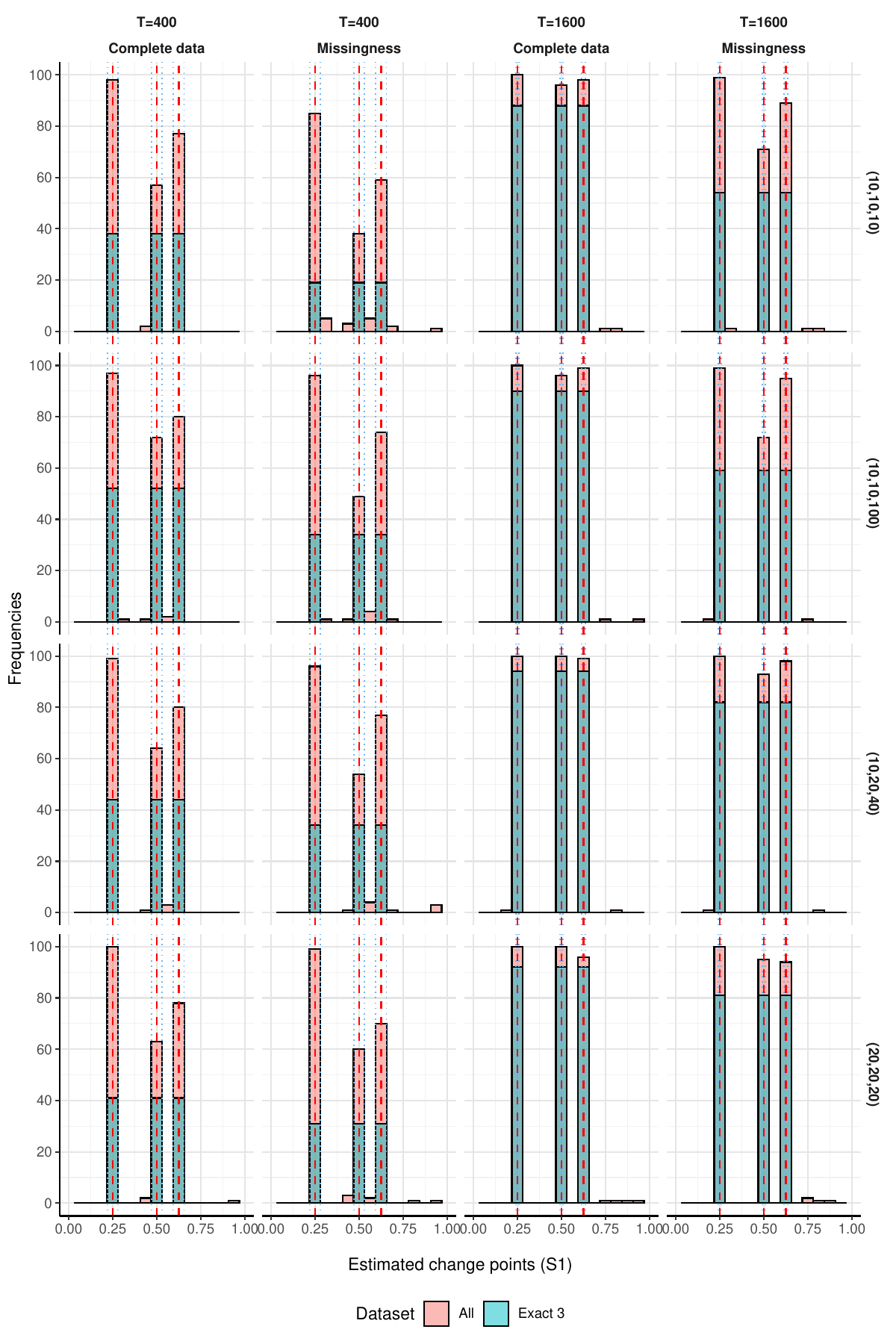}
\caption{\ref{s:one} Barplots of the scaled change point estimators $\{ \wh\theta^{(n)}_j/T, \, j \in [\wh q^{(n)}]\}$ returned by TFMseg for the complete and missing data, when $\Theta = \big\{\lfloor 0.25T \rfloor, \lfloor 0.5T \rfloor, \lfloor 0.625T \rfloor\big\}$, $(\rho_f, \rho_e) = (0, 0)$, $T \in \{400, 1600\}$ and varying $(p_1,p_2,p_3)$ (top to bottom) over $100$ realizations. The red bars give the total frequency of estimated change points for all $n \in [N]$, while the blue bars give the frequency from the subset of realizations $n \in \wh{\cD}$.}
\label{fig: subplot_miss_unbal_rnull_indep_method}
\end{figure}

\clearpage

\section{Additional information for real data applications}
\label{app: real_fama}

\subsection{New York City Yellow Taxi data}\label{app: real_nyc}

For the sensitivity analysis in Section~\ref{sec: real_nyc}, we consider 
\begin{align}
\label{eq:real:rank:taxi}
(r_1,r_2,r_3) \in &\l\{(1,1,1), (2,2,2),(3,3,3), (3,2,2), (4,2,2), (5,2,2), \r. \nonumber\\
&\quad \l. (4,1,2), (4,3,2), (4,2,1), (4,2,3), (2,4,2), (2,5,2), (2,6,2)\r\},
\end{align}
each element of which is obtained according to one of the following rules: 
(i)~We set $r_1 = r_2 = r_3$ with $r_1\in[3]$; 
(ii)~Removing the common component driven by the core factor estimated with $(\wh{r}_1, \wh{r}_2, \wh{r}_3) = (1,1,1)$, we re-estimate the factor numbers from the residuals (see the two-step estimation in Section~4 of \citeauthor{LamYao2012}, \citeyear{LamYao2012}) as $(\wh{r}'_1, \wh{r}'_2, \wh{r}'_3) = (1,4,1)$, which hence gives the specification $(r_1, r_2, r_3) = (2,5,2)$; (iii)~We additionally consider the configurations adjacent to those obtained from (i) and (ii).




\subsection{Fama--French portfolio returns}
\label{app:fama}

Figure~\ref{fig:va_fama_rank_sens} reports the sensitivity analysis similarly performed as that reported in Figure~\ref{fig:eq_fama_rank_sens}, with the value-weighted returns data.

\begin{figure}[h!t!b!]
\centering
\includegraphics[width=.75\linewidth]{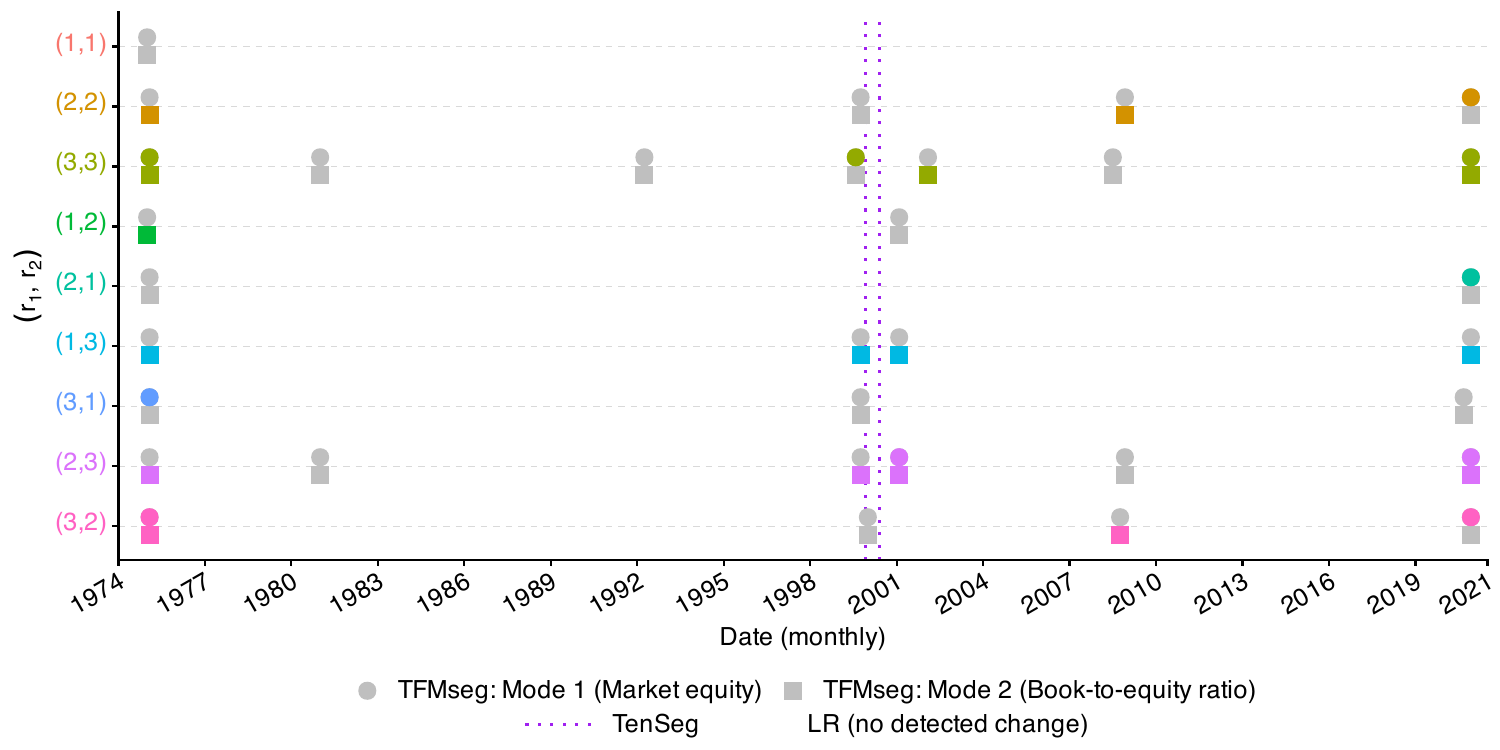}
\caption{Fama--French portfolio returns (value-weighted):
Change points estimated with varying factor number specifications. 
Each pair of circle and square marks the locations of the change points ($x$-axis) detected with given $(r_1, r_2)$ ($y$-axis) and, when the plotting character is highlighted, it indicates that the change is mode-identifiable for the corresponding mode while grey plotting characters indicate mode-unidentifiability for the particular modes.
We also mark the change point estimators returned by TenSeg \citep{anastasiou2025tensor} with the vertical dotted lines; LR does \citep{Baietal2024} not detect any change for any $r = r_1r_2$.}
\label{fig:va_fama_rank_sens}
\end{figure}

\end{document}